%% file: stable-stems.tex
\documentclass{amsbook}

\usepackage{amsrefs,amsmath,amsthm,amssymb}            
\usepackage{enumerate,lscape,calc}
\usepackage[all]{xy}   
\usepackage{longtable,booktabs}

\newtheorem{thm}{Theorem}[section]
\newtheorem{prop}[thm]{Proposition}

\newtheorem{lemma}[thm]{Lemma}

\theoremstyle{definition}
\newtheorem{defn}[thm]{Definition}
\newtheorem{ex}[thm]{Example}

\theoremstyle{remark}
\newtheorem{remark}[thm]{Remark}

\numberwithin{section}{chapter}
\numberwithin{equation}{chapter}

\newcommand{\M}{\mathbb{M}}
\newcommand{\F}{\mathbb{F}}
\newcommand{\C}{\mathbb{C}}
\newcommand{\R}{\mathbb{R}}
\newcommand{\A}{\mathbb{A}}
\newcommand{\Z}{\mathbb{Z}}

\newcommand{\map}{\rightarrow}
\newcommand{\tmf}{\textit{tmf}}
\newcommand{\ol}{\overline}
\newcommand{\sigmabar}{\ol{\sigma}}
\newcommand{\kappabar}{\ol{\kappa}}
\newcommand{\iso}{\cong}
\newcommand{\cl}{\mathrm{cl}}

\newcommand{\clE}[2][{}]{E_{#2}^{#1}(S^0; BP)}
\newcommand{\motE}[2][{}]{E_{#2}^{#1}(S^{0,0}; BPL)}
\newcommand{\motECtau}[2][{}]{E_{#2}^{#1}(C\tau; BPL)}
\newcommand{\olmotE}[2][{}]{\ol{E}_{#2}^{#1}(S^{0,0}; BPL)}

\DeclareMathOperator{\Hom}{Hom}
\DeclareMathOperator{\Ext}{Ext}
\DeclareMathOperator{\Gr}{Gr}
\DeclareMathOperator{\Sq}{Sq}
\DeclareMathOperator{\Ch}{Ch}
\DeclareMathOperator{\coker}{coker}

\makeindex

\begin{document}

\frontmatter

\title{Stable Stems}

\author{Daniel C. Isaksen}
\address{Department of Mathematics\\
Wayne State University\\
Detroit, MI 48202, USA}
\email{isaksen@wayne.edu}
\thanks{The author was supported by NSF grant DMS-1202213.}



\date{}

\begin{abstract}
We present a detailed analysis
of 2-complete stable homotopy groups, both in the classical context
and in the motivic context over $\C$.
We use the motivic May spectral sequence to
compute the cohomology of the motivic Steenrod algebra over $\C$
through the 70-stem.  
We then use the motivic Adams spectral sequence to obtain motivic
stable homotopy groups through the 59-stem.
In addition to finding all Adams differentials in this range,
we also resolve all hidden extensions by $2$, $\eta$, and $\nu$,
except for a few carefully enumerated exceptions that remain unknown.
The analogous classical stable homotopy groups are easy consequences.

We also compute the motivic stable homotopy groups of the cofiber of
the motivic element $\tau$.  This computation is essential
for resolving hidden extensions in the Adams spectral sequence.
We show that the homotopy groups of the cofiber of $\tau$ are
the same as the $E_2$-page of the classical Adams-Novikov spectral sequence.
This allows us to compute the classical Adams-Novikov 
spectral sequence, including differentials and hidden extensions,
in a larger range than was previously known.
\end{abstract}

\maketitle

%

\setcounter{page}{4}

\tableofcontents

\listoftables


\mainmatter

\index{Adams spectral sequence!cofiber of tau@cofiber of $\tau$|see{cofiber of $\tau$}}
\index{Adams-Novikov spectral sequence!cofiber of tau@cofiber of $\tau$|see{cofiber of $\tau$}}
\index{algebraic Steenrod operation|see{Steenrod operation}}
\index{ambiguous generator!Ext@$\Ext$|see{$\Ext$}}
\index{ambiguous generator!cofiber of tau@cofiber of $\tau$|see{cofiber of $\tau$}}
\index{bottom cell|see{cofiber of $\tau$}}
\index{boundary!Adams-Novikov spectral sequence|see{Adams-Novikov spectral sequence}}
\index{c0@$c_0$!hidden extension|see{May spectral sequence}}
\index{category!model|see{model category}}
\index{cell!bottom|see{cofiber of $\tau$}}
\index{cell!top|see{cofiber of $\tau$}}
\index{cellular motivic spectrum|see{motivic spectrum}}
\index{chart!Adams|see{Adams chart}}
\index{chart!Adams-Novikov|see{Adams-Novikov chart}}
\index{chart!Ext@$\Ext$|see{$\Ext$ chart}}
\index{Chow degree|see{degree}}
\index{classical Adams spectral sequence|see{Adams spectral sequence}}
\index{classical Adams-Novikov spectral sequence|see{Adams-Novikov spectral sequence}}
\index{classical stable stem|see{stable stem}}
\index{classical Steenrod algebra|see{Steenrod algebra}}
\index{closed model category|see{model category}}
\index{cobordism!algebraic|see{algebraic cobordism}}
\index{cofiber!homotopy group|see{stable stem}}
\index{cofiber of tau@cofiber of $\tau$!Adams chart|see{Adams chart}}
\index{cohomology!of a point|see{motivic cohomology of a point}}
\index{cohomology!of the Steenrod algebra|see{Steenrod algebra}}
\index{compound hidden extension|see{hidden extension}}
\index{convergence!Adams spectral sequence|see{Adams spectral sequence}}
\index{crossing hidden extension|see{hidden extension}}
\index{computation!machine|see{machine computation}}
\index{cover|see{topology}}
\index{crossing differential!Adams|see{Adams spectral sequence}}
\index{crossing differential!May|see{May spectral sequence}}
\index{class!non-permanent!Adams-Novikov spectral sequence|see{Adams-Novikov spectral sequence}}
\index{differential!Adams|see{Adams spectral sequence}}
\index{differential!Adams-Novikov|see{Adams-Novikov spectral sequence}}
\index{differential!cofiber of tau@cofiber of $\tau$|see{cofiber of $\tau$}}
\index{differential!May|see{May spectral sequence}}
\index{epsilon@$\epsilon$!hidden extension|see{Adams spectral sequence}}
\index{Er-page@$E_r$-page!Adams|see{Adams spectral sequence}}
\index{Er-page@$E_r$-page!Adams-Novikov|see{Adams-Novikov spectral sequence}}
\index{Er-page@$E_r$-page!May|see{May spectral sequence}}
\index{eta@$\eta$!hidden extension!Adams spectral sequence|see{Adams spectral sequence}}
\index{eta@$\eta$!hidden extension!Adams-Novikov spectral sequence|see{Adams-Novikov spectral sequence}}
\index{eta@$\eta$!hidden extension!cofiber of tau@cofiber of $\tau$|see{cofiber of $\tau$}}
\index{eta4@$\eta_4$!hidden extension!Adams spectral sequence|see{Adams spectral sequence}}
\index{extension!hidden|see{hidden extension}}
\index{four!hidden extension|see{Adams spectral sequence}}
\index{generator!Ext@$\Ext$|see{$\Ext$}}
\index{generator!cofiber of tau@cofiber of $\tau$|see{cofiber of $\tau$}}
\index{grading|see{degree}}
\index{h0@$h_0$!hidden extension!cofiber of tau@cofiber of $\tau$|see{cofiber of $\tau$}}
\index{h0@$h_0$!hidden extension!May spectral sequence|see{May spectral sequence}}
\index{h1@$h_1$!hidden extension!cofiber of tau@cofiber of $\tau$|see{cofiber of $\tau$}}
\index{h1@$h_1$!hidden extension!May spectral sequence|see{May spectral sequence}}
\index{h2@$h_2$!hidden extension!cofiber of tau@cofiber of $\tau$|see{cofiber of $\tau$}}
\index{h2@$h_2$!hidden extension!May spectral sequence|see{May spectral sequence}}
\index{h43@$h_4^3$|see{$h_3^2 h_5$}}
\index{hidden extension!Adams|see{Adams spectral sequence}}
\index{hidden extension!Adams-Novikov|see{Adams-Novikov spectral sequence}}
\index{hidden extension!cofiber|see{cofiber}}
\index{hidden extension!cofiber of tau@cofiber of $\tau$|see{cofiber of $\tau$}}
\index{hidden extension!May|see{May spectral sequence}}
\index{higher Chow group|see{Chow group}}
\index{higher differential!Adams spectral sequence|see{Adams spectral sequence}}
\index{homotopy group|see{stable stem}}
\index{homotopy theory!motivic|see{motivic homotopy theory}}
\index{image of J@image of $J$|see{$J$}}
\index{inclusion of the bottom cell|see{cofiber of $\tau$}}
\index{indeterminacy!Massey product|see{Massey product}}
\index{indeterminacy!Toda bracket|see{Toda bracket}}
\index{K-theory@$K$-theory!algebraic|see{algebraic $K$-theory}}
\index{kappa@$\kappa$!hidden extension!Adams spectral sequence|see{Adams spectral sequence}}
\index{kappabar@$\kappabar$!hidden extension!Adams spectral sequence|see{Adams spectral sequence}}
\index{localization!tau@$\tau$|see{$\tau$}}
\index{localization!h1@$h_1$|see{$h_1$}}
\index{matric Massey product|see{Massey product}}
\index{May's Convergence Theorem|see{Convergence Theorem}}
\index{Moss's Convergence Theorem|see{Convergence Theorem}}
\index{Nisnevich topology|see{topology}}
\index{non-permanent class!Adams-Novikov spectral sequence|see{Adams-Novikov spectral sequence}}
\index{nu@$\nu$!cofiber|see{cofiber of $\nu$}}
\index{nu@$\nu$!hidden extension!Adams spectral sequence|see{Adams spectral sequence}}
\index{nu@$\nu$!hidden extension!Adams-Novikov spectral sequence|see{Adams-Novikov spectral sequence}}
\index{nu@$\nu$!hidden extension!cofiber of tau@cofiber of $\tau$|see{cofiber of $\tau$}}
\index{nu4@$\nu_4$!hidden extension!Adams spectral sequence|see{Adams spectral sequence}}
\index{odd differential!May|see{May spectral sequence}}
\index{odd primary Steenrod algebra|see{Steenrod algebra}}
\index{Ph1@$P h_1$!hidden extension!cofiber of tau@cofiber of $\tau$|see{cofiber of $\tau$}}
\index{Ph1@$P h_1$!hidden extension!May spectral sequence|see{May spectral sequence}}
\index{powers of g@powers of $g$|see{$g$}}
\index{projection to the top cell|see{cofiber of $\tau$}}
\index{R-motivic homotopy theory@$\R$-motivic homotopy theory|see{motivic homotopy theory}}
\index{resolution!cobar|see{cobar}}
\index{shuffle!Massey product|see{Massey product}}
\index{sigma@$\sigma$!hidden extension!Adams spectral sequence|see{Adams spectral sequence}}
\index{simplicial presheaf|see{presheaf}}
\index{site|see{topology}}
\index{spectral sequence!Adams|see{Adams spectral sequence}}
\index{spectral sequence!Adams-Novikov|see{Adams-Novikov spectral sequence}}
\index{spectral sequence!May|see{May spectral sequence}}
\index{spectral sequence!rho-Bockstein@$\rho$-Bockstein|see{$\rho$-Bockstein spectral sequence}}
\index{spectral sequence!tau-Bockstein@$\tau$-Bockstein|see{$\tau$-Bockstein spectral sequence}}
\index{spectrum!cellular motivic|see{motivic spectrum}}
\index{spectrum!Eilenberg-Mac Lane|see{Eilenberg-Mac Lane spectrum}}
\index{sphere!unstable motivic|see{motivic sphere}}
\index{squaring operation|see{Steenrod operation}}
\index{stable motivic homotopy theory|see{motivic homotopy theory}}
\index{Steenrod algebra!dual|see{Steenrod algebra}}
\index{stem|see{stable stem}}
\index{strictly defined Massey product|see{Massey product}}
\index{tau@$\tau$!cofiber|see{cofiber of $\tau$}}
\index{tau@$\tau$!hidden extension!Adams spectral sequence|see{Adams spectral sequence}}
\index{tentative hidden extension|see{Adams spectral sequence}}
\index{tmf@$\tmf$|see{topological modular forms}}
\index{top cell|see{cofiber of $\tau$}}
\index{torsion!tau@$\tau$|see{$\tau$}}
\index{two!cofiber|see{cofiber of two}}
\index{two!hidden extension!Adams spectral sequence|see{Adams spectral sequence}}
\index{two!hidden extension!Adams-Novikov spectral sequence|see{Adams-Novikov spectral sequence}}
\index{two!hidden extension!cofiber of tau@cofiber of $\tau$|see{cofiber of $\tau$}}
\index{Zariski topology|see{topology}}
\index{unstable motivic homotopy theory|see{motivic homotopy theory}}
\index{unstable motivic sphere|see{motivic sphere}}

\include{stable-stems-intro}

\include{stable-stems-May}

\include{stable-stems-Adams-diff}

\include{stable-stems-Adams-hidden}

\include{stable-stems-cofiber-tau}

\include{stable-stems-ANSS}

\include{stable-stems-tables}


\backmatter

\printindex

\bibliographystyle{amsalpha}


\end{document}

%% file: stable-stems-intro.tex
\chapter{Introduction}
\label{ch:intro}

One of the fundamental problems of stable homotopy theory is to compute
the stable homotopy groups of the sphere spectrum.  One reason for computing
these groups is that maps between spheres control the construction of finite
cell complexes.

After choosing a prime $p$ and focusing on the $p$-complete stable homotopy
groups instead of the integral homotopy groups,
the Adams spectral sequence 
and the Adams-Novikov spectral sequence have 
proven to be the most effective tools for carrying out such computations.  

At odd primes, the Adams-Novikov spectral sequence has clear computational
advantages over the Adams spectral sequence.  
(Nevertheless, 
the conventional wisdom, derived from 
\index{Mahowald, Mark}
Mark Mahowald,
is that one should compute with both
spectral sequences because they emphasize distinct aspects
of the same calculation.)

Computations at the prime $2$ are generally more difficult than
computations at odd primes.
In this case, 
the Adams spectral sequence and 
the Adams-Novikov spectral sequence seem to be
of equal complexity.
The purpose of this manuscript
is to thoroughly explore the Adams spectral sequence at $2$
in both the classical and motivic contexts.

Motivic techniques are essential to our analysis.
Working motivically instead of classically has both advantages and disadvantages.
The main disadvantage is that the computation is larger and proportionally more
difficult.  On the other hand, there are several advantages.
First, the presence of more non-zero classes allows the detection of otherwise
elusive phenomena.  
Second, the additional motivic weight grading can easily eliminate possibilities
that appear plausible from a classical perspective.

The original motivation for this work was 
to provide input to the
\index{rho-Bockstein spectral sequence@$\rho$-Bockstein spectral sequence}
$\rho$-Bockstein spectral sequence for computing the 
cohomology of the motivic Steenrod algebra over $\R$.
\index{motivic homotopy theory!over R@over $\R$}
The analysis of the $\rho$-Bockstein spectral sequence,
and the further analysis of the motivic Adams spectral sequence
over $\R$, will appear in future work.

This manuscript is a natural sequel to \cite{DI10}, where
the first computational properties of the motivic May spectral sequence, 
as well as of the motivic Adams spectral sequence, were established.

\section{The Adams spectral sequence program}
\label{subsctn:Adams-program}
The Adams spectral sequence
\index{Adams spectral sequence}
starts with the cohomology
of the Steenrod algebra $A$, i.e., $\Ext_A(\F_2, \F_2)$.  
There are two ways of approaching this algebraic object.
\index{machine computation}
First, one can compute by machine.  
This has been carried out to over 200 stems \cite{Bruner97}
\cite{Nassau}.  Machines can also compute the higher structure of products and
Massey products.

\index{May spectral sequence}
The second approach is to compute by hand with the May spectral sequence.
This will be carried out to 70 stems in Chapter \ref{ch:May}.
See also \cite{Tangora70a} for the classical case.
See \cite{Isaksen14a} for a detailed $\Ext$ chart through the 70-stem.

The $E_\infty$-page of the May spectral sequence is the graded object
associated to a filtration on $\Ext_A(\F_2, \F_2)$, which can hide some
of the multiplicative structure. 
\index{May spectral sequence!hidden extension}
One can resolve these hidden multiplicative extensions
with indirect arguments involving higher structure such as Massey products
or algebraic Steenrod operations in the sense of \cite{May70}.
\index{Massey product}
\index{Steenrod operation!algebraic}
A critical ingredient here is May's Convergence Theorem \cite{May69}*{Theorem 4.1},
\index{Convergence Theorem!May}
which allows the computation of Massey products in $\Ext_A(\F_2,\F_2)$
via the differentials in the May spectral sequence.

The cohomology of the Steenrod algebra is the $E_2$-page of the Adams
spectral sequence.  The next step is to compute the Adams differentials.
\index{Adams spectral sequence!differential}
This will be carried out in Chapter \ref{ch:Adams-diff}.
Techniques for establishing differentials include:
\begin{enumerate}
\item
Use knowledge of the image of $J$ \cite{Adams66} to deduce differentials.
\index{J@$J$!image of}
\item
Compare to the completely understood
Adams spectral sequence for the topological modular forms
spectrum $\tmf$ \cite{Henriques07}.
\index{topological modular forms}
\item
Use the relationship between algebraic Steenrod operations and
Adams differentials \cite{BMMS86}*{VI.1}.
\index{Steenrod operation!algebraic}
\item
Exploit Toda brackets to deduce relations in the stable homotopy ring,
which then imply Adams differentials.
\index{Toda bracket}
\end{enumerate}
We have assembled all previously
published results about the Adams differentials in Table
\ref{tab:diff-refs}.  

The $E_\infty$-page of the
Adams spectral sequence is the graded object associated to a filtration
on the stable homotopy groups, which can hide some of the
multiplicative structure.
The final step is to resolve these hidden multiplicative extensions.
\index{Adams spectral sequence!hidden extension}
This will be carried out in Chapter \ref{ch:Adams-hidden}.
Analogously to the extensions that are hidden in the May spectral sequence, 
this generally involves indirect arguments with Toda brackets.
\index{Toda bracket}
We have assembled previously published results
about these hidden extensions in Table \ref{tab:extn-refs}.

The detailed analysis of the Adams spectral sequence requires substantial
technical work with Toda brackets.
\index{Toda bracket}
A critical ingredient for computing Toda brackets is Moss's Convergence Theorem \cite{Moss70},
which allows the computation of Toda brackets 
via the Adams differentials.
\index{Convergence Theorem!Moss}
We
remind the reader to be cautious about indeterminacies in Massey products
and Toda brackets.
\index{Massey product!indeterminacy}
\index{Toda bracket!indeterminacy}

\section{Motivic homotopy theory}
The formal construction of motivic homotopy theory
\index{motivic homotopy theory}
 requires the
heavy machinery of simplicial presheaves and model categories
\cite{MV99} \cite{Jardine87} \cite{DHI04}.  
\index{presheaf!simplicial}
\index{model category}
We give a more intuitive description
of motivic homotopy theory that will suffice for our purposes.

Motivic homotopy theory is a homotopy theory for algebraic varieties.
Start with the category of smooth schemes over a field $k$
(in this manuscript, $k$ always equals $\C$).
This category is inadequate for homotopical purposes because it does
not possess enough gluing constructions, i.e., homotopy colimits.

In order to fix this problem, we can formally adjoin homotopy colimits.
This takes us to the category of simplicial presheaves.
\index{presheaf!simplicial}

The next step is to restore some desired relations.
\index{topology!Zariski}
If
$\{ U, V \}$ is a Zariski cover of a smooth scheme $X$, then 
$X$ is the colimit of the diagram
\begin{equation}
\label{eq:pushout}
\xymatrix@1{
U & U \cap V \ar[r] \ar[l] & V 
}
\end{equation}
in the category of smooth schemes.  However, when we formally adjoined
homotopy colimits, we created a new object, distinct from $X$,
that served as the homotopy pushout of Diagram \ref{eq:pushout}.
This is undesirable, so we formally declare that
$X$ is the homotopy pushout of Diagram \ref{eq:pushout},
from which we obtain the local homotopy theory
of simplicial presheaves.
\index{presheaf!simplicial}
This homotopy theory has some convenient properties such as 
Mayer-Vietoris sequences.
\index{Mayer-Vietoris sequence}

In fact, one needs to work not with Zariski covers but with Nisnevich
covers.  See \cite{MV99} for details on this technical point.
\index{topology!Nisnevich}

The final step is to formally declare that each projection
map $X \times \A^1 \map X$ is a weak equivalence.
This gives the unstable motivic homotopy category.
\index{motivic homotopy theory!unstable}

In unstable motivic homotopy theory,
there are two distinct objects that play the role of circles:
\index{motivic sphere!unstable}
\begin{enumerate}
\item
$S^{1,0}$ is the usual simplicial circle.
\item
$S^{1,1}$ is the punctured affine line $\A^1-0$.
\end{enumerate}
For $p \geq q$, the unstable sphere $S^{p,q}$ is the appropriate 
smash product of copies of $S^{1,0}$ and $S^{1,1}$, so we have
a bigraded family of spheres.

\index{motivic homotopy theory!stable}
Stable motivic homotopy theory is the stabilization of 
unstable motivic homotopy theory with respect to
this bigraded family of spheres.  
As a consequence, calculations such as motivic cohomology and
motivic stable homotopy groups are bigraded.

Motivic homotopy theory over $\C$ comes with a realization functor to
\index{realization functor}
ordinary homotopy theory.  Given a complex scheme $X$, there is an associated
topological space $X(\C)$ of $\C$-valued points.  This construction extends
to a well-behaved functor between unstable and stable homotopy theories.

We will explain at the beginning of Chapter
\ref{ch:Adams-diff} that we have very good calculational
control over this realization functor.  We will use this relationship
in both directions: to deduce motivic facts from classical results,
and to deduce classical facts from motivic results.

One important difference between the classical case and the
motivic case is that not every motivic spectrum is built out of spheres,
i.e., not every motivic spectrum is cellular.
\index{motivic spectrum!cellular}
Stable cellular motivic homotopy theory is more tractable than the
full motivic homotopy theory, and 
many motivic spectra of particular interest, such as
the Eilenberg-Mac Lane spectrum $H\F_2$, 
\index{Eilenberg-Mac Lane spectrum}
the algebraic $K$-theory spectrum $KGL$, 
\index{algebraic K-theory@algebraic $K$-theory}
and the algebraic cobordism spectrum $MGL$, 
\index{algebraic cobordism}
are cellular. 
Stable motivic homotopy group calculations are fundamental to 
cellular motivic homotopy theory.  However, the part of motivic
homotopy theory that is not cellular is essentially invisible
from the perspective of stable motivic homotopy groups.

Although one can study motivic homotopy theory over any base field
(or even more general base schemes), we will work only over 
$\C$, or any algebraically closed field of characteristic $0$.
Even in this simplest case, we find a wealth of exotic phenomena that have
no classical analogues.

\section{The motivic Steenrod algebra}

\index{Steenrod algebra}
The starting point for our Adams spectral sequence work
is the description of the motivic Steenrod algebra over $\C$ at the prime $2$,
which is a variation
on the classical Steenrod algebra.  First, the motivic cohomology of 
a point is $\M_2 = \F_2[\tau]$, where $\tau$ has degree $(0,1)$ \cite{Voevodsky03b}.
\index{motivic cohomology!of a point}

\index{Steenrod algebra!dual}
The (dual) motivic Steenrod algebra over $\C$ is 
\cite{Voevodsky10} \cite{Voevodsky03a} \cite{Borghesi07}*{Section 5.2}
\[
\frac{\M_2 [ \tau_0, \tau_1, \ldots, \xi_1, \xi_2, \ldots ] }
{ \tau_i^2 = \tau \xi_{i+1}}.
\]
The reduced coproduct is determined by
\begin{align*}
\tilde{\phi}_* (\tau_k) &= 
\xi_k \otimes \tau_0 + \xi_{k-1}^2 \otimes \tau_1 + \cdots 
+ \xi_{k-i}^{2^i} \otimes \tau_i + \cdots + \xi_1^{2^{k-1}} \otimes \tau_{k-1} \\
\tilde{\phi}_* (\xi_k) &=
\xi_{k-1}^2 \otimes \xi_1 + \xi_{k-2}^4 \otimes \xi_2 + \cdots
+ \xi_{k-i}^{2^i} \otimes \xi_i + \cdots + \xi_1^{2^{k-1}} \otimes \xi_{k-1}.
\end{align*}

The dual motivic Steenrod algebra has a few interesting features.
\index{tau@$\tau$!localization}
First, if we invert $\tau$, then we obtain a polynomial algebra
that is essentially the same as the classical dual Steenrod algebra.  
This is a general feature.  We will explain at the beginning
of Chapter \ref{ch:Adams-diff}
that one recovers classical calculations from motivic calculations by 
inverting $\tau$.  This fact is useful in both directions:
to deduce motivic facts from classical ones, and to 
deduce classical facts from motivic ones.

Second, if we set $\tau = 0$, we obtain a ``$p = 2$ version" of the
\index{Steenrod algebra!odd primary}
classical odd primary dual Steenrod algebra, with a family of
exterior generators and another family of 
polynomial generators.  This observation suggests that various 
classical techniques that are well-suited for odd primes may 
also work motivically at the prime $2$.

\section{Relationship between motivic and classical calculations}

As a consequence of our detailed analysis of the motivic Adams spectral sequence,
we recover the analysis
of the classical Adams spectral sequence by inverting $\tau$.
\index{Adams spectral sequence!classical}
\index{stable stem!classical}

We will use known results about the classical Adams spectral
sequence from 
\cite{BJM84},
\cite{BMT70},
\cite{Bruner84},
\cite{MT67}, and
\cite{Tangora70b}.
We have carefully collected these results in
Tables \ref{tab:diff-refs} and \ref{tab:extn-refs}.

A few of our calculations are inconsistent with calculations
in \cite{Kochman90} and \cite{KM93},
and we are unable to understand the exact sources of the discrepancies.
For this reason, we have found it prudent to 
avoid relying directly on the calculations in \cite{Kochman90} and \cite{KM93}.
However, we will follow \cite{KM93} in establishing one particularly
difficult Adams differential in Section \ref{subsctn:d5-lemmas}. 

Here is a summary of our calculations that are inconsistent with 
\cite{Kochman90} and \cite{KM93}:
\begin{enumerate}
\item
There is a classical differential $d_3(Q_2) = g t$.
This means that classical $\pi_{56}$ has order 2, not order 4;
and that classical $\pi_{57}$ has order 8, not order 16.
\item
The element $h_1 g_2$ in the 45-stem does not support a hidden
$\eta$ extension to $N$.
\item
The element $C$ of the 50-stem does not support a hidden
$\eta$ extension to $g n$.
\item
\cite{Kochman90} claims that there is a hidden
$\nu$ extension from $h_2 h_5 d_0$ to $g n$
and that there is no hidden $2$ extension on $h_0 h_3 g_2$.
These two claims are incompatible; either both hidden
extensions occur, or neither occur.
(See Lemma \ref{lem:nu-h2h5d0}.)
\end{enumerate}

The proof of the non-existence of the hidden $\eta$ extension
on $h_1 g_2$ is particularly interesting because
it relies inherently on a motivic calculation.  We know of no way
to establish this result only with classical tools.

We draw particular attention to the Adams differential
$d_2(D_1) = h_0^2 h_3 g_2$ in the 51-stem.
Mark Mahowald 
\index{Mahowald, Mark}
privately communicated an argument for the
presence of this differential to the author.  However, this argument fails 
because of the calculation of the Toda bracket 
$\langle \theta_4, 2, \sigma^2 \rangle$ in
Lemma \ref{lem:bracket-theta4-2-sigma^2}, which was unknown to Mahowald.
Zhouli Xu 
\index{Xu, Zhouli}
and the author
discovered an independent proof, which is included in this
manuscript as Lemma \ref{lem:d2-D1}.  
This settles the order of $\pi_{51}$ but not its group structure.
It is possible that $\pi_{51}$ contains an element of order 8.
See \cite{IX14} for a more complete discussion.

Related to the misunderstanding concerning the bracket
$\langle \theta_4, 2, \sigma^2 \rangle$,
the published literature contains incorrect proofs that
$\theta_4^2$ equals zero.  Zhouli Xu
\index{Xu, Zhouli}
has found the first
correct proof of this relation \cite{Xu14}.
This has implications for the strong Kervaire problem.
Xu used the calculation of $\theta_4^2$ to simplify the argument 
given in \cite{BJM84} that establishes the existence of the 
Kervaire class $\theta_5$.
\index{Kervaire problem}
\index{theta4@$\theta_4$}
\index{theta5@$\theta_5$}

We also remark on the hidden $2$ extension in the 62-stem
from $E_1 + C_0$ to $R$ indicated in \cite{KM93}.
We cannot be absolutely certain of the status of this extension
because it lies outside the range of our thorough analysis.
However, it appears implausible from the motivic perspective.
(For entirely different reasons related to $v_2$-periodic homotopy groups,
Mark Mahowald
\index{Mahowald, Mark}
communicated privately to the author that he was also
skeptical of this hidden extension.)

\section{Relationship to the Adams-Novikov spectral sequence}

\index{Adams-Novikov spectral sequence}
We will describe
a rigid relationship between the motivic Adams spectral
sequence and the motivic Adams-Novikov spectral sequence
in Chapter \ref{ch:ANSS}.
In short, 
the $E_2$-page of the classical Adams-Novikov spectral sequence
is isomorphic to 
the bigraded homotopy groups $\pi_{*,*}(C\tau)$
of the cofiber of $\tau$. 
\index{cofiber of tau@cofiber of $\tau$!}
 Here $\tau$ is the element of
the motivic stable homotopy group $\pi_{0,-1}$ that is detected by the
element $\tau$ of $\M_2$.
Moreover, the classical
Adams-Novikov spectral sequence is identical to the
$\tau$-Bockstein spectral sequence converging to stable motivic homotopy groups!

In Chapter \ref{ch:Ctau}, we will extensively compute
$\pi_{*,*}(C\tau)$.  In Chapter \ref{ch:ANSS}, we will apply this
information to 
obtain information about the classical Adams-Novikov spectral sequence
in previously unknown stems.

However, there are two places in earlier chapters where we use specific
calculations from the classical Adams-Novikov spectral sequence.  
We would prefer arguments that are internal to the Adams spectral sequence,
but they have so far eluded us.  The specific calculations that we need are:
\begin{enumerate}
\item
Lemma \ref{lem:tau-D11} shows that a certain possible hidden
$\tau$ extension does not occur in the 57-stem.  See also 
Remark \ref{rem:tau-D11}.  For this, we use that $\beta_{12/6}$ is the only element
in the Adams-Novikov spectral sequence in the 58-stem with filtration 2
that is not divisible by $\alpha_1$
\cite{Shimomura81}.
\item
Lemma \ref{lem:2-h0h5i} establishes a hidden
$2$ extension in the 54-stem.  See also Remark \ref{rem:2-h0h5i}.
For this, we use that $\beta_{10/2}$ is the only element of the 
Adams-Novikov spectral sequence in the 54-stem with filtration $2$
that is not divisible by $\alpha_1$,
and that this element maps to $\Delta^2 h_2^2$ in the
Adams-Novikov spectral sequence for $\tmf$ \cite{Bauer08} \cite{Shimomura81}.
\index{topological modular forms}
\end{enumerate}

The $E_2$-page of the motivic (or classical) Adams spectral sequence is readily
computable by machine.  On the other hand, there seem to be real obstructions
to practical machine computation of the $E_2$-page of the classical
Adams-Novikov spectral sequence.
\index{machine computation}

On the other hand, let us suppose that we did have machine computed data on the 
$E_2$-page of the classical Adams-Novikov spectral sequence.
The rigid relationship between motivic stable homotopy groups and
the classical Adams-Novikov spectral sequence could be exploited to great
effect to determine the pattern of differentials in both the
Adams-Novikov and the Adams spectral sequences.  We anticipate that all
differentials through the 60-stem would be easy to deduce, and we would
expect to be able to compute well past the 60-stem.
For this reason, we foresee that the
next major breakthrough in computing stable stems will involve machine
computation of the Adams-Novikov $E_2$-page.

\section{How to use this manuscript}

The exposition of such a technical calculation creates some
inherent challenges.
In the end, the most important parts of this project
are the Adams charts from \cite{Isaksen14a},
the Adams-Novikov charts from \cite{Isaksen14e},
 and the 
tables in Chapter \ref{ch:table}.
These tables contain a wealth of detailed information in a concise form.
They summarize the essential calculational facts that allow
the computation to proceed.
In fact, the rest of the manuscript merely consists of detailed arguments
that support the claims in the tables.

For readers interested in specific calculational facts, 
the tables in Chapter \ref{ch:table} are the place to start.
These tables include references to more detailed proofs given elsewhere
in the manuscript.
The index also provides references to miscellaneous remarks about specific
elements.

\index{Adams chart}
\index{Adams-Novikov chart}
We draw attention to the following charts from \cite{Isaksen14a} 
and \cite{Isaksen14e}
that are of particular interest:
\begin{enumerate}
\item
A classical Adams $E_2$ chart with differentials.
\item
A classical Adams $E_\infty$ chart with hidden extensions
by $2$, $\eta$, and $\nu$.
\item
A motivic Adams $E_2$ chart.
\item
A motivic Adams $E_\infty$ chart with hidden $\tau$ extensions.
\item
A classical Adams-Novikov $E_2$ chart with differentials.
\item
A classical Adams-Novikov $E_\infty$ chart with hidden extensions
by $2$, $\eta$, and $\nu$.
\end{enumerate}
In each of the charts, we have been careful to document explicitly 
the remaining uncertainties in our calculations.

We also draw attention to the following tables from Chapter \ref{ch:table}
that are of particular interest:
\begin{enumerate}
\item
Tables \ref{tab:Ext-gen}, \ref{tab:Adams-d3}, \ref{tab:Adams-d4}, 
and \ref{tab:Adams-d5} give all of the Adams differentials.
\item
Table \ref{tab:Massey} gives some Massey products in the cohomology of
the motivic Steenrod algebra, including indeterminacies. 
\item
Table \ref{tab:diff-refs} summarizes previously known results about
classical Adams differentials.
\item
Table \ref{tab:bracket-refs} summarizes previously known results about
classical Toda brackets.
\item
Table \ref{tab:Toda} gives some Toda brackets, including indeterminacies.
\item
Table \ref{tab:extn-refs} summarizes previously known results about
hidden extensions in the classical stable homotopy groups.
\item
Table \ref{tab:Adams-ANSS} gives a correspondence between elements of the
classical Adams and Adams-Novikov $E_\infty$ pages.
\end{enumerate}
These tables include specific references to complete proofs of each fact.

\section{Notation}
\label{sctn:notation}

\index{notation}
\index{degree}
\index{weight}
By convention, we give degrees in the form $(s,f,w)$,
where $s$ is the stem; $f$ is the
Adams filtration; and $w$ is the motivic weight.  
An element of degree $(s,f,w)$ will appear on a chart at coordinates
$(s,f)$. 

We will use the following notation extensively:
\begin{enumerate}
\item
$\M_2$ is the mod 2 motivic cohomology of $\C$.
\index{motivic cohomology!of a point}
\item
$A$ is the mod 2 motivic Steenrod algebra over $\C$.
\index{Steenrod algebra}
\item
$A(2)$ is the $\M_2$-subalgebra of $A$ generated by $\Sq^1$, $\Sq^2$, and $\Sq^4$.
\item
$\Ext$ is the trigraded ring $\Ext_A(\M_2,\M_2)$.
\item
$A_{\cl}$ is the classical mod 2 Steenrod algebra.
\index{Steenrod algebra!classical}
\item
$\Ext_{\cl}$ is the bigraded ring $\Ext_{A_{\cl}}(\F_2,\F_2)$.
\item
$\pi_{*,*}$ is the $2$-complete motivic
stable homotopy ring over $\C$.
\index{stable stem}
\item
$E_r(S^{0,0})$ is the $E_r$-page of the motivic
Adams spectral sequence converging to $\pi_{*,*}$.
Note that $E_2(S^{0,0})$ equals $\Ext$.
\index{Adams spectral sequence}
\item
For $x$ in $E_\infty(S^{0,0})$, write
$\{ x\}$ for the set of all elements of $\pi_{*,*}$ that are 
represented by $x$.
\item
$\tau$ is both an element of $\M_2$, as well as the element
of $\pi_{0,-1}$ that it represents in the motivic Adams spectral sequence.
\index{tau@$\tau$}
\item
$C\tau$ is the cofiber of $\tau: S^{0,-1} \map S^{0,0}$.
\index{cofiber of tau@cofiber of $\tau$}
\item
$H^{*,*}(C\tau)$ is the mod 2 motivic cohomology of $C\tau$.
\item
$\pi_{*,*}(C\tau)$ are the 2-complete motivic stable homotopy groups of $C\tau$,
which form a $\pi_{*,*}$-module.
\item
$E_r(C\tau)$
is the $E_r$-page of the 
motivic Adams spectral sequence that converges to
$\pi_{*,*}(C\tau)$.
Note that $E_r(C\tau)$ is an $E_r(S^{0,0})$-module,
and $E_2(C\tau)$ is equal to $\Ext_A (H^{*,*}( C\tau ), \M_2)$.
\item
For $x$ in $E_2(S^{0,0})$,
write $x$ again (or $x_{C\tau}$ when absolutely necessary for clarity)
for the image of $x$ under the map
$E_2(S^{0,0}) \map E_2(C\tau)$ induced by
the inclusion $S^{0,0} \map C\tau$ of the bottom cell.
\item
For $x$ in $E_2(S^{0,0})$ such that $\tau x = 0$,
write $\ol{x}$ 
for a pre-image of $x$ under the map
$E_2(C\tau) \map E_2(S^{0,0})$ induced by
the projection $C\tau \map S^{1,-1}$ to the top cell.
There may be some indeterminacy in the choice of $\ol{x}$.
See Section \ref{subsctn:E2} and Table \ref{tab:Ctau-ambiguous}
for further discussion about
these choices.
\item
$E_r(S^0; BP)$ is the $E_r$-page of the classical
Adams-Novikov spectral sequence.
\index{Adams-Novikov spectral sequence}
\item
$E_r(S^{0,0}; BPL)$ is the $E_r$-page of the motivic Adams-Novikov
spectral sequence converging to $\pi_{*,*}$.
\item
$E_r(C\tau; BPL)$ is the $E_r$-page of the motivic Adams-Novikov
spectral sequence converging to $\pi_{*,*}(C\tau)$.
\end{enumerate}

Table \ref{tab:notation}
lists some traditional notation for specific elements of the motivic
stable homotopy ring.  We will use this notation whenever it is convenient.
A few remarks about these elements are in order:
\begin{enumerate}
\item
See \cite{HKO11}*{p.\ 28} for a geometric construction of $\tau$.
\index{tau@$\tau$}
\item
Over fields that do not contain $\sqrt{-1}$, the motivic stable homotopy
group $\pi_{0,0}$ contains an element that is usually called $\epsilon$.  
\index{epsilon@$\epsilon$}
Our use of the symbol $\epsilon$
follows Toda \cite{Toda62}.  
This should cause no confusion since we are working only
over $\C$.
\item
The element $\eta_4$ is defined to be the element of $\{ h_1 h_4 \}$ 
such that $\eta^3 \eta_4$ is zero. 
\index{eta4@$\eta_4$}
(The other element of $\{ h_1 h_4 \}$
supports infinitely many multiplications by $\eta$.)
\item
Similarly, $\eta_5$ is defined to be the element of $\{ h_1 h_5 \}$
such that $\eta^7 \eta_5$ is zero.
\index{eta5@$\eta_5$}
\end{enumerate}

The element $\theta_{4.5}$ deserves additional discussion.
\index{theta4.5@$\theta_{4.5}$}
We have perhaps presumptuously adopted this notation
for an element of $\{ h_3^2 h_5 \} = \{ h_4^3 \}$.
This element is called $\alpha$ in \cite{BJM84}.
To construct $\theta_{4.5}$,
first choose an element $\theta_{4.5}'$ in $\{ h_3^2 h_5 \}$ such that
$4 \theta_{4.5}'$ is contained in $\{ h_0 h_5 d_0 \}$.
If $\eta \theta_{4.5}'$ is contained in $\{ h_1 h_5 d_0 \}$,
then add an element of $\{ h_5 d_0 \}$ to $\theta_{4.5}'$ and obtain
an element $\theta''_{4.5}$ such that $\eta \theta''_{4.5}$
is contained in $\{B_1\}$.
Next, if $\sigma \theta''_{4.5}$ is contained in
$\{ \tau h_1 h_3 g_2 \}$, then add an
element of $\{\tau h_1 g_2 \}$ to $\theta''_{4.5}$
to obtain an element $\theta_{4.5}$ such that
$\sigma \theta_{4.5}$ is detected in Adams filtration at least 8.
Note that $\sigma \theta_{4.5}$ may in fact be zero.

This does not specify just a single element of $\{h_3^2 h_5\}$.
The indeterminacy in the definition contains even multiples of $\theta_{4.5}$
and the element $\{\tau w \}$, but this indeterminacy does not present
much difficulty.

In addition, we do not know whether $\nu \theta_{4.5}$
is contained in $\{B_2\}$.  We know from Lemma \ref{lem:nu-h3^2h5}
that there is an element $\theta$ of $\{ h_3^2 h_5 \}$
such that $\nu \theta$ is contained in $\{B_2\}$.
It is possible that $\theta$ is of the form $\theta_{4.5} + \beta$,
where $\beta$ belongs to $\{ h_5 d_0 \}$.
We can conclude only that either $\nu \theta_{4.5}$ or
$\nu (\theta_{4.5} + \beta)$ belongs to $\{B_2\}$.

For more details on the properties of $\theta_{4.5}$, see 
Examples \ref{ex:hidden-cross-1} and \ref{ex:hidden-cross-2},
as well as Lemmas \ref{lem:eta-h3^2h5} and \ref{lem:nu-h3^2h5}.

\section{Acknowledgements}
The author would like to acknowledge the invaluable assistance that he received
in the preparation of this manuscript.

Robert Bruner
\index{Bruner, Robert}
generously shared his
extensive library of machine-assisted classical computations.
Many of the results in this article 
would have been impossible to discover without the guidance
of this data.
Discussions with Robert Bruner led to the realization that the cofiber
of $\tau$ is a critical computation.

Dan Dugger's 
\index{Dugger, Daniel}
machine-assisted
motivic computations were also essential.

Martin Tangora 
\index{Tangora, Martin}
offered several key insights into technical 
classical May spectral sequence computations.

Zhouli Xu
\index{Xu, Zhouli}
listened to and critiqued a number of the more subtle arguments
involving delicate properties of Massey products and Toda brackets.
Conversations with Bert Guillou 
\index{Guillou, Bert}
also helped to clarify many of the arguments.

Peter May
\index{May, Peter}
supplied some historical motivation and helped the author understand
how some of the bigger ideas fit together.

Mark Behrens
\index{Behrens, Mark}
encouraged the author to dare to reach beyond the 50-stem.
Similarly, Mike Hopkins
\index{Hopkins, Mike}
and Haynes Miller
\index{Miller, Haynes}
were also supportive.

Finally, and most importantly, the author
is privileged to have discussed some of these results with
Mark Mahowald \index{Mahowald, Mark}
shortly before he passed away.

%% file: stable-stems-May.tex
\chapter{The cohomology of the motivic Steenrod algebra}
\label{ch:May}



This chapter
applies the motivic May spectral sequence 
\index{May spectral sequence}
to obtain the 
cohomology of the motivic Steenrod algebra through the 70-stem.
We will freely borrow results from the classical
May spectral sequence, i.e., from  \cite{May64} and \cite{Tangora70a}.
We will also need
some facts from the cohomology of the classical Steenrod
algebra that have been verified only by machine \cite{Bruner97} 
\cite{Bruner04}.

The $\Ext$ chart in \cite{Isaksen14a} is an essential companion to this chapter.
\index{Ext@$\Ext$!chart}

\subsection*{Outline}

We begin in Section \ref{sctn:May-SS} with a review of the
basic facts about the motivic Steenrod algebra over $\C$,
\index{Steenrod algebra}
the motivic May spectral sequence over $\C$, and the cohomology
of the motivic Steenrod algebra.
\index{Steenrod algebra!cohomology}
A critical ingredient is May's Convergence Theorem \cite{May69}*{Theorem 4.1},
\index{Convergence Theorem!May}
which allows the computation of Massey products in $\Ext_A(\F_2,\F_2)$
via the differentials in the May spectral sequence.
\index{Massey product}
We will thoroughly review this result in Section \ref{sctn:Massey-May}.

Next, in Section \ref{sctn:May-diff} we describe the main points in
computing the motivic May spectral sequence through the
70-stem.  We rely heavily on results of \cite{May64} and \cite{Tangora70a}, 
but we must also
compute several exotic differentials, i.e., 
differentials that do not occur in the classical situation.
\index{May spectral sequence!differential}

Having obtained the $E_\infty$-page of the motivic May spectral sequence,
the next step is to consider hidden extensions.
\index{May spectral sequence!hidden extension}
In Section \ref{sctn:hidden-extn},
we are able to resolve every possible hidden extension by
$\tau$, $h_0$, $h_1$, and $h_2$ through the range that we are considering,
i.e., up to the 70-stem.
The primary tools here are:
\begin{enumerate}
\item
shuffling relations among Massey products.
\index{Massey product}
\item
Steenrod operations on $\Ext$ groups in the sense of \cite{May70}.
\index{Steenrod operation!algebraic}
\item
classical hidden extensions established by machine computation \cite{Bruner97}.
\index{machine computation}
\end{enumerate}

Chapter \ref{ch:table} contains a series of tables that are essential
for bookkeeping throughout the computations:
\begin{enumerate}
\item
Tables \ref{tab:May-E2-gen} and \ref{tab:May-E2-reln} 
describe the May $E_2$-page in terms of generators and relations
and give the values of the May $d_2$ differential.
\item
Tables \ref{tab:May-E4} through \ref{tab:May-higher}
describe the May differentials $d_r$ for $r \geq 4$.
\item
Table \ref{tab:Ext-gen} lists the multiplicative generators of the 
cohomology of the motivic Steenrod algebra over $\C$.
\item
Table \ref{tab:May-Einfty-temp} lists multiplicative generators
of the May $E_\infty$-page that become decomposable in $\Ext$
by hidden relations.
\item
Table \ref{tab:Ext-ambiguous} 
lists all examples of multiplicative generators of the May $E_\infty$-page
that represent more than one element in $\Ext$.
See Section \ref{subsctn:May-Einfty} for more explanation.
\item
Tables \ref{tab:May-tau} through \ref{tab:May-misc} list all
extensions by $\tau$, $2$, $\eta$, and $\nu$ that are hidden
in the May spectral sequence.  A few miscellaneous hidden
extensions are included as well.
\item
Table \ref{tab:Massey} summarizes some Massey products.
\index{Massey product}
\item
Table \ref{tab:Ctau-matric} summarizes some matric Massey products.
\index{Massey product!matric}
\end{enumerate}

Table \ref{tab:Massey} deserves additional explanation.  
In all cases, we have been careful to describe
the indeterminacies accurately.  
\index{Massey product!indeterminacy}
The fifth column refers to an argument
for establishing the Massey product, in one of the following forms:
\begin{enumerate}
\item
An explicit proof given elsewhere in this manuscript.
\item
A May differential implies the Massey product via
May's Convergence Theorem \ref{thm:3-converge}.
\index{Convergence Theorem!May}
\end{enumerate}
The last column of Table \ref{tab:Massey} lists the specific results
that rely on each Massey product.  Frequently, these results are just
a Toda bracket from Table \ref{tab:Toda}.
\index{Toda bracket}

\subsection*{Some examples}

In this section, we describe several of the computational 
intricacies that are established later in the chapter.  
We also present a few questions that deserve further study.

\begin{ex}
\label{ex:h2g^2}
An obvious question, which already arose in \cite{DI10}, is to find
elements that are killed by $\tau^n$ but not by $\tau^{n-1}$,
for various values of $n$.

The element $h_2 g^2$, which is multiplicatively indecomposable,
is the first example
of an element that is killed by $\tau^3$ but not by $\tau^2$.
This occurs because of a hidden extension
$\tau \cdot \tau h_2 g^2 = P h_1^4 h_5$.
There is an analogous relation $\tau^2 h_2 g = P h_4$
that is not hidden.
We do not know if this generalizes to a family of relations of the
form $\tau^2 h_2 g^{2^k} = P h_1^{2^{k+2}-4} h_{k+4}$.

We will show in Chapter \ref{ch:Adams-diff} that $h_2 g^2$
\index{g2@$g^2$} represents an element
in motivic stable homotopy that is killed by $\tau^3$ but not by
$\tau^2$.  This requires an analysis of the motivic Adams spectral
sequence.
In the vicinity of $g^{2^k}$, one might hope to find elements that 
are killed by $\tau^n$ but not by $\tau^{n-1}$,
for large values of $n$.
\index{tau@$\tau$!torsion}
\index{g@$g$!powers of}
\end{ex}

\begin{ex}
\label{ex:h3e0}
Classically, there is a relation $h_3 \cdot e_0 = h_1 h_4 c_0$
\index{h3e0@$h_3 e_0$}
\index{h4c0@$h_4 c_0$}
in the 24-stem 
of the cohomology of the Steenrod algebra.  This relation is hidden on the
$E_\infty$-page of the May spectral sequence.
We now give a proof of this classical relation that uses the cohomology
of the motivic Steenrod algebra. 

Motivically, it turns out that $h_2^3 e_0$ is non-zero,
even though it is zero classically.  
This follows from the hidden extension $h_0 \cdot h_2^2 g = h_1^3 h_4 c_0$
(see Lemma \ref{lem:hidden-h0h2^2g}).
The relation $h_2^3 = h_1^2 h_3$ then implies that
$h_1^2 h_3 e_0$ is non-zero.  Therefore,
$h_3 e_0$ is non-zero as well, and the only possibility is that
$h_3 e_0 = h_1 h_4 c_0$.
\end{ex}

\begin{ex}
\label{ex:h1^7h5c0}
Notice the hidden extension
$h_0 \cdot h_2^2 g^2 = h_1^7 h_5 c_0$ (and similarly,
the hidden extension $h_0 \cdot h_2^2 g = h_1^3 h_4 c_0$
that we discussed above in Example \ref{ex:h3e0}).
\index{g2@$g^2$}
\index{h5c0@$h_5 c_0$}

The next example in this family is
$h_0 \cdot h_2^2 g^3 = h_1^9 D_4$,
\index{g3@$g^3$}
\index{D4@$D_4$}
which at first does not
appear to fit a pattern.  However,
there is a hidden extension $c_0 \cdot i_1 = h_1^4 D_4$,
\index{May spectral sequence!hidden extension!c0@$c_0$}
\index{i1@$i_1$}
so we have
$h_0 \cdot h_2^2 g^3 = h_1^5 c_0 i_1$.
Presumably, 
there is an infinitely family of hidden extensions
in which $h_0 \cdot h_2^2 g^k$ equals 
some power of $h_1$ times $c_0$ times an element related to 
$\Sq^0$ of elements associated to the image of $J$.
\index{Steenrod operation!algebraic}
\index{J@$J$!image of}

It is curious that $c_0 \cdot i_1$ is divisible by $h_1^4$.
An obvious question for further study is to determine the
$h_1$-divisibility of $c_0$ times elements related to
$\Sq^0$ of elements associated to the image of $J$.
For example, what is the largest power of $h_1$ that divides
$g^2 i_1$?
\end{ex}

\begin{ex}
\label{ex:g^2}
Beware that $g^2$ and $g^3$ are not actually elements 
\index{g@$g$!powers of}
of the 40-stem and 60-stem respectively.
Rather, it is only $\tau g^2$ and $\tau g^3$ that exist
(similarly, $g$ does not exist in the 20-stem, but $\tau g$ does exist).
The reason is that there are May differentials
taking $g^2$ to $h_1^8 h_5$, and $g^3$ to $h_1^6 i_1$.
In other words,
$\tau g^2$ and $\tau g^3$ are multiplicatively indecomposable elements.
More generally, we anticipate that the element $g^k$ does not exist
because it supports a May differential related to $\Sq^0$ of
an element in the image of $J$.
\index{J@$J$!image of}
\end{ex}

\begin{ex}
There is an isomorphism from the cohomology of the classical Steenrod
algebra to the cohomology of the motivic Steenrod algebra over $\C$
concentrated in degrees of the form
$\left( 2s+f, f, s+f \right)$.
This isomorphism preserves all higher structure,
including algebraic Steenrod operations and Massey products.
See Section \ref{subsctn:Chow-deg-zero} for more details.
\index{Massey product}
\index{Steenrod operation!algebraic}

For example, the existence of the classical element $P h_2$
\index{Ph2@$P h_2$}
immediately implies that $h_3 g$
\index{h3g@$h_3 g$}
 must be non-zero in the motivic setting;
no calculations are necessary.

Another example is that
$h_1^{2^k-1} h_{k+2}$ is non-zero motivically for all $k \geq 1$,
because $h_0^{2^k-1} h_{k+1}$ is non-zero classically.
\end{ex}

\begin{ex}
Many elements are $h_1$-local in the sense that they
\index{h1@$h_1$!localization}
support infinitely many multiplications by $h_1$.  
In fact, any product of the 
symbols $h_1$, $c_0$, $P$, $d_0$, $e_0$, and $g$, if it
exists, is non-zero.  This is detectable in the cohomology
of motivic $A(2)$ \cite{Isaksen09}.

Moreover, the element $B_1$
\index{B1@$B_1$}
in the 46-stem is $h_1$-local, and any
product of $B_1$ with elements in the previous paragraph is again
$h_1$-local.  We explore $h_1$-local elements in great detail in \cite{GI14}.
\end{ex}

\begin{ex}
\index{wedge subalgebra}
The motivic analogue of the ``wedge'' subalgebra \cite{MT68}
appears to be more complicated than the classical version.
For example, none of the wedge elements support
multiplications by $h_0$ in the classical case.  
Motivically, many wedge elements do support $h_0$ multiplications.
The results in this chapter naturally call for further study of 
the structure of the motivic wedge.  
\end{ex}

\section{The motivic May spectral sequence}
\label{sctn:May-SS}

The following two deep theorems of Voevodsky are the starting points
of our calculations.

\index{motivic cohomology!of a point}

\begin{thm}[\cite{Voevodsky03b}]
$\M_2$ is the bigraded ring $\F_2[\tau]$, where
$\tau$ has bidegree $(0,1)$.
\end{thm}

\index{Steenrod algebra}

\begin{thm}[\cite{Voevodsky03a} \cite{Voevodsky10}]
The motivic Steenrod algebra $A$ is the $\M_2$-algebra generated by
elements $\Sq^{2k}$ and $\Sq^{2k-1}$ for all $k \geq 1$, of bidegrees
$(2k,k)$ and $(2k-1,k-1)$ respectively, and satisfying the following
relations for $a< 2b$:
\[ 
\Sq^a \Sq^b = 
\sum_{c} \binom{b-1-c}{a-2c} \tau^{?}\Sq^{a+b-c} \Sq^c.
\]
\end{thm}

\index{Adem relation}

The symbol $?$ stands for either $0$ or $1$, depending on which value
makes the formula balanced in weight.
See \cite{DI10} for a more detailed discussion of the motivic Adem relations.

The $A$-module structure on $\M_2$ is trivial, i.e.,
every $\Sq^k$ acts by zero.  This follows for simple degree reasons.

\index{Steenrod algebra!dual}
It is often helpful to work with the dual motivic Steenrod algebra
$A_{*,*}$
\cite{Borghesi07}*{Section 5.2}
\cite{Voevodsky03a} 
\cite{Voevodsky10},
which equals
\[
\frac{\M_2 [ \tau_0, \tau_1, \ldots, \xi_1, \xi_2, \ldots ] }
{ \tau_i^2 = \tau \xi_{i+1}}.
\]
The reduced coproduct in $A_{*,*}$ is determined by
\begin{align*}
\tilde{\phi}_* (\tau_k) &= 
\xi_k \otimes \tau_0 + \xi_{k-1}^2 \otimes \tau_1 + \cdots 
+ \xi_{k-i}^{2^i} \otimes \tau_i + \cdots + \xi_1^{2^{k-1}} \otimes \tau_{k-1} \\
\tilde{\phi}_* (\xi_k) &=
\xi_{k-1}^2 \otimes \xi_1 + \xi_{k-2}^4 \otimes \xi_2 + \cdots
+ \xi_{k-i}^{2^i} \otimes \xi_i + \cdots + \xi_1^{2^{k-1}} \otimes \xi_{k-1}.
\end{align*}

\subsection{$\Ext$ groups}
\label{subsctn:Ext-group}

\index{Ext@$\Ext$}
We are interested in computing
$\Ext_A(\M_2,\M_2)$, which we abbreviate as $\Ext$.
This is a trigraded object.  We will consistently use degrees of the
form $(s,f,w)$, where:
\index{degree}
\begin{enumerate}
\item
$f$ is the Adams filtration, i.e., the homological degree.
\item
$s+f$ is the internal degree, i.e., corresponds to the
first coordinate in the bidegrees of $A$.
\item
$s$ is the stem, i.e., the internal degree minus
the Adams filtration.
\item
\index{weight}
$w$ is the weight.
\end{enumerate}

Note that $\Ext^{*,0,*} = \Hom_A^{*,*}(\M_2,\M_2)$
is dual to $\M_2$.  We will abuse notation and write
$\M_2$ for this dual.  Beware that now $\tau$, which is really
the dual of the $\tau$ that we discussed earlier, has degree $(0,0,-1)$.
Since $\Ext$ is a module over $\Ext^{*,0,*}$, i.e., over $\M_2$,
we will always describe $\Ext$ as an $\M_2$-module.

The following result is the key tool for comparing classical
and motivic computations.  The point is that the motivic
and classical computations become the same after inverting $\tau$.
\index{tau@$\tau$!localization}

\begin{prop}[\cite{DI10}]
\label{prop:compare-ext}
There is an isomorphism of rings
\[
\Ext \otimes_{\M_2} \M_2 [\tau^{-1}] \iso
\Ext_{A_{\cl}} \otimes_{\F_2} \F_2 [\tau,\tau^{-1}].
\]
\end{prop}

\subsection{The motivic May spectral sequence}

\index{May spectral sequence}
The classical May spectral sequence arises by filtering the
classical Steenrod algebra by powers of the augmentation ideal.
The same approach can be applied in the motivic setting to 
obtain the motivic May spectral sequence.  
Details appear in \cite{DI10}.  Next we review the main points.

The motivic May spectral
sequence is quadruply graded.  We will always use gradings
of the form $(m,s,f,w)$, where $m$ is the May filtration,
and the other coordinates are as explained in Section \ref{subsctn:Ext-group}.

Let $\Gr(A)$ be the associated graded algebra of $A$
with respect to powers of the augmentation ideal.

\begin{thm}
\label{thm:may}
The motivic May spectral sequence takes the form
\[
E_2 = 
\Ext^{(m,s,f,w)}_{\Gr(A)}(\M_2,\M_2) \Rightarrow
\Ext_A^{(s,f,w)}(\M_2,\M_2).
\]
\end{thm}

\begin{remark}
As in the classical May spectral sequence,
the odd differentials must be trivial for degree reasons.
\index{May spectral sequence!differential!odd}
\end{remark}

\begin{prop}
\label{prop:may-compare}
After inverting $\tau$,
there is an isomorphism of spectral sequences between
the motivic May spectral sequence of Theorem \ref{thm:may}
and the classical May spectral sequence, tensored over
$\F_2$ with $\F_2[\tau, \tau^{-1}]$.
\end{prop}
\index{tau@$\tau$!localization}

\begin{proof}
Start with the fact that $A [\tau^{-1}]$ is isomorphic to
$A_{\cl} \otimes_{\F_2} \F_2[\tau, \tau^{-1}]$, 
with the same May filtrations.
\end{proof}

This proposition means that differentials in the motivic May
spectral sequence must be compatible with the classical differentials.
This fact is critical to the success of our computations.
\index{May spectral sequence!differential}

\subsection{$\Ext$ in degrees with $s+f-2w=0$}
\label{subsctn:Chow-deg-zero}

\begin{defn}
Let $A'$ be the subquotient $\M_2$-algebra of $A$ generated by
$\Sq^{2k}$ for all $k \geq 0$, subject to the relation $\tau = 0$.
\end{defn}

\begin{lemma}
\label{lem:Acl-A'}
There is an isomorphism $A_{\cl} \map A'$ that takes
$\Sq^k$ to $\Sq^{2k}$.
\end{lemma}

The isomorphism takes elements of degree $n$ to elements
of bidegree $(2n,n)$.

\begin{proof}
Modulo $\tau$, the motivic Adem relation for $\Sq^{2a} \Sq^{2b}$
takes the form
\index{Adem relation}
\[ 
\Sq^{2a} \Sq^{2b} = 
\sum_{c} \binom{2b-1-2c}{2a-4c} \Sq^{2a+2b-2c} \Sq^2c.
\]
A standard fact from combinatorics says that
\[
\binom{2b-1-2c}{2a-4c} = \binom{b-1-c}{a-2c}
\]
modulo $2$.
\end{proof}

\begin{remark}
Dually,
$A'$ corresponds to the quotient
$\F_2 [ \xi_1, \xi_2, \ldots ]$ of $A_{*,*}$,
where we have set $\tau$ and $\tau_0$, $\tau_1$, \ldots to be zero.
The dual to $A'$ is visibly isomorphic to the dual of the
classical Steenrod algebra.
\index{Steenrod algebra!classical}
\end{remark}

\begin{defn}
Let $M$ be a bigraded $A$-module.  The Chow degree
of an element $m$ in degree $(t,w)$ is equal to $t-2w$.
\end{defn}
\index{degree!Chow}

The terminology arises from the fact that the Chow degree is
fundamental in Bloch's higher Chow group perspective on 
motivic cohomology \cite{Bloch86}.
\index{Chow group!higher}

\begin{defn}
Let $M$ be an $A$-module.  Define the $A'$-module
$\Ch_0(M)$ to be the subset of $M$ consisting of elements
of Chow degree zero, with $A'$-module structure induced from the
$A$-module structure on $M$.
\end{defn}

The $A'$-module structure on $\Ch_0(M)$ is well-defined since
$\Sq^{2k}$ preserves Chow degrees.

\begin{thm}
\label{thm:Chow-0}
There is an isomorphism from
$\Ext_{A_{\cl}}$ to the subalgebra of $\Ext$ consisting of elements
in degrees $(s,f,w)$ with $s+f-2w=0$.
This isomorphism takes classical elements of degree
$(s,f)$ to motivic elements of degree $(2s+f,f, s+f )$,
and it preserves all higher structure, including
products, squaring operations, and Massey products.
\end{thm}

\begin{proof}
There is a natural transformation
\[
\Hom_A( -, \M_2 ) \map \Hom_{A'} ( \Ch_0(-), \F_2 ),
\]
since $\Ch_0(\M_2) = \F_2$.
Since $\Ch_0$ is an exact functor, the derived functor of the
right side is $\Ext_{A'} ( \Ch_0(-), \F_2 )$.
The universal property
of derived functors gives a natural transformation
$\Ext_A (-, \M_2 ) \map \Ext_{A'} ( \Ch_0(-), \F_2 )$.
Apply this natural transformation to $\M_2$ to obtain
$\Ext_A (\M_2, \M_2 ) \map \Ext_{A'} ( \Ch_0(\M_2), \F_2 )$.
The left side is $\Ext$, and the right side 
is isomorphic to $\Ext_{A_{\cl}}$ since $A'$ is isomorphic to $A_{\cl}$
by Lemma \ref{lem:Acl-A'}.

We have now obtained a map $\Ext \map \Ext_{A_{\cl}}$.
We will verify that this map is an isomorphism on the 
part of $\Ext$ in degrees $(s,f,w)$ with $s+f-2w=0$.
Compare
the classical May spectral sequence with 
the part of the motivic May spectral sequence
in degrees $(m,s,f,w)$ with $s+f-2w=0$.
By direct inspection,
the motivic $E_1$-page in these degrees
is the polynomial algebra over $\F_2$
generated by $h_{ij}$ for $i > 0$ and $j > 0$.
This is isomorphic to the classical $E_1$-page, where
the motivic element $h_{ij}$ corresponds to the classical 
element $h_{i,j-1}$.
\end{proof}

\begin{remark}
Similar methods show that $\Ext$ is concentrated in 
degree $(s,f,w)$ with $s+f-2w \geq 0$.
The map $\Ext \map \Ext_{A_{\cl}}$ constructed
in the proof annihilates elements in degrees
$(s,f,w)$ with $s+f-2w > 0$.
Thus, $\Ext_{A_{\cl}}$ is isomorphic to the quotient 
of $\Ext$ by elements of degree $(s,f,w)$ with $s+f-2w > 0$.
\end{remark}
\index{Steenrod algebra!classical}

\section{Massey products in the motivic May spectral sequence}
\label{sctn:Massey-May}

\index{Massey product}
We will frequently compute Massey products in $\Ext$ in order to resolve
hidden extensions and to determine May differentials.  The 
absolutely essential tool for
computing such Massey products 
is May's Convergence Theorem \cite{May69}*{Theorem 4.1}.
\index{Convergence Theorem!May}
The point of this theorem is that under certain hypotheses,
Massey products in $\Ext$ can be computed in the $E_r$-page of the motivic
May spectral sequence.  For the reader's convenience, 
we will state the theorem in the specific forms that we will use.
We have slightly generalized the result of \cite{May69}*{Theorem 4.1}
to allow for brackets that are not strictly defined.
In order to avoid unnecessarily heavy notation, 
we have intentionally avoided the most general possible statements.
The interested reader is encouraged to carry out these generalizations.

\index{Convergence Theorem!May}
\begin{thm}[May's Convergence Theorem]
\label{thm:3-converge}
Let $\alpha_0$, $\alpha_1$, and $\alpha_2$ be elements of $\Ext$ such that
the Massey product $\langle \alpha_0, \alpha_1, \alpha_2 \rangle$
is defined.
For each $i$, let $a_i$ be a permanent cycle on the May $E_r$-page that detects
$\alpha_i$.
Suppose further that:
\begin{enumerate}
\item
there exist elements $a_{01}$ and $a_{12}$ on the May $E_r$-page
such that $d_r(a_{01}) = a_0 a_1$ and $d_r (a_{12}) = a_1 a_2$.
\item
if $(m,s,f,w)$ is the degree of either $a_{01}$ or $a_{12}$;
$m' \geq m$; and $m' - t < m - r$; then every 
May differential $d_t: E_t^{(m',s,f,w)} \map E_t^{(m'-t+1,s-1,f+1,w)}$
is zero.
\end{enumerate}
Then $a_0 a_{12} + a_{01} a_2$ detects an element of
$\langle \alpha_0, \alpha_1, \alpha_2 \rangle$ in $\Ext$.
\end{thm}

The point of condition (1) is that 
the bracket $\langle a_0, a_1, a_2 \rangle$ is defined in the
differential graded algebra $(E_r, d_r)$.
Condition (2) is an equivalent reformulation of condition (*) in
\cite{May69}*{Theorem 4.1}.
When computing $\langle a_0, a_1, a_2 \rangle$, one uses a differential
$d_r: E_r^{(m,s,f,w)} \map E_r^{(m-r+1,s-1,f+1,w)}$.
The idea of condition (2) is that there are no later 
``crossing" differentials
\index{May spectral sequence!differential!crossing}
$d_t$ whose source has higher May filtration and whose target
has strictly lower May filtration.

The proof of May's Convergence Theorem \ref{thm:3-converge} is exactly
the same as in \cite{May69} because every threefold Massey product is
strictly defined in the sense that its subbrackets have no indeterminacy.
\index{Massey product!strictly defined}

\index{Convergence Theorem!May}
\begin{thm}[May's Convergence Theorem]
\label{thm:4-converge}
Let $\alpha_0$, $\alpha_1$, $\alpha_2$, and $\alpha_3$ be elements of $\Ext$ such that
the Massey product $\langle \alpha_0, \alpha_1, \alpha_2, \alpha_3 \rangle$
is defined.
For each $i$, let $a_i$ be a permanent cycle on the $E_r$-page
that detects $\alpha_i$.
Suppose further that:
\begin{enumerate}
\item
there are elements $a_{01}$, $a_{12}$, and $a_{23}$ on the $E_r$-page
such that $d_r(a_{01}) = a_0 a_1$, $d_r (a_{12}) = a_1 a_2$, and
$d_r (a_{23}) = a_2 a_3$.
\item
there are elements $a_{02}$ and $a_{13}$ on the $E_r$-page such that
$d_r (a_{02}) = a_0 a_{12} + a_{01} a_2$ and
$d_r (a_{13}) = a_1 a_{23} + a_{12} a_3$.
\item
if $(m,s,f,w)$ is the degree of $a_{01}$, $a_{12}$, $a_{23}$, $a_{02}$, or $a_{13}$;
and $m' \geq m$; and $m' - t < m - r$; then every differential
\[
d_t: E_t^{(m',s,f,w)} \map E_t^{(m-t+1,s-1,f+1,w)}
\]
is zero.
\item
The subbracket $\langle \alpha_0, \alpha_1, \alpha_2 \rangle$
has no indeterminacy.
\item
the indeterminacy of
$\langle \alpha_1, \alpha_2, \alpha_3 \rangle$ is generated
by elements
of the form $\alpha_1 \beta$ and $\gamma \alpha_3$,
where $\beta$ and $\gamma$ are detected in
May filtrations strictly lower than the May filtrations of 
$a_{23}$ and $a_{12}$ respectively.
\end{enumerate}
Then $a_0 a_{13} + a_{01} a_{23} + a_{02} a_3$ detects an element of
$\langle \alpha_0, \alpha_1, \alpha_2, \alpha_3 \rangle$ in $\Ext$.
\end{thm}

The point of conditions (1) and (2) is that 
the bracket $\langle a_0, a_1, a_2, a_3 \rangle$ is defined in the
differential graded algebra $(E_r, d_r)$.
Condition (3) is an equivalent reformulation of condition (*) in
\cite{May69}*{Theorem 4.1}.
The point of this condition is
that there are no later ``crossing"
differentials whose source has higher May filtration
and whose target has strictly lower May filtration.

Condition (5) does not appear in \cite{May69}, which only deals
with strictly defined brackets.
\index{Massey product!indeterminacy}
\index{Massey product!strictly defined}
Of course, the theorem has a symmetric version in which the bracket
$\langle \alpha_1, \alpha_2, \alpha_3 \rangle$ has no indeterminacy.
It is probably possible to state a version of the theorem in which both
threefold subbrackets have non-zero indeterminacy.  However, additional
conditions are required for such a 
fourfold bracket to be well-defined \cite{Isaksen14b}.

\begin{proof}
Let $C$ be the cobar resolution of the motivic Steenrod algebra whose
homology is $\Ext$.
\index{cobar resolution}
Let $\tilde{\alpha}_0$, $\tilde{\alpha}_1$, $\tilde{\alpha}_2$, 
and $\tilde{\alpha_3}$ be explicit
cycles in $C$ representing $\alpha_0$, $\alpha_1$, $\alpha_2$, and $\alpha_3$.
As in the proof of \cite{May69}*{Theorem 4.1}, we may choose
an element 
$\tilde{\alpha}_{01}$ of $C$ such that $d(\tilde{\alpha}_{01}) = \tilde{\alpha}_0 \tilde{\alpha}_1$,
and $\tilde{\alpha}_{01}$ is detected by $a_{01}$ in the May $E_r$-page.
We may similarly choose $\tilde{\alpha}_{12}$ and $\tilde{\alpha}_{23}$
whose boundaries are $\tilde{\alpha}_1 \tilde{\alpha}_2$ and $\tilde{\alpha}_2 \tilde{\alpha}_3$ and
that are detected by $a_{12}$ and $a_{23}$.

Next, we want to choose $\tilde{\alpha}_{13}$ in $C$ whose boundary
is $\tilde{\alpha}_1 \tilde{\alpha}_{23} + \tilde{\alpha}_{12} \tilde{\alpha}_3$ and that is
detected by $a_{13}$.
Because of the possible indeterminacy 
in $\langle \alpha_1, \alpha_2, \alpha_3 \rangle$,
the cycle
$\tilde{\alpha}_1 \tilde{\alpha}_{23} + \tilde{\alpha}_{12} \tilde{\alpha}_3$ may not be a boundary in $C$.
However, since we are assuming that
$\langle \alpha_0, \alpha_1, \alpha_2, \alpha_3 \rangle$ is defined,
we can add cycles to $\tilde{\alpha}_{12}$ and $\tilde{\alpha}_{23}$
to ensure that
$\tilde{\alpha}_1 \tilde{\alpha}_{23} + \tilde{\alpha}_{12} \tilde{\alpha}_3$ is a boundary.
When we do this, condition (5) guarantees that
$\tilde{\alpha}_{12}$ and $\tilde{\alpha}_{23}$ are still detected
by $a_{12}$ and $a_{23}$.
Then we may choose $\tilde{\alpha}_{13}$ as in the proof of \cite{May69}*{Theorem 4.1}.

Finally, we may choose $\tilde{\alpha}_{02}$ as in the proof
of \cite{May69}*{Theorem 4.1}.  Because
$\langle \alpha_0, \alpha_1, \alpha_2 \rangle$ has no indeterminacy,
we automatically know that $\tilde{\alpha}_0 \tilde{\alpha}_{12} + \tilde{\alpha}_{01} \tilde{\alpha}_2$
is a boundary in $C$.
\end{proof}

For completeness, we will now also state May's Convergence Theorem
for fivefold brackets.  This result is used only in
Lemma \ref{lem:c0-i1}.
The proof is essentially the same as the proof for fourfold brackets.

\index{Convergence Theorem!May}
\begin{thm}[May's Convergence Theorem]
\label{thm:5-converge}
Let $\alpha_0$, $\alpha_1$, $\alpha_2$, $\alpha_3$, and $\alpha_4$ 
be elements of $\Ext$ such that
the Massey product $\langle \alpha_0, \alpha_1, \alpha_2, \alpha_3, \alpha_4 \rangle$
is defined.
For each $i$, let $a_i$ be a permanent cycle on the $E_r$-page
that detects $\alpha_i$.
Suppose further that:
\begin{enumerate}
\item
there are elements $a_{01}$, $a_{12}$, $a_{23}$, and $a_{34}$ in $E_r$
such that $d_r(a_{01}) = a_0 a_1$, $d_r (a_{12}) = a_1 a_2$, 
$d_r (a_{23}) = a_2 a_3$, and $d_r(a_{34}) = a_3 a_4$.
\item
there are elements $a_{02}$, $a_{13}$, and $a_{24}$ in $E_r$ such that
$d_r (a_{02}) = a_0 a_{12} + a_{01} a_2$,
$d_r (a_{13}) = a_1 a_{23} + a_{12} a_3$, and
$d_r (a_{24}) = a_2 a_{34} + a_{23} a_4$.
\item
there are elements $a_{03}$ and $a_{14}$ in $E_r$ such that
$d_r (a_{03}) = a_0 a_{13} + a_{01} a_{23} + a_{02} a_3$ and
$d_r (a_{14}) = a_1 a_{24} + a_{12} a_{34} + a_{13} a_4$.
\item
if $(m,s,f,w)$ is the degree of $a_{01}$, $a_{12}$, $a_{23}$, $a_{34}$,
$a_{02}$, $a_{13}$, $a_{24}$, $a_{03}$, or $a_{14}$;
$m' \geq m$; and $m' - t < m - r$; then every differential
$d_t: E_t^{(m',s,f,w)} \map E_t^{(m-t+1,s-1,f+1,w)}$
is zero.
\item
the threefold subbrackets 
$\langle \alpha_0, \alpha_1, \alpha_2 \rangle$,
$\langle \alpha_1, \alpha_2, \alpha_3 \rangle$, and
$\langle \alpha_2, \alpha_3, \alpha_4 \rangle$
have no indeterminacy.
\item
the subbracket
$\langle \alpha_0, \alpha_1, \alpha_2, \alpha_3 \rangle$
has no indeterminacy.
\item
the indeterminacy of
$\langle \alpha_1, \alpha_2, \alpha_3, \alpha_4 \rangle$ is generated
by elements contained in
$\langle \beta, \alpha_3, \alpha_4 \rangle$,
$\langle \alpha_1, \gamma, \alpha_4 \rangle$, and
$\langle \alpha_1, \alpha_2, \delta \rangle$,
where $\beta$, $\gamma$, and $\delta$ are detected in
May filtrations strictly lower than the May filtrations
of $a_{12}$, $a_{23}$, and $a_{34}$ respectively.
\end{enumerate}
Then $a_0 a_{14} + a_{01} a_{24} + a_{02} a_{34} + a_{03} a_4$
detects an element of
$\langle \alpha_0, \alpha_1, \alpha_2, \alpha_3, \alpha_4 \rangle$ in $\Ext$.
\end{thm}

Although we will use May's Convergence Theorem to compute most of the Massey
For a few Massey products, we also need occasionally the following result
\cite{Hirsch55} \cite{Adams60}*{Lemma 2.5.4}.

\begin{prop}
\label{prop:Massey-hn}
Let $x$ be an element of $\Ext_A(\M_2, \M_2)$.
\begin{enumerate}
\item
If $h_0 x = 0$, then $\tau h_1 x$ belongs to $\langle h_0, x, h_0 \rangle$.
\item
If $n \geq 1$ and $h_n x = 0$, then $h_{n+1} x$ belongs to 
$\langle h_n, x, h_n \rangle$.
\end{enumerate}
\end{prop}

We will need the following results about shuffling higher brackets
that are not strictly defined.
\index{Massey product!shuffle}

\begin{lemma}
\label{lem:4fold-shuffle}
Suppose that $\langle \alpha_0, \alpha_1, \alpha_2, \alpha_3 \rangle$ and
$\langle \alpha_1, \alpha_2, \alpha_3, \alpha_4 \rangle$ are defined and that
the indeterminacy of $\langle \alpha_0, \alpha_1, \alpha_2 \rangle$ consists
of multiples of $\alpha_2$.
Then
\[
\alpha_0 \langle \alpha_1, \alpha_2, \alpha_3, \alpha_4 \rangle \subseteq
\langle \alpha_0, \alpha_1, \alpha_2, \alpha_3 \rangle \alpha_4.
\]
\end{lemma}

\begin{proof}
For each $i$, choose an element $a_i$ that represents $\alpha_i$.
Let $\beta$ be an element of 
$\alpha_0 \langle \alpha_1, \alpha_2, \alpha_3, \alpha_4 \rangle$.
There exist elements 
$a_{12}$, $a_{23}$, $a_{34}$,
$a_{13}$, and $a_{24}$ such that
$d(a_{i,i+1}) = a_i a_{i+1}$,
$d(a_{i,i+2}) = a_i a_{i+1,i+2} + a_{i,i+1} a_{i+2}$,
and $\beta$ is represented by
\[
b = a_0 a_1 a_{24} + a_0 a_{12} a_{34} +
a_0 a_{13} a_4. 
\]

By the assumption on the indeterminacy of 
$\langle \alpha_0, \alpha_1, \alpha_2 \rangle$, 
we can then choose $a_{01}$ and $a_{02}$ such that
$d(a_{01}) = a_0 a_1$ and
$d(a_{02}) = a_0 a_{12} + a_{01} a_2$.
Then
\[
a_0 a_{13} a_4 + a_{01} a_{23} a_4 +
a_{02} a_3 a_4
\]
represents a class in 
$\langle \alpha_0, \alpha_1, \alpha_2, \alpha_3 \rangle \alpha_4$
that is homologous to $b$.
\end{proof}

\begin{lemma}
\label{lem:4fold-shuffle2}
Suppose that $\langle \alpha_0, \alpha_1, \alpha_2, \alpha_3 \rangle$ and
$\langle \alpha_1, \alpha_2, \alpha_3, \alpha_4 \rangle$
are defined, and suppose that
$\langle \alpha_1, \alpha_2, \alpha_3 \rangle$ is strictly zero.
Then
\[
\alpha_0 \langle \alpha_1, \alpha_2, \alpha_3, \alpha_4 \rangle \cap
\langle \alpha_0, \alpha_1, \alpha_2, \alpha_3 \rangle \alpha_4
\]
is non-empty.
\end{lemma}

\begin{proof}
Choose elements $a_i$ that represent $\alpha_i$.
Choose $a_{01}$, $a_{12}$, and $a_{02}$ such that
$d(a_{01}) = a_0 a_1$,
$d(a_{12}) = a_1 a_2$, and
$d(a_{02}) = a_0 a_{12} + a_{01} a_2$.
Also,
choose $a_{23}$, $a_{34}$, and $a_{24}$ such that
$d(a_{23}) = a_2 a_3$,
$d(a_{34}) = a_3 a_4$, and
$d(a_{24}) = a_2 a_{34} + a_{23} a_4$.
Since $\langle \alpha_1, \alpha_2, \alpha_3 \rangle$ is strictly zero,
there exists $a_{13}$ such that
$d(a_{13}) = a_1 a_{23} + a_{12} a_3$.

Then
\[
a_0 a_1 a_{24} + a_0 a_{12} a_{34} + a_0 a_{13} a_4
\]
represents an element of 
$\alpha_0 \langle \alpha_1, \alpha_2, \alpha_3, \alpha_4 \rangle$, and it
is homologous to
\[
a_0 a_{13} a_4 + a_{01} a_{23} a_4 + a_{02} a_3 a_4,
\]
which represents an element of
$\langle \alpha_0, \alpha_1, \alpha_2, \alpha_3 \rangle \alpha_4$.
\end{proof}

\begin{lemma}
\label{lem:5fold-shuffle}
Suppose that $\langle \alpha_0, \alpha_1, \alpha_2, \alpha_3, \alpha_4 \rangle$ and
$\langle \alpha_1, \alpha_2, \alpha_3, \alpha_4, \alpha_5 \rangle$
are defined and that
$\langle \alpha_0, \alpha_1, \alpha_2, \alpha_3 \rangle$ is strictly zero.
Then
\[
\alpha_0 \langle \alpha_1, \alpha_2, \alpha_3, \alpha_4, \alpha_5 \rangle \subseteq
\langle \alpha_0, \alpha_1, \alpha_2, \alpha_3, \alpha_4 \rangle \alpha_5.
\]
\end{lemma}

\begin{proof}
The proof is essentially the same as the proof of
Lemma \ref{lem:4fold-shuffle}.
\end{proof}

\section{The May differentials}
\label{sctn:May-diff}

\subsection{The May $E_1$-page}

\index{May spectral sequence!E1-page@$E_1$-page}
The $E_2$-page of the May spectral sequence is the cohomology
of a differential graded algebra.  In other words, the May spectral sequence really
starts with an $E_1$-page.
As described in \cite{DI10},
the motivic $E_1$-page is essentially the same as the classical
$E_1$-page.  Specifically, the motivic $E_1$-page is 
a polyonomial algebra over $\M_2$ with
generators $h_{ij}$ for all $i > 0$ and $j \geq 0$,
where:
\begin{enumerate}
\item
$h_{i0}$ has degree $(i, 2^i-2, 1, 2^{i-1} - 1 )$.
\item
$h_{ij}$ has degree $(i, 2^j ( 2^i - 1) - 1, 1, 2^{j-1} ( 2^i - 1) )$
for $j > 0$.
\end{enumerate}
The $d_1$-differential is described by the formula:
\[
d_1 ( h_{ij} ) = \sum_{0 < k < i } h_{kj} h_{i-k,k+j}.
\]

\subsection{The May $E_2$-page}

\index{May spectral sequence!E2-page@$E_2$-page}
We now describe the $E_2$-page of the motivic
May spectral sequence.  As explained in \cite{DI10},
it turns out that the motivic $E_2$-page is essentially
the same as the classical $E_2$-page.  The following
proposition makes this precise.

Recall that $\Gr(A)$ is the associated graded object of the motivic
Steenrod algebra with respect to powers of the augmentation ideal.
Similarly,
let $\Gr(A_{\cl})$ be the associated graded object of the classical
Steenrod algebra with respect to powers of the augmentation ideal.

\begin{prop}[\cite{DI10}]
\label{prop:assoc-graded}
There are graded ring isomorphisms
\begin{enumerate}[(a)]
\item 
$\Gr(A) \iso \Gr(A_{\cl})\otimes_{\F_2} \F_2[\tau]$.
\item 
$\Ext_{\Gr(A)}(\M_2,\M_2) \iso 
\Ext_{\Gr(A_{\cl})}(\F_2,\F_2)\otimes_{\F_2} \M_2$.
\end{enumerate}
\end{prop}

In other words, explicit generators and relations for
the $E_2$-page can be lifted directly from the
classical situation \cite{Tangora70a}.  

Moreover,
because of Proposition \ref{prop:may-compare},
the values of the May $d_2$ differential
\index{May spectral sequence!differential!d2@$d_2$}
can also be lifted from the classical
situation,
except that a few factors of $\tau$ show up to give the necessary weights.
For example, classically we have the differential
\index{b20@$b_{20}$}
\[
d_2 ( b_{20} ) = h_1^3 + h_0^2 h_2.
\]
Motivically, this does not make sense, since
$b_{20}$ and $h_0^2 h_2$ have weight $2$, while $h_1^3$ has weight $3$.
Therefore, the motivic differential must be
\[
d_2 ( b_{20} ) = \tau h_1^3 + h_0^2 h_2.
\]

Table \ref{tab:May-E2-gen} lists the multiplicative
generators of the $E_2$-page
through the 70-stem, and 
Table \ref{tab:May-E2-reln} lists a generating set of relations 
for the $E_2$-page in the same range.
Table \ref{tab:May-E2-gen} also gives the values of the May $d_2$ differential,
all of which are easily deduced from the classical situation
\cite{Tangora70a}.

\subsection{The May $E_4$-page}

\index{May spectral sequence!E4-page@$E_4$-page}
Although the $E_2$-page is quite large, the May $d_2$ differential
is also very destructive.  As a result, the $E_4$-page becomes
manageable.
We obtain the $E_4$-page by 
direct computation with the $d_2$ differential.

\begin{remark}
\index{B@$B$}
As in \cite{Tangora70a}, we use the notation
$B = b_{30} b_{31} + b_{21} b_{40}$.
\end{remark}

Having described the $E_4$-page, it is now necessary to
find the values of the May $d_4$ differential on the multiplicative
generators.
\index{May spectral sequence!differential!d4@$d_4$}
Most of the values of $d_4$ follow
from comparison to the classical case \cite{Tangora70a},
together with a few factors of $\tau$ to balance the weights.
There is only one differential that is not classical.

\begin{lemma}
\label{lem:d4-g}
$d_4 ( g ) = h_1^4 h_4$.
\index{g@$g$}
\end{lemma}

\begin{proof}
By the isomorphism of
Theorem \ref{thm:Chow-0},
we know that $h_1^4 h_4$ cannot survive the
motivic May spectral sequence because $h_0^4 h_3$ 
is zero classically.  
There is only one possible differential
that can kill $h_1^4 h_4$.

See also \cite{DI10} for a different proof of Lemma \ref{lem:d4-g}.
\end{proof}

Table \ref{tab:May-E4}
lists the values of the $d_4$ differential on multiplicative
generators of the $E_4$-page.

\subsection{The May $E_6$-page}

\index{May spectral sequence!E6-page@$E_6$-page}
We can now obtain the $E_6$-page
by direct computation with the May $d_4$ differential and the Leibniz rule.

Having described the $E_6$-page, it is now necessary to
find the values of the May $d_6$ differential on the multiplicative
generators.
\index{May spectral sequence!differential!d6@$d_6$}
Most of these values follow
from comparison to the classical case \cite{Tangora70a},
together with a few factors of $\tau$ to balance the weights.
There are only a few differentials that are not classical.

\begin{lemma}
\mbox{}
\begin{enumerate}
\item
$d_6 ( x_{56} ) = h_1^2 h_5 c_0 d_0$.
\item
$d_6 ( P x_{56} ) = P h_1^2 h_5 c_0 d_0$.
\item
$d_6 ( B_{23} ) = h_1^2 h_5 d_0 e_0$.
\end{enumerate}
\end{lemma}

\begin{proof}
We have the relation $h_1 x_{56} = c_0 \phi$.
The $d_6$ differential on $\phi$ then implies that
$d_6 ( h_1 x_{56} ) = h_1^3 h_5 c_0 d_0$,
from which it follows that 
$d_6 ( x_{56} ) = h_1^2 h_5 c_0 d_0$.

The arguments for the other two differentials are similar,
using the relations
$h_1 \cdot P x_{56} = P c_0 \cdot \phi$
and
$h_1 B_{23} = e_0 \phi$.
\end{proof}

\begin{lemma}
\label{lem:d6-c0g^3}
$d_6 ( c_0 g^3 ) = h_1^{10} D_4$.
\end{lemma}

\begin{proof}
Lemma \ref{lem:c0-i1} shows that
$c_0 \cdot i_1 = h_1^4 D_4$.
Since $h_1^6 i_1 = 0$, we conclude that $h_1^{10} D_4$ must
be zero in $\Ext$.
There is only one possible differential that can hit $h_1^{10} D_4$.
\end{proof}

\begin{remark}
The value of $d_6 ( \Delta h_0^2 Y )$ given in
\cite{Tangora70a}*{Proposition 4.37(c)} is incorrect because it is 
inconsistent with machine computations of $\Ext_{A_{\cl}}$ \cite{Bruner97}.  
The value for $d_6 (\Delta h_0^2 Y)$
given in Table \ref{tab:May-E6} is the only possibility that is consistent 
with the machine computations.
\end{remark}

Table \ref{tab:May-E6}
lists the values of the May $d_6$ differential on multiplicative
generators of the $E_6$-page.


\subsection{The May $E_8$-page}

\index{May spectral sequence!E8-page@$E_8$-page}
We can now obtain the May $E_8$-page
by direct computation with the May $d_6$ differential and the Leibniz rule.
Once we reach the $E_8$-page, we are nearly done.  There are 
just a few more higher differentials to deal with.

Having described the $E_8$-page, it is now necessary to
find the values of the May $d_8$ differential on the multiplicative
generators.
\index{May spectral sequence!differential!d8@$d_8$}
Once again, most of these values follow
from comparison to the classical case \cite{Tangora70a},
together with a few factors of $\tau$ to balance the weights.
There are only a few differentials that are not classical.

\begin{lemma}
\mbox{}
\begin{enumerate}
\item
$d_8 (g^2) = h_1^8 h_5$.
\item
$d_8 (w) = P h_1^5 h_5$.
\item
$d_8 ( \Delta c_0 g) = P h_1^4 h_5 c_0$.
\item
$d_8 ( Q_3 ) = h_1^4 h_5^2$.
\end{enumerate}
\end{lemma}

\begin{proof}
It follows from
Theorem \ref{thm:Chow-0} that
$h_1^8 h_5$ must be zero in $\Ext$, since
$h_0^8 h_4$ is zero classically.
There is only one differential that can possibly
hit $h_1^8 h_5$.

We now know that $P h_1^9 h_5 = 0$ in $\Ext$ since
$h_1^8 h_5 = 0$.  
There is only one differential that can hit this.
This shows that $d_8 ( w ) = P h_1^5 h_5$.

Using the relation $c_0 w = h_1 \cdot \Delta c_0 g$, it follows
that $d_8 ( h_1 \cdot \Delta c_0 g ) = P h_1^5 h_5 c_0$, and then that
$d_8 ( \Delta c_0 g ) = P h_1^4 h_5 c_0$.

Since $h_0^4 h_4^2$ is zero classically, it follows
from Theorem \ref{thm:Chow-0} that
$h_1^4 h_5^2$ must be zero in $\Ext$.
There is only one differential that can possibly hit
$h_1^4 h_5^2$.
\end{proof}

Table \ref{tab:May-E8}
lists the values of the May $d_8$ differential on multiplicative
generators of the $E_8$-page.


\subsection{The May $E_\infty$-page}
\label{subsctn:May-Einfty}

\index{May spectral sequence!differential!higher}
Most of the higher May differentials are zero through the 70-stem.
The exceptions are the May $d_{12}$ differential,
the May $d_{16}$ differential, and the May $d_{32}$ differential.
All of the non-zero values of these differentials are easily
deduced by comparison to the classical case \cite{Tangora70a}.

Table \ref{tab:May-higher} lists the values of these higher
differentials on multiplicative generators of the higher pages.
There are no more differentials to consider in our range,
and we have determined the May $E_\infty$-page.  

\index{May spectral sequence!Einfinity-page@$E_\infty$-page}
The multiplicative generators for the $E_\infty$-page 
through the 70-stem break into two groups.
The first group consists of generators that are still 
multiplicative generators in $\Ext$ 
after hidden extensions have been considered;
these are listed in Table \ref{tab:Ext-gen}.
The second group consists of multiplicative generators of the
$E_\infty$-page that become decomposable in $\Ext$ because
of a hidden extension;
these are listed in Table \ref{tab:May-Einfty-temp}.

\index{Ext@$\Ext$!ambiguous generator}
It is traditional to use the same symbols for elements of 
the $E_\infty$-page
and for the elements of $\Ext$ that they represent.  Generally, there is
no ambiguity with this abuse of notation, but there are several exceptions.
These exceptions occur when a multiplicative generator for 
the $E_\infty$-page
lies in the same degree as another element of 
the $E_\infty$-page with lower May filtration.

The first such example occurs in the 18-stem, where the element 
$f_0$
\index{f0@$f_0$}
 of the $E_\infty$-page represents two elements of $\Ext$ because of the
presence of the element $\tau h_1^3 h_4$ of lower May filtration.
This particular example does not cause much difficulty.  Just arbitrarily
choose one of these elements to be the generator of $\Ext$.  
The element disappears quickly from further analysis because
$f_0$ supports an Adams $d_2$ differential.  

However, later examples involve
more subtlety and call for a careful distinction between the possibilities.
There are no wrong choices, but it is important to be consistent with the 
notation in different arguments.
For example, the element $u'$
\index{u'@$u'$}
 of the $E_\infty$-page represents two elements of 
$\Ext$ because of the presence of $\tau d_0 l$.  One of these elements is 
killed by $\tau$, while the other element is killed by $h_0$.  Sloppy notation
might lead to the false conclusion that there is a multiplicative generator of $\Ext$
in that degree that is killed by both $\tau$ and by $h_0$.

Table \ref{tab:Ext-ambiguous}
lists all such examples of multiplicative generators of the $E_\infty$-page
that represent more than one element in $\Ext$.  In many of these examples, we
have given an algebraic specification of one element of $\Ext$ to serve as the
multiplicative generator, sometimes by comparing
to $\Ext_{A(2)}$ \cite{Isaksen09}.  
In some examples, we have not given a definition
because an algebraic description is not readily available, and also because it
does not seem to matter for later analysis.  The reader is strongly warned to
be cautious when working with these undefined elements.

The example $\tau Q_3$
\index{tauQ3@$\tau Q_3$}
 deserves an additional remark.  Here we
have defined the element in terms of an Adams differential.  This 
is merely a matter of convenience for later work with the Adams spectral
ßsequence in Chapter \ref{ch:Adams-diff}.

\section{Hidden May extensions}
\label{sctn:hidden-extn}

\index{May spectral sequence!hidden extension}
In order to pass from the $E_\infty$-page to $\Ext$,
we must resolve some hidden extensions.
In this section, we deal with all possible hidden extensions
by $\tau$, $h_0$, $h_1$, and $h_2$.
We will use several different tools, including:
\begin{enumerate}
\item
Classical hidden extensions \cite{Bruner97}.
\item
Shuffle relations with Massey products.
\index{Massey product!shuffle}
\item
Steenrod operations in the sense of \cite{May70}.
\index{Steenrod operation!algebraic}
\item
Theorem \ref{thm:Chow-0} for hidden extensions among
elements in degrees $(s,f,w)$ with $s+f-2w=0$.
\end{enumerate}

\subsection{Hidden May $\tau$ extensions}
\label{subsctn:t-hidden}

By exhaustive search, the following results give all of the 
hidden $\tau$ extensions.
\index{May spectral sequence!hidden extension!tau@$\tau$}

\begin{prop}
Table \ref{tab:May-tau} lists all of the hidden
$\tau$ extensions through the 70-stem.
\end{prop}

\begin{proof}
Many of the extensions follow by comparison to the classical case
as described in \cite{Bruner97}.
For example, there is a classical hidden extension
$h_0 \cdot e_0 g = h_0^4 x$.  This implies that
$\tau^2 \cdot h_0 e_0 g = h_0^4 x$ motivically.

Proofs for the more subtle cases are given below.
\end{proof}

\begin{lemma}
\label{lem:t^2-h2g^2}
\mbox{}
\begin{enumerate}
\item
$\tau \cdot \tau h_2 g^2 = P h_1^4 h_5$.
\item
$\tau \cdot \tau h_0 g^3 = P h_1^4 h_5 e_0$.
\end{enumerate}
\end{lemma}

\begin{proof}
Start with the relation $h_1 \cdot \tau g + h_2 f_0 = 0$,
and apply the squaring operation $\Sq^4$.
One needs that $\Sq^3 ( \tau g ) = P h_1^2 h_5$ \cite{Bruner04}.
The result is the first hidden extension.

For the second, multiply the first hidden extension by $e_0$.
\end{proof}

\begin{lemma}
\label{lem:t-B8}
\mbox{}
\begin{enumerate}
\item
$\tau \cdot B_8 = P h_5 d_0$.
\item
$\tau \cdot h_1^2 B_{21} = P h_5 c_0 d_0$.
\item
$\tau \cdot B_8 d_0 = h_0^4 X_3$.
\end{enumerate}
\end{lemma}

\begin{proof}
There is a classical hidden extension
$c_0 \cdot B_1 = P h_1 h_5 d_0$ \cite{Bruner97}.
Motivically, there is a non-hidden relation
$c_0 \cdot B_1 = h_1 B_8$.
It follows that
$\tau \cdot h_1 B_8 = P h_1 h_5 d_0$ motivically.

For the second hidden extension, multiply the
first hidden extension by $c_0$.
Note that $c_0 B_8 = h_1^2 B_{21}$ is
detected in the $E_\infty$-page of the May spectral sequence.

For the third hidden extension, 
multiply the first hidden extension by $d_0$,
and observe that 
$P h_5 d_0^2 = h_0^4 X_3$,
which is detected in the $E_\infty$-page of the May spectral 
sequence.
\end{proof}

\begin{lemma}
\label{lem:tau-Pu'}
\mbox{}
\begin{enumerate}
\item
$\tau \cdot P u' = h_0^5 R_1$.
\item
$\tau \cdot P^2 u' = h_0^9 R$.
\item
$\tau \cdot P^3 u' = h_0^6 R_1'$.
\end{enumerate}
\end{lemma}

\begin{proof}
We first compute that
$\langle \tau, u', h_0^3 \rangle = \{ Q', Q' + \tau Pu \}$.
One might try to apply May's Convergence Theorem \ref{thm:3-converge} with the 
May differential $d_2 (b_{20} b_{30}^3 h_0(1) ) = \tau u'$, 
but condition (2) of the theorem is not satisfied because
of the ``crossing" May differential $d_4 ( P \Delta h_0 h_4 ) = P^2 h_0 h_4^2$.
\index{Convergence Theorem!May}
\index{May spectral sequence!differential!crossing}

Instead, note that $h_0 \cdot u' = \tau h_0 d_0 l$ by comparison to
$\Ext_{A(2)}$ \cite{Isaksen09}, so we have that
$\langle \tau, u', h_0^3 \rangle = \langle \tau, \tau h_0 d_0 l, h_0^2 \rangle$.
The latter bracket is given in Table \ref{tab:Massey}.

Next, Table \ref{tab:Massey} shows that 
$P u' = \langle u', h_0^3, h_0 h_3 \rangle$,
with no indeterminacy.  
Use the previous paragraph and a shuffle to get that
$\tau \cdot P u' = h_0 h_3 Q'$.
Finally, there is a classical hidden extension 
$h_3 \cdot Q' = h_0^4 R_1$ \cite{Bruner97}, which implies that the same
formula holds motivically.

The argument for the second hidden extension is similar,
using the shuffle
\[
\tau \cdot P^2 u' = 
\tau \langle u', h_0^3, h_0^5 h_4 \rangle =
\langle \tau, u', h_0^3 \rangle h_0^5 h_4 =
h_0^5 h_4 Q'.
\]
The first equality comes from Table \ref{tab:Massey}.
Also, we need the classical hidden
extension $h_4 \cdot Q' = h_0^4 R$ \cite{Bruner97}, which implies that the
same formula holds motivically.

The argument for the third hidden extension is also similar,
using the shuffle
\[
\tau \cdot P^3 u' =
\tau \langle u', h_0^3, h_0^3 i \rangle =
\langle \tau, u', h_0^3 \rangle h_0^3 i =
h_0^3 i Q'.
\]
The first equality comes from Table \ref{tab:Massey}.
Also, we need the classical hidden extension 
$i \cdot Q' = h_0^3 R_1'$ \cite{Bruner97}.
\end{proof}

\begin{lemma}
\label{lem:tau-k1}
$\tau \cdot k_1 = h_2 h_5 n$.
\end{lemma}

\begin{proof}
First, Table \ref{tab:Massey} shows that 
$k = \langle d_0, h_3, h_0^2 h_3 \rangle$,
with no indeterminacy.  It follows from \cite{Milgram68} that
$\Sq^0 k = 
\langle \Sq^0 d_0, \Sq^0 h_3, \Sq^0 h_0^2 h_3 \rangle$,
with no indeterminacy.
In other words, 
$\Sq^0 k  = \langle \tau^2 d_1, h_4, \tau^2 h_1^2 h_4 \rangle$.
From the classical calculation \cite{Bruner04},
$\Sq^0 k$ also equals $\tau^3 h_2 h_5 n$.

On the other hand, Table \ref{tab:Massey} show that 
$k_1 = \langle d_1, h_4, h_1^2 h_4 \rangle$,
with no indeterminacy.

This shows that $\tau^4 \cdot k_1 = \tau^3 h_2 h_5 n$
in $\Ext$, from which it follows that
$\tau \cdot k_1 = h_2 h_5 n$.
\end{proof}

\begin{remark}
In the 46-stem, $\tau \cdot u'$ does not equal
$\tau^2 d_0 l$.  
Similarly, in the 49-stem,
$\tau \cdot v'$ does not equal $\tau^2 e_0 l$.
This is true by definition; see Table \ref{tab:Ext-ambiguous}.
\index{u'@$u'$}
\index{v'@$v'$}
\end{remark}

\subsection{Hidden May $h_0$ extensions}
\label{subsctn:h0-hidden}

By exhaustive search, the following results give all of the
hidden $h_0$ extensions.
\index{May spectral sequence!hidden extension!h0@$h_0$}

\begin{prop}
Table \ref{tab:May-h0} lists all of the hidden
$h_0$ extensions through the 70-stem.
\end{prop}

\begin{proof}
Many of the extensions follow by comparison to the classical case
as described in \cite{Bruner97}.
For example, there is a classical hidden extension
$h_0 \cdot r = s$.  This implies that
$h_0 \cdot r = s$ motivically as well.

Several other extensions are implied by the hidden
$\tau$ extensions established in 
Section \ref{subsctn:t-hidden}.
For example, the extensions $\tau \cdot P u' = h_0^4 S_1$
and $\tau \cdot \tau h_0 d_0^2 j = h_0^5 S_1$ imply
that $h_0 \cdot P u' = \tau h_0 d_0^2 j$.

Proofs for the more subtle cases are given below.
\end{proof}

\begin{lemma}
\label{lem:h0-u'}
\mbox{}
\begin{enumerate}
\item
$h_0 \cdot u' = \tau h_0 d_0 l$.
\item
$h_0 \cdot v' = \tau h_0 e_0 l$.
\item
$h_0 \cdot P v' = \tau h_0 d_0^2 k$.
\item
$h_0 \cdot P^2 v' = \tau h_0 d_0^3 i$.
\end{enumerate}
\end{lemma}

\begin{proof}
These follow by comparison to $\Ext_{A(2)}$ \cite{Isaksen09}.
\end{proof}

\begin{lemma}
\label{lem:hidden-h0h2^2g}
\mbox{}
\begin{enumerate}
\item
$h_0 \cdot h_2^2 g = h_1^3 h_4 c_0$.
\item
$h_0 \cdot h_2^2 g^2 = h_1^7 h_5 c_0$.
\item
$h_0 \cdot h_2^2 g^3 = h_1^9 D_4$.
\end{enumerate}
\end{lemma}

\begin{proof}
For the first hidden extension, use the shuffle
\[
h_1^3 h_4 \langle h_1, h_0, h_2^2 \rangle = 
\langle h_1^3 h_4, h_1, h_0 \rangle h_2^2.
\]
Similarly, for the second hidden section, use the shuffle
\[
h_1^7 h_5 \langle h_1, h_0, h_2^2 \rangle = 
\langle h_1^7 h_5, h_1, h_0 \rangle h_2^2.
\]

For the third hidden extension, there is a 
hidden extension $c_0 \cdot i_1 = h_1^4 D_4$
that will be established in Lemma \ref{lem:c0-i1}.
Use this relation to compute that
\[
h_1^9 D_4 = h_1^5 i_1 \langle h_1, h_2^2, h_0 \rangle =
\langle h_1^5 i_1, h_1, h_2^2 \rangle h_0.
\]
Finally, Table \ref{tab:Massey} shows that 
$h_2^2 g^3 = \langle h_1^5 i_1, h_1, h_2^2 \rangle$.
\end{proof}

\begin{lemma}
\label{lem:h0-gr}
\mbox{}
\begin{enumerate}
\item
$h_0 \cdot g r = P h_1^3 h_5 c_0$.
\item
$h_0 \cdot l m = h_1^6 X_1$.
\item
$h_0 \cdot m^2 = h_1^5 c_0 Q_2$.
\end{enumerate}
\end{lemma}

\begin{remark}
The three parts may seem unrelated, but note that
$l m = e_0 g r$ and $m^2 = g^2 r$ on the $E_8$-page of the May spectral sequence.
\end{remark}

\begin{proof}
Table \ref{tab:Massey} shows that 
$e_0 r = \langle \tau^2 g^2, h_2^2, h_0 \rangle$.
Next observe that
\[
h_2 \cdot e_0 r = \langle \tau^2 g^2, h_2^2, h_0 \rangle h_2 =
\langle \tau^2 g^2, h_2^2, h_0 h_2 \rangle =
\langle \tau^2 h_2 g^2, h_2, h_0 h_2 \rangle.
\]
None of these brackets have indeterminacy.

Use the relation $P h_1^4 h_5 = \tau^2 h_2 g^2$ from
Lemma \ref{lem:t^2-h2g^2}
to write
\[
h_2 \cdot e_0 r = \langle Ph_1^4 h_5, h_2, h_0 h_2 \rangle =
P h_1^3 h_5 \langle h_1, h_2, h_0 h_2 \rangle = 
P h_1^3 h_5 c_0.
\]
The last step is to show that $h_2 \cdot e_0 r = h_0 \cdot g r$.
This follows from the calculation
\[
h_0 \cdot g r = h_0 \langle h_1^3 h_4, h_1, r \rangle =
\langle h_0, h_1^3 h_4, h_1 \rangle r = h_2 e_0 \cdot r,
\]
where the brackets are given in Table \ref{tab:Massey}.
This finishes the proof of part (1).

For part (2), we will prove below in Lemma \ref{lem:h1-th1G}
that $h_1^3 X_1 = P h_5 c_0 e_0$.
So we wish to show that $h_0 \cdot l m = P h_1^3 h_5 c_0 e_0$.
This follows immediately from part (1), using that
$l m = e_0 g r$.

The proof of part (3) is similar to the proof of part (1).
First, $l m$ equals $\langle \tau^2 g^3, h_2^2, h_0 \rangle$.
As above, this implies that
$h_2 l m = \langle \tau^2 h_2 g^3, h_2, h_0 h_2 \rangle$.
Now use the (not hidden) relation $\tau^2 h_2 g^3 = h_1^6 Q_2$
to deduce that $h_2 l m = h_1^5 c_0 Q_2$.
The desired formula now follows since $h_2 l = h_0 m$.
\end{proof}

\begin{lemma}
\label{lem:h0-h0^2B22}
$h_0 \cdot h_0^2 B_{22} = P h_1 h_5 c_0 d_0$.
\end{lemma}

\begin{proof}
This follows from the hidden $\tau$ extension
$\tau \cdot h_1^3 B_{21} = P h_1 h_5 c_0 d_0$
that follows from Lemma \ref{lem:t-B8},
together with the relation $\tau h_1^3 = h_0^2 h_2$.
\end{proof}


\subsection{Hidden May $h_1$ extensions}

By exhaustive search, the following results give all of the
hidden $h_1$ extensions.
\index{May spectral sequence!hidden extension!h1@$h_1$}

\begin{prop}
Table \ref{tab:May-h1} lists all of the hidden
$h_1$ extensions through the 70-stem.
\end{prop}

\begin{proof}
Many of the extensions follow by comparison to the classical case
as described in \cite{Bruner97}.
For example, there is a classical hidden extension
$h_1 \cdot x = h_2^2 d_1$.
This implies that there is a motivic hidden
extension $h_1 \cdot x = \tau h_2^2 d_1$.

Proofs for the more subtle cases are given below.
\end{proof}

\begin{lemma}
\label{lem:h1-th1G}
\mbox{}
\begin{enumerate}
\item
$h_1 \cdot \tau h_1 G = h_5 c_0 e_0$.
\item
$h_1 \cdot h_1 B_3 = h_5 d_0 e_0$.
\item
$h_1 \cdot \tau P h_1 G = P h_5 c_0 e_0$.
\item
$h_1 \cdot h_1^2 X_3 = h_5 c_0 d_0 e_0$.
\end{enumerate}
\end{lemma}

\begin{proof}
Table \ref{tab:Massey} shows that 
$\tau h_1 G = \langle h_5, h_2 g, h_0^2 \rangle$.
Shuffle to obtain
\[
h_1 \cdot \tau h_1 G =
\langle h_5, h_2 g, h_0^2 \rangle h_1 =
h_5 \langle h_2 g, h_0^2, h_1 \rangle.
\]
Finally, 
Table \ref{tab:Massey} shows that 
$c_0 e_0 = \langle h_2 g, h_0^2, h_1 \rangle$.
This establishes the first hidden extension.

For the second hidden extension,
Table \ref{tab:Massey} shows that
$\langle h_5 c_0 e_0, h_0, h_2^2 \rangle$ equals
$h_1 h_5 d_0 e_0$.
From part (1), this equals
$\langle \tau h_1^2 G, h_0, h_2^2 \rangle$, which equals
$h_1^2 \langle \tau G, h_0, h_2^2 \rangle$
because there is no indeterminacy.
This shows that $h_5 d_0 e_0$ is divisible by $h_1$.
The only possible hidden extension is
$h_1 \cdot h_1 B_3 = h_5 d_0 e_0$.

For the third hidden extension,
start with the relation $h_1 \cdot \tau P G = P h_1 \cdot \tau G$
because there is no possible hidden relation.
Therefore, using part (1),
\[
h_1^3 \cdot \tau P G = P h_1 \cdot h_1^2 \cdot \tau G 
= P h_1 h_5 c_0 e_0.
\]
It follows that $h_1 \cdot \tau P h_1 G = P h_5 c_0 e_0$.

For the fourth hidden extension,
use part (1) to conclude that
$h_5 c_0 d_0 e_0$ is divisible by $h_1$.
The only possibility is that
$h_1 \cdot h_1^2 X_3 = h_5 c_0 d_0 e_0$.
\end{proof}

\begin{lemma}
\label{lem:h1-h1^2B6} 
$h_1 \cdot h_1^2 B_6 = \tau h_2^2 d_1 g$.
\end{lemma}

\begin{proof}
Table \ref{tab:Massey} shows that 
$d_1 g = \langle d_1, h_1^3, h_1 h_4 \rangle$.
Using the hidden extension
$h_1 \cdot x = \tau h_2^2 d_1$ \cite{Bruner97},
it follows that
$\tau h_2^2 d_1 g = \langle h_1 x, h_1^3, h_1 h_4 \rangle$,
which equals
$h_1 \langle x, h_1^3, h_1 h_4 \rangle$
because there is no indeterminacy.
Therefore, $\tau h_2^2 d_1 g$ is divisible by $h_1$, and 
the only possibility is that
$h_1 \cdot h_1^2 B_6 = \tau h_2^2 d_1 g$.
\end{proof}

\begin{lemma}
\label{lem:h1-h1D11}
$h_1 \cdot h_1 D_{11} = \tau^2 c_1 g^2$.
\end{lemma}

\begin{proof}
Begin by computing that
$h_1 D_{11} = \langle y, h_1^2, h_1^2 h_4 \rangle$,
using May's Convergence Theorem \ref{thm:3-converge}, 
\index{Convergence Theorem!May}
the May differential $d_4(g) = h_1^4 h_4$,
and the relation
$\Delta h_3^2 g = \Delta h_1^2 d_1$.
Also recall the hidden extension
$h_1 \cdot y = \tau^2 c_1 g$, which follows by comparison to the
classical case \cite{Bruner97}.

It follows that
\[
h_1^2 D_{11} = \langle h_1 y, h_1^4, h_4 \rangle =
\langle \tau^2 c_1 g, h_1^4, h_4 \rangle
\]
because there is no indeterminacy.
Finally, Table \ref{tab:Massey} shows that
$\tau^2 c_1 g^2$ equals $\langle \tau^2 c_1 g, h_1^2, h_1^2 h_4 \rangle$.
\end{proof}

\begin{lemma}
\label{lem:h1-C0}
$h_1 \cdot C_0 = 0$.
\end{lemma}

\begin{proof}
The only other possibility is that $h_1 \cdot C_0$ equals
$h_0 h_5 l$.

First compute that $C_0$ belongs to
$\langle h_0 h_3^2, h_0, h_1, \tau h_1 g_2 \rangle$
using May's Convergence Theorem \ref{thm:4-converge} and 
the May differentials $d_4 ( \nu ) = h_0^2 h_3^2$ and 
$d_4 (x_{47} ) = \tau h_1^2 g_2$.
\index{Convergence Theorem!May}
The subbracket
$\langle h_0 h_3^2, h_0, h_1 \rangle$ is strictly zero.
On the other hand, 
the subbracket
$\langle h_0, h_1, \tau h_1 g_2 \rangle$ equals
$\{ 0, \tau h_0 h_2 g_2 \}$.
Condition (5) of May's Convergence Theorem \ref{thm:4-converge} is satisfied because
the May filtration of $\tau h_2 g_2$ is less than the May filtration of $x_{47}$.

Because $\langle h_1, h_0 h_3^2, h_0 \rangle$ is zero,
the hypothesis of Lemma \ref{lem:4fold-shuffle} is satisfied.
This implies that 
$h_1 \cdot C_0$ belongs to
$\langle h_1, h_0 h_3^2, h_0, h_1 \rangle \tau h_1 g_2$.
For degree reasons, the bracket
$\langle h_1, h_0 h_3^2, h_0, h_1 \rangle$
consists of elements spanned by $f_0$ and $\tau h_1^3 h_4$.
But the products $h_1 \cdot f_0$ and $h_1 \cdot \tau h_1^3 h_4$
are both zero, so
$\langle h_1, h_0 h_3^2, h_0, h_1 \rangle \tau h_1 g_2$ must be zero.
Therefore, $h_1 \cdot C_0$ is zero.
\end{proof}

\begin{lemma}
\label{lem:h1-r1}
$h_1 \cdot r_1 = s_1$.
\end{lemma}

\begin{proof}
This follows immediately from Theorem \ref{thm:Chow-0}
and the classical relation $h_0 \cdot r = s$.
\end{proof}

\begin{lemma}
\label{lem:h1-h1^2q1}
$h_1 \cdot h_1^2 q_1 = h_0^4 X_3$.
\end{lemma}

\begin{proof}
Apply $\Sq^6$ to the relation
$h_2 r = h_1 q$ to obtain that
$h_3 r^2 = h_1^2 \Sq^5 (q)$.
Next, observe that $Sq^5 ( q ) = h_1 q_1$
by comparison to the classical case \cite{Bruner04}.

By comparison to the classical case, there is a relation
$h_3 r = h_0^2 x + \tau h_2^2 n$, so
$h_3 r^2 = h_0^2 r x$.
Finally, use the hidden extension $h_0 \cdot r = s$
and the non-hidden relation $s x = h_0^3 X_3$.
\end{proof}

\subsection{Hidden May $h_2$ extensions}

By exhaustive search, the following results give all of the
hidden $h_2$ extensions.
\index{May spectral sequence!hidden extension!h2@$h_2$}

\begin{prop}
Table \ref{tab:May-h2} lists all of the hidden
$h_2$ extensions through the 70-stem.
\end{prop}

\begin{proof}
Many of the extensions follow by comparison to the classical case
as described in \cite{Bruner97}.
For example, there is a classical hidden extension
$h_2 \cdot Q_2 = h_5 k$.  This implies that the same formula
holds motivically.

Also, many extensions are implied by hidden $h_0$ extensions
that we already established in Section \ref{subsctn:h0-hidden}.
For example, there is a hidden extension
$h_0 \cdot h_2^2 g = h_1^3 h_4 c_0$.
This implies that there is also a hidden extension
$h_2 \cdot h_0 h_2 g = h_1^3 h_4 c_0$.

Proofs for the more subtle cases are given below.
\end{proof}

\begin{remark}
\label{rem:h2-e0r}
We established the extensions
\begin{enumerate}
\item
$h_2 \cdot e_0 r = P h_1^3 h_5 c_0$
\item
$h_2 \cdot l m = h_1^5 c_0 Q_2$
\end{enumerate}
in the proof of Lemma \ref{lem:h0-gr}.
The extension
$h_2 \cdot k m = h_1^6 X_1$ follows from
Lemma \ref{lem:h0-gr} and the relation
$h_2 k = h_0 l$.
\end{remark}

\begin{lemma}
\label{lem:h2-h2B2}
$h_2 \cdot h_2 B_2 = h_1 h_5 c_0 d_0$.
\end{lemma}

\begin{proof}
Table \ref{tab:Massey} shows that 
$h_2 B_2 = \langle g_2, h_0^3, h_2^2 \rangle$, with no indeterminacy.
Then $h_2 \cdot h_2 B_2$ equals 
$\langle g_2, h_0^3, h_2^3 \rangle$, 
because there is no indeterminacy.
This bracket equals
$\langle g_2, h_0^3, h_1^2 h_3 \rangle$,
which equals
$\langle g_2, h_0^3, h_1 \rangle h_1 h_3$ 
since there is no indeterminacy.
Table \ref{tab:Massey} also shows that 
the bracket $\langle g_2, h_0^3, h_1 \rangle$
equals $B_1$.

We have now shown that $h_2 \cdot h_2 B_2$ equals
$h_1 h_3 \cdot B_1$.
It remains to show that there is a hidden extension
$h_3 \cdot B_1 = h_5 c_0 d_0$.
First observe that
$B_1 \cdot h_1^2 d_0 = h_1^3 B_{21}$ by a non-hidden relation.
This implies that 
$B_1 \cdot \tau h_1^2 d_0 = P h_1 h_5 c_0 d_0$ by 
Lemma \ref{lem:t-B8}.

Now there is a hidden extension
$h_3 \cdot P h_1 = \tau h_1^2 d_0$, so
$B_1 \cdot h_3 \cdot P h_1 = P h_1 h_5 c_0 d_0$.
The only possibility is that
$h_3 \cdot B_1 = h_5 c_0 d_0$.
\end{proof}

\begin{lemma}
\label{lem:h2-B6}
$h_2 \cdot B_6 = \tau e_1 g$.
\end{lemma}

\begin{proof}
Table \ref{tab:Massey} shows that 
$\langle \tau, B_6, h_1^2 h_3 \rangle = h_2 C_0$
with no indeterminacy.  
This means that
$\langle \tau, B_6, h_2^3 \rangle = h_2 C_0$.
If $h_2 \cdot B_6$ were zero, then this would imply that
$\langle \tau, B_6, h_2 \rangle h_2^2 = h_2 C_0$.
However, $h_2 C_0$ cannot be divisible by $h_2^2$.
\end{proof}

\subsection{Other hidden May extensions}

We collect here a few miscellaneous extensions that are needed for
various arguments.

\index{May spectral sequence!hidden extension!c0@$c_0$}
\index{May spectral sequence!hidden extension!Ph1@$Ph_1$}

\begin{lemma}
\label{lem:c0-i1}
\mbox{}
\begin{enumerate}
\item
$c_0 \cdot i_1 = h_1^4 D_4$.
\item
$P h_1 \cdot i_1 = h_1^5 Q_2$.
\item
$c_0 \cdot Q_2 = P D_4$.
\end{enumerate}
\end{lemma}

\begin{proof}
Start by computing that
$h_1^2 D_4$ belongs to $\langle c_0, h_4^2, h_3, h_1^3, h_1 h_3 \rangle$;
we will not need to worry about the indeterminacy.
One can use May's Convergence Theorem \ref{thm:5-converge} 
and the May $d_2$ differential
to make this computation. 
\index{Convergence Theorem!May}
 All of the threefold subbrackets are 
strictly zero, and one of the fourfold subbrackets is also strictly
zero.  However, $\langle c_0, h_4^2, h_3, h_1^3 \rangle$ equals
$\{ 0, h_1^2 h_5 e_0 \}$.
Condition (7) of May's Convergence Theorem \ref{thm:5-converge} is satisfied because
$h_1^2 h_5 e_0 = \langle h_5 c_0, h_3, h_1^3 \rangle$, and
the May filtration of $h_5 c_0$ is less than the May filtration of
$h_1 h_0(1,3)$.

The hypothesis of Lemma \ref{lem:5fold-shuffle} is satisfied
because $\langle h_3, h_1^3, h_1 h_3, h_1^2 \rangle$ is strictly zero.
Therefore, $h_1^4 D_4$ is contained in
$c_0 \langle h_4^2, h_3, h_1^3, h_1 h_3, h_1^2 \rangle$.
The main point is that $h_1^4 D_4$ is divisible by $c_0$.
The only possibility is that $c_0 \cdot i_1 = h_1^4 D_4$.
This establishes the first formula.

For the second formula,
compute that 
$h_1 Q_2$ equals $\langle h_4, h_1^2 h_4, h_4, P h_1 \rangle$ 
with no indeterminacy,
using May's Convergence Theorem \ref{thm:4-converge}
\index{Convergence Theorem!May}
 and the May differentials
$d_4(\nu_1) = h_1^2 h_4^2$ and
 $d_4( \Delta h_1 ) = P h_1 h_4$.
The subbracket $\langle h_1^2 h_4, h_4, P h_1 \rangle$ equals
$\{ 0, P h_1^3 h_5 \}$.
Condition (5) of May's Convergence Theorem \ref{thm:4-converge}
is satisfied because the May filtration of $h_1^2 h_5$ is less than
the May filtration of $\nu_1$.

Next, compute that 
$i_1 = \langle h_1^4, h_4, h_1^2 h_4, h_4 \rangle$ 
with no indeterminacy, using May's Convergence Theorem \ref{thm:4-converge}
and the May differentials
$d_4(g) = h_1^4 h_4$ and
$d_4(\nu_1) = h_1^2 h_4^2$.
The subbracket
$\langle h_1^4, h_4, h_1^2 h_4\rangle$ equals $\{ 0, h_1^6 h_5 \}$.
Condition (5) of May's Convergence Theorem \ref{thm:4-converge}
is satisfied because the May filtration of $h_1^2 h_5$ is less than
the May filtration of $\nu_1$.

The hypothesis of Lemma \ref{lem:4fold-shuffle2} is satisfied because
$\langle h_4, h_1^2 h_4, h_4 \rangle$ is strictly zero.
Therefore,
\[
h_1^4 \langle h_4, h_1^2 h_4, h_4, P h_1 \rangle =
\langle h_1^4, h_4, h_1^2 h_4, h_4 \rangle P h_1,
\]
and $h_1^5 Q_2 = P h_1 \cdot i_1$.
This establishes the second formula.

The third formula now follows easily.
Compute that $P h_1 \cdot c_0 \cdot i_1$ equals 
$P h_1 \cdot h_1^4 D_4$ and also $c_0 \cdot h_1^5 Q_2$.
\end{proof}

\begin{remark}
Part (3) of Lemma \ref{lem:c0-i1}
shows that the multiplicative generator $P D_4$ of the $E_\infty$-page
becomes decomposable in $\Ext$ by a hidden extension.
\end{remark}

\begin{lemma}
\label{lem:c0-B6}
$c_0 \cdot B_6 = h_1^3 B_3$.
\end{lemma}

\index{May spectral sequence!hidden extension!c0@$c_0$}

\begin{proof}
Table \ref{tab:Massey} shows that 
$h_1^3 Q_2 = \langle \tau, B_6, h_1^4 \rangle$.
This bracket has no indeterminacy.
It follows that
$h_1^3 c_0 Q_2 = \langle \tau, B_6 \cdot c_0, h_1^4 \rangle$,
since this bracket also has no indeterminacy.

The element $h_1^3 c_0 Q_2$ is non-zero by part (3) of 
Lemma \ref{lem:c0-i1}.
Therefore, 
$\langle h_1^4, B_6 \cdot c_0, \tau \rangle$ is not zero,
so $B_6 \cdot c_0$ is non-zero.
The only possibility is that it equals $h_1^3 B_3$.
\end{proof}

\begin{lemma}
\label{lem:c0-G3}
$c_0 \cdot G_3 = P h_1^3 h_5 e_0$.
\end{lemma}

\begin{proof}
Start with the relation $h_1^2 G_3 = h_2 g r$.  This implies
that $h_1^2 d_0 G_3 = h_2 d_0 g r$, which equals $h_1^6 X_1$ 
by Table \ref{tab:May-h2}.
Therefore, $c_0^2 G_3$ is non-zero, which means that $c_0 G_3$ is also.
The only possibility is that $c_0 G_3$ equals $P h_1^3 h_5 e_0$.
\end{proof}

\begin{lemma}
\label{lem:g2'}
$h_0^2 B_4 + \tau h_1 B_{21} = g_2'$.
\end{lemma}

\begin{proof}
On the $E_\infty$-page, there is a relation
$h_0^2 B_4 + \tau h_1 B_{21} = 0$.
The hidden extension follows from 
the analogous classical hidden relation \cite{Bruner97}.
\end{proof}

\index{May spectral sequence!hidden extension!compound}

\begin{remark}
Through the 70-stem,
Lemma \ref{lem:g2'} is the only example
of a hidden relation of the form
$h_0 \cdot x + h_1 \cdot y$,
$h_0 \cdot x + h_2 \cdot y$, or
$h_1 \cdot x + h_2 \cdot y$.
\end{remark}


%% file: stable-stems-Adams-diff.tex
\chapter{Differentials in the Adams spectral sequence}
\label{ch:Adams-diff}

\setcounter{thm}{0}

The main goal of this chapter is to compute the differentials
in the motivic Adams spectral sequence.
\index{Adams spectral sequence!differential}
We will rely heavily on the computation of the Adams $E_2$-page carried
out in Chapter \ref{ch:May}.
We will borrow results from the classical Adams spectral sequence where
necessary. 
Tables \ref{tab:diff-refs} and \ref{tab:bracket-refs}
summarize previously established results about the classical Adams spectral
sequence, including differentials and Toda brackets.
The tables give specific references to proofs.  The main sources are
\cite{BJM84}, \cite{BMT70}, \cite{Bruner84}, \cite{MT67},
\cite{Tangora70b}, and \cite{Toda62}.

The Adams charts in \cite{Isaksen14a} are essential companions to this chapter.

\subsection*{The motivic Adams spectral sequence}

We refer to \cite{DI10}, \cite{HKO11}, and \cite{Morel99} for background on 
the construction and convergence of the 
motivic Adams spectral sequence over $\C$.
\index{Adams spectral sequence!convergence}
In this section, we review just enough to proceed with
our computations in later sections.

\index{Adams spectral sequence}
\index{Adams spectral sequence!differential}
\begin{thm}[ \cite{DI10} \cite{HKO11} \cite{Morel99} ]
The motivic Adams spectral sequence takes the form
\[
E_2^{s,f,w} = \Ext_A^{s,f,w} ( \M_2, \M_2 ) 
\Rightarrow
\pi_{s,w},
\]
with differentials of the form
$d_r: E_r^{s,f,w} \map E_r^{s-1,f+r,w}$.
\end{thm}

We will need to compare the
motivic Adams spectral sequence to the classical
Adams spectral sequence.  
The following proposition is implicit in \cite{DI10}*{Sections 3.2 and 3.4}.

\begin{prop}
\label{prop:compare}
After inverting $\tau$,
the motivic Adams spectral sequence becomes isomorphic to
the classical Adams spectral sequence tensored over $\F_2$
with $\M_2[\tau^{-1}]$.
\end{prop}
\index{Adams spectral sequence!classical}
\index{tau@$\tau$!localization}

In particular, Proposition \ref{prop:compare} implies that
motivic differentials and motivic hidden extensions must be
compatible with their classical analogues.  
This comparison will be a key tool.

\subsection*{Outline}

 A critical ingredient
is Moss's Convergence Theorem \cite{Moss70},
\index{Convergence Theorem!Moss}
which allows the computation of Toda brackets in $\pi_{*,*}$
via the differentials in the Adams spectral sequence.
We will thoroughly review this result in 
Section \ref{sctn:Adams-bracket}.
\index{Toda bracket}

Section \ref{sctn:differentials} describes the main points in
establishing the Adams differentials.  We postpone the numerous technical
lemmas to Section \ref{sctn:d-lemmas}.

Chapter \ref{ch:table} contains a series of tables that summarize the essential
computational facts in a concise form.
Tables \ref{tab:Ext-gen}, \ref{tab:Adams-d3}, \ref{tab:Adams-d4},
and \ref{tab:Adams-d5} give the values of the motivic Adams differentials.
The fourth columns of these tables refer to one argument that
establishes each differential, which is not necessarily the first known proof.
This takes one of the following forms:
\begin{enumerate}
\item
An explicit proof given elsewhere in this manuscript.
\item
``image of $J$" means that the differential is easily
deducible from the structure of the image of $J$ \cite{Adams66}.
index{J@$J$!image of}
\item
``$\tmf$ " means that the differential can be detected
in the Adams spectral sequence for $\tmf$ \cite{Henriques07}.
\index{topological modular forms}
\item
``Table \ref{tab:diff-refs}" means that the differential is easily deduced
from the analogous classical result.
\item
``\cite{BMMS86}*{VI.1}" means that the differential can be computed
using the relationship between algebraic Steenrod
operations and Adams differentials.
\index{Steenrod operation!algebraic}
\end{enumerate}

Table \ref{tab:Toda} summarizes some calculations of
Toda brackets.
\index{Toda bracket}
In all cases, we have been careful to describe the indeterminacies
accurately.
\index{Toda bracket!indeterminacy}
The fifth column refers to an argument for establishing
this differential, in one of the following forms:
\begin{enumerate}
\item
An explicit proof given elsewhere in this manuscript. 
\item
A Massey product (which appears in Table \ref{tab:Massey})
implies the Toda bracket via Moss's Convergence Theorem \ref{thm:Moss} with $r=2$.
\index{Massey product}
\index{Convergence Theorem!Moss}
\item
An Adams differential implies the Toda bracket via Moss's Convergence Theorem
\ref{thm:Moss} with $r >2$.
\end{enumerate}
The last column of Table \ref{tab:Toda} lists the specific results
that rely on each Toda bracket.


\section{Toda brackets in the motivic Adams spectral sequence}
\label{sctn:Adams-bracket}

We will frequently compute Toda brackets in the motivic stable homotopy
groups in order to resolve
hidden extensions and to determine Adams differentials. 
\index{Toda bracket}
The absolutely essential tool for
computing such Toda brackets
is Moss's Convergence Theorem \cite{Moss70}*{Theorem 1.2}.
\index{Convergence Theorem!Moss}
The point of this theorem is that under certain hypotheses,
Toda brackets can be computed via Massey products in the $E_r$-page of the motivic
Adams spectral sequence. 
\index{Massey product}
For the reader's convenience, 
we will state the Convergence Theorem in the specific forms that we will use.

The $E_2$-page of the motivic Adams spectral sequence possesses Massey products,
since it equals the cohomology of the motivic Steenrod algebra.
Moreover, since $(E_r, d_r)$ is a differential graded algebra for $r \geq 2$,
the $E_{r+1}$-page of the motivic Adams spectral sequence also possesses
Massey products that are computed with the Adams $d_r$ differential.
When necessary for clarity,
we will use the notation $\langle a_0, \ldots, a_n \rangle_{E_{r+1}}$
to refer to Massey products in the $E_{r+1}$-page in this sense.
Similarly, $\langle a_0, \ldots, a_n \rangle_{E_2}$ indicates
a Massey product in $\Ext$.

\index{Convergence Theorem!Moss}
\begin{thm}[Moss's Convergence Theorem]
\label{thm:Moss}
Let $\alpha_0$, $\alpha_1$, and $\alpha_2$ be elements of the motivic
stable homotopy groups such that
the Toda bracket $\langle \alpha_0, \alpha_1, \alpha_2 \rangle$
is defined.
Let $a_i$ be a permanent cycle on the Adams $E_r$-page that detects
$\alpha_i$ for each $i$.
Suppose further that:
\begin{enumerate}
\item
the Massey product $\langle a_0, a_1, a_2 \rangle_{E_{r}}$ is defined
(in $\Ext$ when $r=2$, or 
using the Adams $d_{r-1}$ differential when $r \geq 3$).
\item
if $(s,f,w)$ is the degree of either $a_0 a_1$ or $a_1 a_2$;
$f' < f - r + 1$; $f'' > f$; and $t = f''-f'$; then every 
Adams differential 
$d_{t} : E_{t}^{(s+1,f',w)} \map E_{t}^{(s,f'',w)}$
is zero.
\end{enumerate}
Then
$\langle a_0, a_1, a_2 \rangle_{E_{r}}$ contains a permanent cycle
that detects an element of the Toda bracket
$\langle \alpha_0, \alpha_1, \alpha_2 \rangle$.
\end{thm}

Condition (2) is an equivalent reformulation of condition (1.3) in
\cite{Moss70}*{Theorem 1.2}.  
When computing $\langle a_0, a_1, a_2 \rangle$, one uses a differential
$d_{r-1}: E_r^{(s-1,f-r+1,w)} \map E_r^{(s,f,w)}$.
The idea of condition (2) is that there are no later 
``crossing" differentials
$d_t$ whose source has strictly lower Adams filtration and whose target
has strictly higher Adams filtration.
\index{Adams spectral sequence!differential!crossing}

\begin{ex}
\index{h4@$h_4$}
\index{eta4@$\eta_4$}
Consider the differential $d_2(h_4) = h_0 h_3^2$.
This shows that $\langle \eta, 2, \sigma^2 \rangle$ intersects
$\{ h_1 h_4 \}$.  In fact, Table \ref{tab:Toda} shows that
the bracket equals $\{ h_1 h_4 \} = \{ \eta_4, \eta_4 + \eta \rho_{15} \}$.
\end{ex}

\begin{ex}
Consider the Massey product $\langle h_2, h_3, h_0^2 h_3 \rangle$.
Using the May differential $d_4(\nu) = h_0^2 h_3^2$ and
May's Convergence Theorem \ref{thm:3-converge},
\index{Convergence Theorem!May}
this Massey product contains $f_0$ with indeterminacy $\tau h_1^3 h_4$.
\index{f0@$f_0$}
However, this calculation tells us nothing about
the Toda bracket $\langle \nu, \sigma, 4 \sigma \rangle$.
The presence of the later Adams differential
$d_3( h_0 h_4) = h_0 d_0$ means that
condition (2) of Moss's Convergence Theorem \ref{thm:Moss} is not satisfied.
\index{Adams spectral sequence!differential!crossing}
\index{Convergence Theorem!Moss}
\end{ex}

\begin{ex}
Consider the Toda bracket $\langle \theta_4, 2, \sigma^2 \rangle$.
The relation $h_4^3 + h_3^2 h_5 = 0$ and 
the Adams differentials $d_2( h_5) = h_0 h_4^2$ and
$d_2(h_4) = h_0 h_3^2$ show that the expression 
$\langle h_4^2, h_0, h_3^2 \rangle_{E_3}$ is zero.
This implies that $\langle \theta_4, 2, \sigma^2 \rangle$
consists entirely of elements of Adams filtration strictly greater than 3.
In particular, the Toda bracket is disjoint from $\{ h_3^2 h_5 \}$.
See Lemma \ref{lem:bracket-theta4-2-sigma^2} for more discussion of 
this Toda bracket.
\end{ex}

One case of Moss's Convergence Theorem \ref{thm:Moss} says that Massey products in 
$\Ext_A(\M_2,\M_2)$
are compatible with Toda brackets in $\pi_{*,*}$, assuming that there are
no interfering Adams differentials.  Thus, we will use many Massey products 
in $\Ext_A(\M_2,\M_2)$, 
most of which are computed using 
May's Convergence Theorem \ref{thm:3-converge}.
\index{Convergence Theorem!May}
\index{Convergence Theorem!Moss}
\index{Toda bracket}
\index{Massey product}

We will also need the following lemma.

\begin{lemma}
\label{lem:2-bracket}
If $2 \alpha$ is zero, then $\tau \eta \alpha$ belongs to
$\langle 2, \alpha, 2 \rangle$.
\end{lemma}

\begin{proof}
The motivic case follows immediately from the classical case,
which is proved in \cite{Toda62}.
\end{proof}

\subsection{Toda brackets and cofibers}
\label{subsctn:Toda-cofiber}

\index{Toda bracket}
\index{stable stem!cofiber}
The purpose of this section is to establish a relationship
between Toda brackets of the form $\langle \alpha_0, \ldots, \alpha_n \rangle$
and properties of the stable homotopy groups of the cofiber $C\alpha_0$ of 
$\alpha_0$.  This relationship is well-known to those who use it.
See \cite{Toda62}*{Proposition 1.8} for essentially the same result.

Suppose given a map $\alpha_0: S^{p,q} \map S^{0,0}$.
Then we have a cofiber sequence
\[
\xymatrix@1{
S^{p,q} \ar[r]^{\alpha_0} & S^{0,0} \ar[r]^j & C\alpha_0 \ar[r]^-q & 
S^{p+1,q} \ar[r]^{\alpha_0} & S^{1,0}
}
\]
where $j$ is the inclusion of the bottom cell, and $q$ is projection
onto the top cell.
Note that $\pi_{*,*} (C\alpha_0)$ is a $\pi_{*,*}$-module.

\begin{prop}
\label{prop:3bracket-cofiber}
Let $\alpha_0$, $\alpha_1$, and $\alpha_2$ be elements of $\pi_{*,*}$ such that
$\alpha_0 \alpha_1$ and $\alpha_1 \alpha_2$ are zero.  
Let $\ol{\alpha_1}$ be an element of $\pi_{*,*}(C\alpha_0)$
such that $q_*(\ol{\alpha_1}) = \alpha_1$.
In $\pi_{*,*} (C \alpha_0)$, the element
$\ol{\alpha_1} \cdot \alpha_2$ belongs to 
$j_* ( \langle \alpha_0, \alpha_1, \alpha_2 \rangle)$.
\end{prop}

\begin{proof}
The proof is described by the following diagram.
The composition $\ol{\alpha_1} \alpha_2$ can be lifted
to $S^{0,0}$ because $\alpha_1 \alpha_2$ was assumed to be zero.
This shows that $\ol{\alpha_1} \cdot \alpha_2$
is equal to $j_*(\beta)$.
Finally, $\beta$ is one possible
definition of the Toda bracket $\langle \alpha_0, \alpha_1, \alpha_2 \rangle$.
\[
\xymatrix{
& & S^{0,0} \ar[d]^j \\
& & C \alpha_0 \ar[d]^q \\
S^{*,*} \ar[r]_{\alpha_2} \ar@{-->}[rruu]^\beta & 
   S^{*,*} \ar[r]_{\alpha_1}  \ar[ur]^{\ol{\alpha_1}} &
  S^{p+1,q} \ar[r]_{\alpha_0} & S^{1,0} \\
}
\]
\end{proof}

\begin{remark}
\label{rem:Massey-3bracket-cofiber}
We have presented Proposition \ref{prop:3bracket-cofiber} in the context
of stable motivic homotopy groups, but the proof works in 
the much greater generality of a stable model category.
\index{model category}
For example, the same result holds for Massey products, where one works 
in the derived category of a graded algebra $A$, and maps
correspond to elements of $\Ext$ groups over $A$.
\end{remark}

\begin{remark}
\label{rem:higher-bracket-cofiber}
Proposition \ref{prop:3bracket-cofiber} 
can be generalized to higher compositions.
Suppose that $\langle \alpha_0, \ldots, \alpha_n \rangle$ is 
defined.
Then the bracket $\langle \ol{\alpha_1}, \alpha_2, \ldots, \alpha_n \rangle$
is contained in $j_* (\langle \alpha_0, \ldots, \alpha_n \rangle)$.
The proof is similar to the proof of Proposition \ref{prop:3bracket-cofiber},
using the definition of higher Toda brackets \cite{Shipley02}*{Appendix A}.
\end{remark}

\section{Adams differentials}
\label{sctn:differentials}

The $E_2$-page of the motivic Adams spectral sequence
is described in Chapter \ref{ch:May} (see also \cite{DI10}). 
\index{Adams spectral sequence!E2-page@$E_2$-page} 
See \cite{Isaksen14a} for a chart of the $E_2$-page through the 70-stem.
\index{Adams chart}
A list of multiplicative generators for the $E_2$-page 
is given in Table \ref{tab:Ext-gen}.

Our next task is to compute the Adams differentials.
The main point is to compute the Adams $d_r$ differentials on
the multiplicative generators of the $E_r$-page.  Then one can compute
the entire Adams $d_r$ differential using that $d_r$ is a derivation.

\subsection{Adams $d_2$ differentials}
\label{subsctn:Adams-d2}

\index{Adams spectral sequence!differential!d2@$d_2$}
Most of the Adams $d_2$ differentials are lifted directly from
the classical situation,
in the sense of Proposition \ref{prop:compare}.
We provide a few representative
examples of this phenomenon.

\begin{ex}
\label{ex:d2-1}
\index{h4@$h_4$}
The classical differential $d_2 (h_4) = h_0 h_3^2$
immediately implies that there is a motivic differential
$d_2 (h_4) = h_0 h_3^2$.
\end{ex}

\begin{ex}
\label{ex:d2-2}
\index{d0@$d_0$}
\index{e0@$e_0$}
Unlike the classical situation, 
the elements $h_1^k d_0$ and $h_1^k e_0$ are non-zero
in the $E_2$-page for all $k \geq 0$.
The classical differential $d_2 ( e_0 ) = h_1^2 d_0$
implies that there is a motivic differential
$d_2 ( e_0 ) = h_1^2 d_0$, from which it follows that
$d_2 ( h_1^k e_0 ) = h_1^{k+2} d_0$ for all $k \geq 0$.
Technically, these are ``exotic" differentials, although
we will soon see subtler examples.
\end{ex}

\begin{ex}
\label{ex:d2-3}
\index{c2@$c_2$}
\index{e1@$e_1$}
Consider the classical differential
$d_2( h_0 c_2 ) = h_1^2 e_1$.  Motivically, this formula
does not make sense because the weights of
$h_0 c_2$ and $h_1^2 e_1$ are $22$ and $23$ respectively.
It follows that there is a motivic differential
$d_2 ( h_0 c_2 ) = \tau h_1^2 e_1$.
Then $h_1^2 e_1$ is non-zero on the $E_3$-page.
\end{ex}

\index{Adams spectral sequence!differential!d2@$d_2$}
\begin{prop}
\label{prop:Adams-d2}
Table \ref{tab:Ext-gen} lists some values of the motivic Adams $d_2$
differential.  The motivic Adams $d_2$ differential is zero
on all other multiplicative generators of the $E_2$-page, through
the 70-stem.
\end{prop}

\begin{proof}
Table \ref{tab:Ext-gen}
cites one possible argument (but not necessarily the earliest
published result) for each non-zero differential on a multiplicative
generator of the $E_2$-page.  These arguments break into several types:
\begin{enumerate}
\item
Some differentials are consequences of the image of $J$ calculation 
\cite{Adams66}.
\index{J@$J$!image of}
\item
Some differentials follow by comparison to the Adams spectral sequence
for $\tmf$ \cite{Henriques07}.
\index{topological modular forms}
\item
Some differentials follow by comparison to an analogous classical result.
\index{Adams spectral sequence!differential!classical}
\item
One differential follows from the relationship
between Adams differentials and algebraic Steenrod operations
\cite{BMMS86}*{VI.1}.
\index{Steenrod operation!algebraic}
\item
The remaining differentials are proved in Section \ref{subsctn:d2-lemmas}.
\end{enumerate}

For the differentials whose values are zero, 
Section \ref{subsctn:d2-lemmas} includes proofs 
for the cases that are not obvious.
\end{proof}

In order to maintain the flow of the narrative, we have collected the
technical computations of miscellaneous $d_2$ differentials
in Section \ref{subsctn:d2-lemmas}.

The $E_2$ chart in \cite{Isaksen14a} indicates the Adams $d_2$
differentials, all of which are implied by the calculations 
in Table \ref{tab:Ext-gen}.
\index{Adams chart}

\begin{remark}
\index{h3g@$h_3 g$}
\index{h3g2@$h_3 g^2$}
\index{h3g3@$h_3 g^3$}
Lemma \ref{lem:d2-h3g} establishes three differentials
$d_2 ( h_3 g ) = h_0 h_2^2 g$,
$d_2 ( h_3 g^2 ) = h_0 h_2^2 g^2$, and
$d_2 ( h_3 g^3 ) = h_0 h_2^2 g^3$.
Presumably there is an infinite family of exotic differentials of the
form 
\[
d_2 ( h_3 g^k ) = h_0 h_2^2 g^k.
\]
\end{remark}

\begin{remark}
\index{X1@$X_1$}
\index{B4@$B_4$}
\index{B21@$B_{21}$}
The differential $d_2 (X_1) = h_0^2 B_4 + \tau h_1 B_{21}$ is inconsistent
with the results of \cite{KM93}.  
\end{remark}

\begin{remark}
\index{D1@$D_1$}
\index{g2@$g_2$}
In the 51-stem,
we draw particular attention to the Adams differential
$d_2(D_1) = h_0^2 h_3 g_2$.
Mark Mahowald
\index{Mahowald, Mark}
privately communicated an argument for the
presence of this differential to the author.  However, this argument fails 
because of the calculation of the Toda bracket 
\index{Toda bracket}
$\langle \theta_4, 2, \sigma^2 \rangle$ in
Lemma \ref{lem:bracket-theta4-2-sigma^2}, which was unknown to Mahowald.
Zhouli Xu
\index{Xu, Zhouli}
and the author
discovered an independent proof, which is included in 
Lemma \ref{lem:d2-D1}.
See \cite{IX14} for a full discussion.
\end{remark}

\begin{remark}
\label{rem:d2-tQ3}
\index{Q3@$Q_3$}
As noted in Table \ref{tab:Ext-ambiguous},
the element $\tau Q_3$ is defined in $\Ext$ such that $d_2 (\tau Q_3) = 0$.
\index{Ext@$\Ext$!ambiguous generator}
\end{remark}

\begin{remark}
\index{topological modular forms}
Quite a few of the $d_2$ differentials in this section follow by comparison
to the Adams spectral sequence for $\tmf$, i.e., the Adams spectral sequence
whose $E_2$-page is the cohomology of the subalgebra $A(2)$ of the Steenrod algebra.
See \cite{Henriques07} for detailed computations with this spectral sequence.

\index{motivic modular forms}
Presumably, there is a ``motivic modular forms" spectrum that is the motivic
analogue of $\tmf$.  If such a motivic spectrum existed, then the
$E_2$-page of its Adams spectral sequence would be the cohomology of 
motivic $A(2)$, as described in \cite{Isaksen09}.  Such a spectral sequence would
help significantly in calculating the differentials in the 
motivic Adams spectral sequence for $S^{0,0}$ that we are considering here.
\end{remark}


\subsection{Adams $d_3$ differentials}
\label{subsctn:Adams-d3}
See \cite{Isaksen14a} for a chart of the $E_3$-page.
This chart is complete through the 70-stem;
however, the Adams $d_3$ differentials are complete only through 
the 65-stem.
\index{Adams chart}
\index{Adams spectral sequence!E3-page@$E_3$-page}

The next step is to compute the Adams $d_3$ differential
on the multiplicative generators of the $E_3$-page.
\index{Adams spectral sequence!differential!d3@$d_3$}

\begin{prop}
\label{prop:Adams-d3}
Table \ref{tab:Adams-d3} lists some values of the motivic Adams $d_3$
differential.  The motivic Adams $d_3$ differential is zero
on all other multiplicative generators of the $E_3$-page, through
the 65-stem, except that 
$d_3(D_3)$ might equal $B_3$.
\index{D3@$D_3$}
\index{B3@$B_3$}
\end{prop}

\begin{proof}
Table \ref{tab:Adams-d3} 
cites one possible argument (but not necessarily the earliest
published result) for each non-zero differential on a multiplicative
generator of the $E_3$-page.  These arguments break into several types:
\begin{enumerate}
\item
Some differentials are consequences of the image of $J$ calculation 
\cite{Adams66}.
\index{J@$J$!image of}
\item
Some differentials follow by comparison to the Adams spectral sequence
for $\tmf$ \cite{Henriques07}.
\index{topological modular forms}
\item
Some differentials follow by comparison to an analogous classical result.
\index{Adams spectral sequence!differential!classical}
\item
The remaining differentials are proved in Section \ref{subsctn:d3-lemmas}.
\end{enumerate}

For the differentials whose values are zero, 
Section \ref{subsctn:d3-lemmas} includes proofs 
for the cases that are not obvious.
\end{proof}

In order to maintain the flow of the narrative, we have collected the technical
computations of miscellaneous $d_3$ differentials in Section \ref{subsctn:d3-lemmas}.

The $E_3$ chart in \cite{Isaksen14a} indicates the Adams $d_3$ differentials,
all of which are implied by the calculations in Table \ref{tab:Adams-d3}.
The differentials are complete only through the 65-stem.
Beyond the 65-stem, there are a number of unknown differentials.
\index{Adams chart}

\begin{remark}
\label{rem:d3-D3}
\index{D3@$D_3$}
\index{B3@$B_3$}
\index{B1@$B_1$}
\index{Toda bracket}
\index{theta4.5@$\theta_{4.5}$}
\index{eta4@$\eta_4$}
The chart in \cite{KM93} indicates a differential $d_3(D_3) = B_3$.
However, we have been unable to independently verify this differential.
Because of the relation $h_1 B_3 = h_4 B_1$ and
because $\{ B_1 \}$ contains $\eta \theta_{4.5}$,
we know that
$h_1 B_3$ detects $\langle \eta \theta_{4.5}, \sigma^2, 2 \rangle$,
as shown in Table \ref{tab:Toda}.
It follows that
$B_3$ detects $\langle \theta_{4.5}, \sigma^2, 2 \rangle$ and that
$h_1 B_3$ detects $\eta_4 \theta_{4.5}$.
We have so far been unable to show that either
$\langle \theta_{4.5}, \sigma^2, 2 \rangle$ or
$\eta_4 \theta_{4.5}$ is zero.
\end{remark}

\begin{remark}
\index{h5e0@$h_5 e_0$}
\index{B1@$B_1$}
We draw attention to the differential
$d_3( h_1 h_5 e_0) = h_1^2 B_1$.
This can be derived from its classical analogue, which is
carefully proved in \cite{BJM84}.  
Lemma \ref{lem:d3-h1h5e0} provides an independent proof.
This proof
originates
from an algebraic hidden extension in the $h_1$-local cohomology
of the motivic Steenrod algebra \cite{GI14}.
\index{h1@$h_1$!localization}
\end{remark}

\begin{remark}
\index{Q2@$Q_2$}
\index{gt@$g t$}
The differential $d_3(Q_2) = \tau^2 g t$ given in
Lemma \ref{lem:d3-C0} is inconsistent with the chart in \cite{KM93}.
We do not understand the source of this discrepancy.
\end{remark}

\begin{remark}
\index{r1@$r_1$}
\index{X2@$X_2$}
We claim that $d_3(r_1)$ is zero; this is tentative because our analysis
is incomplete in the relevant range.
The only other possibility is that $d_3(r_1)$ equals $h_1^2 X_2$.
However, we show in Lemma \ref{lem:tau-h1^2X2}
that $h_1^2 X_2$ supports a hidden $\tau$ extension and must therefore
be non-zero on the $E_\infty$-page.
\end{remark}


\subsection{Adams $d_4$ differentials}
\label{subsctn:Adams-d4}
See \cite{Isaksen14a} for a chart of the $E_4$-page.
This chart is complete through the 65-stem.
Beyond the 65-stem, because of unknown earlier differentials,
the actual $E_4$-page is a subquotient of what is shown
in the chart.
\index{Adams chart}
\index{Adams spectral sequence!E4-page@$E_4$-page}

The next step is to compute the Adams $d_4$ differentials on the
multiplicative generators of the $E_4$-page.
\index{Adams spectral sequence!differential!d4@$d_4$}

\begin{prop}
\label{prop:Adams-d4}
Table \ref{tab:Adams-d4} lists some values of the motivic Adams $d_4$
differential.  The motivic Adams $d_4$ differential is zero
on all other multiplicative generators of the $E_4$-page, through
the 65-stem, with the possible exceptions that:
\begin{enumerate}
\item
$d_4(\tau h_1 X_1)$ or $d_4(R)$ might equal $\tau^2 d_0 e_0 r$.
\index{X1@$X_1$}
\index{R@$R$}
\index{d0e0r@$d_0 e_0 r$}
\item
$d_4(C')$ or $d_4(\tau X_2)$ might equal $h_2 B_{21}$ or $\tau h_2 B_{21}$
respectively.
\index{C'@$C'$}
\index{X2@$X_2$}
\index{B21@$B_{21}$}
\end{enumerate}
\end{prop}

\begin{proof}
Table \ref{tab:Adams-d4}
cites one possible argument (but not necessarily the earliest
published result) for each non-zero differential on a multiplicative
generator of the $E_4$-page.  These arguments break into several types:
\begin{enumerate}
\item
\index{J@$J$!image of}
Some differentials are consequences of the image of $J$ calculation 
\cite{Adams66}.
\item
Some differentials follow by comparison to the Adams spectral sequence
for $\tmf$ \cite{Henriques07}.
\index{topological modular forms}
\item
Some differentials follow by comparison to an analogous classical result.
\index{Adams spectral sequence!differential!classical}
\item
The remaining differentials are proved in Section \ref{subsctn:d4-lemmas}.
\end{enumerate}

For the differentials whose values are zero, 
Section \ref{subsctn:d4-lemmas} includes proofs 
for the cases that are not obvious.
\end{proof}

The $E_4$ chart in \cite{Isaksen14a} indicates the Adams $d_4$
differentials, all of which are implied by the calculations 
in Table \ref{tab:Adams-d4}.
The differentials are complete only through the 65-stem.  Beyond the
65-stem, there are a number of unknown differentials.
\index{Adams chart}

\begin{remark}
\label{rem:d4-h1X1}
The chart in \cite{KM93} indicates a 
classical differential $d_4(h_1 X_1) = d_0 e_0 r$.
However, we have been unable to independently verify this differential.
\index{X1@$X_1$}
\index{d0e0r@$d_0 e_0 r$}

Because of the differential $d_5(\tau P h_5 e_0) = \tau d_0 z$
\index{Ph5e0@$P h_5 e_0$}
\index{d0z@$d_0 z$}
from Lemma \ref{lem:d5-tPh5e0}, we strongly suspect that
$\tau^2 d_0 e_0 r$ is hit by some differential, but there is
more than one possibility.

Note that $\tau^2 d_0 e_0 r$ detects $\tau^2 \eta \kappabar^3$.
\end{remark}

\begin{remark}
\label{rem:Adams-d4-C'}
The chart in \cite{KM93} indicates a 
classical differential $d_4(C') = h_2 B_{21}$.
However, we have been unable to independently verify this differential.
\index{C'@$C'$}
\index{B21@$B_{21}$}

Because $B_{21}$ detects $\kappa \theta_{4.5}$,
\index{theta4.5@$\theta_{4.5}$}
we know that $h_2 B_{21}$ detects $\nu \kappa \theta_{4.5}$.
If we could show that $\nu \kappa \theta_{4.5}$ is zero,
then we could conclude that there is a differential
$d_4(C') = h_2 B_{21}$.
\end{remark}


\subsection{Adams $d_5$ differentials}
\label{subsctn:d5}

Because the $d_4$ differentials are relatively sparse, \cite{Isaksen14a} does not
provide a separate chart for the $E_5$-page. 

The next step is to compute the Adams $d_5$ differentials on the multiplicative
generators of the $E_5$-page.
\index{Adams spectral sequence!E5-page@$E_5$-page}
\index{Adams spectral sequence!differential!d5@$d_5$}

\begin{prop}
\label{prop:Adams-d5}
Table \ref{tab:Adams-d5} lists some values of the motivic Adams $d_5$
differential.  The motivic Adams $d_5$ differential is zero
on all other multiplicative generators of the $E_5$-page, through
the 65-stem, with the possible exceptions that:
\begin{enumerate}
\item
$d_5(A')$ might equal $\tau h_1 B_{21}$.
\index{A'@$A'$}
\index{B21@$B_{21}$}
\item
$d_5(\tau h_1 H_1)$ might equal $\tau h_2 B_{21}$.
\index{H1@$H_1$}
\item
$d_5(\tau h_1^2 X_1)$ might equal $\tau^3 d_0^2 e_0^2$.
\index{X1@$X_1$}
\index{d02e02@$d_0^2 e_0^2$}
\end{enumerate}
\end{prop}

\begin{proof}
The differential $d_5( h_0^{22} h_6 ) = P^6 d_0$ 
follows from the calculation of the image of $J$ \cite{Adams66}.
The differential $d_5(h_1 h_6) = 0$ follows from the existence of 
the classical element $\eta_6$ \cite{Mahowald77}.
\index{eta6@$\eta_6$}

The remaining cases are computed in Section \ref{subsctn:d5-lemmas}.
\end{proof}

The chart of the $E_4$-page in \cite{Isaksen14a}
indicates the very few $d_5$ differentials along with the 
$d_4$ differentials.
\index{Adams chart}

\begin{remark}
\label{rem:d5-A'}
The chart in \cite{KM93} indicates a 
classical differential $d_5(A') = h_1 B_{21}$.
\index{A'@$A'$}
\index{B21@$B_{21}$}
\index{theta4.5@$\theta_{4.5}$}
However, we have been unable to independently verify this differential.
Because $B_{21}$ detects $\kappa \theta_{4.5}$,
we know that $h_1 B_{21}$ detects $\eta \kappa \theta_{4.5}$.
We have so far been unable to show that 
$\eta \kappa \theta_{4.5}$ is zero.
\end{remark}

\begin{remark}
\index{H1@$H_1$}
\index{B21@$B_{21}$}
We suspect that $d_5(\tau h_1 H_1)$ equals zero, not $\tau h_2 B_{21}$.
This would follow immediately if we knew that
$d_4(C') = h_2 B_{21}$ 
(see Proposition \ref{prop:Adams-d4} and Remark \ref{rem:Adams-d4-C'}).
\index{C'@$C'$}
\end{remark}

\begin{remark}
We show in Lemma \ref{lem:t^3d0^2e0^2-hit} that
$\tau^3 d_0^2 e_0^2$ is hit by some differential.
\index{d02e02@$d_0^2 e_0^2$}
We suspect that $d_5 (\tau h_1^2 X_1)$ equals $\tau^3 d_0^2 e_0^2$.
\index{X1@$X_1$}
The other possibilities are 
$d_9(\tau X_2)$ and $d_{10} (\tau h_1 H_1)$.
\index{X2@$X_2$}
\index{H1@$H_1$}
\end{remark}


\subsection{Higher Adams differentials}

At this point, we are almost done.
\index{Adams spectral sequence!differential!higher}
\index{Adams spectral sequence!E6-page@$E_6$-page}
\index{Adams spectral sequence!Einfinity-page@$E_\infty$-page}

\begin{prop}
Through the 59-stem, the $E_6$-page equals the $E_\infty$-page.
\end{prop}

\begin{proof}
The only possible higher differential 
is that $d_6(h_5 c_1)$ might equal $P h_1^2 h_5 c_0$.
\index{h5c1@$h_5 c_1$}
\index{Ph5c0@$P h_5 c_0$}
However, we will show in the proof of Lemma \ref{lem:d4-C} that
$P h_1^2 h_5 c_0$ cannot be hit by a differential.
\end{proof}

The calculations of Adams differentials lead
immediately to our main theorem.

\begin{thm}
\label{thm:Adams-Einfty}
The $E_\infty$-page of the motivic Adams spectral sequence
over $\C$ is depicted in the chart in \cite{Isaksen14a}
through the 59-stem.
Beyond the 59-stem, the actual $E_\infty$-page is a subquotient of
what is shown in the chart.
\end{thm}


\section{Adams differentials computations}
\label{sctn:d-lemmas}

In this section, we collect the technical computations that establish
the Adams differentials discussed in Section \ref{sctn:differentials}.

\subsection{Adams $d_2$ differentials computations}
\label{subsctn:d2-lemmas}

\index{Adams spectral sequence!differential!d2@$d_2$}

The first two lemmas establish well-known facts from the classical situation.
However, explicit proofs are not readily available in the literature, so we
supply them here.

\begin{lemma}
\label{lem:d2-e0}
$d_2 ( P^k e_0) = h_1^2 P^k d_0$.
\end{lemma}

\begin{proof}
Because of the relation $2 \kappa = 0$, 
there must be a differential $d_2(\beta) = h_0 d_0$
in the Adams spectral sequence for $\tmf$.
Here $\beta$ is the class in the 15-stem as labeled in \cite{Henriques07}.  
Then $d_2( h_2 \beta ) = h_0^2 e_0$.

Now $f_0$ maps to $h_2 \beta$, so it follows that
$d_2(f_0) = h_0^2 e_0$ in the classical Adams spectral sequence
for the sphere.  The same formula must hold motivically.

The relation $h_0 f_0 = \tau h_1 e_0$ then implies that
$d_2(e_0) = h_1^2 d_0$.  This establishes the formula for $k = 0$.

The argument for larger values of $k$ is similar,
using that $d_2( P^k h_2 \beta ) = P^k h_0^2 e_0$
in the Adams spectral sequence for $\tmf$;
$P^k h_0 j$ maps to $P^{k+1} h_2 \beta$; and
$P^k h_0^2 j = \tau P^{k+1} h_1 e_0$.
\end{proof}

\begin{lemma}
\label{lem:d2-l}
$d_2(l) = h_0 d_0 e_0$.
\end{lemma}

\begin{proof}
The differential $d_2(k) = h_0 d_0^2$ follows by comparison to the
Adams spectral sequence for $\tmf$.  The relation
$h_2 k = h_0 l$ then implies that $d_2(l) = h_0 d_0 e_0$.
\end{proof}

\begin{lemma}
\label{lem:d2-h3g}
\mbox{}
\begin{enumerate}
\item
$d_2 (h_3 g) = h_0 h_2^2 g$.
\item
$d_2 ( h_3 g^2 ) = h_0 h_2^2 g^2$.
\item
$d_2 ( h_3 g^3 ) = h_0 h_2^2 g^3$.
\end{enumerate}
\end{lemma}

\begin{proof}
Table \ref{tab:Adams-misc-extn} indicates that
$\sigma \eta_4$ is contained in $\{ h_4 c_0 \}$.
Therefore, $h_1^3 h_4 c_0$
detects $\eta^3 \sigma \eta_4$.
However, $\eta^3 \eta_4$ is zero.

This means that $h_1^3 h_4 c_0$ must be zero
on the $E_\infty$-page.
The only possible differential is
$d_2 ( h_3 g ) = h_1^3 h_4 c_0$.
Finally, note that $h_1^3 h_4 c_0 = h_0 h_2^2 g$ in
the $E_2$-page.  This establishes the first differential.

The argument for the second differential is essentially the same.
The product $\eta^6 \epsilon \eta_5$ is detected by $h_1^7 h_5 c_0$.
Since $\eta^3 \sigma = \eta^2 \epsilon$,
we get that $\eta^7 \sigma \eta_5$ is also detected by $h_1^7 h_5 c_0$.
However, $\eta^7 \eta_5$ is zero,
so $h_1^7 h_5 c_0$ must be hit by some differential.

For the third differential,
Table \ref{tab:May-misc}
shows that 
$c_0 i_1 = h_1^4 D_4$ on the $E_2$-page.
This implies that
$\eta^5 \epsilon \{i_1\}$ is contained in $\{ h_0 h_2^2 g^3 \}$ in $\pi_{66,40}$.
Using that $\eta^3 \sigma = \eta^2 \epsilon$,
we get that 
$\eta^6 \sigma \{ i_1 \}$ is contained in $\{ h_0 h_2^2 g^3 \}$.
However, $\eta^6 \{ i_1 \}$ equals zero,
so some differential must hit
$h_0 h_2^2 g^3$.
\end{proof}

\begin{lemma}
\label{lem:d2-e0g}
$d_2 ( e_0 g) = h_1^2 e_0^2$.
\end{lemma}

\begin{proof}
First note that $Ph_1 \cdot e_0 g = h_1 d_0^2 e_0 + h_1^4 v$;
this is true in the May $E_\infty$-page.
Now apply $d_2$ to this formula to get
\[
Ph_1 \cdot d_2 ( e_0 g ) = h_1^3 d_0^3 + h_1^6 u.
\]
In particular, it follows that $d_2 (e_0 g)$ is non-zero.
The only possibility is that $d_2(e_0 g) = h_1^2 e_0^2$.
\end{proof}

\begin{lemma}
\label{lem:d2-u'}
\mbox{}
\begin{enumerate}
\item
$d_2 ( u' ) = \tau h_0 d_0^2 e_0$.
\item
$d_2 ( P u' ) = \tau P h_0 d_0^2 e_0$.
\item
$d_2 (P^2 u') = \tau P^2 h_0 d_0^2 e_0$.
\item
$d_2 (P^3 u') = \tau P^3 h_0 d_0^2 e_0$.
\item
$d_2 ( v' ) = h_1^2 u' + \tau h_0 d_0 e_0^2$.
\item
$d_2(P v') = P h_1^2 u' + \tau h_0 d_0^4$.
\item
$d_2(P^2 v') = P^2 h_1^2 u' + \tau P h_0 d_0^4$.
\end{enumerate}
\end{lemma}

\begin{proof}
The first four formulas follow easily from the relations
$h_0 u' = \tau h_0 d_0 l$,
$h_0 \cdot P u' = \tau d_0^2 j$, 
$h_0 \cdot P^2 u' = \tau P h_0 d_0^2 j$, and
$h_0 \cdot P^3 u' = \tau P^2 h_0 d_0^2 j$.

For the fifth formula,
start with the relation $c_0 v = h_1 v'$,
which holds already in the May $E_\infty$-page.
Apply $d_2$ to obtain
$h_1^2 c_0 u = h_1 d_2 ( v')$.
We have
$c_0 u = h_1 u'$ (also from the May $E_\infty$-page),
so 
$h_1 d_2 ( v' ) = h_1^3 u'$.
It follows that $d_2 ( v' )$
equals either $h_1^2 u'$ or
$h_1^2 u' + \tau h_0 d_0 e_0^2$.
Because of the relation
$h_0 v' = \tau h_0 e_0 l$,
it must be the latter.

The proofs of the sixth and seventh formulas are essentially the same,
using the relations
$P h_1 \cdot v' = h_1 \cdot P v'$,
$P^2 h_1 \cdot v' = h_1 \cdot P^2 v'$,
$h_0 \cdot P v' = \tau h_0 d_0^2 k$, and
$h_0 \cdot P^2 v' = \tau h_0 d_0^3 i$.
\end{proof}

\begin{lemma}
\label{lem:d2-G3}
$d_2 ( G_3 ) = h_0 g r$.
\end{lemma}

\begin{proof}
The argument is similar to the 
proof of Lemma \ref{lem:d2-h3g}.

Let $\alpha$ be an element of $\{ P h_1 h_5 \}$
such that
$\eta^3 \alpha$ is contained in $\nu \{ \tau^2 g^2 \}$.
Now $\epsilon \alpha$ is contained in $\{ P h_1 h_5 c_0 \}$.
Using that $\eta^3 \sigma = \eta^2 \epsilon$ from 
Table \ref{tab:Adams-compound-extn},
we get that 
$\eta^2 \epsilon \alpha = \eta^3 \sigma \alpha$,
which is contained in $\nu \sigma \{ \tau^2 g^2 \}$.
This is zero, since $\nu \sigma$ is zero.

This means that $P h_1^3 h_5 c_0 = h_0 g r$ must be zero
on the $E_\infty$-page of the Adams spectral sequence,
but there are several possible differentials.
We cannot have 
$d_2(\tau g n) = h_0 g r$, since $\tau g \cdot n$
is the product of two permanent cycles.
We cannot have
$d_3( h_2 B_2) = h_0 g r$, since we will show later in
Lemma \ref{lem:d3-B2}
that $B_2$ does not support a $d_3$ differential.
We cannot have
$d_4( h_0^2 h_3 g_2) = h_0 g r$,
$d_5( h_0 h_3 g_2) = h_0 g r$, or
$d_6 ( h_3 g_2 ) = h_0 g r$, since we will show later in 
Lemma \ref{lem:d5-g2} that $g_2$ is a 
permanent cycle.

There is just one remaining possibility, so we conclude that
$d_2(G_3) = h_0 g r$.
\end{proof}

\begin{lemma}
\label{lem:Adams-d2-B6}
$d_2(B_6) = 0$.
\end{lemma}

\begin{proof}
The only other possibility is that 
$d_2(B_6)$ equals $h_1 h_5 c_0 d_0$.
If this were the case, then
$d_2 (\ol{B_6})$ would equal $h_1 h_5 \cdot \ol{c_0 d_0}$
in the motivic Adams spectral sequence for the cofiber of $\tau$ 
analyzed in Chapter \ref{ch:Ctau}.
\index{cofiber of tau@cofiber of $\tau$!Adams spectral sequence}
This is impossible because
$h_1^4 \cdot \ol{B_6} = 0$ while
$h_1^5 h_5 \cdot \ol{c_0 d_0}$ is non-zero.
\end{proof}

\begin{lemma}
\label{lem:d2-i1}
$d_2(i_1) = 0$.
\end{lemma}

\begin{proof}
The only other possibility is that $d_2 (i_1) = h_1^4 h_5 e_0$.
However, we will see below in Lemma \ref{lem:d3-h1h5e0} that
$h_1^4 h_5 e_0$ must survive to the $E_3$-page.
\end{proof}

\begin{lemma}
\label{lem:d2-gm}
$d_2 ( g m ) = h_0 e_0^2 g$.
\end{lemma}

\begin{proof}
This follows easily from the 
relation $h_0 g m = h_2 e_0 m$ and 
the differential
$d_2 ( m ) = h_0 e_0^2$.
\end{proof}

\begin{lemma}
\label{lem:d2-Q1}
\mbox{}
\begin{enumerate}
\item
$d_2(Q_1) = \tau h_1^2 x'$.
\item
$d_2(U) = P h_1^2 x'$.
\item
$d_2(R_2) = h_0 U$.
\item
$d_2(G_{11}) = h_0 d_0 x'$.
\end{enumerate}
\end{lemma}

\begin{proof}
First note that $d_2(R_1) = h_0^2 x'$, which follows from the classical
case as shown in Table \ref{tab:diff-refs}.
Then the relation $h_2 R_1 = h_1 Q_1$ implies that
$d_2(Q_1) = \tau h_1^2 x'$.  This establishes the first formula.

Next, there is a relation $\tau h_1 U = P h_1 Q_1$, which is not hidden
in the motivic May spectral sequence.
Therefore, $\tau h_1 d_2(U)$ equals $\tau P h_1^3 x'$.
It follows that $d_2(U)$ equals $P h_1^2 x'$.
This establishes the second formula.

For the third formula, start with the relation
$h_0^2 R_2 = \tau h_1 U$.
This implies that $h_0^2 d_2(R_2)$ equals $\tau P h_1^3 x'$,
which equals $h_0^3 U$.  
Therefore, $d_2(R_2)$ equals $h_0 U$.

For the fourth formula, start with the relation
$h_0 G_{11} = h_2 R_2$.
This implies that $h_0 d_2(G_{11})$ equals $h_0 h_2 U$, which equals
$h_0^2 d_0 x'$.
Therefore, $d_2(G_{11})$ equals $h_0 d_0 x'$.
\end{proof}

\begin{lemma}
\label{lem:d2-D4}
\mbox{}
\begin{enumerate}
\item
$d_2(H_1) = B_7$.
\item
$d_2 (D_4) = h_1 B_6$.
\end{enumerate}
\end{lemma}

\begin{proof}
First note that classically $h_3^2 H_1$ equals $h_4 A'$ \cite{Bruner97}.
Therefore, $h_3^2 d_2(H_1)$ equals
$h_0 h_3^2 A' + h_4 d_2(A')$ classically, which 
equals $h_0 h_3^2 A'$ because $h_4 d_2(A')$ must be zero.
This implies that $d_2(H_1)$ is non-zero classically.
The only motivic possibility is that $d_2(H_1)$ equals $B_7$.
This establishes the first formula.

Next, consider the relation $h_1^2 H_1 = h_3 D_4$,
which is not hidden in the motivic May spectral sequence.
It follows that $ h_3 \cdot d_2(D_4) = h_1^2 B_7$.
The only possibilities are that $d_2(D_4)$ equals $h_1 B_6$ or 
$h_1 B_6 + \tau h_1^2 G$.

Table \ref{tab:May-misc} gives the
hidden extension $c_0 \cdot i_1 = h_1^4 D_4$.
Since $d_2(i_1) = 0$ from Lemma \ref{lem:d2-i1},
it follows that $d_2 (h_1^4 D_4) = 0$.
Then $d_2 (D_4)$ cannot equal $h_1 B_6 + \tau h_1^2 G$
since $h_1^4 \cdot \tau h_1^2 G$ is non-zero.
This establishes the second formula.
\end{proof}

\begin{lemma}
\label{lem:d2-X1}
\mbox{}
\begin{enumerate}
\item
$d_2(X_1) = h_0^2 B_4 + \tau h_1 B_{21}$.
\item
$d_2(G_{21}) = h_0 X_3$.
\item
$d_2(\tau G) = h_5 c_0 d_0$.
\end{enumerate}
\end{lemma}

\begin{proof}
First consider the relation $h_3 R_1 = h_0^2 X_1$ \cite{Bruner97}.
Table \ref{tab:diff-refs} shows that $d_2(R_1) = h_0^2 x'$, so
$h_0^2 d_2(X_1)$ equals $h_0^2 h_3 x'$.
There is another relation $h_0^2 h_3 x' = h_0^4 B_4$, 
which is not hidden in the May spectral sequence.
It follows that $d_2(X_1)$ equals either
$h_0^2 B_4$ or $h_0^2 B_4 + \tau h_1 B_{21}$.

Next consider the relation $h_1^2 X_1 = h_3 Q_1$ \cite{Bruner97}.
We know from Lemma \ref{lem:d2-Q1} that
$d_2(Q_1)$ equals $\tau h_1^2 x'$, so
$h_1^2 d_2(X_1)$ equals $\tau h_1^2 h_3 x'$.
There is another relation $\tau h_1^2 h_3 x' = \tau h_1^3 B_{21}$,
which is not hidden in the May spectral sequence.
It follows that $d_2(X_1)$ equals either $\tau h_1 B_{21}$
or $\tau h_1 B_{21} + h_0^2 B_4$.

Now combine the previous two paragraphs to obtain the first formula.

For the second formula, start with the relation 
$h_0^2 G_{21} = h_3 X_1 + \tau e_1 r$ from \cite{Bruner97}.
Then
$d_2 (h_0^2 G_{21})$ equals $h_3 d_2(X_1) = h_0^2 h_3 B_4$,
which equals $h_0^3 X_3$ \cite{Bruner97}.
The second formula follows.

For the third formula,
Table \ref{tab:May-h1} gives the 
relation $P h_1 \cdot \tau G = h_1^2 X_1$.
The first formula implies that $P h_1 \cdot d_2 ( \tau G ) = \tau h_1^3 B_{21}$, 
which equals $P h_1 h_5 c_0 d_0$ by Table \ref{tab:May-tau}.
\end{proof}

\begin{lemma}
\label{lem:d2-D1}
$d_2(D_1) = h_0^2 h_3 g_2$.
\end{lemma}

\begin{proof}
This proof is due to Z.\ Xu \cite{IX14}.

Start with the Massey product $\tau G = \langle h_1, h_0, D_1 \rangle$.
The higher Leibniz rule \cite{Moss70}*{Theorem 1.1} then implies that
$d_2 (\tau G) = \langle h_1, h_0, d_2 (D_1) \rangle$
because there is no possible indeterminacy.
We showed in Lemma \ref{lem:d2-X1} that
$d_2(\tau G)$ equals $h_5 c_0 d_0$.
This means that $d_2(D_1)$ is non-zero, and the only possibility
is that $d_2(D_1)$ equals $h_0^2 h_3 g_2$.

In fact, note that $h_0^2 h_3 g_2 = h_2^2 h_5 d_0$
and that $h_5 c_0 d_0 = \langle h_1, h_0, h_2^2 h_5 d_0 \rangle$,
but this is not essential for the proof.
\end{proof}

\begin{remark}
\label{rem:bracket-D1}
The proof of Lemma \ref{lem:d2-D1} relies on the Massey product
$\tau G = \langle h_1, h_0, D_1 \rangle$.
One might attempt to prove this with
May's Convergence Theorem \ref{thm:3-converge}
and the May differential $d_2(h_2 b_{22} b_{40} ) = h_0 D_1$.
However, there is a later differential
$d_4 ( \Delta_1 h_1 ) = h_1 h_3 g_2 + h_1 h_5 g$,
so the hypotheses of May's Convergence Theorem \ref{thm:3-converge}
are not satisfied.
\index{Convergence Theorem!May}
\index{May spectral sequence!differential!crossing}

This bracket can be computed via the lambda algebra \cite{IX14}.
Moreover, it has been verified by computer calculation.
\index{lambda algebra}
\index{machine computation}
\end{remark}

\begin{lemma}
\label{lem:d2-D2}
\mbox{}
\begin{enumerate}
\item
$d_2(D_2) = h_0 Q_2$.
\item
$d_2(A) = h_0 B_3$.
\item
$d_2(A'') = h_0 X_2$.
\end{enumerate}
\end{lemma}

\begin{proof}
There is a classical relation $e_0 D_2 = h_0 h_3 G_{21}$ \cite{Bruner97}.
Since $d_2(G_{21}) = h_0 X_3$ by Lemma \ref{lem:d2-X1},
it follows that
$e_0 d_2(D_2)$ equals $h_0^2 h_3 X_3$, which is non-zero.
The only possibilities are that $d_2(D_2)$ equals either
$h_0 Q_2$ or $h_5 j$.

Next, there is a classical relation $i D_2 = 0$ \cite{Bruner97}.
It follows that $i d_2(D_2)$ equals $P h_0 d_0 D_2$, which is non-zero.
The only possibilities are that $d_2(D_2)$ equals
either $h_0 Q_2$ or $h_0 Q_2 + h_5 j$.

We obtain a classical differential
$d_2(D_2) = h_0 Q_2$ by combining the previous two paragraphs.
The same formula must hold motivically.
This establishes the first claim.

For the second claim, use the first claim together with the relations
$h_2 D_2 = h_0 A$ and $h_2 Q_2 = h_0 B_3$.

For the third claim, use the second claim together with the relations
$h_0 A'' = h_2 (A+A')$ and $h_2 B_3 = h_0 X_2$.
\end{proof}

\begin{lemma}
\label{lem:d2-B22}
\mbox{}
\begin{enumerate}
\item
$d_2(B_4) = h_0 B_{21}$.
\item
$d_2 (B_{22}) = h_1^2 B_{21}$.
\end{enumerate}
\end{lemma}

\begin{proof}
There is a relation $P h_2 B_4 = i B_2$,
which is not hidden in the May spectral sequence.
It follows that $P h_2 d_2(B_4)$ equals $P h_0 d_0 B_2$,
which equals $P h_0 h_2 B_{21}$.
Therefore, $d_2(B_4)$ equals $h_0 B_{21}$.
This establishes the first formula.

Now consider the relation
$h_0 h_2 B_4 = \tau h_1 B_{22}$, which is not hidden in the
May spectral sequence.  This implies that
$\tau h_1 d_2(B_{22})$ equals $h_0^2 h_2 B_{21}$,
which equals $\tau h_1^3 B_{21}$.
It follows that $d_2(B_{22})$ equals $h_1^2 B_{21}$.
\end{proof}

\begin{lemma}
\label{lem:Adams-d2-C'}
$d_2(C') = 0$.
\end{lemma}

\begin{proof}
First note that $h_1 C' = \tau d_1^2$ is a permanent cycle.
Therefore, $h_1 d_2(C')$ must equal zero, so
$d_2(C')$ does not equal $h_1^2 B_3$.
\end{proof}

\begin{lemma}
\label{lem:d2-X2}
\mbox{}
\begin{enumerate}
\item
$d_2(X_2) = h_1^2 B_3$.
\item
$d_2(D_3') = h_1 X_3$.
\end{enumerate}
\end{lemma}

\begin{proof}
Note that $h_1 B_3 = h_4 B_1$.
We will show in Lemma \ref{lem:eta-h3^2h5} that
$\{ B_1 \}$ contains $\eta \theta_{4.5}$.
As shown in Table \ref{tab:Toda},
$\{ h_1 B_3 \}$ intersects
the bracket $\langle \eta \theta_{4.5}, 2, \sigma^2 \rangle$.
In fact, $\langle \eta \theta_{4.5}, 2, \sigma^2 \rangle$
is contained in $\{ h_1 B_3 \}$ because all of the possible indeterminacy
is in strictly higher Adams filtration.
This shows that
$\theta_{4.5} \langle \eta, 2, \sigma^2 \rangle$ intersects $\{h_1 B_3\}$.

For the first formula, note that
$\{ h_1^3 B_3 \}$ intersects
$\langle \eta^2 \theta_{4.5}, \eta, 2 \rangle \sigma^2$.
This last expression must be zero for degree reasons.
Therefore, $h_1^3 B_3$ must be killed by some differential.
The only possibility is that 
$d_2(h_1 X_2) = h_1^3 B_3$, which implies that 
$d_2(X_2) = h_1^2 B_3$.

For the second formula, 
note that $h_1^2 X_3 = h_1 c_0 B_3 = h_4 c_0 B_1$.
Lemma \ref{lem:sigma-h1h4} says that
$h_4 c_0$ detects $\sigma \eta_4$.
Therefore, $h_1^2 X_3$ detects
$\eta \sigma \eta_4 \theta_{4.5}$.
The Adams filtration of $\eta_4 \theta_{4.5}$ is at least 8;
the Adams filtration of $\eta \eta_4 \theta_{4.5}$ is at least 11;
and the Adams filtration of $\eta \sigma \eta_4 \theta_{4.5}$ is
at least 12.
Since the Adams filtration of $h_1^2 X_3$ is 11,
it follows that
$h_1^2 X_3$ must be hit by some differential.
The only possibility is that $d_2 (h_1 D_3')$ equals $h_1^2 X_3$.
\end{proof}

\begin{lemma}
\label{lem:d2-tG0}
\mbox{}
\begin{enumerate}
\item
$d_2 ( \tau G_0 ) = h_2 C_0 + h_1 h_3 Q_2$.
\item
$d_2 ( h_2 G_0 ) = h_1 C''$.
\end{enumerate}
\end{lemma}

\begin{proof}
We work in the motivic Adams spectral sequence for the cofiber of $\tau$ 
of Chapter \ref{ch:Ctau},
\index{cofiber of tau@cofiber of $\tau$!Adams spectral sequence}
where we have the relation $h_1 h_3 \cdot \ol{D_4} = \tau G_0$.
From Table \ref{tab:Ctau-E2}, we know that
$d_2 (\ol{D_4}) = h_1 \cdot \ol{B_6} + Q_2$.
It follows that 
$d_2 (h_1 h_3 \cdot \ol{D_4} ) = h_1^2 h_3 \cdot \ol{B_6} + h_1 h_3 Q_2$.
Finally, observe that
$h_1^2 h_3 \cdot \ol{B_6} = h_2^3 \cdot \ol{B_6} = h_2 C_0$.
This establishes the first formula.

For the second formula, 
use the first formula together with the relations
$\tau \cdot h_2 G_0 = h_2 \cdot \tau G_0$
and $h_2^2 C_0 = \tau h_1 C''$.
\end{proof}

\begin{lemma}
\label{lem:d2-tB5}
\mbox{}
\begin{enumerate}
\item
$d_2(\tau B_5) = \tau h_0^2 B_{23}$.
\item
$d_2(D_2') = \tau^2 h_0^2 B_{23}$.
\item
$d_2(P(A+A')) = \tau^2 h_0 h_2 B_{23}$.
\end{enumerate}
\end{lemma}

\begin{proof}
Classically, there is a relation $i B_5 = 0$ \cite{Bruner97}.
Using that $d_2(i) = P h_0 d_0$, we get that
$i d_2(B_5)$ equals $P h_0 d_0 B_5$ classically, which is non-zero.  
The only possibility
is that there is a 
motivic differential $d_2(\tau B_5) = \tau h_0^2 B_{23}$.
Note that the $P h_5 j$ term is eliminated because of the motivic weight.
This establishes the first formula.

Classically, there is a relation $i D_2' = 0$ \cite{Bruner97}.
As in the previous paragraph, we get that
$i d_2(D_2')$ equals $P h_0 d_0 D_2'$ classically, which is non-zero.
However, this time the motivic weights allow for two possibilities.
It follows that $d_2(D_2')$ equals either
$\tau^2 h_0^2 B_{23}$ or $\tau^2 h_0^2 B_{23} + P h_5 j$.

We know from \cite{Bruner04} that classically,
$\Sq^4(q)$ is non-zero, and $\Sq^5(q)$ is a multiple of $h_1$.
From \cite{BMMS86}*{VI.1}, we have that
$d_2 (\Sq^4(q)) = h_0 \Sq^5(q)$, which is zero.
From the previous two paragraphs, it follows that
$\Sq^4(q)$ must be $B_5 + D_2'$ classically,
and $d_2(D_2')$ must be $h_0^2 B_{23}$.
The motivic formula $d_2(D_2') = \tau^2 h_0^2 B_{23}$
follows immediately.
This establishes the second formula.

For the third formula, 
there is a classical relation $i P(A+A') = 0$ \cite{Bruner97}.
As before,
we get that $i d_2(P(A+A'))$ equals $P^2 h_0 d_0 (A+A')$, which is non-zero.
It follows that
$d_2(P(A+A'))$ equals $\tau^2 h_0 h_2 B_{23}$ or $h_0^4 G_{21}$.
The relation $h_2 D_2' = h_0 P(A+A')$ and the
calculation of $d_2(D_2')$ in the previous paragraph imply
that $d_2(P(A+A'))$ equals $\tau^2 h_0 h_2 B_{23}$.
\end{proof}

\begin{lemma}
\label{lem:d2-P^3v}
$d_2(P^3 v) = P^3 h_1^2 u$.
\end{lemma}

\begin{proof}
We will show in Lemma \ref{lem:d3-t^2P^2d0m} that
$d_3(\tau^2 P^2 d_0 m ) = P^3 h_1 u$.
Since $h_1 \cdot \tau^2 P^2 d_0 m$ is zero,
$h_1 \cdot P^3 h_1 u$ must be zero on the $E_3$-page.
Therefore, some $d_2$ differential must hit it.
The only possibility is that $d_2(P^3 v) = P^3 h_1^2 u$.
\end{proof}

\begin{lemma}
\label{lem:Adams-d2-X3}
$d_2 (X_3) = 0$.
\end{lemma}

\begin{proof}
Start with the relation $h_1 X_3 = B_3 c_0$.
This shows that $h_1 d_2 (X_3)$ is zero.
Therefore, $d_2(X_3)$ cannot equal $h_1 c_0 Q_2$.
\end{proof}

\begin{lemma}
\label{lem:d2-R1'}
$d_2(R_1') = P^2 h_0 x'$.
\end{lemma}

\begin{proof}
First, there is a relation $h_0^6 R_1' = \tau P^3 u'$,
as shown in Table \ref{tab:May-tau}.
Lemma \ref{lem:d2-u'} says that
$d_2(P^3 u') = \tau P^3 h_0 d_0^2 e_0$,
so $h_0^6 d_2(R_1')$ equals $\tau^2 P^3 h_0 d_0^2 e_0$.
There is another relation $\tau^2 P^3 h_0 d_0^2 e_0 = P^2 h_0^7 x'$,
as shown in Table \ref{tab:May-tau}.
It follows that
$d_2(R_1')$ equals $P^2 h_0 x'$.
\end{proof}

The next lemma computes a few $d_2$ differentials on decomposable
elements.  In principle, these differentials are consequences of the
previous lemmas.
However, the results of the calculations are unexpected because
of some extensions that are hidden in the motivic May spectral sequence.

\begin{lemma}
\mbox{}
\begin{enumerate}
\item
$d_2( e_0^2 g ) = h_1^7 B_1$.
\item
$d_2 ( c_0 e_0^2 g ) = h_1^8 B_8$.
\item
$d_2 (e_0 v ) = h_1^5 x'$.
\item
$d_2 ( e_0 v' ) = h_1^4 c_0 x' + \tau h_0 d_0 e_0^3$.
\end{enumerate}
\end{lemma}

\begin{proof}
In the first formula, we have
$d_2 ( e_0 \cdot e_0 g ) = h_1^2 d_0 \cdot e_0 g + e_0 \cdot h_1^2 e_0^2$.
This simplifies to $h_1^7 B_1$, as shown in \cite{GI14}.

The second formula follows immediately from the first formula,
using that $B_1 c_0 = h_1 B_8$.

For the third formula, 
start with the relation $P h_1 \cdot B_1 = h_1^2 x'$.
Since $h_1^7 B_1$ is hit by a $d_2$ differential, it follows that
$h_1^9 x'$ must also be hit by a $d_2$ differential.
The only possibility is that 
$d_2 ( e_0 v ) = h_1^5 x'$.

For the fourth formula,
$h_1^5 c_0 x'$ must be hit by the $d_2$ differential,
since $h_1^5 x'$ is hit by the $d_2$ differential.
The only possibility is that 
$d_2 ( h_1 e_0 v' ) = h_1^5 c_0 x'$.
Therefore,
$d_2 ( e_0 v')$ equals either $h_1^4 c_0 x'$ or 
$h_1^4 c_0 x' + \tau h_0 d_0 e_0^3$.
The extension $h_0 \cdot e_0 v' = \tau h_0 d_0 e_0 m$
implies that the second possibility is correct.
\end{proof}


\subsection{Adams $d_3$ differentials computations}
\label{subsctn:d3-lemmas}

\index{Adams spectral sequence!differential!d3@$d_3$}

\begin{lemma}
$d_3 ( h_4 c_0 ) = 0$.
\end{lemma}

\begin{proof}
The only other possibility is that $d_3( h_4 c_0 )$ equals $c_0 d_0$.
Table \ref{tab:Ctau-E2}
shows that $P d_0$ is hit by a differential
in the Adams spectral sequence for the cofiber $C\tau$ of $\tau$.
\index{cofiber of tau@cofiber of $\tau$!Adams spectral sequence}
Therefore, $\{ P d_0 \}$ must be divisible by $\tau$ in the homotopy
groups of $S^{0,0}$.
The only possibility 
is that $c_0 d_0$ is a non-zero permanent cycle and that
$\tau \cdot \{ c_0 d_0 \} = \{ P d_0 \}$.
\end{proof}

\begin{lemma}
\label{lem:d3-te0g}
\mbox{}
\begin{enumerate}
\item
$d_3 ( \tau e_0 g ) = c_0 d_0^2$.
\item
$d_3 ( \tau d_0 v ) = P h_1 u'$.
\item
$d_3( \tau^2 g m ) = h_1 d_0 u$.
\item
$d_3 ( \tau e_0 g^2 ) = c_0 d_0 e_0^2$.
\item
$d_3 ( \tau g v ) = h_1 d_0 u'$.
\item
$d_3 ( \tau P d_0 v) = P^2 h_1 u'$.
\end{enumerate}
\end{lemma}

\begin{proof}
For the first formula,
there is a classical differential
$d_4( e_0 g ) = P d_0^2$ given in Table \ref{tab:diff-refs}.
Motivically, there must be a differential
$d_4 ( \tau^2 e_0 g ) = P d_0^2$.  This shows that
$\tau e_0 g$ cannot survive to $E_4$.

The arguments for the remaining formulas are similar,
using the existence of the classical differentials
$d_4 ( d_0 v ) = P^2 u$,
$d_4 (g m ) = d_0^2 j + h_0^5 R_1$, 
$d_4 (e_0 g^2) = d_0^4$, 
$d_4 (g v) = P d_0 u$, and
$d_4 ( P d_0 v ) = P^3 u$.
All of these classical differentials
can be detected in the Adams spectral sequence for $\tmf$ \cite{Henriques07},
except for the third one, which is an easy consequence of $d_4(e_0 g) = P d_0^2$.
\end{proof}

\begin{lemma}
\label{lem:d3-tPd0e0}
\mbox{}
\begin{enumerate}
\item
$d_3 ( \tau P d_0 e_0 ) = P^2 c_0 d_0$.
\item
$d_3 ( \tau P^2 d_0 e_0 ) = P^3 c_0 d_0$.
\item
$d_3 ( \tau P^3 d_0 e_0 ) = P^4 c_0 d_0$.
\item
$d_3 ( \tau P^4 d_0 e_0 ) = P^5 c_0 d_0$.
\end{enumerate}
\end{lemma}

\begin{proof}
For the first formula,
we know that 
$d_3( \tau d_0 e_0 ) = P c_0 d_0$
by comparison to the classical case.
Therefore, $d_3 (\tau P h_1 d_0 e_0 ) = P^2 h_1 c_0 d_0$.
The desired formula follows immediately.
The arguments for the second, third, and fourth formulas are essentially the same.
\end{proof}

\begin{lemma}
$d_3 ( P h_5 c_0 ) = 0$.
\end{lemma}

\begin{proof}
The only other possibility is that $d_3 ( P h_5 c_0 ) = \tau d_0 l + u'$.
However, $c_0 ( \tau d_0 l + u' ) = h_1 d_0 u$ is non-zero, while
$P h_5 c_0^2 = P h_1^2 h_5 d_0 = 0$ since $P h_5 d_0 = \tau B_8$
by Table \ref{tab:May-tau}.
\end{proof}

\begin{lemma}
\label{lem:d3-B2}
$d_3( B_2 ) = 0$.
\end{lemma}

\begin{proof}
First, $B_{21}$ cannot support a $d_3$ differential,
so $h_2 B_{21} = d_0 B_2$ cannot support a $d_3$ differential.
This implies that $d_3(B_2)$ cannot equal $e_0 r$, since
$d_0 e_0 r$ is non-zero on the $E_3$-page.
\end{proof}

\begin{lemma}
\label{lem:d3-h1h5e0}
\mbox{}
\begin{enumerate}
\item
$d_3 ( h_1 h_5 e_0 ) = h_1^2 B_1$.
\item
$d_3 ( h_5 c_0 e_0 ) = h_1^2 B_8$.
\item
$d_3 ( P h_5 e_0 ) = h_1^2 x'$.
\item
$d_3( h_1 X_1 + \tau B_{22}) = c_0 x'$.
\end{enumerate}
\end{lemma}

\begin{proof}
We pass to the motivic Adams spectral sequence for the cofiber of $\tau$,
as discussed in Chapter \ref{ch:Ctau}.
\index{cofiber of tau@cofiber of $\tau$!Adams spectral sequence}
Note that $h_1^6 h_5 \cdot \ol{h_1^2 e_0} = \tau e_0 g^2$ 
in the $E_2$-page for the cofiber of $\tau$.
Also,
$h_1^6 \cdot \ol{h_1^3 B_1} = c_0 d_0 e_0^2$ in the
$E_3$-page for the cofiber of $\tau$.

Now $d_3(\tau e_0 g^2) = c_0 d_0 e_0^2$ on the $E_3$-page for $S^{0,0}$,
as shown in Lemma \ref{lem:d3-te0g}.
It follows that 
$d_3(h_5 \cdot \ol{h_1^2 e_0} ) = \ol{h_1^3 B_1}$ on
the $E_3$-page for the cofiber of $\tau$,
and then
$d_3 ( h_1^2 h_5 e_0 ) = h_1^3 B_1$ for $S^{0,0}$ as well.
This establishes the first formula.

After multiplying by $h_1$, 
the second and third formulas follow easily from the first.

For the fourth formula, start with the relation
$P h_5 c_0 e_0 = h_1^3 X_1$ 
from Table \ref{tab:May-h1}.
Multiply the third formula by $c_0$ to obtain the desired formula.
\end{proof}

\begin{lemma}
\label{lem:d3-gr}
\mbox{}
\begin{enumerate}
\item
$d_3 ( g r ) = \tau h_1 d_0 e_0^2$.
\item
$d_3 (m^2) = \tau h_1 e_0^4$.
\end{enumerate}
\end{lemma}

\begin{proof}
We have $d_3 ( \tau g r ) = \tau^2 h_1 d_0 e_0^2$
because $d_3(r) = \tau h_1 d_0^2$.
The first formula follows immediately.

For the second formula, multiply the first formula
by $\tau g$ and use multiplicative relations
from \cite{Bruner97} to obtain that $d_3( \tau m^2 ) = \tau^2 h_1 e_0^4$.
The second formula follows immediately.
\end{proof}

\begin{lemma}
\label{lem:d3-t^2G}
$d_3 ( \tau^2 G) = \tau B_8$.
\end{lemma}

\begin{proof}
From Table \ref{tab:Adams-tau},
the product $\tau \epsilon \kappa$ belongs to $\{ P d_0\}$ in $\pi_{22,12}$.
Since $h_1 h_5 c_0 d_0 = 0$ in the $E_\infty$-page by Lemma \ref{lem:d2-X1},
we know that
$\eta_5 \epsilon \kappa$
is either zero or represented in $E_\infty$ in higher filtration.
It follows that 
$\tau \eta_5 \epsilon \kappa = \eta_5 \{ P d_0\}$ 
is either zero or represented in $E_\infty$ in higher filtration.
Now $P h_1 h_5 d_0 = \tau h_1 B_8$ by Table \ref{tab:May-tau},
so $\tau h_1 B_8$ must be hit by some differential.
The only possibility is that
$d_3(\tau^2 G) = \tau B_8$.
\end{proof}

\begin{lemma}
\label{lem:d3-B6}
\mbox{}
\begin{enumerate}
\item
$d_3 ( e_1 g ) = h_1 g t$.
\item
$d_3(B_6) = \tau h_2 g n$.
\item
$d_3( g t ) = 0$.
\end{enumerate}
\end{lemma}

\begin{proof}
Start with the differential $d_3 (e_1) = h_1 t$ from Table \ref{tab:diff-refs}.
Then $d_3 ( \tau e_1 g )$ equals $\tau h_1 t g$, which implies 
the first formula.
The second formula follows easily, using that 
$\tau e_1 g = h_2 B_6$ from Table \ref{tab:May-h2}.

For the third formula, 
we know that $d_3( \tau g t) = 0$
because $d_3(\tau g) = 0$ and $d_3(t) = 0$.
Therefore, $d_3( g t)$ cannot equal $\tau h_1 e_0^2 g$.
\end{proof}

\begin{lemma}
\label{lem:d3-h5i}
\mbox{}
\begin{enumerate}
\item
$d_3 ( h_5 i ) = h_0 x'$.
\item
$d_3 ( h_5 j ) = h_2 x'$.
\end{enumerate}
\end{lemma}

\begin{proof}
We proved in Lemma \ref{lem:d3-h1h5e0} that
$d_3( h_5 c_0 e_0 ) = h_1 B_8$.
This implies that 
$d_3 (\ol{h_5 c_0 e_0}) = \ol{h_1^2 B_8}$
in the motivic Adams spectral sequence for 
the cofiber of $\tau$, which is discussed in Chapter \ref{ch:Ctau}.
\index{cofiber of tau@cofiber of $\tau$!Adams spectral sequence}
The hidden extensions 
$h_0 \cdot \ol{h_5 c_0 e_0} = h_5 j$ and
$h_0 \cdot \ol{h_1^2 B_8} = h_2 x'$ then imply that
$d_3 (h_5 j ) = h_2 x'$ for the cofiber of $\tau$,
which means that the same formula must hold for $S^{0,0}$.
This establishes the second formula.

The first formula now follows easily, using the relation
$h_2 h_5 i = h_0 h_5 j$.
\end{proof}

\begin{lemma}
\label{lem:d3-B3}
$d_3(B_3) = 0$.
\end{lemma}

\begin{proof}
The only other possibility is that $d_3(B_3) = B_{21}$.
On the $E_3$-page, $h_2 B_3$ is zero while $h_2 B_{21}$ is non-zero.
\end{proof}

\begin{lemma}
\label{lem:d3-tg^3}
$d_3 ( \tau g^3 ) = h_1^6 B_8$.
\end{lemma}

\begin{proof}
Start with the hidden extension
$\tau \eta^2 \cdot \{ \tau g^2 \} = \{ d_0^3 \}$, which follows from the
analogous classical extension given in Table \ref{tab:extn-refs}.
This implies that $\tau \eta^2 \{ \tau g^2 \} \kappabar = \{\tau d_0^2 e_0^2\}$.
In particular, $\eta^2 \{\tau^2 g^3 \}$ must be non-zero.

Either $\tau g^3$ or $\tau g^3 + h_1^4 h_5 c_0 e_0$ survives the motivic 
Adams spectral sequence.  In the first case,
there is no possible non-zero value for a hidden extension
of the form $\eta^2 \{ \tau g^3 \}$.
The only remaining possibility is that $\tau g^3 + h_1^4 h_5 c_0 e_0$
survives, in which case
$\eta^2 \{ \tau g^3 + h_1^4 h_5 c_0 e_0 \} = \{ h_1^6 h_5 c_0 e_0 \}$
is a non-hidden extension.
\end{proof}

\begin{lemma}
\label{lem:d3-C0}
\mbox{}
\begin{enumerate}
\item
$d_3 ( C_0 ) = n r$.
\item
$d_3 ( E_1 ) = n r$.
\item
$d_3 ( Q_2 ) =\tau^2 g t$.
\item
$d_3 (C'') = n m$.
\end{enumerate}
\end{lemma}

\begin{proof}
There are relations
$C_0 = h_2^2 \cdot \ol{B_6}$ and
$n r = h_2^2 \cdot \ol{\tau h_2 g n}$
in the motivic Adams spectral sequence for
the cofiber of $\tau$, as discussed in Chapter \ref{ch:Ctau}.
\index{cofiber of tau@cofiber of $\tau$!Adams spectral sequence}
From Lemma \ref{lem:d3-B6}, 
we know that $d_3 (\ol{B_6}) = \ol{\tau h_2 g n}$.
The first formula now follows easily.

For the second formula, note that $g E_1 = g C_0$ classically, and
that $g n r$ is non-zero on the $E_3$-page \cite{Bruner97}.
We already know that $d_3 ( g C_0 ) = g n r$ classically, so it follows that
$d_3 (E_1)$ also equals $n r$ classically.
The motivic formula is an immediate consequence.

For the third formula, there are classical relations
$w Q_2 = g^2 C_0$ and $w g t = g^2 n r$ \cite{Bruner97}.
We already know that $d_3 ( g^2 C_0 ) = g^2 n r$,
so it follows that $w d_3(Q_2) = w \cdot g t$.
The desired formula follows immediately.

For the fourth formula, there is a classical relation
$g C'' = r Q_2$ \cite{Bruner97}.  The $d_3$ differentials on $r$ and $Q_2$
imply that
$d_3 (g C'') = g r t$ classically, 
which equals $g n m$ \cite{Bruner97}.
The desired formula follows immediately.
\end{proof}

\begin{lemma}
\mbox{}
\begin{enumerate}
\item
$d_3 ( \tau h_1 X_1 ) = 0$.
\item
$d_3 ( R ) = 0$.
\end{enumerate}
\end{lemma}

\begin{proof}
We have classical relations $h_1 r X_1 = 0$ and $h_1^2 d_0^2 X_1 = 0$ \cite{Bruner97}.
Therefore, $r d_3 (h_1 X_1) = 0$ classically.
On the other hand, $r c_0 x'$ is non-zero on the $E_3$-page \cite{Bruner97}.
This shows that $d_3(h_1 X_1)$ cannot equal $c_0 x'$ classically,
which establishes the first formula.

An identical argument works for the second formula, using that
$r R = 0$ and $h_1 d_0^2 R = 0$.
\end{proof}

\begin{lemma}
\label{lem:d3-tgw}
$d_3 (\tau g w) = h_1^3 c_0 x'$.
\end{lemma}

\begin{proof}
There is a hidden extension $\eta \cdot \{ \tau w \} = \{ \tau d_0 l + u'\}$,
which follows from the analogous classical extension given in Table \ref{tab:extn-refs}.
This implies that $\eta \{\tau^2 g w\} = \{ \tau^2 d_0 e_0 m \}$.
If $\tau g w$ were a permanent cycle, then
$\eta \cdot \{\tau g w\}$ would be a non-zero hidden extension.
But there is no possible value for this hidden extension.
\end{proof}

\begin{lemma}
\label{lem:d3-t^2P^2d0m}
$d_3(\tau^2 P^2 d_0 m) = P^3 h_1 u$.
\end{lemma}

\begin{proof}
Note that $P^2 d_0 m$ supports a $d_4$ differential in the Adams
spectral sequence for $\tmf$ \cite{Henriques07}. 
However, $P^2 d_0 m$ cannot support a $d_4$ differential in the
classical Adams spectral sequence for the sphere.
Therefore, $P^2 d_0 m$ cannot surive to the $E_4$-page.
The only possibility is that there is a classical
differential $d_3 (P^2 d_0 m) = P^3 h_1 u$, 
from which the motivic analogue follows immediately.
\end{proof}

\begin{lemma}
$d_3(\tau^2 B_5 + D'_2) = 0$.
\end{lemma}

\begin{proof}
Classically, $(B_5 + D_2') d_0$ is zero while
$d_0 g w$ is non-zero on the $E_3$-page \cite{Bruner97}.
Therefore
$d_3(\tau^2 B_5 + D'_2)$ cannot equal $\tau^3 g w$.
\end{proof}

\begin{lemma}
\label{lem:d3-X3}
$d_3(X_3) = 0$.
\end{lemma}

\begin{proof}
The only other possibility is that $d_3(X_3)$ equals
$\tau n m$.
However, $g  n m$ is non-zero on the classical $E_3$-page,
while $g X_3$ is zero \cite{Bruner97}.
\end{proof}

\begin{lemma}
\label{lem:d3-h2B23}
$d_3(h_2 B_{23} ) = 0$.
\end{lemma}

\begin{proof}
This follows easily from the facts 
that $d_3( \tau B_{23} ) = 0$ and that
$h_2 \cdot \tau B_{23} = \tau \cdot h_2 B_{23}$.
\end{proof}

\begin{lemma}
\label{lem:d3-h2B5}
\mbox{}
\begin{enumerate}
\item
$d_3 (h_2 B_5 ) = h_1 B_8 d_0$.
\item
$d_3 ( \tau e_0 x' ) = P c_0 x'$.
\end{enumerate}
\end{lemma}

\begin{proof}
We will show in Lemma \ref{lem:d4-h2B5} that
$d_4 (\tau h_2 B_5) = h_1 d_0 x'$.
This means that $h_2 B_5$ cannot survive to $E_4$.
The only possibility is that $d_3(h_2 B_5) = h_1 B_8 d_0$.
This establishes the first formula.

The proof of the second formula is similar.
We will show in Lemma \ref{lem:d4-h2B5} that
$d_4 ( \tau^2 e_0 x') = P^2 x'$, so
$\tau e_0 x'$ cannot survive to $E_4$.
The only possibility is that $d_3(\tau e_0 x') = P c_0 x'$.
\end{proof}

\subsection{Adams $d_4$ differentials computations}
\label{subsctn:d4-lemmas}

\index{Adams spectral sequence!differential!d4@$d_4$}

\begin{lemma}
\label{lem:d4-C}
$d_4(C) = 0$.
\end{lemma}

\begin{proof}
The other possibility is that $d_4(C)$ equals $P h_1^2 h_5 c_0$.
We will show that 
$P h_1^2 h_5 c_0$ survives and is non-zero in the $E_\infty$-page.

Let $\alpha$ be an element of $\{ P h_1 h_5 c_0 \}$.
From Table \ref{tab:Toda}, the bracket
$\langle \eta^2, \alpha, \epsilon \rangle$
contains the element $\{ P h_1^3 h_5 e_0 \}$.
In order to compute this bracket, we
need the relation $c_0 \cdot G_3 = P h_1^3 h_5 e_0$ 
from Table \ref{tab:May-misc}.
Note that the bracket has indeterminacy generated by
$\eta^2 \{ D_{11} \}$.

If $P h_1^2 h_5 c_0$ were hit by a differential, then
$\eta \alpha$ would be zero.
Then $\eta \langle \eta, \alpha, \epsilon \rangle$
would equal $\langle \eta^2, \alpha, \epsilon \rangle$.
But $\{ P h_1^3 h_5 e_0 \}$ cannot be divisible by $\eta$.
By contradiction, $P h_1^2 h_5 c_0$ cannot be hit by a differential.
\end{proof}

\begin{lemma}
\mbox{}
\begin{enumerate}
\item
$d_4( h_0 h_5 i) = 0$.
\item
$d_4 (C_{11}) = 0$.
\end{enumerate}
\end{lemma}

\begin{proof}
The only non-zero possibility for $d_4 ( h_0 h_5 i)$ is $\tau d_0 u$.
However, $d_0 u$ survives to a non-zero homotopy class in
the Adams spectral sequence for $\tmf$ \cite{Henriques07}.  
This implies that $\tau d_0 u$ survives to a non-zero homotopy class in the
motivic Adams spectral sequence.
This establishes the first formula.

The proof of the second formula is similar.
The only non-zero possibility for $d_4 (C_{11})$ is $\tau^3 d_0 e_0 m$, 
but $d_0 e_0 m$ survives to a non-zero homotopy class in
the Adams spectral sequence for $\tmf$ \cite{Henriques07}.  
\end{proof}

\begin{lemma}
\label{lem:d4-t^2e0g^2}
$d_4 ( \tau^2 e_0 g^2 ) = d_0^4$.
\end{lemma}

\begin{proof}
First note that $d_4 (\tau^2 e_0 g) = P d_0^2$, which
follows from its classical analogue given in Table \ref{tab:diff-refs}.
Multiply this formula by $h_1 d_0^2$ to obtain that
$d_4(\tau^2 h_1 d_0 e_0^3)$ equals $P h_1 d_0^4$.
Finally, note that
$\tau^2 h_1 d_0 e_0^3$ equals $P h_1 \cdot \tau^2 e_0 g^2$.
The desired formula follows.
\end{proof}

\begin{lemma}
\mbox{}
\label{lem:d4-h2B5}
\begin{enumerate}
\item
$d_4 (\tau^2 h_1 B_{22}) = P h_1 x'$.
\item
$d_4 (\tau h_2 B_5 ) = h_1 d_0 x'$.
\item
$d_4 (\tau^2 e_0 x') = P^2 x'$.
\end{enumerate}
\end{lemma}

\begin{proof}
In the classical situation, $d_0 \cdot P h_1 x'$ is non-zero on the $E_4$-page,
and $d_0 \cdot h_1 B_{22} = (d_0 e_0 + h_0^7 h_5) \cdot B_1$ \cite{Bruner97}.
Using that $d_4(d_0 e_0 + h_0^7 h_5 ) = P^2 d_0$,
it follows that there is a classical differential
$d_4 (h_1 B_{22} ) = P h_1 x'$.
The motivic differential follows immediately.

The arguments for the second and third differentials are similar.
For the second, use that
$P d_0 \cdot h_1 d_0 x'$ is non-zero on the $E_4$-page;
$P d_0 \cdot h_2 B_5 = h_1 x' \cdot e_0 g$ \cite{Bruner97}; and
$d_4 (e_0 g) = P d_0^2$ classically.

For the third formula, use that
$h_1 d_0 \cdot P^2 x'$ is non-zero on the $E_4$-page;
$h_1 d_0 \cdot e_0 x' = P d_0 \cdot h_1 B_{22}$ \cite{Bruner97};
and $d_4 ( h_1 B_{22} )=  P h_1 x'$ classically from the first part of the lemma.
\end{proof}

\begin{lemma}
\label{lem:d4-t^2m^2}
$d_4(\tau^2 m^2) = d_0^2 z$.
\end{lemma}

\begin{proof}
First note that $d_4(\tau^2 g r) = i j$, which follows by comparison
to $\tmf$ \cite{Henriques07}.
Multiply by $\tau g$ to obtain that
$d_4 (\tau^3 m^2) = \tau d_0^2 z$, using multiplicative 
relations from \cite{Bruner97}.
The desired formula follows.
\end{proof}


\subsection{Adams $d_5$ differentials computations}
\label{subsctn:d5-lemmas}

\index{Adams spectral sequence!differential!d4@$d_4$}

\begin{lemma}
\label{lem:d5-g2}
$d_5(g_2) = 0$.
\end{lemma}

\begin{proof}
The only non-zero possibility is that $d_5 ( g_2 )$ equals $\tau^2 h_2 g^2$.
However, $e_0 \cdot g_2$ is zero in the $E_5$-page, while
$e_0 \cdot \tau^2 h_2 g^2$ is non-zero in the $E_5$-page.
\end{proof}

\begin{lemma}
$d_5 (B_2) = 0$.
\end{lemma}

\begin{proof}
The only other possibility is that $d_5(B_2)$ equals $h_1 u'$.
Recall from Table \ref{tab:extn-refs} that there is a 
classical hidden extension $\eta \{d_0 l \} = \{ P u\}$.
This implies that $\eta \{ \tau d_0 l + u' \}$
is non-zero motivically.  Therefore, $h_1 u'$ cannot be zero in $E_\infty$.
\end{proof}

Our next goal is to show that 
$d_5 (\tau P h_5 e_0) = \tau d_0 z$.
We will need a few preliminary lemmas.
This approach follows \cite{KM93}*{Theorem 2.2}, but we have corrected and clarified
the details in that argument.

\begin{lemma}
\label{lem:bracket-2kappabar^2}
$\langle \{q\}, 2, 8 \sigma \rangle
 = \{ 0, 2 \tau \kappabar^2 \}$.
\end{lemma}

\begin{proof}
Table \ref{tab:Massey} shows that 
$\langle q, h_0, h_0^3 h_3 \rangle$ equals $\tau h_1 u$.
Then Moss's Convergence Theorem \ref{thm:Moss} implies that
$\langle \{q\}, 2, 8 \sigma \rangle$
contains $\{ \tau h_1 u \}$.
\index{Convergence Theorem!Moss}
Table \ref{tab:Adams-2} shows that 
$\{ \tau h_1 u \}$ equals $2 \tau \kappabar^2$.
Finally, use 
Lemma \ref{lem:epsilon-q} to
show that $2 \tau \kappabar^2 = \tau \epsilon \{q \}$
is in the indeterminacy of the bracket.
\end{proof}

\begin{lemma}
\label{lem:bracket-tau-nu-kappabar}
The bracket
$\langle 2, 8 \sigma, 2, \sigma^2 \rangle$ contains 
$\tau \nu \kappabar$.
\end{lemma}

\begin{proof}
The subbracket $\langle 2, 8 \sigma, 2 \rangle$ is strictly zero,
as shown in Table \ref{tab:Toda}.
We will next show that 
the subbracket $\langle 8\sigma, 2, \sigma^2 \rangle$ is also
strictly zero.
First, the shuffle
\[
\langle 8 \sigma, 2, \sigma^2 \rangle \eta =
8 \sigma \langle 2, \sigma^2, \eta \rangle
\]
implies that the subbracket is annihilated by $\eta$.
This rules out $P d_0$.
Moss's Convergence Theorem \ref{thm:Moss} with the Adams differential
$d_2(h_4) = h_0 h_3^2$ implies that the subbracket
is detected in Adams filtration strictly greater than $5$.
\index{Convergence Theorem!Moss}
This rules out $\tau h_2 c_1$.
The only remaining possibility is that the subbracket contains zero,
and there is no possible indeterminacy.

We will work in the motivic Adams spectral sequence
for the cofiber $C2$ of $2$.
\index{cofiber of two}
We write $E_2(C2)$ for $\Ext_A ( H^{*,*} (C2), \M_2)$,
i.e., the $E_2$-page of the motivic Adams spectral sequence for $C2$.
The cofiber sequence
\[
\xymatrix@1{
S^{0,0} \ar[r]^2 & S^{0,0} \ar[r]^j & C2 \ar[r]^q & S^{1,0} }
\]
induces a map $q_*: E_2(C2) \map \Sigma^{1,0} E_2$.
Let $\ol{h_0^3 h_3}$ be an element of $E_2(C2)$ such that
$q_*(\ol{h_0^3 h_3})$ equals $h_0^3 h_3$, and 
let $\{ \ol{h_0^3 h_3} \}$ be the corresponding element in $\pi_{8,4}(C2)$.
Then
$j_* \langle 2, 8 \sigma, 2, \sigma^2 \rangle$ equals
$\langle \{ \ol{h_0^3 h_3} \}, 2, \sigma^2 \rangle$
in $\pi_{23,12}(C2)$
by Remark \ref{rem:higher-bracket-cofiber}.

Because of the Adams differential $d_2(h_4) = h_0 h_3^2$,
we know that 
$\langle \{ \ol{h_0^3 h_3} \}, 2, \sigma^2 \rangle$
is detected by
$h_4 \cdot \ol{h_0^3 h_3}$ in $E_\infty(C2)$.
Here we are using a slight generalization of Moss's Convergence Theorem \ref{thm:Moss},
in which one considers Toda brackets of maps between different objects 
(see \cite{Moss70} for the classical case).
\index{Convergence Theorem!Moss}

Finally, we need to compute $h_4 \cdot \ol{h_0^3 h_3 }$ in $E_2(C2)$.
This equals $j_* \langle h_4, h_0^3 h_3, h_0 \rangle$
by Remark \ref{rem:Massey-3bracket-cofiber} (see also 
Proposition \ref{prop:hidden-Massey} for an analogous result).
Table \ref{tab:Massey} shows that 
$\langle h_4, h_0^3 h_3, h_0 \rangle$ equals $\tau^2 h_2 g$.
\end{proof}

\begin{lemma}
\label{lem:bracket-sigma-eta4}
$\langle \epsilon, 2, \sigma^2 \rangle =
\{ \sigma \eta_4, \sigma \eta_4 + 4 \nu \kappabar \}$.
\end{lemma}

\begin{proof}
Using the Adams differential $d_2(h_4) = h_0 h_3^2$ and
Moss's Convergence Theorem \ref{thm:Moss},
we know that
$\langle \epsilon, 2, \sigma^2 \rangle$
intersects $\{ h_4 c_0 \}$.
\index{Convergence Theorem!Moss}

Lemma \ref{lem:sigma-h1h4}
shows that $\sigma \eta_4$ is contained in $\{ h_4 c_0 \}$.
The indeterminacy of $\{ h_4 c_0 \}$ is generated
by $\tau h_2 g$, $\tau h_0 h_2 g$, and $P h_1 d_0$.
By Lemma \ref{lem:2-h0h2g}, the indeterminacy consists of 
multiples of $\nu \kappabar$.
Therefore, $\{ h_4 c_0 \}$ consists of elements of the form
$\sigma \eta_4 + k \nu \kappabar$ for $0 \leq k \leq 7$.

Note that 
$\langle \epsilon, 2, \sigma^2 \rangle 2$
equals
$\epsilon \langle 2, \sigma^2, 2 \rangle$,
which is zero
because $\langle 2, \sigma^2, 2 \rangle$ contains $0$
by Table \ref{tab:Toda}.
Therefore,
if $\sigma \eta_4 + k \nu \kappabar$ belongs to 
$\langle \epsilon, 2, \sigma^2 \rangle$,
then $k$ equals $0$ or $4$.

We now know that either
$\sigma \eta_4$ or $\sigma \eta_4 + 4 \nu \kappabar$ belongs to
$\langle \sigma^2, 2, \epsilon \rangle$.
But $\tau \eta \epsilon \kappa$ equals $4 \nu \kappabar$
by Lemma \ref{lem:2-h0h2g} and the hidden $\tau$ extension from $h_1 c_0 d_0$
to $P h_1 d_0$ given in Table \ref{tab:Adams-tau},
so $4 \nu \kappabar$ belongs to the indeterminacy of the bracket.
It follows that both 
$\sigma \eta_4$ and $\sigma \eta_4 + 4 \nu \kappabar$ belong to
the bracket.
\end{proof}

\begin{lemma}
\label{lem:d5-tPh5e0}
$d_5(\tau P h_5 e_0) = \tau d_0 z$.
\end{lemma}

\begin{proof}
By Lemma \ref{lem:eta-t^3e0^2g},
$\tau \eta \kappa \kappabar^2$ is detected by $d_0 z$.
On the other hand,
$\nu \{ q \} \kappabar$ equals $\tau \eta \kappa \kappabar^2$,
by Table \ref{tab:extn-refs}.
We will show that $\tau \nu \{q\} \kappabar$ must be zero.  It will follow that
some differential must hit $\tau d_0 z$, and there is just one possibility.

From Lemma \ref{lem:bracket-tau-nu-kappabar},
we know that $\tau \nu \{q \} \kappabar$ is contained in
$\{ q \} \langle 2, 8\sigma, 2, \sigma^2 \rangle$,
which is contained in
$\langle \alpha, 2, \sigma^2 \rangle$
for some element $\alpha$ in 
$\langle \{q\}, 2, 8\sigma \rangle$.
By Lemma \ref{lem:bracket-2kappabar^2}, the two possible values for 
$\alpha$ are $0$ and $2 \tau \kappabar^2$.

First suppose that $\alpha$ is zero.  Then
$\tau \nu \{q \} \kappabar$ is contained in
$\langle 0, 2, \sigma^2 \rangle$, which is strictly zero.

Next suppose that $\alpha$ is $2 \tau \kappabar^2$.
By Lemma \ref{lem:epsilon-q},
we know that $\epsilon \{q \}$ equals $\{ h_1 u \}$,
which equals $2 \kappabar^2$ by Table \ref{tab:Adams-2}.
Therefore, the element
$\tau \nu \{q \} \kappabar$ is contained in
$\langle \tau \epsilon \{q\}, 2, \sigma^2 \rangle$.
This bracket has no indeterminacy, so it equals
$\tau \{q \} \langle \epsilon, 2, \sigma^2 \rangle$.
Using that $4 \nu \kappabar \cdot \tau \{q\}$ is zero,
Lemma \ref{lem:bracket-sigma-eta4} implies that
$\tau \nu \{q\} \kappabar$ equals
$\sigma \eta_4 \cdot \tau \{ q\}$.

We will show in the proof of Lemma \ref{lem:epsilon-q}
that $\sigma \{q \}$ equals $\nu \{t\}$.
So
$\sigma \eta_4 \cdot \tau \{ q\}$ equals
$\tau \nu \eta_4 \{ t\}$, which is zero
because $\nu \eta_4$ is zero.
\end{proof}

\begin{lemma}
$d_5(r_1) = 0$
\end{lemma}

\begin{proof}
The only other possibility is that $d_5(r_1)$ equals $h_2 B_{22}$.
However, $h_2 r_1$ is zero, while $h_2^2 B_{22}$ is non-zero in the $E_5$-page.
\end{proof}

\begin{lemma}
$d_5(\tau h_2 C') = 0$.
\end{lemma}

\begin{proof}
We do not know whether $d_4(C')$ equals $h_2 B_{21}$, so there are two
situations to consider
(see Proposition \ref{prop:Adams-d4} and Remark \ref{rem:Adams-d4-C'}).

In the first case,
assume that $d_4(C') = 0$.  
Then $C'$ survives to the $E_5$-page, and $d_5(C')$ must equal zero because
there are no other possibilities.
It follows that $d_5(\tau h_2 C') = 0$.

In the second case,
assume that $d_4(C')$ equals $h_2 B_{21}$ 
Then $d_4(h_2 C') = h_0 h_2 B_{22}$, but $\tau h_2 C'$ survives to 
the $E_5$-page.

The only possible non-zero value for $d_5(\tau h_2 C')$ is $\tau^4 g w$.
However, $g w$ survives to a non-zero homotopy class in 
the classical Adams spectral sequence for $\tmf$ \cite{Henriques07}.
Therefore, $g w$ survives to a non-zero homotopy class in the
classical Adams spectral sequence for $S^0$. This implies that
$\tau^4 g w$ cannot be hit by a motivic differential.
\end{proof}

Our next goal is to show that the element $\tau^3 d_0^2 e_0^2$
is hit by some differential.  We include it in this section
because it is likely hit by $d_5(\tau h_1^2 X_1)$.

\begin{lemma}
\label{lem:t^3d0^2e0^2-hit}
The element $\tau^3 d_0^2 e_0^2$ is hit by some differential.
\end{lemma}

\begin{proof}
Classically, $\eta^2 \kappabar^2$ is detected by 
$d_0^3$ because of the hidden $\eta$ extension from $g^2$ to $z$
and from $z$ to $d_0^3$ shown in Table \ref{tab:extn-refs}.
This implies that
motivically, $\tau^2 \eta^2 \kappabar^3$ is detected by $\tau^3 d_0^2 e_0^2$.
Therefore, we need to show that $\tau^2 \eta^2 \kappabar^3$ is zero.

Compute that $\langle \kappa, 2\nu, \nu \rangle$ equals
$\{ \eta \kappabar, \eta \kappabar + \nu \nu_4 \}$,
using Moss's Convergence Theorem \ref{thm:Moss} and the Adams differential
$d_2(f_0) = h_0^2 e_0$.
\index{Convergence Theorem!Moss}
It follows that
$\langle \kappa, 2, \nu^2 \rangle$ equals either
$\eta \kappabar$ or $\eta \kappabar + \nu \nu_4$.
Using Moss's Convergence Theorem \ref{thm:Moss} and the Adams differential
$d_3(h_0 h_4) = h_0 d_0$, we get that
$\langle \kappa, 2, \nu^2 \rangle$ is detected in Adams
filtration strictly greater than $4$.
Therefore,
$\langle \kappa, 2, \nu^2 \rangle$ equals $\eta \kappabar$.

This means that
$\tau^2 \eta^2 \kappabar^3$ equals
$\tau^2 \eta \kappabar^2 \langle \kappa, 2, \nu^2 \rangle$.
We showed in the proof of Lemma \ref{lem:d5-tPh5e0} that
$\tau^2 \eta \kappa \kappabar^2$ is zero.
Therefore,
$\tau^2 \eta^2 \kappabar^3$ is 
contained in $\langle 0, 2, \nu^2 \rangle$, which is strictly zero.
\end{proof}


%% file: stable-stems-Adams-hidden.tex
\chapter{Hidden extensions in the Adams spectral sequence}
\label{ch:Adams-hidden}

\index{Adams spectral sequence!hidden extension}
The main goal of this chapter is to compute hidden extensions
in the Adams spectral sequence.  We rely on the computation
of the Adams $E_\infty$-page carried out in Chapter \ref{ch:Adams-diff}.
We will borrow results from the classical Adams spectral sequence where
necessary.
Table \ref{tab:extn-refs}
summarizes previously established hidden extensions in the classical Adams spectral sequence.
\index{Adams spectral sequence!hidden extension!classical}
The table gives specific references to proofs.  The main sources are
\cite{BJM84}, \cite{BMT70}, and \cite{MT67}.

The Adams $E_\infty$ chart in \cite{Isaksen14a} is an essential companion
to this chapter.
\index{Adams chart}

\subsection*{Outline}

Section \ref{sctn:hidden} describes the main points in establishing the
hidden extensions.
We postpone the numerous technical proofs to Section \ref{sctn:extn-lemmas}.

Chapter \ref{ch:table} contains a series of tables that summarize the essential
computational facts in a concise form.
Tables \ref{tab:Adams-tau}, \ref{tab:Adams-2}, \ref{tab:Adams-eta},
and \ref{tab:Adams-nu} list the hidden extensions by
$\tau$, $2$, $\eta$, and $\nu$.  The fourth columns of these tables
refer to one argument that establishes each hidden extension, which is not
necessarily the first known proof.  This takes one of the following forms:
\begin{enumerate}
\item
An explicit proof given elsewhere in this manuscript.
\item
``image of $J$" means that the hidden extension is easily
deducible from the structure of the image of $J$ \cite{Adams66}.
\index{J@$J$!image of}
\item
``cofiber of $\tau$" means that the hidden extension is easily
deduced from the structure of the homotopy groups of the cofiber of 
$\tau$, as described in Chapter \ref{ch:Ctau}.
\index{cofiber of tau@cofiber of $\tau$}
\item
``Table \ref{tab:extn-refs}" means that the hidden extension
is easily deduced from an analogous classical hidden extension.
\end{enumerate}
Tables \ref{tab:Adams-misc-extn} and \ref{tab:Adams-compound-extn}
give some additional miscellaneous hidden extensions, again with 
references to a proof.

Tables \ref{tab:Adams-tau-tentative},
\ref{tab:Adams-2-tentative},
\ref{tab:Adams-eta-tentative}, and
\ref{tab:Adams-nu-tentative}
give partial information about hidden extensions in stems 59 through 70.
These results should be taken as tentative, since the
analysis of Adams differentials in this range is incomplete.
\index{Adams spectral sequence!hidden extension!tentative}


\section{Hidden Adams extensions}
\label{sctn:hidden}

\subsection{The definition of a hidden extension}
\label{subsctn:extn-defn}

First we will be precise about the exact nature of a hidden extension.
The most naive notion of a hidden extension is a non-zero product
$\alpha \beta$ in $\pi_{*,*}$ such that $\alpha$ and $\beta$ are detected in the
$E_\infty$-page by $a$ and $b$ respectively and $a b =0$ in the $E_\infty$-page.
However, this notion is too general, as the following example illustrates.

\begin{ex}
\label{ex:naive-hidden}
Consider $\{ h_3^2 \}$ in $\pi_{14,8}$, which consists of the two 
elements $\sigma^2$ and $\sigma^2 + \kappa$.
We have $h_1 h_3^2 = 0$ in $E_\infty$, but
$\eta (\sigma^2 + \kappa)$ is non-zero in $\pi_{15,9}$.

This type of situation is not usually considered
a hidden extension.  Because of the non-zero product $h_1 d_0$ in $E_\infty$,
one can see immediately
that there exists an element $\beta$ of $\{ h_3^2 \}$ such that
$\eta \beta$ is non-zero.

However, it is not immediately clear whether $\beta$ is
$\sigma^2$ or $\sigma^2 + \kappa$.  Distinguishing these possibilities
requires further analysis.  In fact, $\eta \sigma^2 = 0$ \cite{Toda62}.
\end{ex}

In order to avoid an abundance of not very interesting situations similar to
Example \ref{ex:naive-hidden}, we make the following formal definition
of a hidden extension.

\begin{defn}
\label{defn:hidden}
Let $\alpha$ be an element of $\pi_{*,*}$ that is detected
by an element $a$ of the $E_\infty$-page of the motivic Adams spectral sequence.
A hidden extension by $\alpha$ is a pair
of elements $b$ and $c$ of $E_\infty$ such that:
\begin{enumerate}
\item
$a b = 0$ in the $E_\infty$-page.
\item
There exists an element $\beta$ of $\{ b \}$ 
such that $\alpha \beta$ is contained in $\{ c \}$.
\item
If there exists an element $\beta'$ of $\{b'\}$ such
that $\alpha \beta'$ is contained in $\{ c\}$, then
the Adams filtration of $b'$ is less than or equal to the
Adams filtration of $b$.
\end{enumerate}
\end{defn}
\index{Adams spectral sequence!hidden extension!definition}

In other words, $b$ is the element of highest filtration such that
there is an $\alpha$ multiplication from $\{b\}$ into $\{ c\}$.
Consider the situation of Example \ref{ex:naive-hidden}.
Because $\eta \{ d_0 \}$ is contained in $\{ h_1 d_0 \}$
and the Adams filtration of $d_0$ is greater than the
Adams filtration of $h_3^2$,
condition (3) of Definition \ref{defn:hidden} implies
that there is not a hidden $\eta$ extension from $h_3^2$ to $h_1 d_0$.

\begin{remark}
Condition (3) of Definition \ref{defn:hidden} implies that $b'$ is not divisible
by $a$ in $E_\infty$.  
This allows one to easily reduce the number of cases that must be checked
when searching for hidden extensions.
\end{remark}

\begin{lemma}
\label{lem:hidden-not-exist}
Let $\alpha$ be an element of $\pi_{*,*}$.
Let $b$ be an element of the $E_\infty$-page of the motivic Adams spectral
sequence, and suppose that 
there exists an element $\beta$ of $\{ b\}$ such that $\alpha \beta$ is zero.
Then there is no hidden $\alpha$ extension on $b$.
\end{lemma}

\begin{proof}
Suppose that there exists some element $\beta'$ of $\{b \}$ such that
$\alpha \beta'$ is in $\{ c\}$.  Then 
$\alpha (\beta + \beta')$ is also in $\{ c\}$,
and the Adams filtration of $\beta + \beta'$ is strictly greater
than the Adams filtration of $\beta'$.
This implies that there is not a hidden $\alpha$ extension
from $b$ to $c$.
\end{proof}

\begin{lemma}
\label{lem:hidden-exist}
Let $\alpha$ be an element of $\pi_{*,*}$ that is detected
by an element $a$ of the $E_\infty$-page of the motivic Adams spectral
sequence.
Suppose that 
$b$ and $c$ are elements of $E_\infty$ such that:
\begin{enumerate}
\item
$a b = 0$ in the $E_\infty$-page.
\item
$\alpha \{ b \}$ is contained in $\{ c\}$.
\end{enumerate}
Then there is a hidden $\alpha$ extension from $b$ to $c$.
\end{lemma}

\begin{proof}
Let $\beta$ be any element of $\{b\}$, so
$\alpha \beta$ is contained in $\{c\}$.
Let $b'$ be an element of the $E_\infty$-page,
and let $\beta'$ in $\{b'\}$ be an element such that
$\alpha \beta'$ is also contained in $\{c \}$.

Since both $\alpha \beta$ and $\alpha \beta'$ are contained
in $\{c\}$, their sum $\alpha (\beta + \beta')$ is detected
in Adams filtration strictly greater than the Adams filtration of
$c$.  Therefore, $\beta + \beta'$ is not an element of $\{ b\}$,
which means that the Adams filtration of $b'$ must be less than
or equal to the Adams filtration of $b$.
\end{proof}

\begin{ex}
\label{ex:hidden-cross-1}
The conditions of Lemma \ref{lem:hidden-exist} are not quite equivalent
to the conditions of Definition \ref{defn:hidden}.
The difference is well illustrated by an example.
We will show in Lemma \ref{lem:eta-h3^2h5} that there is a hidden
$\eta$ extension from $h_3^2 h_5$ to $B_1$,
\index{h32h5@$h_3^2 h_5$}
\index{B1@$B_1$}
so there exists an element $\beta$ of $\{h_3^2 h_5\}$
such that $\eta \beta$ is contained in $\{B_1\}$.
Now let $\beta'$ be an element of $\{h_5 d_0 \}$,
\index{h5d0@$h_5 d_0$}
so $\beta + \beta'$ is an element of
$\{ h_3^2 h_5 \}$ because the Adams filtration of $h_5 d_0$ is 
greater than the Adams filtration of $h_3^2 h_5$.

Note that $\eta \beta'$ is contained in $\{h_1 h_5 d_0\}$.
The Adams filtration of $h_1 h_5 d_0$ is less than the
Adams filtration of $B_1$, so
$\eta (\beta + \beta')$ is contained in $\{ h_1 h_5 d_0 \}$.
This shows that the conditions of Lemma \ref{lem:hidden-exist}
are not satisfied.

The difference between Definition \ref{defn:hidden} and
Lemma \ref{lem:hidden-exist} occurs precisely when there are
``crossing" $\alpha$ extensions.  
\index{Adams spectral sequence!hidden extension!crossing}
In the chart of the $E_\infty$-page in \cite{Isaksen14a},
the straight line from $h_3^2 h_5$ to $B_1$ crosses the straight
line from $h_5 d_0$ to $h_1 h_5 d_0$.
\end{ex}

\begin{ex}
\label{ex:hidden-cross-2}
Through the 59-stem,
the issue of ``crossing" extensions occurs in only two other places.
First, we will show in Lemma \ref{lem:nu-h3^2h5} that there is a hidden
extension from $h_3^2 h_5$ to $B_2$.  
\index{B2@$B_2$}
In the $E_\infty$ chart in \cite{Isaksen14a},
the straight line from $h_3^2 h_5$ to $B_2$ crosses the straight line
from $h_5 d_0$ to $h_2 h_5 d_0$.
Therefore there exists an element $\beta$ of
$\{ h_3^2 h_5 \}$ such that $\nu \beta$ is not contained in $\{B_2\}$.

Second, we will show in Lemma \ref{lem:eta-h1f1} that there is a hidden
extension from $h_1 f_1$ to $\tau h_2 c_1 g$.
\index{f1@$f_1$}
\index{c1g@$c_1 g$}
In the $E_\infty$ chart in \cite{Isaksen14a},
the straight line from $h_1 f_1$ to $\tau h_2 c_1 g$
crosses the straight line from $h_1^2 h_5 c_0$ to $h_1^3 h_5 c_0$.
Therefore, there exists an element $\beta$ of
$\{ h_1 f_1 \}$ such that $\eta \beta$ is not contained in $\{\tau h_2 c_1 g\}$.
\end{ex}

We will thoroughly explore hidden extensions in the sense of 
Definition \ref{defn:hidden}.  However, such hidden extensions do not
completely determine the multiplicative structure of $\pi_{*,*}$.
For example, the relation $\eta \sigma^2 = 0$ discussed in
Example \ref{ex:naive-hidden} does not fit into this formal framework.

Something even more complicated occurs with the relation
$h_2^3 + h_1^2 h_3 = 0$ in the $E_\infty$-page.  There is a hidden relation
here, in the sense that $\nu^3 + \eta^2 \sigma$ does not equal zero;
rather, it equals $\eta \epsilon$ \cite{Toda62}.  We do not attempt
to systematically address these types of compound relations.
\index{Adams spectral sequence!hidden extension!compound}

\subsection{Hidden Adams $\tau$ extensions}
\label{subsctn:hidden-tau}

\index{Adams spectral sequence!hidden extension!tau@$\tau$}
For hidden $\tau$ extensions,
the key tool is the homotopy of the cofiber $C\tau$ of $\tau$.
\index{cofiber of tau@cofiber of $\tau$}
This calculation is fully explored in Chapter \ref{ch:Ctau}.
Let $\alpha$ be an element of $\pi_{*,*}$.
Then $\alpha$ maps to zero under the inclusion $S^{0,0} \map C\tau$
of the bottom cell
if and only if $\alpha$ is divisible by $\tau$ in $\pi_{*,*}$.

\begin{prop}
\label{prop:hidden-tau}
Table \ref{tab:Adams-tau} shows some
hidden $\tau$ extensions in $\pi_{*,*}$,
through the 59-stem.  
These are the only hidden $\tau$ extensions in this range,
with the possible exceptions that there might be
hidden $\tau$ extensions:
\begin{enumerate}
\item
from $h_1 i_1$ to $h_1 B_8$.
\index{i1@$i_1$}
\index{B8@$B_8$}
\item
from $j_1$ to $B_{21}$.
\index{j1@$j_1$}
\index{B21@$B_{21}$}
\end{enumerate}
\end{prop}

\begin{proof}
Table \ref{tab:Adams-tau}
cites one possible argument for each hidden $\tau$ extension.
These arguments break into two types:
\begin{enumerate}
\item
In many cases, we know from Chapter \ref{ch:Ctau} that 
an element of $\{ b' \}$ maps to zero in $\pi_{*,*}(C\tau)$,
where $C\tau$ is the cofiber of $\tau$.
\index{cofiber of tau@cofiber of $\tau$}
Therefore,
this element of $\{ b' \}$ is divisible by $\tau$ in $\pi_{*,*}$,
which implies that there must be a hidden $\tau$ extension.
Usually there is just one possible hidden $\tau$ extension.
\item
Other more difficult cases are proved in
Section \ref{subsctn:tau-lemmas}.
\end{enumerate}

For many of the possible hidden $\tau$ extensions
from $b$ to $b'$, we know from Chapter \ref{ch:Ctau}
that none of the elements of $\{ b' \}$ map to zero in 
$\pi_{*,*}(C\tau)$.
Therefore,
none of the elements of $\{ b' \}$ is divisible by $\tau$,
so none of these possible hidden $\tau$ extensions 
are actual hidden $\tau$ extensions.
A number of more difficult non-existence proofs are given
in Section \ref{subsctn:tau-lemmas}.
\end{proof}

In order to maintain the flow of the narrative, we have collected
the technical computations of hidden extensions in 
Section \ref{subsctn:tau-lemmas}.

\begin{remark}
Table \ref{tab:Adams-tau-tentative}
shows some additional hidden $\tau$ extensions in stems 60 through 69.
These results are tentative because the analysis of the $E_\infty$-page
is incomplete in this range.  Tentative proofs in Section \ref{subsctn:tau-lemmas}
are clearly indicated.
\index{Adams spectral sequence!hidden extension!tentative}
\end{remark}

\begin{remark}
We show in Lemma \ref{lem:tau-h1^2g2} that there is no
hidden $\tau$ extension on $h_1^2 g_2$.  This contradicts
the claim in \cite{Kochman90} that there is a classical hidden
$\eta$ extension from $h_1 g_2$ to $N$.
We do not understand the source of this discrepancy.
See also Remark \ref{rem:eta-th1g2}.
\index{g2@$g_2$}
\index{N@$N$}
\end{remark}

\begin{remark}
\label{rem:tau-h1i1}
There may be a hidden $\tau$ extension from $h_1 i_1$ to $h_1 B_8$.
This extension occurs if and only if 
$d_3( \ol{h_1 i_1} )$ equals $h_1 B_8$
in the Adams spectral sequence for the cofiber of $\tau$
(see Proposition \ref{prop:Ctau-Adams-d3}).
\index{cofiber of tau@cofiber of $\tau$!Adams spectral sequence}
If this extension occurs, then it implies that there is a hidden
relation
$\nu \{C\} + \tau \{ i_1 \} = \{ B_8\}$.
\index{i1@$i_1$}
\index{B8@$B_8$}
\index{C@$C$}
\end{remark}

\begin{remark}
\label{rem:tau-D11}
We show in Lemma \ref{lem:tau-D11} that there is no hidden $\tau$ extension on
$D_{11}$. 
\index{D11@$D_{11}$}
 This proof is different in spirit from the rest of this
manuscript because it uses specific calculations in the classical
Adams-Novikov spectral sequence.
\index{Adams-Novikov spectral sequence!classical}
  This is especially relevant
since Chapter \ref{ch:ANSS} uses the calculations here to derive
Adams-Novikov calculations, so there is some danger of circular arguments.
We would prefer to have a proof that is internal to the motivic
Adams spectral sequence.
\end{remark}

\begin{remark}
\index{j1@$j_1$}
\index{B21@$B_{21}$}
\index{C'@$C'$}
Remark \ref{rem:Adams-d4-C'} explains that the following three claims are equivalent:
\begin{enumerate}
\item
there is a hidden $\tau$ extension from $j_1$ to $B_{21}$.
\item
$d_4( \ol{j_1} ) = B_{21}$ 
in the motivic Adams spectral sequence for the cofiber of $\tau$.
\index{cofiber of tau@cofiber of $\tau$!Adams spectral sequence}
\item
$d_4(C') = h_2 B_{21}$ in the motivic Adams spectral sequence
for the sphere.
\end{enumerate}
\end{remark}


\subsection{Hidden Adams $2$ extensions}
\label{subsctn:hidden-2}

\index{Adams spectral sequence!hidden extension!two}
\begin{prop}
\label{prop:hidden-2}
Table \ref{tab:Adams-2} shows some
hidden $2$ extensions in $\pi_{*,*}$,
through the 59-stem.  
These are the only hidden $2$ extensions in this range,
with the possible exceptions that there might be
hidden $2$ extensions:
\begin{enumerate}
\item
from $h_0 h_3 g_2$ to $\tau g n$.
\index{h3g2@$h_3 g_2$}
\index{gn@$g n$}
\item
from $j_1$ to $\tau^2 c_1 g^2$.
\index{j1@$j_1$}
\index{c1g2@$c_1 g^2$}
\end{enumerate}
\end{prop}

\begin{proof}
Table \ref{tab:Adams-2} cites one possible argument (but not necessarily the
earliest published result) for each hidden $2$ extension.
One extension follows from its classical analogue given 
in Table \ref{tab:extn-refs}.
\index{Adams spectral sequence!hidden extension!classical}
The remaining cases are proved in Section \ref{subsctn:2-lemmas}.

A number of non-existence proofs are given
in Section \ref{subsctn:2-lemmas}.
\end{proof}

In order to maintain the flow of the narrative, we have collected the
technical computations of various hidden $2$ extensions in 
Section \ref{subsctn:2-lemmas}.

\begin{remark}
Table \ref{tab:Adams-2-tentative}
shows some additional hidden $2$ extensions in stems 60 through 69.
These results are tentative because the analysis of the $E_\infty$-page
is incomplete in this range.
Tentative proofs in Section \ref{subsctn:2-lemmas} are clearly indicated.
\index{Adams spectral sequence!hidden extension!tentative}
\end{remark}

\begin{remark}
\index{Adams spectral sequence!hidden extension!four}
Recall from Table \ref{tab:extn-refs} that there is a hidden
$4$ extension from $h_3^2 h_5$ to $h_0 h_5 d_0$.
\index{h32h5@$h_3^2 h_5$}
\index{h5d0@$h_5 d_0$}
It is tempting to consider this as a hidden $2$ extension
from $h_0 h_3^2 h_5$ to $h_0 h_5 d_0$, but this is not consistent with
Definition \ref{defn:hidden}.
\end{remark}

\begin{remark}
There is a possible
hidden $2$ extension from $h_0 h_3 g_2$ to $\tau g n$.
\index{h3g2@$h_3 g_2$}
\index{gn@$g n$}
We show in Lemma \ref{lem:nu-h2h5d0} that
this hidden extension occurs if and only if 
there is a hidden $\nu$ extension from $h_2 h_5 d_0$ to $\tau g n$.
\index{h5d0@$h_5 d_0$}
Lemma \ref{lem:nu-h2h5d0} is inconsistent with results of \cite{Kochman90},
which indicates the hidden $\nu$ extension
but not the hidden $2$ extension.  We do not understand the source of
this discrepancy.
\end{remark}

\begin{remark}
\label{rem:2-h0h5i}
We show in Lemma \ref{lem:2-h0h5i} that there is a hidden $2$
extension from $h_0 h_5 i$ to $\tau^3 e_0^2 g$.
\index{h5i@$h_5 i$}
\index{e02g@$e_0^2 g$}
This proof is different in spirit from the rest of this
manuscript because it uses specific calculations in the classical
Adams-Novikov spectral sequence.
\index{Adams-Novikov spectral sequence!classical}
This is especially relevant since
Chapter \ref{ch:ANSS} uses the calculations here to derive
Adams-Novikov calculations, so there is some danger of circular arguments.
We would prefer to have a proof that is internal to the motivic 
Adams spectral sequence.

We point out one other remarkable property of this hidden extension.
Up to the 59-stem, it is the only example of a $2$ extension that is hidden
in both the Adams spectral sequence and the Adams-Novikov spectral sequence.
(There are several $\eta$ extensions and $\nu$ extensions
that are hidden in both spectral sequences.)
\end{remark}


\subsection{Hidden Adams $\eta$ extensions}
\label{subsctn:hidden-eta}

\index{Adams spectral sequence!hidden extension!eta@$\eta$}

\begin{prop}
\label{prop:hidden-eta}
Table \ref{tab:Adams-eta} shows some
hidden $\eta$ extensions in $\pi_{*,*}$,
through the 59-stem.  
These are the only hidden $\eta$ extensions in this range,
with the possible exceptions that there might be
a hidden $\eta$ extension
from $\tau h_1 Q_2$ to $\tau B_{21}$.
\index{Q2@$Q_2$}
\index{B21@$B_{21}$}
\end{prop}

\begin{proof}
Table \ref{tab:Adams-eta} cites one possible argument (but not necessarily the
earliest published result) for each hidden $\eta$ extension.
These arguments break into three types:
\begin{enumerate}
\item
Some hidden extensions follow from the calculation of the image of $J$
\cite{Adams66}.
\index{J@$J$!image of}
\item
Some hidden extensions follow from their classical analogues given 
in Table \ref{tab:extn-refs}.
\index{Adams spectral sequence!hidden extension!classical}
\item
The remaining more difficult cases are proved in Section \ref{subsctn:eta-lemmas}.
\end{enumerate}

A number of non-existence proofs are given
in Section \ref{subsctn:eta-lemmas}.
\end{proof}

In order to maintain the flow of the narrative, we collect the technical 
results establishing various hidden $\eta$ extensions in Section
\ref{subsctn:eta-lemmas}.

\begin{remark}
Table \ref{tab:Adams-eta-tentative}
shows some additional hidden $\eta$ extensions in stems 60 through 69.
These results are tentative because the analysis of the $E_\infty$-page
is incomplete in this range.
Tentative proofs in Section \ref{subsctn:eta-lemmas}
are clearly indicated.
\index{Adams spectral sequence!hidden extension!tentative}
\end{remark}

\begin{remark}
We show in Lemma \ref{lem:eta-th1g2} that there is no
hidden $\eta$ extension on $\tau h_1 g_2$.  This contradicts
the claim in \cite{Kochman90} that there is a classical hidden
$\eta$ extension from $h_1 g_2$ to $N$.
We do not understand the source of this discrepancy.
\index{g2@$g_2$}
\index{N@$N$}
\end{remark}

\begin{remark}
\label{rem:eta-th1g2}
The element $\eta^2 \{g_2\}$ is considered in \cite{BJM84}*{Lemma 4.3},
where it is shown to be equal to $\sigma^2 \{ d_1 \}$.
\index{d1@$d_1$}
\index{g2@$g_2$}
Our results indicate that both are zero classically; this is consistent with a 
careful reading of \cite{BJM84}*{Lemma 4.3}. 

Motivically, $\eta^2 \{g_2\} = \sigma^2 \{d_1 \}$ is non-zero because
they are detected by $h_1^2 g_2 = h_3^2 d_1$.  However,
Lemma \ref{lem:eta-th1g2} implies that $\tau \eta^2 \{g_2\}$
and $\tau \sigma^2 \{d_1\}$ are both zero.
\end{remark}

\begin{remark}
We show in Lemma \ref{lem:eta-C} that there is no
hidden $\eta$ extension on $C$.  This contradicts
the claim in \cite{Kochman90} that there is a classical hidden
$\eta$ extension from $C$ to $g n$.
We do not understand the source of this discrepancy.
\index{C@$C$}
\index{gn@$g n$}
\end{remark}


\subsection{Hidden Adams $\nu$ extensions}
\label{subsctn:hidden-nu}

\index{Adams spectral sequence!hidden extension!nu@$\nu$}
\begin{prop}
\label{prop:hidden-nu}
Table \ref{tab:Adams-nu} shows some
hidden $\nu$ extensions in $\pi_{*,*}$,
through the 59-stem.  
These are the only hidden $\nu$ extensions in this range,
with the possible exceptions that there might be
hidden $\nu$ extensions:
\begin{enumerate}
\item
from $h_2 h_5 d_0$ to $\tau g n$.
\index{h5d0@$h_5 d_0$}
\index{gn@$g n$}
\item
from $i_1$ to $g t$.
\index{i1@$i_1$}
\index{gt@$g t$}
\end{enumerate}
\end{prop}

\begin{proof}
Table \ref{tab:Adams-nu} cites one possible argument (but not necessarily the
earliest published result) for each hidden $\nu$ extension.
These arguments break into two types:
\begin{enumerate}
\item
Some hidden extensions follow from their classical analogues given 
in Table \ref{tab:extn-refs}.
\index{Adams spectral sequence!hidden extension!classical}
\item
The remaining more difficult cases are proved in Section \ref{subsctn:nu-lemmas}.
\end{enumerate}

A number of non-existence proofs are given
in Section \ref{subsctn:nu-lemmas}.
\end{proof}

In order to maintain the flow of the narrative, we have collected the
technical results establishing various hidden $\nu$ extensions in 
Section \ref{subsctn:nu-lemmas}.

\begin{remark}
Table \ref{tab:Adams-nu-tentative}
shows some additional hidden $\nu$ extensions in stems 60 through 69.
These results are tentative because the analysis of the $E_\infty$-page
is incomplete in this range.
Tentative proofs in Section \ref{subsctn:nu-lemmas}
are clearly indicated.
\index{Adams spectral sequence!hidden extension!tentative}
\end{remark}

\begin{remark}
We draw the reader's attention to the curious hidden $\nu$ extensions
on $h_2 c_1$, $h_2 c_1 g$, and $N$. 
\index{c1@$c_1$}
\index{c1g@$c_1 g$}
\index{N@$N$}
These are ``exotic" extensions that
have no classical analogues.  The hidden extension on $N$
contradicts the claim in \cite{Kochman90} that there is a hidden
$\eta$ extension from $h_1 g_2$ to $N$.
\index{g2@$g_2$}
In addition to the proof provided in Lemma \ref{lem:nu-h2c1},
one can also establish these hidden extensions by computing in the 
motivic Adams spectral sequence for the cofiber of $\nu$.
\index{cofiber of nu@cofiber of $\nu$}
One can show that $\{ h_1^2 h_4 c_0 \}$, $\{ P h_1^2 h_5 c_0 \}$,
and $\{ h_1^6 h_5 c_0 \}$ all map to zero in the cofiber of $\nu$,
which implies that they are divisible by $\nu$.
index{h4c0@$h_4 c_0$}
\index{Ph5c0@$P h_5 c_0$}
\index{h5c0@$h_5 c_0$}
\end{remark}

\begin{remark}
There is a possible
hidden $\nu$ extension from $h_2 h_5 d_0$ to $\tau g n$.
\index{h5d0@$h_5 d_0$}
\index{gn@$g n$}
We show in Lemma \ref{lem:nu-h2h5d0} that
this hidden extension occurs if and only if 
there is a hidden $2$ extension from $h_0 h_3 g_2$ to $\tau g n$.
\index{h3g2@$h_3 g_2$}
Lemma \ref{lem:nu-h2h5d0} is inconsistent with results of \cite{Kochman90},
which indicates the hidden $\nu$ extension
but not the hidden $2$ extension.  We do not understand the source of
this discrepancy.
\end{remark}


\section{Hidden Adams extensions computations}
\label{sctn:extn-lemmas}

In this section, we collect the technical computations that establish
the hidden extensions discussed in Section \ref{sctn:hidden}.

\subsection{Hidden Adams $\tau$ extensions computations}
\label{subsctn:tau-lemmas}

\index{Adams spectral sequence!hidden extension!tau@$\tau$}

\begin{lemma}
\label{lem:t-h1h3g}
\mbox{}
\begin{enumerate}
\item
There is a hidden $\tau$ extension from $h_1 h_3 g$ to $d_0^2$.
\item
There is a hidden $\tau$ extension from $h_1 h_3 g^2$ to $d_0 e_0^2$.
\end{enumerate}
\end{lemma}

\begin{proof}
We will show in Lemma \ref{lem:epsilon-kappabar}
that $\epsilon \kappabar = \kappa^2$ in $\pi_{28,16}$.
Therefore, $\kappa^2$ is contained in $\kappabar \langle 2, \nu^2, \eta \rangle$.

Let $C\tau$ be the cofiber of $\tau$, whose homotopy is studied
thoroughly in Chapter \ref{ch:Ctau}.
\index{cofiber of tau@cofiber of $\tau$}
Let $\kappabar_{C\tau}$ be the image of $\kappabar$ in
$\pi_{20,11}(C\tau)$.
Then the image of $\kappa^2$ in $\pi_{28,16}(C\tau)$ is contained in
$\kappabar_{C\tau} \langle 2, \nu^2, \eta \rangle$.
Because $\kappabar_{C\tau} \cdot 2$ is zero, we can shuffle to obtain
$\langle \kappabar_{C\tau}, 2, \nu^2 \rangle \eta$.

Now $\pi_{27, 15}(C\tau)$ consists only of the element
$\{ P^3 h_1^3\}$.  However, this element cannot belong to
$\langle \kappabar_{C\tau}, 2, \nu^2 \rangle$
because $\{P^3 h_1^3\}$ supports infinitely many multiplications by
$\eta$, while the elements in the bracket cannot.
Therefore,
$\langle \kappabar_{C\tau}, 2, \nu^2 \rangle$ is zero,
and the image of $\kappa^2$ in $\pi_{28,16}(C\tau)$ is zero.

Therefore, $\kappa^2$ in $\pi_{28,14}$ is divisible by $\tau$,
and there is just one possible hidden $\tau$ extension.
This completes the proof of the first claim.

The proof for the second claim is analogous, using that
$\epsilon \{ \tau g^2\} = \kappa \{e_0^2\}$ 
from Lemma \ref{lem:epsilon-kappabar}.
The bracket
$\langle \{\tau g^2\}_{C\tau}, 2, \nu^2 \rangle$ in $\pi_{47, 27}(C\tau)$
must be zero because there are no other possibilities.
\end{proof}

\begin{lemma}
\label{lem:t.eta-d1}
\mbox{}
\begin{enumerate}
\item
There is no hidden $\tau$ extension on $h_1 d_1$.
\item
There is no hidden $\tau$ extension on $h_1 d_1 g$.
\end{enumerate}
\end{lemma}

\begin{proof}
For the first formula, 
the only other possibility is that 
there is a hidden $\tau$ extension from $h_1 d_1$ to $h_1 q$.
We will show that this is impossible.

Proposition \ref{prop:ANSS-Ctau} shows that 
the element $\{ d_1 \}$ of $\pi_{32,18}$
is detected in Adams-Novikov filtration $4$. 
\index{Adams-Novikov spectral sequence}
Therefore, $\{ d_1 \}$ realizes
to zero in $\pi_{32} \tmf$ \cite{Bauer08}, so
$\tau \eta \{d_1\}$ also realizes to zero in $\pi_{33} \tmf$.
\index{topological modular forms}

On the other hand, $\{ h_1 e_0^2\}$ realizes to a non-zero element
of $\pi_{35} \tmf$.  The classical hidden extension
$\nu \{ q\} = \{ h_1 e_0^2 \}$ given in Table \ref{tab:extn-refs}
then implies that $\{ q\}$ realizes to a non-zero element of
$\pi_{32} \tmf$.  Then $\{ h_1 q \}$ also realizes to a non-zero
element of $\pi_{33} \tmf$.  

This shows that $\tau \eta \{d_1 \}$ cannot belong to $\{h_1 q\}$,
so it must be zero.  Now Lemma \ref{lem:hidden-not-exist} 
establishes the first claim.

For the second claim,
Table \ref{tab:Toda} shows that 
$\{ d_1 g \} = \langle \{d_1 \}, \eta^3, \eta_4 \rangle$,
again with no indeterminacy.
Now shuffle to obtain
$\tau \eta \{ d_1 g \} = \langle \tau \eta, \{d_1\}, \eta^3 \rangle \eta_4$.
The element $\{ \tau h_2 e_0^2 \}$ is 
the only non-zero element that could possibly be contained
in $\langle \tau \eta, \{d_1\}, \eta^3 \rangle$.
In any case,
$\langle \tau \eta, \{d_1\}, \eta^3 \rangle \eta_4$ 
is zero.
This shows that $\tau \eta \{d_1 g\}$ is zero.  
Lemma \ref{lem:hidden-not-exist} establishes the second claim.
\end{proof}

\begin{lemma}
\label{lem:t-th0g^2}
There is a hidden $\tau$ extension from $\tau h_0 g^2$ 
to $h_1 u$.
\end{lemma}

\begin{proof}
Classically, there is a hidden $2$ extension from
$g^2$ to $h_1 u$ given in Table \ref{tab:extn-refs}.
This implies that there is a motivic hidden $2$ extension
from $\tau^2 g^2$ to $h_1 u$.
The desired hidden $\tau$ extension follows.
\end{proof}

\begin{lemma}
\label{lem:t-th1g^2}
\mbox{}
\begin{enumerate}
\item
There is a hidden $\tau$ extension from $\tau h_1 g^2$ to $z$.
\item
There is a hidden $\tau$ extension from $\tau h_1 e_0^2 g$ to $d_0 z$.
\end{enumerate}
\end{lemma}

\begin{proof}
There is a classical hidden $\eta$ extension 
from $g^2$ to $z$ given in Table \ref{tab:extn-refs}.
It follows that there is a motivic hidden $\eta$ extension
from $\tau^3 g^2$ to $z$.
The first claim follows immediately.

For the second claim, multiply the first hidden extension by $d_0$.
\end{proof}

\begin{lemma}
\label{lem:tau-h1^2g2}
There is no hidden $\tau$ extension on $h_1^2 g_2$.
\end{lemma}

\begin{proof}
We will show in Lemma \ref{lem:eta-th1g2} that there is no hidden
$\eta$ extension on $\tau h_1 g_2$.  This implies that there is no
hidden $\tau$ extension on $h_1^2 g_2$.
\end{proof}

\begin{lemma}
There is no hidden $\tau$ extension on $\tau h_2 d_1 g$.
\end{lemma}

\begin{proof}
The only other possibility is that there is a hidden $\tau$  extension
from $\tau h_2 d_1 g$ to $d_0 z$.
However, we showed in Lemma \ref{lem:t-th1g^2} that there 
is a hidden $\tau$ extension from $\tau h_1 e_0^2 g$ to $d_0 z$.
Since $\{ \tau h_1 e_0^2 g \}$ is contained in the indeterminacy
of $\{ \tau h_2 d_1 g \}$, there exists an element of
$\{ \tau h_2 d_1 g\}$ that is annihilated by $\tau$.
Lemma \ref{lem:hidden-not-exist} finishes the proof.
\end{proof}

\begin{lemma}
\label{lem:tau-D11}
There is no hidden $\tau$ extension on $D_{11}$.
\end{lemma}

\begin{proof}
This proof is different in spirit from the rest of the manuscript
because it relies on specific calculations in the classical
Adams-Novikov spectral sequence.
\index{Adams-Novikov spectral sequence!classical}

There is an element $\beta_{12/6}$
in the Adams-Novikov spectral sequence in the 58-stem with
filtration 2 \cite{Shimomura81}.  
Using Proposition \ref{prop:ANSS-Ctau},
if this class survives, then it would 
correspond to an element of
$\pi_{58,30}$ that is not divisible by $\tau$.
By inspection of the $E_\infty$-page of the
motivic Adams spectral sequence, there is no such 
element in $\pi_{58,30}$.
Therefore, $\beta_{12/6}$ must support a differential
in the Adams-Novikov spectral sequence.

Using the framework of Chapter \ref{ch:ANSS},
an Adams-Novikov $d_{2r+1}$ differential on $\beta_{12/6}$
would correspond to an element of $\pi_{57,r+30}$
that is not divisible by $\tau$; that is not killed by
$\tau^{r-1}$; and that is annihilated by $\tau^r$.
By inspection of the $E_\infty$-page of the motivic
Adams spectral sequence, the only possibility is that
$r = 1$, and the corresponding element of $\pi_{57,31}$
is detected by $D_{11}$.
\end{proof}

\begin{lemma}
There is no hidden $\tau$ extension on $h_3^2 g_2$.
\end{lemma}

\begin{proof}
The only other possibility is that there is a hidden $\tau$
extension from $h_3^2 g_2$ to $h_1 Q_2$.
However, $\eta \{h_1 Q_2\}$ equals $\{ h_1^2 Q_2 \}$, which is non-zero.
On the other hand, $\{ h_3^2 g_2 \}$ contains the element $\sigma^2 \{g_2\}$.
This is annihilated by $\eta$ because
$\eta \sigma^2 = 0$ \cite{Toda62}.
It follows that $\tau \{ h_3^2 g_2 \}$ cannot intersect $\{h_1 Q_2 \}$.
\end{proof}

\begin{lemma}
\label{lem:t.h3d1g}
There is no hidden $\tau$ extension on $h_3 d_1 g$.
\end{lemma}

\begin{proof}
The only other possibility is that there is a hidden $\tau$ extension
from $h_3 d_1 g$ to $P h_1^3 h_5 e_0$.  We will argue that this cannot occur.

Let $\alpha$ be an element of $\{h_3 d_1 g\}$.
Note that $\alpha$ equals either $\sigma \{d_1 g\}$ or 
$\sigma \{d_1 g\} + \nu \{ g t\}$.  In either case,
$\eta \alpha$ equals $\eta \sigma \{d_1 g\}$, which is a non-zero element
of $\{h_1 h_3 d_1 g\}$.
We know from Lemma \ref{lem:t.eta-d1} that $\tau \eta \{d_1 g\}$ is zero,
so $\tau \eta \alpha$ is zero.

On the other hand, 
$\eta \{ P h_1^3 h_5 e_0\}$ equals $\{ \tau^2 h_0 g^3 \}$, which
is non-zero.  Therefore,
$\tau \alpha$ cannot equal $\{ P h_1^3 h_5 e_0\}$.
\end{proof}

\begin{lemma}
\label{lem:t-Ph1^3h5e0}
There is a hidden $\tau$ extension from
$P h_1^3 h_5 e_0$ to $\tau d_0 w$.
\end{lemma}

\begin{proof}
Table \ref{tab:Toda} shows that 
$\langle \tau, \nu \kappabar^2, \eta \rangle$ in $\pi_{45,24}$
contains the element $\{ \tau w \}$.
This bracket has indeterminacy generated by $\tau \eta \{ g_2 \}$.
Table \ref{tab:Toda} also shows that the bracket
$\langle \nu \kappabar^2, \eta, \eta \kappa \rangle$ equals
$\{ P h_1^4 h_5 e_0 \}$,
with no indeterminacy.

Now use the shuffle
$\tau \langle \nu \kappabar^2, \eta, \eta \kappa \rangle =
\langle \tau, \nu \kappabar^2, \eta \rangle \eta \kappa$
to conclude that
$\tau \{ P h_1^4 h_5 e_0 \}$ equals $\eta \kappa \{\tau w\}$.

There is a classical extension $\eta \{ w \} = \{ d_0 l \}$,
as shown in Table \ref{tab:extn-refs}.
It follows that there is a motivic relation
$\eta \kappa \{ \tau w\} = \{ \tau d_0^2 l + d_0 u' \}$;
in particular, it is non-zero.

We have shown that $\tau \{ P h_1^4 h_5 e_0 \}$ is non-zero.
But this equals $\tau \eta \{ P h_1^3 h_5 e_0 \}$, so
$\tau \{ P h_1^3 h_5 e_0 \}$ is also non-zero.
There is just one possible non-zero value.
\end{proof}

\begin{lemma}
There is no hidden $\tau$ extension on $\tau^2 c_1 g^2$.
\end{lemma}

\begin{proof}
The only other possibility is that there is a hidden $\tau$
extension from $\tau^2 c_1 g^2$ to $\tau d_0 w$.
We showed in Lemma \ref{lem:t-Ph1^3h5e0} that there is a hidden
$\tau$ extension from $P h_1^3 h_5 e_0$ to $\tau d_0 w$.
Since $\{ P h_1^3 h_5 e_0 \}$ is contained in the indeterminacy of
$\{ \tau^2 c_1 g^2 \}$, there exists an element
of $\{ \tau^2 c_1 g^2 \}$ that is annihilated by $\tau$.
\end{proof}

\begin{lemma}
\label{lem:tau-h1^2X2}
Tentatively, there is a hidden $\tau$ extension
from $h_1^2 X_2$ to $\tau B_{23}$.
\end{lemma}

\begin{proof}
The claim is tentative because our analysis of Adams differentials
is incomplete in the relevant range.

There exists an element of $\{ \tau B_{23} \}$ 
that maps to zero in the homotopy groups of the cofiber of $\tau$,
which is described in Chapter \ref{ch:Ctau}.
\index{cofiber of tau@cofiber of $\tau$}
Therefore, this element of $\{ \tau B_{23} \}$ is divisible by $\tau$.
The only possibility is that there is a hidden
$\tau$ extension from $h_1^2 X_2$ to $\tau B_{23}$.
\end{proof}

\begin{lemma}
\label{lem:tau-h1^4X2}
Tentatively, there is a hidden $\tau$ extension from
$h_1^4 X_2$ to $B_8 d_0$.
\end{lemma}

\begin{proof}
The claim is tentative because our analysis of Adams differentials
is incomplete in the relevant range.

This follows from the hidden $\tau$ extension
from $h_1^2 X_2$ to $\tau B_{23}$ given in Lemma \ref{lem:tau-h1^2X2}
and the hidden $\eta$ extension from $\tau h_1 B_{23}$ to $B_8 d_0$
given in Lemma \ref{lem:eta-th1B23}.
\end{proof}

\begin{lemma}
\label{lem:tau-B8d0}
Tentatively, there is a hidden $\tau$ extension from
$B_8 d_0$ to $d_0 x'$.
\end{lemma}

\begin{proof}
The claim is tentative because our analysis of Adams differentials
is incomplete in the relevant range.

This follows immediately from the hidden $\tau$ extension from $B_8$
to $x'$ given in Table \ref{tab:Adams-tau}.
\end{proof}


\subsection{Hidden Adams $2$ extensions computations}
\label{subsctn:2-lemmas}

\begin{lemma}
\label{lem:2-h2^2h4}
There is no hidden $2$ extension on $h_2^2 h_4$.
\end{lemma}

\begin{proof}
The only other possibility is that there is a hidden $2$ extension from
$h_2^2 h_4$ to $\tau h_1 g$.
However, we will show later in Lemma \ref{lem:eta-th1g} that
there is a hidden $\eta$ extension on $\tau h_1 g$.
\end{proof}

\begin{lemma}
There is no hidden $2$ extension on $h_4 c_0$.
\end{lemma}

\begin{proof}
We showed in Lemma \ref{lem:sigma-h1h4}
that $\sigma \eta_4$ belongs to $\{ h_4 c_0 \}$,
and $2 \eta_4$ equals zero.
Now use Lemma \ref{lem:hidden-not-exist} to finish the claim.
\end{proof}

\begin{lemma}
\label{lem:2-h0h2g}
\mbox{}
\begin{enumerate}
\item
There is a hidden $2$ extension from 
$h_0 h_2 g$ to $h_1 c_0 d_0$.
\item
There is a hidden $2$ extension from 
$\tau h_0 h_2 g$ to $P h_1 d_0$.
\item
There is a hidden $2$ extension from
$h_0 h_2 g^2$ to $h_1 c_0 e_0^2$.
\item
There is a hidden $2$ extension from
$\tau h_0 h_2 g^2$ to $h_1 d_0^3$.
\end{enumerate}
\end{lemma}

\begin{proof}
Recall that there is a classical hidden $2$ extension from
$h_0 h_2 g$ to $P h_1 d_0$, as shown in Table \ref{tab:extn-refs}.
This immediately implies the second claim.
The first formula now follows from the second, using the hidden
$\tau$ extension from $h_1 c_0 d_0$ to $P h_1 d_0$ given in Table \ref{tab:Adams-tau}.

For the last two formulas,
recall that there is a classical hidden extension
$\eta^2 \kappabar^2 = \{ d_0^3 \}$, as shown in Table \ref{tab:extn-refs}.
This implies that there is a motivic hidden extension
$\tau \eta^3 \{ \tau g^2 \} = \{ h_1 d_0^3 \}$.
Use the relation $\tau \eta^3 = 4 \nu$ to deduce the fourth formula.

The third formula follows from the fourth, 
using the hidden $\tau$ extension from $h_1 c_0 e_0^2$ to $h_1 d_0^3$
given in Table \ref{tab:Adams-tau}.
\end{proof}

\begin{lemma}
\label{lem:2-h1h5}
\mbox{}
\begin{enumerate}
\item
There is no hidden $2$ extension on $h_1 h_5$.
\item
There is no hidden $2$ extension on $h_1 h_3 h_5$.
\end{enumerate}
\end{lemma}

\begin{proof}
For the first claim, the only other possibility is that there is a 
hidden $2$ extension from $h_1 h_5$ to $\tau d_1$.
Table \ref{tab:Toda} shows that 
the Toda bracket $\langle \eta, 2, \theta_4 \rangle$
intersects $\{h_1 h_5\}$.
Shuffle to obtain
\[
2 \langle \eta, 2, \theta_4 \rangle =
\langle 2, \eta, 2 \rangle \theta_4.
\]
This expression equals $\tau \eta^2 \theta_4$ by 
Table \ref{tab:Toda},
which must be zero.  Lemma \ref{lem:hidden-not-exist}
now finishes the first claim.

The second claim follows easily since $\sigma \eta_5$ is contained
in $\{ h_1 h_3 h_5 \}$.
\end{proof}

\begin{lemma}
There is no hidden $2$ extension on $p$.
\end{lemma}

\begin{proof}
The only other possibility is that there is a hidden $2$ extension
from $p$ to $h_1 q$.
Table \ref{tab:extn-refs} shows that
$\nu \theta_4$ is contained in $\{ p \}$.
Also, $2 \theta_4$ is zero.
Lemma \ref{lem:hidden-not-exist} now finishes the proof.
\end{proof}

\begin{lemma}
\mbox{}
\begin{enumerate}
\item
There is no hidden $2$ extension on $h_2 d_1$.
\item
There is no hidden $2$ extension on $h_3 d_1$.
\item
There is no hidden $2$ extension on $h_2 d_1 g$.
\item
There is no hidden $2$ extension on $h_3 d_1 g$.
\end{enumerate}
\end{lemma}

\begin{proof}
These follow immediately from Lemma \ref{lem:hidden-not-exist},
together with the facts that
$\nu \{d_1\}$ is contained in $\{ h_2 d_1\}$;
$\sigma \{d_1\}$ is contained in $\{ h_3 d_1\}$;
$\nu \{d_1 g\}$ is contained in $\{ h_2 d_1 g\}$;
$\sigma \{d_1 g\}$ is contained in $\{ h_3 d_1 g\}$;
and $2 \{d_1\}$ and $2 \{d_1 g\}$ are both zero.
\end{proof}

\begin{lemma}
\label{lem:2-h5c0}
There is no hidden $2$ extension on $h_5 c_0$.
\end{lemma}

\begin{proof}
The only other possibility is that there is a hidden $2$ extension
from $h_5 c_0$ to $\tau^2 c_1 g$.
Table \ref{tab:Toda} shows that 
$\langle \epsilon, 2, \theta_4 \rangle$ intersects $\{h_5 c_0 \}$.
Now shuffle to obtain that
\[
2 \langle \epsilon, 2, \theta_4 \rangle =
\langle 2, \epsilon, 2 \rangle \theta_4.
\]
By Table \ref{tab:Toda}, this equals 
$\tau \eta \epsilon \theta_4$, which must be zero.
Lemma \ref{lem:hidden-not-exist} now finishes the argument.
\end{proof}

\begin{lemma}
There is no hidden $2$ extension on $P h_1 h_5$.
\end{lemma}

\begin{proof}
The other possiblities are hidden $2$ extensions to $\tau^3 g^2$ or $h_1 u$.
We will show that neither can occur.

We already know from Table \ref{tab:extn-refs} that
there is a hidden $2$ extension from $\tau^3 g^2$ to $h_1 u$.
Therefore, there cannot be a hidden $2$ extension from $P h_1 h_5$
to $h_1 u$.

Table \ref{tab:Adams-eta} shows a hidden
$\eta$ extension from $\tau^3 g^2$ to $z$.  This implies that
$\tau^3 g^2$ cannot be the target of a hidden $2$ extension.
\end{proof}

\begin{lemma}
\label{lem:2-h0^2h5d0}
There is no hidden $2$ extension on $h_0^2 h_5 d_0$.
\end{lemma}

\begin{proof}
As shown in Table \ref{tab:extn-refs}, $\tau w$ supports a hidden $\eta$ extension.
Therefore, it cannot be the target of a hidden $2$ extension.
\end{proof}

\begin{lemma}
\label{lem:2-h2g2}
There is no hidden $2$ extension on $h_2 g_2$.
\end{lemma}

\begin{proof}
From the relation $h_2 g_2 + h_3 f_1 = 0$ in the $E_2$-page, we know that
$\{ h_2 g_2\}$ contains $\sigma \{f_1\}$.
Also, $\{f_1\}$ contains an element that is annihilated by $2$,
so $\{ h_2 g_2 \}$ contains an element that is annihilated by $2$.
Lemma \ref{lem:hidden-not-exist} finishes the argument.
\end{proof}

\begin{lemma}
\label{lem:2-Ph5c0}
There is no hidden $2$ extension on $P h_5 c_0$.
\end{lemma}

\begin{proof}
The element $P h_5 c_0$ detects $\rho_{15} \eta_5$ \cite{Tangora70b}*{Lemma 2.5}.
Also, $2 \eta_5$ is zero, so $\{ P h_5 c_0 \}$ contains an element
that is annihilated by $2$.  Lemma \ref{lem:hidden-not-exist}
finishes the proof.
\end{proof}

\begin{lemma}
\label{lem:2-e0r}
\mbox{}
\begin{enumerate}
\item
There is a hidden $2$ extension from
$e_0 r$ to $h_1 u'$.
\item
There is a hidden $2$ extension from
$\tau e_0 r$ to $P u$.
\end{enumerate}
\end{lemma}

\begin{proof}
Table \ref{tab:extn-refs} shows that there is a hidden $\eta$ extension
from $\tau w$ to $\tau d_0 l + u'$.
Also, from Table \ref{tab:Adams-tau}, there is a hidden $\tau$ extension
from $h_1 u'$ to $P u$.
Therefore, $\tau \eta^2 \{ \tau w\} = \{ P u \}$.

Recall from Table \ref{tab:Toda} that 
$\tau \eta^2 = \langle 2, \eta, 2 \rangle$.
Since $2 \{ \tau w \}$ is zero, we can shuffle to obtain
\[
\tau \eta^2 \{ \tau w \} = 
\langle 2, \eta, 2 \rangle \{ \tau w \} = 
2 \langle \eta, 2, \{ \tau w \} \rangle.
\]
This shows that $\{ P u \}$ is divisible by $2$.

By Lemmas \ref{lem:2-h2g2} and \ref{lem:2-Ph5c0},
the only possibility is that there is a hidden $2$ extension
from $\tau e_0 r$ to $P u$.  This establishes the second claim.

The first claim now follows from the second, using the hidden $\tau$
extension from $h_1 u'$ to $P u$ given in Table \ref{tab:Adams-tau}.
\end{proof}

\begin{lemma}
\label{lem:2-h2h5d0}
There is no hidden $2$ extension on $h_2 h_5 d_0$.
\end{lemma}

\begin{proof}
There is an element of $2 \{ h_5 d_0\}$ that is divisible by $4$,
as shown in Table \ref{tab:extn-refs}.
Therefore,
there is an element of $2 \nu \{ h_5 d_0 \}$ that is divisible by $4$.
However, zero is the only element of $\pi_{48,26}$
that is divisible by $4$.
\end{proof}

\begin{lemma}
\label{lem:2-h0B2}
There is no hidden $2$ extension on $h_0 B_2$.
\end{lemma}

\begin{proof}
We will show later in Lemma \ref{lem:nu-h3^2h5} that
there exists an element $\alpha$ of $\{h_3^2 h_5\}$ such that 
$\nu \alpha$ belongs to $\{B_2\}$. (We do not know whether
$\alpha$ equals $\theta_{4.5}$, but that does not matter here.
See Section \ref{sctn:notation} for further discussion.)
Therefore, $2 \nu \alpha$ belongs to $\{ h_0 B_2\}$.

Now $2 \cdot 2 \nu \alpha$ equals $\tau \eta^3 \alpha$.
There is a classical relation $\eta^3 \alpha = 0$ \cite{BJM84}*{Lemma 3.5},
which implies that $\eta^3 \alpha$ equals zero motivically as well.
\end{proof}

\begin{lemma}
\label{lem:2-h5c1}
\mbox{}
\begin{enumerate}
\item
There is no hidden $2$ extension on $h_5 c_1$.
\item
There is no hidden $2$ extension on $h_2 h_5 c_1$.
\end{enumerate}
\end{lemma}

\begin{proof}
Table \ref{tab:Toda} show that 
$\langle \sigmabar, 2, \theta_4 \rangle$ intersects $\{ h_5 c_1 \}$.
Now shuffle to obtain
\[
2 \langle \sigmabar, 2, \theta_4 \rangle =
\langle 2, \sigmabar, 2 \rangle \theta_4.
\]
Table \ref{tab:Toda} shows that
$\langle 2, \sigmabar, 2 \rangle$ consists of multiples of $2$,
and $2 \theta_4$ is zero.
Therefore,
$2 \langle \sigmabar, 2, \theta_4 \rangle$ is zero.
Lemma \ref{lem:hidden-not-exist}
now establishes the first claim.

The second claim follows immediately from the first because
$h_2 h_5 c_1$ contains $\nu \{h_5 c_1\}$.
\end{proof}

\begin{lemma}
\label{lem:2-h0h3g2}
There is no hidden $2$ extension from
$h_0 h_3 g_2$ to $h_2 B_2$.
\end{lemma}

\begin{proof}
The element $h_0 h_3 g_2$ detects $2 \sigma \{ g_2 \}$.
We will show later in Lemma \ref{lem:nu-h2B2} that $h_2 B_2$ supports a 
hidden $\nu$ extension.  Therefore, none of the elements of $\{ h_2 B_2 \}$
are divisible by $\sigma$.
\end{proof}

\begin{lemma}
\label{lem:nu-h2h5d0}
There is a hidden $2$ extension on $h_0 h_3 g_2$ if and only if
there is a hidden $\nu$ extension on $h_2 h_5 d_0$.
\end{lemma}

\begin{proof}
Let $\beta$ be an element of $\{ h_5 d_0 \}$.
Table \ref{tab:Toda} shows that
$\langle 2, \eta, \eta \beta \rangle$ intersects 
$\{h_2 h_5 d_0 \}$,
and
$\langle \eta, \eta \beta, \nu \rangle$ intersects $\{ h_0 h_3 g_2 \}$.
Now consider the shuffle
\[
2 \langle \eta, \eta \beta, \nu \rangle = 
\langle 2, \eta, \eta \beta \rangle \nu.
\]
The indeterminacy here is zero.
\end{proof}

\begin{lemma}
There is no hidden $2$ extension on $i_1$.
\end{lemma}

\begin{proof}
The only other possibility is that there is a hidden $2$ extension from
$i_1$ to $h_1^2 G_3$.  Because of the hidden $\tau$ extension
from $h_1^2 G_3$ to $d_0 u$ given in Table \ref{tab:Adams-tau}, this would imply a
hidden $2$ extension from $\tau i_1$ to $d_0 u$.

However, $\tau i_1$ detects $\nu \{C\}$,
and $2 \{ C \}$ is zero.  Therefore,
$\{ \tau i_1 \}$ contains an element that is annihilated by $2$.
Lemma \ref{lem:hidden-not-exist} implies that there cannot be
a hidden $2$ extension on $\tau i_1$.
\end{proof}

\begin{lemma}
\label{lem:2-B8}
\mbox{}
\begin{enumerate}
\item
There is no hidden $2$ extension on $B_8$.
\item
There is no hidden $2$ extension on $x'$.
\end{enumerate}
\end{lemma}

\begin{proof}
We will show in Lemma \ref{lem:epsilon-h3^2h5} that
$B_8$ detects $\epsilon \theta_{4.5}$.
Since $2 \epsilon$ is zero, it follows that 
$\{ B_8\}$ contains an element that is annihilated by $2$.
Lemma \ref{lem:hidden-not-exist} establishes the first claim.

The second claim follows from the first, using the hidden
$\tau$ extension from $B_8$ to $x'$ given in Table \ref{tab:Adams-tau}.
\end{proof}

\begin{lemma}
There is no hidden $2$ extension on $h_2 g n$.
\end{lemma}

\begin{proof}
Note that
$\nu \{ g n\}$ is contained in $\{ h_2 g n\}$,
and $2 \{ g n\}$ is zero.
Therefore, $\{ h_2 g n \}$ contains an element that is annihilated
by $2$, and Lemma \ref{lem:hidden-not-exist} finishes the proof.
\end{proof}

\begin{lemma}
\label{lem:2-h0h5i}
There is a hidden $2$ extension from $h_0 h_5 i$ to $\tau^4 e_0^2 g$.
\end{lemma}

\begin{proof}
This proof is different in spirit from the rest of the manuscript because it
relies on specific calculations in the classical Adams-Novikov spectral sequence.
\index{Adams-Novikov spectral sequence!classical}

The class $h_0 h_5 i$ detects an element of $\pi_{54,28}$
that is not divisible by $\tau$.
By Proposition \ref{prop:ANSS-Ctau},
this corresponds to an element 
in the classical Adams-Novikov spectral sequence 
in the 54-stem with filtration 2.  
The only possibility is the element $\beta_{10/2}$ \cite{Shimomura81}.

The image of $\beta_{10/2}$ in the Adams-Novikov
spectral sequence for $\tmf$ \cite{Bauer08} is
$\Delta^2 h_2^2$.  Since there is no filtration shift, this
is detectable in the chromatic spectral sequence.
In the Adams-Novikov spectral sequence for $\tmf$, there 
is a hidden $2$ extension from $\Delta^2 h_2^2$ 
to the class that detects $\kappa \kappabar^2$.
Therefore, in the Adams-Novikov spectral sequence for the sphere,
there must also be a hidden $2$ extension from 
$\beta_{10/2}$ to the class that detects $\kappa \kappabar^2$.

Since $\beta_{10/2}$ corresponds to $h_0 h_5 i$,
it follows that 
in the Adams spectral sequence, there is a hidden $2$ extension from
$h_0 h_5 i$ to $\tau^4 e_0^2 g$.
\end{proof}

\begin{lemma}
There is no hidden $2$ extension on $B_{21}$.
\end{lemma}

\begin{proof}
We showed in Lemma \ref{lem:kappa-h3^2h5} that
$B_{21}$ detects a multiple of $\kappa$.
Since $2 \kappa$ is zero, it follows that $\{ B_{21}\}$
contains an element that is annihilated by $2$.
Lemma \ref{lem:hidden-not-exist} finishes the proof.
\end{proof}

\begin{lemma}
\label{lem:2-t^3g^3}
\mbox{}
\begin{enumerate}
\item
Tentatively, there is a hidden $2$ extension from $\tau^3 g^3$ to
$d_0 u' + \tau d_0^2 l$.
\item
Tentatively, there is a hidden $2$ extension from $\tau h_0 h_2 g^3$ to
$h_1 d_0^2 e_0^2$.
\item
Tentatively, there is a hidden $2$ extension 
from $\tau e_0 g r$ to $d_0^2 u$.
\end{enumerate}
\end{lemma}

\begin{proof}
The claims are tentative because our analysis of Adams differentials
is incomplete in the relevant range.

The first formula  follows immediately from the hidden
$\tau$ extension from $\tau^2 h_0 g^3$ to $d_0 u' + \tau d_0^2 l$
given in Table \ref{tab:Adams-tau-tentative}.
The second formula follows immediately from the hidden
$\tau$ extension from $h_0^2 h_2 g^3$ to $h_1 d_0^2 e_0^2$
given in Table \ref{tab:Adams-tau-tentative}.
The third formula follows immediately from the hidden $\tau$
extension from $h_0 e_0 g r$ to $d_0^2 u$ given in Table \ref{tab:Adams-tau-tentative}.
\end{proof}


\subsection{Hidden Adams $\eta$ extensions computations}
\label{subsctn:eta-lemmas}

\index{Adams spectral sequence!hidden extension!eta@$\eta$}

\begin{lemma}
\label{lem:eta-c1}
There is no hidden $\eta$ extension on $c_1$.
\end{lemma}

\begin{proof}
The only other possibility is that there is a hidden $\eta$
extension from $c_1$ to $h_0^2 g$.
We will show in Lemma \ref{lem:nu-h0^2g} that
there is a hidden $\nu$ extension on $h_0^2 g$.
Therefore, it cannot be the target of a hidden $\eta$ extension.
\end{proof}

\begin{lemma}
\label{lem:eta-th1g}
\mbox{}
\begin{enumerate}
\item
There is a hidden $\eta$ extension 
from $\tau h_1 g$ to $c_0 d_0$.
\item
There is a hidden $\eta$ extension 
from $\tau^2 h_1 g$ to $P d_0$.
\item
There is a hidden $\eta$ extension 
from $\tau h_1 g^2$ to $c_0 e_0^2$.
\item
There is a hidden $\eta$ extension 
from $z$ to $\tau d_0^3$.
\end{enumerate}
\end{lemma}

\begin{proof}
There is a classical hidden $\eta$ extension from
$h_1 g$ to $P d_0$, as shown in Table \ref{tab:extn-refs}.
This implies that there is a motivic hidden $\eta$
extension from $\tau^2 h_1 g$ to $P d_0$.  This establishes
the second claim.

The first claim follows from the second claim, using the hidden
$\tau$ extension from $c_0 d_0$ to $P d_0$ given in Table \ref{tab:Adams-tau}.

Next,
there is a classical hidden $\eta$ extension from $z$ to $d_0^3$,
as shown in Table \ref{tab:extn-refs}.
This implies that there is a motivic hidden $\eta$ extension from
$z$ to $\tau d_0^3$.  This establishes the fourth claim.

The third claim follows from the fourth, using the hidden $\tau$
extensions from $\tau^2 h_1 g^2$ to $z$ and from
$c_0 e_0^2$ to $d_0^3$ given in Table \ref{tab:Adams-tau}.
\end{proof}

\begin{lemma}
There is no hidden $\eta$ extension on $p$.
\end{lemma}

\begin{proof}
Classically, $\nu \theta_4$ belongs to $\{ p\}$,
as shown in Table \ref{tab:extn-refs},
so the same formula holds motivically.
Therefore, $\{ p \}$ contains an element that is annihilated by $\eta$.
Lemma \ref{lem:hidden-not-exist} finishes the argument.
\end{proof}

\begin{lemma}
\label{lem:eta-h0^2h3h5}
There is a hidden $\eta$ extension from $h_0^2 h_3 h_5$ to
$\tau^2 c_1 g$.
\end{lemma}

\begin{proof}
First, $\tau^2 c_1 g$ equals $h_1 y$ on the $E_2$-page \cite{Bruner97}.
Use Moss's Convergence Theorem \ref{thm:Moss} together with the Adams differential
$d_2(y) = h_0^3 x$ to conclude that
$\{ \tau^2 c_1 g \}$ intersects $\langle \eta, 2, \alpha \rangle$,
where $\alpha$ is any element of $\{ h_0^2 x \}$.
\index{Convergence Theorem!Moss}

However, the later Adams differential $d_4(h_0 h_3 h_5) = h_0^2 x$
implies that $0$ belongs to $\{ h_0^2 x\}$.
Therefore,
$\{ \tau^2 c_1 g\}$ intersects $\langle \eta, 2, 0 \rangle$.
In other words, there exists an element of
$\{ \tau^2 c_1 g\}$ that is a multiple of $\eta$.
The only possibility is that there is a hidden
$\eta$ extension from $h_0^2 h_3 h_5$ to $\tau^2 c_1 g$.
\end{proof}

\begin{remark}
\label{rem:eta-h0^2h3h5}
Lemma \ref{lem:eta-h0^2h3h5}
shows that $\eta \{h_0^2 h_3 h_5 \}$ 
is an element of $\{\tau^2 c_1 g\}$.
\index{h3h5@$h_3 h_5$}
\index{c1g@$c_1 g$}
However, $\{ \tau^2 c_1 g\}$ contains two elements because
$u$ is in higher Adams filtration.  
The sum $\eta \{h_0^2 h_3 h_5 \} + \tau \sigmabar \kappabar$ is either zero
or equal to $\{ u\}$.  Both $\{h_0^2 h_3 h_5 \}$ and $\sigmabar$
map to zero in $\pi_{*,*}(\tmf)$, while $\{u\}$ is non-zero
in $\pi_{*,*}(\tmf)$.
\index{topological modular forms}
Therefore,
$\eta \{h_0^2 h_3 h_5 \} + \tau \sigmabar \kappabar$ must be zero.
We will need this observation in Lemmas \ref{lem:Ctau-eta-h0y} 
and \ref{lem:Ctau-eta-_th2c1g}.
\end{remark}

\begin{lemma}
There is no hidden $\eta$ extension on $\tau h_3 d_1$.
\end{lemma}

\begin{proof}
We know that $\sigma \{ \tau d_1 \}$ is contained in 
$\{ \tau h_3 d_1\}$, and there exists an element
of $\{ \tau d_1\}$ that is annihilated by $\eta$.
Therefore, $\{ \tau h_3 d_1 \}$ contains an element that 
is annihilated by $\eta$.
Lemma \ref{lem:hidden-not-exist} finishes the proof.
\end{proof}

\begin{lemma}
There is no hidden $\eta$ extension on $c_1 g$.
\end{lemma}

\begin{proof}
Since $\nu \{ t\}$ is contained in $\{ \tau c_1 g\}$,
Lemma \ref{lem:hidden-not-exist} implies that
there is no hidden $\eta$ extension on $\tau c_1 g$.
In particular,
there cannot be a hidden $\eta$ extension from 
$\tau c_1 g$ to $\tau h_0^2 g^2$.
Therefore, there cannot be a hidden $\eta$ extension from
$c_1 g$ to $h_0^2 g^2$.
\end{proof}

\begin{lemma}
There is no hidden $\eta$ extension on $\tau h_1 h_5 c_0$.
\end{lemma}

\begin{proof}
Table \ref{tab:Adams-eta} shows a
hidden $\eta$ extension from $\tau^3 g^2$ to $z$.
Therefore, there cannot be a hidden $\eta$ extension
from $\tau h_1 h_5 c_0$ to $z$.
\end{proof}

\begin{lemma}
\label{lem:eta-h1f1}
There is a hidden $\eta$ extension from $h_1 f_1$ to $\tau h_2 c_1 g$.
\end{lemma}

\begin{proof}
Note that $\{ \tau h_2 c_1 g\}$ contains $\nu^2 \{ t\}$.
Table \ref{tab:Toda} shows that
$\nu^2 = \langle \eta, \nu, \eta \rangle$.
Shuffle to compute that
\[
\nu^2 \{ t \} = \langle \eta, \nu, \eta \rangle \{ t \} =
\eta \langle \nu, \eta, \{t \} \rangle,
\]
so $\nu^2 \{t\}$ is divisible by $\eta$.
The only possibility is that there is a hidden $\eta$
extension from $h_1 f_1$ to $\tau h_2 c_1 g$.
\end{proof}

\begin{lemma}
\label{lem:eta-th1g2}
There is no hidden $\eta$ extension on $\tau h_1 g_2$.
\end{lemma}

\begin{proof}
We will show in Lemma \ref{lem:nu-h2c1} that $N$ supports a hidden
$\nu$ extension.  Therefore, $N$ cannot be the target of a hidden
$\eta$ extension.

For degree reasons, $\eta^3 \{g_2\}$ must be zero.
Therefore, $\tau \eta^3 \{g_2 \}$ must be zero.
This implies that 
the target of a hidden $\eta$ extension on $\tau h_1 g_2$
cannot support an $h_1$ multiplication.
Hence, 
there cannot be a hidden $\eta$ extension
from $\tau h_1 g_2$ to $B_1$ or to $\tau d_0 l + u'$.
\end{proof}

\begin{lemma}
\label{lem:eta-h3^2h5}
There is a hidden $\eta$ extension from $h_3^2 h_5$ to $B_1$.
\end{lemma}

\begin{proof}
This follows immediately from the analogous classical hidden 
extension given in Table \ref{tab:extn-refs},
but we repeat the interesting proof from \cite{Tangora70b} here for completeness.

First, Table \ref{tab:Massey} shows that 
$B_1 = \langle h_1, h_0, h_0^2 g_2 \rangle$.
Then Moss's Convergence Theorem \ref{thm:Moss} implies that
$\{ B_1 \}$ intersects $\langle \eta, 2, \{ h_0^2 g_2 \} \rangle$.
\index{Convergence Theorem!Moss}

Next, the classical product $\sigma \theta_4$ belongs to $\{ x\}$,
as shown in Table \ref{tab:extn-refs}.
Since $h_3 x = h_0^2 g_2$ on the 
$E_2$-page \cite{Bruner97}, it follows that 
$\sigma^2 \theta_4$ equals $\{ h_0^2 g_2 \}$.  The same formula
holds motivically.

Now $\langle \eta, 2, \sigma^2 \theta_4 \rangle$ is contained in
$\langle \eta, 2 \sigma^2, \theta_4 \rangle$, which equals
$\langle \eta, 0, \theta_4 \rangle$.
Therefore, $\{ B_1 \}$ contains an element of the form
$\theta_4 \alpha + \eta \beta$.

The possible non-zero values for $\alpha$ are $\eta_4$ or $\eta \rho_{15}$.
In the first case, $\theta_4 \alpha$ equals
$\theta_4 \langle 2, \sigma^2, \eta \rangle$, which equals
$\langle \theta_4, 2, \sigma^2 \rangle \eta$.
Therefore, in either case, $\theta_4 \alpha$ is a multiple of $\eta$,
so we can assume that $\alpha$ is zero.

We have now shown that $\{B_1\}$ contains a multiple of $\eta$.
Because of Lemma \ref{lem:eta-th1g2}, 
the only possibility is that there is a hidden $\eta$
extension from $h_3^2 h_5$ to $B_1$.
\end{proof}

\begin{lemma}
\mbox{}
\begin{enumerate}
\item
There is no hidden $\eta$ extension on $h_1 h_5 d_0$.
\item
There is no hidden $\eta$ extension on $N$.
\end{enumerate}
\end{lemma}

\begin{proof}
We showed in Lemma \ref{lem:2-e0r} that there is a hidden
$2$ extension on $e_0 r$.  Therefore, $e_0 r$ cannot
be the target of a hidden $\eta$ extension.
\end{proof}

\begin{lemma}
There is no hidden $\eta$ extension on $h_1 B_1$.
\end{lemma}

\begin{proof}
Classically, there is no hidden $\eta$ extension on $h_1 B_1$
\cite{BJM84}*{Theorem 3.1(i) and Lemma 3.5}.
Therefore, there cannot be a motivic hidden $\eta$ extension
from $h_1 B_1$ to $\tau d_0 e_0^2$.
\end{proof}

\begin{lemma}
\label{lem:eta-h5c1}
There is no hidden $\eta$ extension on $h_5 c_1$.
\end{lemma}

\begin{proof}
Table \ref{tab:Toda} shows that 
$\{ h_5 c_1 \}$ is contained in $\langle \nu, \sigma, \sigma \eta_5 \rangle$.
Next compute that
\[
\eta \langle \nu, \sigma, \sigma \eta_5 \rangle =
\langle \eta, \nu, \sigma \rangle \eta_5,
\]
which equals zero because $\langle \eta, \nu, \sigma \rangle$ is zero.
Therefore, $\{ h_5 c_1 \}$ contains an element that is annihilated by
$\eta$, so Lemma \ref{lem:hidden-not-exist} says that there cannot 
be a hidden $\eta$ extension on $h_5 c_1$.
\end{proof}

\begin{lemma}
\label{lem:eta-C}
There is no hidden $\eta$ extension on $C$.
\end{lemma}

\begin{proof}
Table \ref{tab:Toda} shows that 
$\{ C \}$ equals $\langle \nu, \eta, \tau \eta \alpha \rangle$,
where $\alpha$ is any element of $\{ g_2 \}$.
Compute that
\[
\eta \langle \nu, \eta, \tau \eta \alpha \rangle =
\langle \eta, \nu, \eta \rangle \tau \eta \alpha =
\nu^2 \cdot \tau \eta \alpha = 0.
\]
This shows that $\{ C \}$ contains an element that is annihilated by $\eta$,
so Lemma \ref{lem:hidden-not-exist} implies that there cannot be a hidden
$\eta$ extension on $C$.
\end{proof}

\begin{lemma}
\label{lem:eta-t^2e0m}
There is a hidden $\eta$ extension from $\tau^2 e_0 m$ to $d_0 u$.
\end{lemma}

\begin{proof}
This follows immediately from the hidden $\tau$ extensions from
$h_1 G_3$ to $\tau^2 e_0 m$ and 
from $h_1^2 G_3$ to $d_0 u$ given in Table \ref{tab:Adams-tau}.
\end{proof}

\begin{lemma}
\label{lem:eta-ti1}
There is no hidden $\eta$ extension on $\tau i_1$.
\end{lemma}

\begin{proof}
Suppose that there exists an element $\alpha$ of $\{ \tau i_1 \}$
such that $\eta \alpha$ belongs to $\{ \tau^2 e_0^2 g \}$.
Using the hidden $\tau$ extension from $\tau^2 h_1 e_0^2 g$
to $d_0 z$ given in Table \ref{tab:Adams-tau}, this would imply that
$\tau \eta^2 \alpha$ equals
$\{ d_0 z\}$.  

Recall from Table \ref{tab:Toda} 
that $\tau \eta^2 = \langle 2, \eta, 2 \rangle$.
Then the shuffle
\[
\tau \eta^2 \alpha =
\langle 2, \eta, 2 \rangle \{ \tau i_1 \} =
2 \langle \eta, 2, \alpha \rangle
\]
shows that $\{d_0 z\}$ is divisible by $2$.
However, this is not possible.
\end{proof}

\begin{lemma}
\label{lem:eta-t^3e0^2g}
There is a hidden $\eta$ extension from $\tau^3 e_0^2 g$
to $d_0 z$.
\end{lemma}

\begin{proof}
This follows immediately from the hidden $\tau$ extension
from $\tau^2 h_1 e_0^2 g$ to $d_0 z$ given in Table \ref{tab:Adams-tau}.
\end{proof}

\begin{lemma}
There is no hidden $\eta$ extension on $h_1 x'$.
\end{lemma}

\begin{proof}
We already showed in Lemma \ref{lem:eta-t^3e0^2g}
that there is a hidden $\eta$ extension from $\tau^3 e_0^2 g$ to $d_0 z$.
Therefore,
there cannot be a 
hidden $\eta$ extension
from $h_1 x'$ to $d_0 z$.
\end{proof}

\begin{lemma}
There is no hidden $\eta$ extension on $h_3^2 g_2$.
\end{lemma}

\begin{proof}
Note that $\sigma^2 \{g_2\}$ is contained in $\{ h_3^2 g_2 \}$,
and $\eta \sigma^2$ is zero \cite{Toda62}.
Therefore, $\{ h_3^2 g_2 \}$ contains an element that is annihilated
by $\eta$, and Lemma \ref{lem:hidden-not-exist} implies that there 
cannot be a hidden $\eta$ extension on $h_3^2 g_2$.
\end{proof}

\begin{lemma}
There is no hidden $\eta$ extension from $\tau h_1 Q_2$ to $\tau^2 d_0 w$.
\end{lemma}

\begin{proof}
Classically, $d_0 w$ maps to a non-zero element in the 
$E_\infty$-page of the Adams spectral sequence for $\tmf$ \cite{Henriques07}.
In $\tmf$, this class cannot be the target of an $\eta$ extension.
\end{proof}

\begin{lemma}
\label{lem:eta-td0w}
\mbox{}
\begin{enumerate}
\item
Tentatively, there is a hidden $\eta$ extension
from $\tau d_0 w$ to $d_0 u' + \tau d_0^2 l$.
\item
Tentatively, there is a hidden $\eta$ extension
from $\tau^3 g^3$ to $d_0 e_0 r$.
\item
Tentatively, there is a hidden $\eta$ extension
from $\tau^2 h_1 g^3$ to $d_0^2 e_0^2$.
\item
Tentatively, there is a hidden $\eta$ extension
from $d_0 e_0 r$ to $\tau^2 d_0^2 e_0^2$.
\item
Tentatively, there is a hidden $\eta$ extension
from $\tau^2 g w$ to $\tau^2 d_0 e_0 m$.
\item
Tentatively, there is a hidden $\eta$ extension
from $\tau^2 d_0 e_0 m$ to $d_0^2 u$.
\end{enumerate}
\end{lemma}

\begin{proof}
The claims are tentative because our analysis of Adams differentials
is incomplete in the relevant range.

For the first formula, use the hidden $\tau$ extensions
from $P h_1^3 h_5 e_0$ to $\tau d_0 w$ given in Lemma \ref{lem:t-Ph1^3h5e0}
and from $\tau^2 h_0 g^3$ to $d_0 u' + \tau d_0^2 l$ given in Table \ref{tab:Adams-tau-tentative}.
The second formula follows immediately from the
hidden $\tau$ extension from $\tau^2 h_1 g^3$ to $d_0 e_0 r$ given in Table \ref{tab:Adams-tau-tentative}.
The third formula follows immediately from the
hidden $\tau$ extension from $h_1^6 h_5 c_0 e_0$ to $d_0^2 e_0^2$ given in Table \ref{tab:Adams-tau-tentative}.
The fourth formula follows immediately from the 
hidden $\tau$ extensions from 
$\tau^2 h_1 g^3$ to $d_0 e_0 r$ and
from $h_1^6 h_5 c_0 e_0$ to $d_0^2 e_0^2$ given in Table \ref{tab:Adams-tau-tentative}.
The fifth formula follows immediately from the
hidden $\tau$ extension from
$h_1^5 X_1$ to $\tau^2 d_0 e_0 m$ given in Table \ref{tab:Adams-tau-tentative}.
The sixth formula follows immediately from the hidden
$\tau$ extensions from
$h_1^5 X_1$ to $\tau^2 d_0 e_0 m$ and 
from $h_0 e_0 g r$ to $d_0^2 u$ given in Table \ref{tab:Adams-tau-tentative}.
\end{proof}

\begin{lemma}
\label{lem:eta-th1B23}
\mbox{}
\begin{enumerate}
\item
Tentatively, there is a hidden $\eta$ extension from
$\tau h_1 B_{23}$ to $B_8 d_0$.
\item
Tentatively, there is a hidden $\eta$ extension from
$\tau^2 h_1 B_{23}$ to $d_0 x'$.
\end{enumerate}
\end{lemma}

\begin{proof}
The claims are tentative because our analysis of Adams differentials
is incomplete in the relevant range.

The first formula follows from the hidden $\eta$ extension from
$\tau h_1 g$ to $c_0 d_0$ given in Lemma \ref{lem:eta-th1g},
using that $\theta_{4.5} \{ \tau h_1 g \}$ is contained in 
$\{ \tau h_1 B_{23} \}$ by Lemma \ref{lem:kappabar-h3^2h5}
and that
$\theta_{4.5} \{ c_0 d_0 \}$ is contained in
$\{ B_8 d_0 \}$ by Lemma \ref{lem:epsilon-h3^2h5}.

The second formula follows from the first,
using the hidden $\tau$ extension from $B_8 d_0$ to $d_0 x'$
given in Lemma \ref{lem:tau-B8d0}.
\end{proof}


\subsection{Hidden Adams $\nu$ extensions computations}
\label{subsctn:nu-lemmas}

\index{Adams spectral sequence!hidden extension!nu@$\nu$}

\begin{lemma}
There is no hidden $\nu$ extension on $h_0 h_2 h_4$.
\end{lemma}

\begin{proof}
This follows immediately from Lemma \ref{lem:2-h2^2h4}, where we showed
that there is no hidden $2$ extension on $h_2^2 h_4$.
\end{proof}

\begin{lemma}
\label{lem:nu-h0^2g}
\mbox{}
\begin{enumerate}
\item
There is a hidden $\nu$ extension from $h_0^2 g$ to $h_1 c_0 d_0$.
\item
There is a hidden $\nu$ extension from $\tau h_0^2 g$ to $P h_1 d_0$.
\item
There is a hidden $\nu$ extension from $h_0^2 g^2$ to $h_1 c_0 e_0^2$.
\item
There is a hidden $\nu$ extension from $\tau h_0^2 g^2$ to $h_1 d_0^3$.
\end{enumerate}
\end{lemma}

\begin{proof}
These follow immediately from the hidden $2$ extensions
established in Lemma \ref{lem:2-h0h2g}.
\end{proof}

\begin{lemma}
\label{lem:nu-h2c1}
\mbox{}
\begin{enumerate}
\item
There is a hidden $\nu$ extension from $h_2 c_1$ to $h_1^2 h_4 c_0$.
\item
There is a hidden $\nu$ extension from $h_2 c_1 g$ to $h_1^6 h_5 c_0$.
\item
There is a hidden $\nu$ extension from $N$ to $P h_1^2 h_5 c_0$.
\end{enumerate}
\end{lemma}

\begin{proof}
Table \ref{tab:Massey} shows that 
$\langle h_2, h_2 c_1, h_1 \rangle$
equals $h_3 g$.
This Massey product contains no permanent cycles because
$h_3 g$ supports an Adams differential by Lemma \ref{lem:d2-h3g}.
Therefore, (the contrapositive of) Moss's Convergence Theorem \ref{thm:Moss}
implies that the Toda bracket $\langle \nu, \nu \sigmabar, \eta \rangle$
is not well-defined.  
\index{Convergence Theorem!Moss}
The only possibility is that
$\nu^2 \sigmabar$ is non-zero.  This implies
that there is a hidden $\nu$ extension on $h_2 c_1$,
and the only possible target for this hidden extension is $h_1^2 h_4 c_0$.
This finishes the first claim.

The proof of the second claim is similar.
Table \ref{tab:Massey} show that 
$\langle h_2, h_2 c_1 g, h_1 \rangle$
equals $h_3 g^2$.
Since $h_3 g^2$ supports a differential by Lemma \ref{lem:d2-h3g},
Moss's Convergence Theorem \ref{thm:Moss}
implies that the Toda bracket
$\langle \nu, \alpha, \eta \rangle$
is not well-defined for any $\alpha$ in $\{ h_2 c_1 g\}$.
\index{Convergence Theorem!Moss}
This implies that there is a hidden $\nu$ extension on $h_2 c_1 g$,
and the only possible target is $h_1^6 h_5 c_0$.
This finishes the second claim.

For the third claim,
we will 
first compute $\langle h_2, N, h_1 \rangle$ on the $E_2$-page.
The 
May differential $d_2( \Delta b_{21} h_1(1) ) = h_2 N$ and
May's Convergence Theorem \ref{thm:3-converge} imply that
$\langle h_2, N, h_1 \rangle$ equals an element that is detected
by $G_3$ in the $E_\infty$-page of the May spectral sequence.
\index{Convergence Theorem!May}
Because of the presence of $\tau g n$ in lower May filtration,
the bracket equals either $G_3$ or $G_3 + \tau g n$.
In any case, both of these elements support an Adams $d_2$ differential
by Lemma \ref{lem:d2-G3}
because $\tau g n$ is a product of permanent cycles.
Moss's Convergence Theorem \ref{thm:Moss} then
implies that the Toda bracket
$\langle \nu, \alpha, \eta \rangle$ is not well-defined for any
$\alpha$ in $\{ N \}$.
\index{Convergence Theorem!Moss}
This implies that there is a hidden $\nu$ extension on $N$,
and the only possible target is $P h_1^2 h_5 c_0$.
This finishes the third claim.
\end{proof}

\begin{lemma}
\label{lem:nu-th2^2g}
\mbox{}
\begin{enumerate}
\item
There is a hidden $\nu$ extension from $\tau h_2^2 g$ to $h_1 d_0^2$.
\item
There is a hidden $\nu$ extension from $\tau h_2^2 g^2$ to $h_1 d_0 e_0^2$.
\end{enumerate}
\end{lemma}

\begin{proof}
These follow from the hidden $\tau$ extensions from $h_1^2 h_3 g$
to $h_1 d_0^2$ 
and from $h_1^2 h_3 g^2$ to $h_1 d_0 e_0^2$
given in Table \ref{tab:Adams-tau}.
\end{proof}

\begin{lemma}
There is no hidden $\nu$ extension on $h_1 h_5$.
\end{lemma}

\begin{proof}
The only other possibility is that there is a hidden $\nu$ extension
from $h_1 h_5$ to $\tau^2 h_1 e_0^2$.
We know from Table \ref{tab:extn-refs}
that $\{ \tau^2 h_1 e_0^2 \}$ contains $\nu \{q \}$.
Since $\{q \}$ belongs to the indeterminacy of
$\{ h_1 h_5\}$, there
exists an element of $\{ h_1 h_5 \}$ that is annihilated by $\nu$.
Lemma \ref{lem:hidden-not-exist} finishes the proof.
\end{proof}

\begin{lemma}
There is no hidden $\nu$ extension on $p$.
\end{lemma}

\begin{proof}
Recall that there is a hidden $\nu$ extension from $h_4^2$ to $p$,
as shown in Table \ref{tab:extn-refs}.
If there were a hidden $\nu$ extension from $p$ to $t$,
then $\nu^4 \theta_{4}$ would belong to $\{ \tau h_2 c_1 g \}$.
This is impossible since $\nu^4$ is zero.
\end{proof}

\begin{lemma}
There is no hidden $\nu$ extension on $x$.
\end{lemma}

\begin{proof}
Recall from Table \ref{tab:extn-refs} that
there is a classical hidden $\sigma$ extension from $h_4^2$ to $x$.
Therefore, $\sigma \theta_4$ belongs to $\{ x\}$ motivically as well,
so $x$ cannot support a hidden $\nu$ extension.
\end{proof}

\begin{lemma}
There is no hidden $\nu$ extension on $h_0^2 h_3 h_5$.
\end{lemma}

\begin{proof}
The only other possibility is that there is a hidden $\nu$ extension
from $h_0^2 h_3 h_5$ to $P h_1^2 h_5$ or to $z$.
However, 
from Lemma \ref{lem:eta-th1g},
both $\{ P h_1^2 h_5 \}$ and $\{ z\}$ support multiplications 
by $\eta$.
Therefore, neither $P h_1^2 h_5$ nor $z$ can be the target of a 
hidden $\nu$ extension.
\end{proof}

\begin{lemma}
\label{lem:nu-h1h3h5}
\mbox{}
\begin{enumerate}
\item
There is no hidden $\nu$ extension on $h_1 h_3 h_5$.
\item
There is no hidden $\nu$ extension on $h_3 d_1$.
\item
There is no hidden $\nu$ extension on $\tau^2 c_1 g$.
\item
There is no hidden $\nu$ extension on $h_3 g_2$.
\end{enumerate}
\end{lemma}

\begin{proof}
In each case, the possible source of the hidden extension
detects an element that is divisible by $\sigma$.
Therefore, each possible source cannot support a hidden
$\nu$ extension.

Note that 
$h_3 q = \tau^2 c_1 g$ on the $E_2$-page \cite{Bruner97}.
\end{proof}

\begin{lemma}
\label{lem:nu-h5c0}
There is no hidden $\nu$ extension on $h_5 c_0$.
\end{lemma}

\begin{proof}
First, Table \ref{tab:Toda} shows that
$\{ h_5 c_0 \}$ contains $\langle \epsilon, 2, \theta_4 \rangle$.
Then shuffle to obtain
\[
\nu \langle \epsilon, 2, \theta_4 \rangle =
\langle \nu, \epsilon, 2 \rangle \theta_4.
\]
Since $\langle \nu, \epsilon, 2 \rangle$ is zero,
there is an element of $\{ h_5 c_0 \}$ that is annihilated by $\nu$.
Lemma \ref{lem:hidden-not-exist} finishes the proof.
\end{proof}

\begin{lemma}
\label{lem:nu-u}
\mbox{}
\begin{enumerate}
\item
There is a hidden $\nu$ extension from $u$ to $\tau d_0^3$.
\item
There is a hidden $\nu$ extension from $\tau w$ to $\tau^2 d_0 e_0^2$.
\end{enumerate}
\end{lemma}

\begin{proof}
First shuffle to compute that
\[
\nu \langle \eta, \nu, \{\tau^2 e_0^2 \} \rangle =
\langle \nu, \eta, \nu \rangle \{ \tau^2 e_0^2 \} =
(\epsilon + \eta \sigma) \{ \tau^2 e_0^2 \}.
\]
This last expression equals $\tau^2 \{ c_0 e_0^2 \}$,
which equals $\{ \tau d_0^3 \}$ because of the 
hidden $\tau$ extension from $c_0 e_0^2$ to $d_0^3$
given in Table \ref{tab:Adams-tau}.

Therefore, $\{ \tau d_0^3 \}$ is divisible by $\nu$.
Lemmas \ref{lem:nu-h1h3h5} and \ref{lem:nu-h5c0}
eliminate most of the possibilities.  
The only remaining possibility is that there is a hidden
$\nu$ extension on $u$.  This establishes the first claim.

The proof of the second claim is similar.  Shuffle to compute that
\[
\nu \langle \eta, \nu, \tau \kappabar^2 \rangle =
\langle \nu, \eta, \nu \rangle \tau \kappabar^2 =
(\epsilon + \eta \sigma) \tau \kappabar^2 =
\tau \epsilon \kappabar^2.
\]
By 
Lemma \ref{lem:epsilon-kappabar}, 
this last expression is detected by
$\tau^2 d_0 e_0^2$.

Therefore, $\{ \tau^2 d_0 e_0^2 \}$ is divisible by $\nu$.
Because of Lemmas 
\ref{lem:nu-Ph0h2h5} and
\ref{lem:nu-h3^2h5},
the only possibility is that there is a hidden $\nu$ extension
from $\tau w$ to $\tau^2 d_0 e_0^2$.
\end{proof}

\begin{lemma}
\label{lem:nu-Ph0h2h5}
\mbox{}
\begin{enumerate}
\item
There is no hidden $\nu$ extension on $P h_0 h_2 h_5$.
\item
There is no hidden $\nu$ extension on $h_0 g_2$.
\item
There is no hidden $\nu$ extension on $h_0 h_5 d_0$.
\end{enumerate}
\end{lemma}

\begin{proof}
These follow immediately from Lemmas \ref{lem:2-h0^2h5d0},
\ref{lem:2-h2g2}, and \ref{lem:2-h2h5d0}.
\end{proof}

\begin{lemma}
\mbox{}
\label{lem:nu-h3^2h5}
\begin{enumerate}
\item
There is a hidden $\nu$ extension from $h_3^2 h_5$ to $B_2$.
\item
There is a hidden $\nu$ extension from $h_0 h_3^2 h_5$ to $h_0 B_2$.
\end{enumerate}
\end{lemma}

\begin{proof}
The proof is similar in spirit to the proof of 
Lemma \ref{lem:eta-h3^2h5}.
Table \ref{tab:Massey} shows that 
$\langle h_2, h_0^2 g_2, h_0 \rangle$
equals $\{ B_2, B_2 + h_0^2 h_5 e_0 \}$.
Then Moss's Convergence Theorem \ref{thm:Moss} implies that
$\langle \nu, \sigma^2 \theta_4, 2 \rangle$
intersects $\{ B_2 \}$. 
\index{Convergence Theorem!Moss}
 Here we are using
that $\sigma^2 \theta_4$ belongs to
$\{ h_0^2 g_2 \}$, as shown in the proof of Lemma \ref{lem:eta-h3^2h5}.

This bracket contains
$\langle \nu, \sigma^2, 2 \theta_4 \rangle$,
which contains zero since $2 \theta_4$ is zero.
It follows that
$\{ B_2 \}$ contains an element in the indeterminacy of
$\langle \nu, \sigma^2 \theta_4, 2 \rangle$.
The only possibility is that there is a hidden $\nu$
extension from $h_3^2 h_5$ to $B_2$.
This finishes the proof of the first hidden extension.

The second hidden extension follows immediately from the first.
\end{proof}

\begin{lemma}
There is no hidden $\nu$ extension on $B_1$.
\end{lemma}

\begin{proof}
We showed in Lemma \ref{lem:eta-h3^2h5} that
$\{ B_1 \}$ contains an element that is divisible by $\eta$.
Therefore, $B_1$ cannot support a hidden $\nu$ extension.
\end{proof}

\begin{lemma}
\label{lem:nu-h2B2}
\mbox{}
\begin{enumerate}
\item
There is a hidden $\nu$ extension from $h_2 B_2$ to $h_1 B_8$.
\item
There is a hidden $\nu$ extension from $\tau h_2 B_2$ to $h_1 x'$.
\end{enumerate}
\end{lemma}

\begin{proof}
As discussed in Section \ref{sctn:notation},
$\sigma \theta_{4.5}$ is detected in Adams filtration greater than 6.
Thus,
$\eta^2 \sigma \theta_{4.5}$ is zero, even though
$\sigma \theta_{4.5}$ itself could possibly be
detected by $\tau^2 d_1 g$ or $\tau^2 e_0 m$.

Recall from Table \ref{tab:Adams-compound-extn}
that $\eta^2 \sigma + \nu^3 = \eta \epsilon$.
Therefore,
$\nu^3 \theta_{4.5}$ equals $\eta \epsilon \theta_{4.5}$.
Lemma \ref{lem:epsilon-h3^2h5} implies that
$\eta \epsilon \theta_{4.5}$ is detected by $h_1 B_8$,
so $\{ h_1 B_8 \}$ contains an element that is divisible by $\nu$.
The only possibility is that 
there must be a hidden $\nu$ extension from $h_2 B_2$ to $h_1 B_8$.
This establishes the first claim.

The second claim follows easily from the first, using the hidden $\tau$
extension from $h_1 B_8$ to $h_1 x'$ given in Table \ref{tab:Adams-tau}.
\end{proof}

\begin{lemma}
\label{lem:nu-h1G3}
\mbox{}
\begin{enumerate}
\item
There is a hidden $\nu$ extension from $h_1 G_3$ to $\tau^2 h_1 e_0^2 g$.
\item
There is a hidden $\nu$ extension from $\tau^2 e_0 m$ to $d_0 z$.
\end{enumerate}
\end{lemma}

\begin{proof}
Table \ref{tab:Toda} shows that 
$\langle \{ q\}, \eta^3, \eta_4 \rangle$ 
equals $\{ h_1 G_3 \}$, and
$\langle \{ \tau^2 h_1 e_0^2 \}, \eta^3, \eta_4 \rangle$
equals $\{ \tau^2 h_1 e_0^2 g \}$.
Neither Toda bracket has indeterminacy; for the second bracket,
one needs that $\eta_4 \{t \}$ is contained in
\[
\langle \eta, \sigma^2, 2 \rangle \{ t \} =
\eta \langle \sigma^2, 2, \{ t \} \rangle,
\]
which must be zero.

Now compute that
\[
\nu \{ h_1 G_3 \} =
\nu \langle \{q \}, \eta^3, \eta_4 \rangle =
\langle \nu \{q \}, \eta^3, \eta_4 \rangle =
\langle \{ \tau^2 h_1 e_0^2 \}, \eta^3, \eta_4 \rangle =
\{ \tau^2 h_1 e_0^2 g \}.
\]
Here we are using that none of the Toda brackets has indeterminacy,
and we are using Table \ref{tab:extn-refs} to identify
$\nu \{q \}$ with $\{\tau^2 h_1 e_0^2 \}$.
This establishes the first claim.

The second claim follows easily from the first, using the 
hidden $\tau$ extensions from $h_1 G_3$ to $\tau^2 e_0 m$
and from $\tau^2 h_1 e_0^2 g$ to $d_0 z$ given in Table \ref{tab:Adams-tau}.
\end{proof}

\begin{lemma}
There is no hidden $\nu$ extension on $\tau^2 d_1 g$.
\end{lemma}

\begin{proof}
We showed in Lemma \ref{lem:nu-h1G3} that 
there is a hidden $\nu$ extension from $\tau^2 e_0 m$ to $d_0 z$.
Therefore, there cannot be a hidden $\nu$ extension
from $\tau^2 d_1 g$ to $d_0 z$.
\end{proof}

\begin{lemma}
\label{lem:nu-h1^6h5e0}
There is a hidden $\nu$ extension from $h_1^6 h_5 e_0$ to $h_2 e_0^2 g$.
\end{lemma}

\begin{proof}
This follows immediately from the hidden $\tau$
extension from $h_1^6 h_5 e_0$ to $\tau e_0^2 g$ 
given in Table \ref{tab:Adams-tau}.
\end{proof}

\begin{lemma}
\label{lem:nu-h0h2h5i}
Tentatively, there is a hidden $\nu$ extension
from $h_0 h_2 h_5 i$ to $\tau^2 d_0^2 l$.
\end{lemma}

\begin{proof}
The claim is tentative because our analysis of Adams differentials
is incomplete in the relevant range.

As explained in the proof of Lemma \ref{lem:2-h0h5i},
the class $h_0 h_5 i$ dectects an element of classical $\pi_{54}$
that maps to an element of $\pi_{54} \tmf$ that is detected by 
$\Delta^2 h_2^2$ in the Adams-Novikov spectral sequence for $\tmf$.
Then $h_0 h_2 h_5 i$ detects an element in $\pi_{57}$ that maps
that maps to an element of $\pi_{57} \tmf$ that is detected by 
$\Delta^2 h_2^3$ in the Adams-Novikov spectral sequence for $\tmf$.

In the classical Adams-Novikov spectral sequence for $\tmf$,
there is a hidden $\nu$ extension from $\Delta^2 h_2^3$ to $2 g^3$ \cite{Bauer08}.
Therefore, the corresponding hidden extension must occur in
the motivic Adams spectral sequence as well.
\end{proof}

\begin{lemma}
\mbox{}
\label{lem:nu-td0w}
\begin{enumerate}
\item
Tentatively, there is a hidden $\nu$ extension
from $P h_1^3 h_5 e_0$ to $\tau d_0^2 e_0^2$.
\item
Tentatively, there is a hidden $\nu$ extension
from $\tau d_0 w$ to $\tau^2 d_0^2 e_0^2$.
\item
Tentatively, there is a hidden $\nu$ extension
from $\tau g w + h_1^4 X_1$ to $\tau^2 e_0^4$.
\end{enumerate}
\end{lemma}

\begin{proof}
The claims are tentative because our analysis of Adams differentials
is incomplete in the relevant range.

The second formula follows from the hidden $\nu$ extension
from $\tau w$ to $\tau^2 d_0 e_0^2$ given in Lemma \ref{lem:nu-u}.
The first formula then follows using the hidden $\tau$
extension from $P h_1^3 h_5 e_0$ given in Lemma \ref{lem:t-Ph1^3h5e0}.

For the third formula, start with the hidden
$\nu$ extension from $\tau w$ to $\tau^2 d_0 e_0^2$.
Multiply by $\tau g$ to obtain a hidden $\nu$
extension from $\tau^2 g w$ to $\tau^3 e_0^4$.
The third formula follows immediately.
\end{proof}

\begin{lemma}
\label{lem:nu-th0^2g^3}
Tentatively, there is a hidden $\nu$ extension
from $\tau h_0^2 g^3$ to $h_1 d_0^2 e_0^2$.
\end{lemma}

\begin{proof}
The claim is tentative because our analysis of Adams differentials
is incomplete in the relevant range.

This follows immediately from the 
hidden $\tau$ extension from $h_0^2 h_2 g^3$ to $h_1 d_0^2 e_0^2$ 
given in Table \ref{tab:Adams-tau-tentative}.
\end{proof}

\begin{lemma}
\label{lem:nu-h2c1g^2}
Tentatively, there is a hidden $\nu$ extension
from $h_2 c_1 g^2$ to $h_1^8 D_4$.
\end{lemma}

\begin{proof}
The claim is tentative because our analysis of Adams differentials
is incomplete in the relevant range.

The argument is essentially the same as the proof of Lemma \ref{lem:nu-h2c1}.
Table \ref{tab:Massey} shows that 
$\langle h_2, h_2 c_1 g^2, h_1 \rangle$ equals
$h_3 g^3$.
Since $h_3 g^3$ supports a differential by Lemma \ref{lem:d2-h3g},
Moss's Convergence Theorem \ref{thm:Moss}
implies that the Toda bracket
$\langle \nu, \alpha, \eta \rangle$
is not well-defined for any $\alpha$ in $\{ h_2 c_1 g^2 \}$.
\index{Convergence Theorem!Moss}
This implies that there is a hidden $\nu$ extension on $h_2 c_1 g^2$,
and the only possible target is $h_1^8 D_4$.
\end{proof}


\subsection{Miscellaneous Adams hidden extensions}
\label{subsctn:hidden-misc}

In this section, we include some miscellaneous hidden extensions.
They are needed at various points for technical arguments, but they
are interesting for their own sakes as well.

\index{Adams spectral sequence!hidden extension!sigma@$\sigma$}

\begin{lemma}
\label{lem:sigma-h1h4}
There is a hidden $\sigma$ extension from $h_1 h_4$ to $h_4 c_0$.
\end{lemma}

\begin{proof}
The product $\eta \epsilon \eta_4$ is contained in $\{h_1^2 h_4 c_0 \}$.
Now recall the hidden relation
$\eta \epsilon = \eta^2 \sigma + \nu^3$ 
from Table \ref{tab:Adams-compound-extn}.
Also $\nu \eta_4$ is zero because there is no other possibility.
Therefore, $\eta^2 \sigma \eta_4$ is contained in $\{h_1^2 h_4 c_0\}$.
It follows that
$\sigma \eta_4$ is contained in $\{ h_4 c_0 \}$.
\end{proof}

\index{Adams spectral sequence!hidden extension!sigma@$\sigma$}

\begin{lemma}
\label{lem:sigma-h1h3h5}
There is no hidden $\sigma$ extension on $h_1 h_3 h_5$.
\end{lemma}

\begin{proof}
The element $\sigma \eta_5$ belongs to $\{ h_1 h_3 h_5 \}$.
We will show that $\sigma^2 \eta_5$ is zero and then apply
Lemma \ref{lem:hidden-not-exist}.

Table \ref{tab:Toda} shows that $\eta_5$ 
belongs to $\langle \eta, 2, \theta_4 \rangle$.
Then $\sigma^2 \eta_5$ belongs to 
\[
\sigma^2 \langle \eta, 2, \theta_4 \rangle =
\langle \sigma^2, \eta, 2 \rangle \theta_4.
\]

Finally, we must show that $\langle \sigma^2, \eta, 2 \rangle$
is zero in $\pi_{16,9}$.  First shuffle to obtain
\[
\langle \sigma^2, \eta, 2 \rangle \eta = 
\sigma^2 \langle \eta, 2, \eta \rangle.
\]
Table \ref{tab:Toda} shows that 
$\langle \eta, 2, \eta \rangle$ equals $\{ 2 \nu, 6 \nu \}$,
so $\sigma^2 \langle \eta, 2, \eta \rangle$ is zero.
Since multiplication by $\eta$ is injective on $\pi_{16,9}$,
this shows that $\langle \sigma^2, \eta, 2 \rangle$ is zero.
\end{proof}

\index{Adams spectral sequence!hidden extension!epsilon@$\epsilon$}

\begin{lemma}
\label{lem:epsilon-kappabar}
\mbox{}
\begin{enumerate}
\item
There is a hidden $\epsilon$ extension from $\tau g$ to $d_0^2$.
\item
There is a hidden $\epsilon$ extension from $\tau g^2$ to $d_0 e_0^2$.
\end{enumerate}
\end{lemma}

\begin{proof}
Table \ref{tab:Toda} shows that $\epsilon$ is contained in
$\langle 2 \nu, \nu, \eta \rangle$.  Therefore,
$\eta \epsilon \kappabar$ equals
$\langle 2 \nu, \nu, \eta \rangle \eta \kappabar$
with no indeterminacy.
This expression equals $\langle 2\nu, \nu, \eta^2 \kappabar \rangle$ because
the latter still has no indeterminacy.

Lemma \ref{lem:eta-th1g} tells us that we can rewrite this bracket as
$\langle 2\nu, \nu, \{c_0 d_0\} \rangle$, which equals
$\langle 2\nu, \nu, \epsilon \rangle \kappa$.
Table \ref{tab:Toda} shows that 
$\langle 2\nu, \nu, \epsilon \rangle$ equals $\eta \kappa$.

Therefore, $\eta \epsilon \kappabar$ equals $\eta \kappa^2$.
It follows that $\epsilon \kappabar$ equals $\kappa^2$.
This establishes the first claim.

The argument for the second claim is essentially the same.
Start with 
$\eta \epsilon \{ \tau g^2 \} = 
\langle 2 \nu, \nu, \eta \rangle \eta \{\tau g^2\}$.
This equals $\langle 2\nu, \nu, \{ c_0 e_0^2\} \rangle$, which is the same
as $\{ h_1 d_0 e_0^2 \}$.
This shows that $\eta \epsilon \{ \tau g^2 \}$ 
equals $\eta \{ d_0 e_0^2 \}$, so
$\epsilon \{ \tau g^2 \}$ equals $\{ d_0 e_0^2 \}$.
\end{proof}

\begin{remark}
\label{rem:c0-tg^3}
Based on the calculations in Lemma \ref{lem:epsilon-kappabar},
one might expect that there is a hidden $\epsilon$
extension from $\tau g^3 + h_1^4 h_5 c_0 e_0$ to 
$e_0^4$.
\end{remark}

\index{Adams spectral sequence!hidden extension!epsilon@$\epsilon$}

\begin{lemma}
\label{lem:epsilon-q}
There is a hidden $\epsilon$ extension from $q$ to $h_1 u$.
\end{lemma}

\begin{proof}
This proof follows the argument of the proof of \cite{KM93}*{Lemma 2.1},
which we include for completeness.

First, recall from Table \ref{tab:Toda} that
$\epsilon + \eta \sigma$ equals $\langle \nu, \eta, \nu \rangle$.
Then
$(\epsilon + \eta \sigma) \{q \}$ equals
$\langle \nu, \eta, \nu \rangle \{q\}$, which is contained in
$\langle \nu, \eta, \nu \{q \} \rangle$.
It follows from Table \ref{tab:extn-refs} that 
$\nu \{q \}$ equals $\tau \eta \kappa \kappabar$,
so
$(\epsilon + \eta \sigma) \{q \}$ belongs to 
$\langle \nu, \eta, \tau \eta \kappa \kappabar \rangle$.

On the other hand, this bracket contains
$\langle \nu, \eta, \tau \eta \kappa \rangle \kappabar$.
Table \ref{tab:Toda} shows that 
$\langle \nu, \eta, \tau \eta \kappa \rangle$
equals $\{ \tau h_0 g \} = \{ 2 \kappabar, 6 \kappabar \}$,
and $4 \kappabar^2$ is zero.
Therefore,
$\langle \nu, \eta, \tau \eta \kappa \rangle \kappabar$
equals $2 \kappabar^2$.

This shows that the difference
$(\epsilon + \eta \sigma) \{q \} - 2 \kappabar^2$
is contained in the indeterminacy of the bracket
$\langle \nu, \eta, \tau \eta \kappa \kappabar \rangle$.
The indeterminacy of this bracket consists of multiples of $\nu$.

Each of the terms in $(\epsilon + \eta \sigma) \{q \} - 2 \kappabar^2$
is in Adams filtration at least 9, and there are no multiples of $\nu$
in those filtrations.
Therefore, 
$(\epsilon + \eta \sigma) \{q \}$ equals $2 \kappabar^2$.

We now need to show that $\eta \sigma \{ q\}$ is zero.
Because $h_3 q = h_2 t$ in $\Ext$, we know that
$\sigma \{q \} + \nu \{ t \}$ either equals zero or $\{ u \}$.
Note that $\kappa (\sigma \{q \} + \nu \{ t \})$ is zero, 
while $ \kappa \{ u \} = \{ d_0 u \}$ is non-zero.
Therefore,
$\sigma \{q \} + \nu \{ t \}$ equals zero, and
$\eta \sigma \{q \}$ is zero as well.
\end{proof}

\index{Adams spectral sequence!hidden extension!epsilon@$\epsilon$}

\begin{lemma}
\label{lem:epsilon-h3^2h5}
There is a hidden $\epsilon$ extension from 
$h_3^2 h_5$ to $B_8$.
\end{lemma}

\begin{proof}
First, there is a relation $h_1 B_8 = c_0 B_1$ on the $E_2$-page,
which is not hidden in the May spectral sequence. Since 
$B_1$ detects $\eta \theta_{4.5}$ by definition of $\theta_{4.5}$
(see Section \ref{sctn:notation}),
we get that $h_1 B_8$ detects $\eta \epsilon \theta_{4.5}$
and that $B_8$ detects $\epsilon \theta_{4.5}$.
\end{proof}

On the $E_\infty$-page, we have the relation
$h_2^3 h_5 = h_1^2 h_3 h_5$ in the 40-stem.  We will next show that
this relation gives rise to a compound hidden extension that is
analogous to Toda's relation
$\nu^3 + \eta^2 \sigma = \eta \epsilon$ (see Table \ref{tab:Adams-compound-extn}). the 
Note that the element $\epsilon \eta_5$ is detected by $h_1 h_5 c_0$,
whose Adams filtration is higher than the Adams filtration of
$h_2^3 h_5 = h_1^2 h_3 h_5$.

\index{Adams spectral sequence!hidden extension!compound}

\begin{lemma}
\label{lem:compound-h1^2h3h5}
$\nu \{ h_2^2 h_5 \} + \eta \sigma \eta_5$ equals
$\epsilon \eta_5$.
\end{lemma}

\begin{proof}
Table \ref{tab:Toda} shows that 
$\langle \nu^2, 2, \theta_4 \rangle$ equals $\{ h_2^2 h_5 \}$.
Note that $\{ x\}$ belongs to the indeterminacy, since there
is a hidden $\sigma$ extension from $h_4^2$ to $x$ as shown in 
Table \ref{tab:extn-refs}.

Similarly, $\langle \nu^3, 2, \theta_4 \rangle$
intersects $\{ h_2^3 h_5 \}$, with no indeterminacy.
In order to compute the indeterminacy, we need to know that
$\eta \mu_9 \theta_4$ is zero.  This follows from the calculation
\[
\eta \theta_4 \langle \eta, 2, 8 \sigma \rangle =
\langle \eta \theta_4, \eta, 2 \rangle 8 \sigma = 0.
\]

Table \ref{tab:Toda} also shows that 
$\langle \eta, 2, \theta_4 \rangle$ equals
$\{ \eta_5, \eta_5 + \eta \rho_{31} \}$.

With these tools, compute that
\[
\nu \{ h_2^2 h_5 \} = 
\nu \langle \nu^2, 2, \theta_4 \rangle = 
\langle \nu^3, 2, \theta_4 \rangle
\]
because there is no indeterminacy in the last bracket.
This equals
$\langle \eta^2 \sigma + \eta \epsilon, 2, \theta_4 \rangle$,
which equals
$(\eta \sigma + \epsilon) \langle \eta, 2, \theta_4 \rangle$,
again because there is no indeterminacy.
Finally, this last expression equals 
$\eta \sigma \eta_5 + \epsilon \eta_5$.
\end{proof}

\index{Adams spectral sequence!hidden extension!nu4@$\nu_4$}

\begin{lemma}
\label{lem:nu4-h4^2}
There is a hidden $\nu_4$ extension from $h_4^2$ to $h_2 h_5 d_0$.
\end{lemma}

\begin{proof}
Table \ref{tab:Toda} shows that $\langle \sigma, \nu, \sigma \rangle$
consists of a single element $\alpha$ contained in $\{ h_2 h_4 \}$.
Then $\alpha$ must be of the form $k \nu_4$ or $k \nu_4 + \eta \mu_{17}$
where $k$ is odd.
Since $2 \theta_4$ and $\eta \mu_{17} \theta_4$ are both zero,
we conclude that
$\langle \sigma, \nu, \sigma \rangle \theta_4$ equals
$\nu_4 \theta_4$.

Table \ref{tab:extn-refs} shows that
$\sigma \theta_4$ equals $\{x\}$.
Therefore, $\nu_4 \theta_4$ is contained in
$\langle \sigma, \nu, \{ x \} \rangle$.

Next compute that $h_2 h_5 d_0 = \langle h_3, h_2, x \rangle$
with no indeterminacy.
This follows from the shuffle
\[
h_2 \langle h_3, h_2, x \rangle =
\langle h_2, h_3, h_2 \rangle x =
h_3^2 x = h_2^2 h_5 d_0.
\]
Then Moss's Convergence Theorem \ref{thm:Moss} implies that the Toda
bracket
$\langle \sigma, \nu, \{x \} \rangle$
intersects $\{ h_2 h_5 d_0 \}$.
\index{Convergence Theorem!Moss}
The indeterminacy in $\langle \sigma, \nu, \{x\} \rangle$
is concentrated in Adams filtration strictly greater than 6,
so $\langle \sigma, \nu, \{ x\} \rangle$ is contained in
$\{ h_2 h_5 d_0 \}$.
This shows that $\nu_4 \theta_4$ is contained in $\{ h_2 h_5 d_0 \}$.
\end{proof}

\begin{lemma}
\label{lem:bracket-theta4-2-sigma^2}
$\langle \theta_4, 2, \sigma^2 \rangle$ 
is contained in $\{ h_0 h_3^2 h_5 \}$, with
indeterminacy generated by $\rho_{15} \theta_4$ in $\{h_0^2 h_5 d_0 \}$.
\end{lemma}

\begin{proof}
Table \ref{tab:Toda} shows that
$\nu_4$ is contained in $\langle 2, \sigma^2, \nu \rangle$.
Therefore,
$\nu_4 \theta_4$ is contained in
$\langle \theta_4, 2, \sigma^2 \rangle \nu$.
On the other hand, 
Lemma \ref{lem:nu4-h4^2} says that 
$\nu_4 \theta_4$ is contained in
$\{ h_2 h_5 d_0 \}$.

We have now shown that $\langle \theta_4, 2, \sigma^2 \rangle$
contains an element $\alpha$ such that $\nu \alpha$ belongs to
$\{ h_2 h_5 d_0 \}$.  In particular, $\alpha$ has Adams filtration at most 5.
In addition,
we know that $2 \alpha$ is zero because of 
the shuffle
\[
\langle \theta_4, 2, \sigma^2 \rangle 2 =
\theta_4 \langle 2, \sigma^2, 2 \rangle = 0.
\]
Here we have used Table \ref{tab:Toda} for the
bracket $\langle 2, \sigma^2, 2\rangle$.
The only possibility is that $\alpha$ belongs to 
$\{h_0 h_3^2 h_5 \}$.

The indeterminacy follows immediately from 
\cite{Tangora70b}*{Corollary 2.8}.
\end{proof}

\index{Adams spectral sequence!hidden extension!eta4@$\eta_4$}

\begin{lemma}
\label{lem:eta4-h4^2}
There is a hidden $\eta_4$ extension from $h_4^2$ to $h_1 h_5 d_0$.
\end{lemma}

\begin{proof}
Table \ref{tab:Toda} shows that
$\eta_4$ belongs to the Toda bracket $\langle \eta, \sigma^2, 2 \rangle$.
Then $\eta_4 \theta_4$ belongs to
$\eta \langle \sigma^2, 2, \theta_4 \rangle$.

Recall from the proof of Lemma \ref{lem:bracket-theta4-2-sigma^2}
that $\langle \sigma^2, 2, \theta_4 \rangle$ consists of
elements $\alpha$ in $\{ h_0 h_3^2 h_5 \}$ of order $2$.
Table \ref{tab:extn-refs} shows that there is a hidden $4$ extension
from $h_3^2 h_5$ to $h_0 h_5 d_0$.
It follows that each $\alpha$ must be of the form
$2 \gamma - \beta$, where $\gamma$ belongs to $\{h_3^2 h_5\}$
and $\beta$ belongs to $\{ h_5 d_0 \}$.
Then $\eta \alpha = \eta \beta$ must belong to $\{h_1 h_5 d_0\}$.
This shows that $\eta_4 \theta_4$ belongs to
$\{ h_1 h_5 d_0 \}$.
\end{proof}

\index{Adams spectral sequence!hidden extension!kappa@$\kappa$}

\begin{lemma}
\label{lem:kappa-h3^2h5}
There is a hidden $\kappa$ extension from either $h_3^2 h_5$ 
or $h_5 d_0$ to $B_{21}$.
\end{lemma}

\begin{proof}
The element $\tau h_1 B_{21}$ may be the target of an Adams differential.
Regardless, the element $h_1 B_{21}$ is non-zero on the $E_\infty$-page.
Note that $h_1 B_{21}$ equals $d_0 B_1$ on the $E_2$-page \cite{Bruner97}.
Since $B_1$ detects $\eta \theta_{4.5}$ by definition of $\theta_{4.5}$ 
(see Section \ref{sctn:notation}),
$d_0 B_{21}$ detects $\eta \kappa \theta_{4.5}$.
This implies that $B_{21}$ detects $\kappa \theta_{4.5}$.

Therefore, $B_{21}$ must be the target of a hidden $\kappa$ extension.
The possible sources of this hidden extension are $h_3^2 h_5$ or
$h_5 d_0$.
\end{proof}

\index{Adams spectral sequence!hidden extension!kappabar@$\kappabar$}
\index{Adams spectral sequence!hidden extension!tentative}

\begin{lemma}
\label{lem:kappabar-h3^2h5}
Tentatively, there is a hidden
$\kappabar$ extension from either $h_3^2 h_5$ or $h_5 d_0$ to $\tau B_{23}$.
\end{lemma}

\begin{proof}
The claim is tentative because our analysis of Adams differentials
is incomplete in the relevant range.

The element $\tau h_1 B_{23}$ equals $\tau g B_1$ on the $E_2$-page
\cite{Bruner97}.
Since $B_1$ detects $\eta \theta_{4.5}$ by definition of $\theta_{4.5}$
(see Section \ref{sctn:notation}),
$\tau g B_1$ detects $\eta \kappabar \theta_{4.5}$.
Therefore, $\tau B_{23}$ detects $\kappabar \theta_{4.5}$.

It follows that $\tau B_{23}$ is the target of a hidden $\kappabar$
extension.  The possible sources for this hidden extension are
$h_3^2 h_5$ or $h_5 d_0$.
\end{proof}

\begin{remark}
Lemma \ref{lem:kappabar-h3^2h5} is tentative because there are unknown
Adams differentials in the relevant range.  
\end{remark}


%% file: stable-stems-cofiber-tau.tex
\chapter{The cofiber of $\tau$}
\label{ch:Ctau}

\setcounter{thm}{0}

\index{cofiber of tau@cofiber of $\tau$}
The purpose of this chapter is to compute the motivic stable homotopy 
groups of the
cofiber $C \tau$ of $\tau$.
We obtain nearly complete results up to the 63-stem, and we have partial
results up to the 70-stem.
The Adams charts for $C\tau$ 
in \cite{Isaksen14a} are essential companions to this chapter.
\index{Adams chart!cofiber of tau@cofiber of $\tau$}

The element $\tau$ realizes to $1$ in the classical stable homotopy groups.
Therefore, $C \tau$ is an ``entirely exotic" object in motivic stable 
homotopy, since it realizes classically to the trivial spectrum.
\index{tau@$\tau$}

There are two main motivations for this calculation.
First, 
it is the key to resolving hidden $\tau$ extensions 
that were discussed in Section \ref{subsctn:hidden-tau}.
\index{Adams spectral sequence!hidden extension!tau@$\tau$}
Second, we will show in Proposition \ref{prop:ANSS-Ctau} that
the motivic homotopy groups of $C\tau$ are isomorphic to the classical
Adams-Novikov $E_2$-page. 
\index{cofiber of tau@cofiber of $\tau$!homotopy group}
 Thus the calculations in this chapter will 
allow us to reverse-engineer the classical Adams-Novikov spectral sequence.
\index{Adams-Novikov spectral sequence!classical}

\index{cofiber of tau@cofiber of $\tau$!Adams spectral sequence}
The computational method will be the motivic Adams spectral
sequence
\cite{DI10} \cite{HKO11} \cite{Morel99}
for $C \tau$, which takes the form
\[
E_2 = \Ext_A ( H^{*,*} ( C\tau ) ; \M_2)
\Rightarrow
\pi_{*,*} (C \tau).
\]
We write $E_2(C\tau)$ for this $E_2$-page
$\Ext_A (H^{*,*} ( C \tau ) ; \M_2)$.
See \cite{HKO11} for convergence properties of this spectral sequence.
\index{Adams spectral sequence!convergence}
\index{cofiber of tau@cofiber of $\tau$!Adams spectral sequence!convergence}

\subsection*{Outline}

The first step in executing the motivic Adams spectral sequence
for $C\tau$ is to algebraically compute
the $E_2$-page, i.e.,
$\Ext_A ( H^{*,*} ( C\tau ) , \M_2)$.
We carry this out in Section \ref{sctn:t-Bockstein} using the
long exact sequence
\[
\xymatrix{
\ar[r] &
\Ext_A ( \M_2, \M_2 ) \ar[r]^\tau &
\Ext_A ( \M_2, \M_2 ) \ar[r] &
\Ext_A ( H^{*,*}(C\tau), \M_2 ) \ar[r] & .
}
\]
Some additional work is required in
resolving hidden extensions for the action of
$\Ext_A ( \M_2, \M_2 )$ on
$E_2(C\tau)$.

\index{cofiber of tau@cofiber of $\tau$!Adams spectral sequence!differential}
The next step is to compute the Adams differentials.
In Section \ref{sctn:Adams},
we use a variety of methods to obtain these computations.
The most important is to borrow results about 
differentials in the motivic Adams spectral sequence
for $S^{0,0}$ from Tables \ref{tab:Ext-gen},
\ref{tab:Adams-d3}, \ref{tab:Adams-d4}, and \ref{tab:Adams-d5}.
In addition,
there are several
computations that require analyses of brackets and 
hidden extensions.

The complete understanding of the Adams differentials allows
for the computation of the $E_\infty$-page of the motivic Adams spectral
sequence for $C\tau$.  The final step, carried out in 
Section \ref{sctn:Ctau-hidden}, is to resolve hidden extensions
by $2$, $\eta$, and $\nu$ 
in $\pi_{*,*}(C\tau)$.
\index{cofiber of tau@cofiber of $\tau$!hidden extension}

Chapter \ref{ch:table} contains a series of tables that summarize the essential
computational facts in a concise form.
Tables \ref{tab:Ctau-h0}, \ref{tab:Ctau-h1}, and \ref{tab:Ctau-h2}
give extensions by $h_0$, $h_1$, and $h_2$ in $E_2(C\tau)$
that are hidden in the long exact sequence that computes
$E_2(C\tau)$.
\index{cofiber of tau@cofiber of $\tau$!hidden extension!h0@$h_0$}
\index{cofiber of tau@cofiber of $\tau$!hidden extension!h1@$h_1$}
\index{cofiber of tau@cofiber of $\tau$!hidden extension!h2@$h_2$}
The fourth columns of these tables refer to one argument that
establishes each hidden extension.
This takes one of the following forms:
\begin{enumerate}
\item
An explicit proof given elsewhere in this manuscript.
\item
A May differential that computes a Massey product of the form
$\langle h_i, x, \tau \rangle$ via May's Convergence Theorem \ref{thm:3-converge}.
\index{Convergence Theorem!May}
\index{Massey product}
This Massey product implies the hidden extension in $E_2(C\tau)$ by
Proposition \ref{prop:hidden-Massey}.
\end{enumerate}
Table \ref{tab:Ctau-misc-extn} gives some additional miscellaneous hidden extensions
in $E_2(C\tau)$, again with references to a proof.

Table \ref{tab:Ctau-E2} lists the generators of 
$E_2(C\tau)$ as a module over 
$\Ext_A ( \M_2, \M_2 )$.  
Table \ref{tab:Ctau-ambiguous} 
lists all examples of generators of $E_2(C\tau)$
for which there is some ambiguity.
See Section \ref{subsctn:E2} for more explanation.
\index{cofiber of tau@cofiber of $\tau$!generator}

Tables \ref{tab:Ctau-E2} and \ref{tab:Ctau-d3} provide 
the values of $d_2$ and $d_3$ differentials 
in the Adams spectral sequence for $C\tau$.
\index{cofiber of tau@cofiber of $\tau$!Adams spectral sequence!differential}
The fourth columns of these tables refer to one argument that
establishes each differential.
This takes one of the following forms:
\begin{enumerate}
\item
An explicit proof given elsewhere in this manuscript.
\item
``top cell'' means that the differential is detected by projection
$E_r(C\tau) \map E_r(S^{0,0})$ to the top cell.
\index{cofiber of tau@cofiber of $\tau$!top cell}
\item
Some differentials can be established with an algebraic relation to another
differential that is detected by the inclusion
$E_r(S^{0,0}) \map E_r(C\tau)$ of the bottom cell.
\index{cofiber of tau@cofiber of $\tau$!bottom cell}
\end{enumerate}

Table \ref{tab:Ctau-top} describes the 
part of the projection $\pi_{*,*}(C\tau) \map \pi_{*,*}$ to the top cell
that are hidden by the map $E_\infty(C\tau) \map E_\infty(S^{0,0})$
of Adams $E_\infty$-pages.
\index{cofiber of tau@cofiber of $\tau$!top cell}
See Proposition \ref{prop:Ctau-top} for more explanation.

Table \ref{tab:Ctau-hidden} gives the extensions by $2$, $\eta$, and
$\nu$ in $\pi_{*,*}(C\tau)$ that are hidden in $E_\infty(C\tau)$.
\index{cofiber of tau@cofiber of $\tau$!Adams spectral sequence!hidden extension}
The fourth column refers to one argument that establishes each
hidden extension.  
This takes one of the following forms:
\begin{enumerate}
\item
An explicit proof given elsewhere in this manuscript.
\item
``top cell'' means that the hidden extension is detected by the projection
$\pi_{*,*} (C\tau) \map \pi_{*,*} (S^{0,0})$ to the top cell.
\index{cofiber of tau@cofiber of $\tau$!top cell}
\item
``bottom cell" means that the hidden extension is detected by the inclusion
$\pi_{*,*} \map \pi_{*,*}(C\tau)$ of the bottom cell.
\index{cofiber of tau@cofiber of $\tau$!bottom cell}
\end{enumerate}

\subsection*{Massey products and cofibers}

We will rely heavily on Massey products and Toda brackets,
using the well-known relationship between
Toda brackets and hidden extensions in the homotopy groups of a cofiber.
See Proposition \ref{prop:3bracket-cofiber} for an explicit statement.
We will also need a similar result for Massey products.
\index{cofiber!hidden extension}
\index{Massey product}

\begin{prop}
\label{prop:hidden-Massey}
Let $y$ and $z$ belong to $E_2(S^{0,0})$
such that 
$\tau y$ and $z y$ are both zero.
In $E_2(C\tau)$, 
there is a hidden extension
\[
z \cdot \ol{y} \in \langle z, y, \tau \rangle,
\]
where the Massey product is computed in $E_2(S^{0,0})$
and then pushed forward along the map
$E_2(S^{0,0}) \map E_2(C\tau)$. 
\end{prop}

\begin{proof}
The proof is identical to the proof of 
Proposition \ref{prop:3bracket-cofiber},
except that we work
in the derived category of chain complexes of $A$-modules 
instead of the motivic stable homotopy category.
In this derived category, the cofiber of
$\tau: \M_2 \map \M_2$ is $H^{*,*}(C\tau)$.
\end{proof}

\section{The Adams $E_2$-page for the cofiber of $\tau$}
\label{sctn:t-Bockstein}

\index{cofiber of tau@cofiber of $\tau$!Adams spectral sequence!E2-page@$E_2$-page}

The main tool for computing 
$E_2(C\tau) = \Ext_A ( H^{*,*} ( C\tau ) , \M_2)$
is the long exact sequence
\[
\xymatrix{
\cdots \ar[r] & E_2(S^{0,0}) \ar[r]^\tau & 
E_2(S^{0,0}) \ar[r] & E_2(C\tau) \ar[r] & \cdots
}
\]
associated to the cofiber sequence
\[
\xymatrix{
S^{0,-1} \ar[r]^\tau & S^{0,0} \ar[r] & C \tau \ar[r] & 
S^{1,-1}.
}
\]
This yields a short exact sequence
\[
\xymatrix{
0 \ar[r] & \coker(\tau) \ar[r] & E_2(C\tau) \ar[r] & \ker(\tau) \ar[r] & 0.
}
\]
The desired $E_2(C\tau)$ is almost completely described by the
previous short exact sequence.
It only remains to compute some hidden extensions.

\subsection{Hidden extensions in the Adams $E_2$-page 
for the cofiber of $\tau$}
\label{subsctn:t-Bockstein-hidden}

We will resolve all possible hidden extensions by
$h_0$, $h_1$, and $h_2$ through the 70-stem.
The reader should refer to the charts in \cite{Isaksen14a}
in order to make
sense of the following results.
\index{cofiber of tau@cofiber of $\tau$!hidden extension!h0@$h_0$}
\index{cofiber of tau@cofiber of $\tau$!hidden extension!h1@$h_1$}
\index{cofiber of tau@cofiber of $\tau$!hidden extension!h2@$h_2$}

\begin{thm}
\label{thm:hidden}
Tables \ref{tab:Ctau-h0}, \ref{tab:Ctau-h1}, and \ref{tab:Ctau-h2}
give some hidden extensions by 
$h_0$, $h_1$, and $h_2$ in $E_2 (C \tau)$.
Through the 70-stem,
all other possible hidden extensions by $h_0$, $h_1$, and $h_2$
are either zero or are easily implied by extensions in the tables,
with the possible exceptions that:
\begin{enumerate}
\item
$h_2 \cdot \ol{\tau^2 h_1 g^2}$ might equal $\tau w$.
\index{g2@$g^2$}
\index{w@$w$}
\item
$h_0 \cdot \ol{c_0 Q_2}$ and
$h_2 \cdot \ol{c_0 Q_2}$ are either both zero, or equal
$D_2'$ and $P(A+A')$ respectively.
\index{c0Q2@$c_0 Q_2$}
\index{D2'@$D_2'$}
\index{P(A+A')@$P(A+A')$}
\item
$h_1^3 c_0 \cdot \ol{D_4}$ equals either $h_2 B_5$ or $h_2 B_5 + h_1^2 X_3$.
\index{D4@$D_4$}
\index{B5@$B_5$}
\index{X3@$X_3$}
\end{enumerate}
\end{thm}

\begin{ex}
In the 14-stem, 
there is a hidden extension
$h_2 \cdot \ol{h_1^2 c_0} = h_0 d_0$, which does not appear in 
Table \ref{tab:Ctau-h2}.
This is easily implied by the hidden extension
$h_0 \cdot \ol{h_1^2 c_0} = P h_2$, which does appear in Table \ref{tab:Ctau-h0}.
\index{d0@$d_0$}
\index{Ph2@$P h_2$}
\index{c0@$c_0$}
\end{ex}

\begin{proof}
Most of these hidden extensions are established with
Proposition \ref{prop:hidden-Massey}, so we just need
to compute Massey products of the form
$\langle h_i, x, \tau \rangle$
in $E_2(S^{0,0})$.
\index{Massey product}
Most of these Massey products are
computed using May's Convergence Theorem \ref{thm:3-converge}.
\index{Convergence Theorem!May}
The fourth columns of 
Tables \ref{tab:Ctau-h0}, \ref{tab:Ctau-h1}, and \ref{tab:Ctau-h2}
indicate which May differentials are relevant for computing each bracket.

A few hidden extensions require more complicated proofs.
These proofs are given in the following lemmas.
\end{proof}

\subsection{Hidden $h_0$ extensions in the Adams $E_2$-page
for the cofiber of $\tau$}
\label{subsctn:h0-extns}

\index{cofiber of tau@cofiber of $\tau$!hidden extension!h0@$h_0$}

\begin{lemma}
\label{lem:_c0e0.h0}
\mbox{}
\begin{enumerate}
\item
$h_0 \cdot \ol{c_0 e_0} = j$.
\item
$h_0 \cdot \ol{P^k c_0 e_0} = P^k j$.
\item
$h_0 \cdot \ol{c_0 e_0 g} = d_0 l$.
\end{enumerate}
\end{lemma}

\begin{proof}
We prove the first formula.  The proofs for the other formulas are essentially
the same.

By Proposition \ref{prop:hidden-Massey}, we must compute
$\langle h_0, c_0 e_0, \tau \rangle$ in $E_2(S^{0,0})$.
We may attempt to compute this bracket using 
May's Convergence Theorem \ref{thm:3-converge}
with the May differential $d_2(b_{30} h_0(1)^2 ) = \tau c_0 e_0$.
However, the hypothesis of
May's Convergence Theorem \ref{thm:3-converge} is not satisfied
because of the later May differential $d_4(\Delta h_1^2 ) = P h_1^2 h_4$.
\index{Convergence Theorem!May}
\index{May spectral sequence!differential!crossing}

Instead, note that 
$h_2^2 \langle h_0 , c_0 e_0, \tau \rangle$ equals
$\langle h_2^2, h_0, c_0 e_0 \rangle \tau$.
Table \ref{tab:Massey} shows that 
the last bracket equals $h_1 d_0 e_0$.

Therefore,
$h_2^2 \langle h_0, c_0 e_0, \tau \rangle$ equals $\tau h_1 d_0 e_0$.
It follows that
$\langle h_0, c_0 e_0, \tau \rangle$ equals $j$.
\end{proof}

\begin{lemma}
\label{lem:h0-h1d1g}
$h_0 \cdot \ol{h_1 d_1 g} = h_1 h_5 c_0 d_0$.
\end{lemma}

\begin{proof}
By Proposition \ref{prop:hidden-Massey}, we must compute
$\langle h_0, h_1 d_1 g, \tau \rangle$ in $E_2(S^{0,0})$.
Because there is no indeterminacy, we have
\[
\langle h_0, h_1 d_1 g, \tau \rangle = 
\langle h_0, d_1, \tau h_1 g \rangle =
\langle h_0, d_1, h_2 f_0 \rangle =
\langle h_0, d_1, f_0 \rangle h_2.
\]
Table \ref{tab:Massey} shows that 
$h_2 B_2 = \langle h_0, d_1, f_0 \rangle$.
Finally, use that 
$h_2 \cdot h_2 B_2 = h_1 h_5 c_0 d_0$ 
from Table \ref{tab:May-h2}.
\end{proof}

\begin{lemma}
\label{lem:h0-h1^2B8}
$h_0 \cdot \ol{h_1^2 B_8} = h_2 x'$.
\end{lemma}

\begin{proof}
By Proposition \ref{prop:hidden-Massey}, we must compute the bracket
$\langle h_0, h_1^2 B_8, \tau \rangle$, which equals
$\langle h_0, h_1, \tau h_1 B_8 \rangle$ because
there is no indeterminacy.
Table \ref{tab:Massey} shows that
$\langle h_0, h_1, \tau h_1 B_8 \rangle$
equals $h_2 x'$.
Note that $\tau h_1 B_8 = P h_1 h_5 d_0$ 
from Table \ref{tab:May-tau}.
\end{proof}

\subsection{Hidden $h_1$ extensions in the Adams $E_2$-page
for the cofiber of $\tau$}
\label{subsctn:h1-extns}

\index{cofiber of tau@cofiber of $\tau$!hidden extension!h1@$h_1$}

\begin{lemma}
\label{lem:h1-th0e0^3}
$h_1 \cdot \ol{\tau h_0 e_0^3} = d_0 u$.
\end{lemma}

\begin{proof}
Using Proposition \ref{prop:hidden-Massey}, we wish to compute the bracket
$\langle h_1, \tau h_0 e_0^3, \tau \rangle$ in $E_2(S^{0,0})$.
We may attempt to use 
May's Convergence Theorem \ref{thm:3-converge}
with the May differential
$d_4( \Delta d_0^2 ) = \tau^2 h_0 e_0^3$.
However, the conditions of
May's Convergence Theorem \ref{thm:3-converge}
are not satisfied
because of the later May differential $d_8( \Delta^2 h_1^4 ) = P^2 h_1^4 h_5$.
\index{Convergence Theorem!May}
\index{May spectral sequence!differential!crossing}

Instead, 
Table \ref{tab:Ctau-h1} shows that
$h_1 \cdot \ol{\tau h_0 d_0 e_0^2}$
equals $P v$.
Next, observe that 
$d_0 \cdot \ol{\tau h_0 e_0^3} + e_0 \cdot \ol{\tau h_0 d_0 e_0^2}$ is either
zero or $h_1^2 U$.  
In either case,
$h_1 d_0 \cdot \ol{\tau h_0 e_0^3}$ must be non-zero.
It follows that 
$h_1 \cdot \ol{\tau h_0 e_0^3}$ is also non-zero,
and there is just one possible non-zero value.
\end{proof}

\begin{lemma}
\label{lem:h1-h1h5c0d0}
$h_1^2 h_5 \cdot \ol{c_0 d_0} = P h_5 e_0$.
\end{lemma}

\begin{proof}
Table \ref{tab:Ctau-misc-extn}
shows that
\[
h_1^2 \cdot \ol{c_0 d_0} + d_0 \cdot \ol{h_1^2 c_0} = P e_0,
\]
which means that
\[
h_1^5 h_5 \cdot \ol{c_0 d_0} + h_1^3 h_5 d_0 \cdot \ol{h_1^2 c_0} = 
P h_1^3 h_5 e_0.
\]
But $h_1^3 h_5 d_0 = 0$, so 
$h_1^5 h_5 \cdot \ol{c_0 d_0} = P h_1^3 h_5 e_0$, 
from which the desired formula follows.
\end{proof}

\begin{lemma}
\label{lem:h5d0e0.h1^2}
$h_1^2 \cdot \ol{h_5 d_0 e_0} = \tau B_{23} + c_0 Q_2$.
\end{lemma}

\begin{proof}
Because of Proposition \ref{prop:hidden-Massey}, we wish
to compute the Massey product $\langle h_1, h_1 h_5 d_0 e_0, \tau \rangle$
in $E_2(S^{0,0})$.
We may attempt to use
May's Convergence Theorem \ref{thm:3-converge}
with the May differential $d_6(B_{23}) = h_1^2 h_5 d_0 e_0$.
\index{Convergence Theorem!May}

However, there is a subtlety here.  The element $\tau B_{23}$ belongs to
the May $E_\infty$-page for $E_2(S^{0,0})$.  It represents
two elements in $E_2(S^{0,0})$ because
of the presence of $P D_4$ with lower May filtration.
Thus, we have only determined so far that
$h_1^2 \cdot \ol{h_5 d_0 e_0}$ equals either
$\tau B_{23}$ or $\tau B_{23} + c_0 Q_2$.
\index{Ext@$\Ext$!ambiguous generator}

This ambiguity is resolved essentially by definition.
In Table \ref{tab:Ext-ambiguous},
the element $\tau B_{23}$ in $E_2(S^{0,0})$ is defined such that
$\langle h_1, h_1 h_5 d_0 e_0, \tau \rangle$ equals
$\tau B_{23} + c_0 Q_2$.
\end{proof}

\begin{lemma}
\label{lem:_h1^2Q2.h1^5}
$h_1^5 \cdot \ol{h_1^2 Q_2} = \tau g w + h_1^4 X_1$.
\end{lemma}

\begin{proof}
Because of Proposition \ref{prop:hidden-Massey}, we wish to compute
the Massey product $\langle h_1^5, h_1^2 Q_2, \tau \rangle$.
We may attempt to use 
May's Convergence Theorem \ref{thm:3-converge}
with the May differential
$d_4(\Delta h_1 g^2) = h_1^7 Q_2$.
\index{Convergence Theorem!May}

As in the proof of Lemma \ref{lem:h5d0e0.h1^2},
there is a subtlety here.
The element $\tau g w$ belongs to
the May $E_\infty$-page for $E_2(S^{0,0})$.  It represents
two elements in $E_2(S^{0,0})$ because
of the presence of $P h_1 h_5 c_0 e_0$ with lower May filtration.
Recall that 
Table \ref{tab:Ext-ambiguous} defines
$\tau g w$ to be the element of
$E_2(S^{0,0})$ such that $h_1 \cdot \tau g w = 0$.
\index{Ext@$\Ext$!ambiguous generator}

We have determined so far that
$h_1^5 \cdot \ol{h_1^2 Q_2}$ equals either
$\tau g w$ or $\tau g w + h_1^4 X_1$.

Table \ref{tab:Adams-d3} gives
a non-zero value for the Adams differential
$d_3(\tau g w)$.
On the other hand,
$d_3 (\ol{h_1^2 Q_2})$ is zero.
Therefore, 
$h_1^5 \cdot \ol{h_1^2 Q_2}$ cannot equal
$\tau g w$.
\index{Adams spectral sequence!differential!d3@$d_3$}
\end{proof}

\begin{remark}
The proof of Lemma \ref{lem:_h1^2Q2.h1^5} is not entirely algebraic in the sense
that it relies on Adams differentials.  We would prefer a
purely algebraic proof, but it has so far eluded us.
\end{remark}

\begin{lemma}
\label{lem:_D4.h1^3c0}
$h_1^3 c_0 \cdot \ol{D_4}$ equals either $h_2 B_5$ or $h_2 B_5 + h_1^2 X_3$.
\end{lemma}

\begin{proof}
Because of Proposition \ref{prop:hidden-Massey},
we wish to compute
$\langle h_1^3 c_0, D_4, \tau \rangle$.
We may attempt to use
May's Convergence Theorem \ref{thm:3-converge} with the May differential
$d_4(\phi g) = h_1^5 X_2$.
\index{Convergence Theorem!May}

As in the proof of Lemma \ref{lem:h5d0e0.h1^2},
there is a subtlety here.
The element $h_2 B_5$ belongs to
the May $E_\infty$-page for $E_2(S^{0,0})$.  It represents
two elements in $E_2(S^{0,0})$ because
of the presence of $h_1^2 X_3$ with lower May filtration
(see Table \ref{tab:Ext-ambiguous}).
\index{Ext@$\Ext$!ambiguous generator}
\end{proof}

\subsection{Other extensions in the Adams $E_2$-page
for the cofiber of $\tau$}
\label{subsctn:other-extns}

We finish this section with some additional miscellaneous hidden extensions.

\index{cofiber of tau@cofiber of $\tau$!hidden extension!compound}

\begin{lemma}
\label{lem:B6.h1^3}
$h_1^3 \cdot \ol{B_6} + h_2 \cdot \ol{\tau h_2 d_1 g} = h_1^2 Q_2$.
\end{lemma}

\begin{proof}
Table \ref{tab:May-h1} gives
the hidden extension $h_1 \cdot h_1^2 B_6 = \tau h_2^2 d_1 g$
in $E_2(S^{0,0})$.
This means that $h_1^3 \cdot \ol{B_6} + h_2 \cdot \ol{\tau h_2 d_1 g}$
belongs to the image of
$E_2(S^{0,0}) \map E_2(C\tau)$.

Next, compute that $h_1^3 Q_2 = \langle h_1^4, B_6, \tau \rangle$
using May's Convergence Theorem \ref{thm:3-converge} with the May differentials
$d_2 (b_{30} b_{40} h_1(1) ) = \tau B_6$ and 
$d_2 ( h_1^2 b_{21}^2 b_{30} b_{31} + h_1^2 b_{21}^3 b_{40}) = h_1^4 B_6$.
\index{Convergence Theorem!May}
Therefore, $h_1^4 \cdot \ol{B_6} = h_1^3 Q_2$ by 
Proposition \ref{prop:hidden-Massey}.
The desired formula now follows.
\end{proof}

\begin{remark}
Through the 70-stem,
Lemma \ref{lem:B6.h1^3} is the only example of 
a hidden relation of the form
$h_0 \cdot \ol{x} + h_1 \cdot \ol{y}$,
$h_0 \cdot \ol{x} + h_2 \cdot \ol{y}$, or
$h_1 \cdot \ol{x} + h_2 \cdot \ol{y}$ in $E_2(C\tau)$.
\end{remark}

\begin{lemma}
\label{lem:hidden-B6-Ph1}
$P h_1 \cdot \ol{B_6} = h_1 q_1$.
\end{lemma}

\index{cofiber of tau@cofiber of $\tau$!hidden extension!Ph1@$Ph_1$}

\begin{proof}
Compute that $h_1 q_1$ is contained in the Massey product
$\langle P h_1, B_6, \tau \rangle$ in $E_2(S^{0,0})$,
using May's Convergence Theorem \ref{thm:3-converge} with the May differentials
$d_2 ( b_{30} b_{40} h_1(1) ) = \tau B_6$ and
$d_2 ( \Delta B h_1^3 ) = P h_1 \cdot B_6$.
\index{Convergence Theorem!May}
The bracket
has indeterminacy generated by $\tau^2 h_0 B_{23}$, so it equals
$\{ h_1 q_1, h_1 q_1 + \tau^2 h_0 B_{23} \}$.

Push forward this bracket into $E_2(C\tau)$, where it collapses
to the single element $h_1 q_1$ since
$\tau^2 h_0 B_{23}$ maps to zero in $E_2(C\tau)$.
Proposition \ref{prop:hidden-Massey} now gives the desired result.
\end{proof}

\index{cofiber of tau@cofiber of $\tau$!hidden extension!compound}

\begin{lemma}
\label{lem:_h1^2e0.c0}
\mbox{}
\begin{enumerate}
\item
$h_1^2 \cdot \ol{c_0 d_0} + d_0 \cdot \ol{h_1^2 c_0} = P e_0$.
\item
$c_0 \cdot \ol{h_1^2 e_0} + e_0 \cdot \ol{h_1^2 c_0} = d_0^2$.
\item
$h_1^2 \cdot \ol{h_1 d_0 u} + d_0 \cdot \ol{h_1^3 u} = P v'$.
\end{enumerate}
\end{lemma}

\begin{proof}
These formulas have essentially the same proof.  We prove only the first
formula.

\index{Massey product!matric}
Table \ref{tab:Ctau-matric} shows that there is a matric bracket
\[
P e_0 = 
\left\langle
\left[
\begin{array}{cc}
h_1^2 & d_0  \\
\end{array}
\right], 
\left[
\begin{array}{c}
c_0 d_0 \\
h_1^2 c_0 
\end{array}
\right],
\tau
\right\rangle.
\]
A matric version of Proposition \ref{prop:hidden-Massey}
gives the desired hidden extension.
\end{proof}

Before considering the next hidden extension, we need a bracket
computation.

\begin{lemma}
\label{lem:h1^4-t-brackets}
$h_1^2 d_0^2 = \langle c_0 e_0, \tau, h_1^4 \rangle$.
\end{lemma}

\begin{proof}
The bracket cannot be computed directly with 
May's Convergence Theorem \ref{thm:3-converge}
because of the the later May differential $d_4(\Delta h_1^2) = P h_1^2 h_4$. 
Therefore, we must follow a more complicated route.
\index{Convergence Theorem!May}
\index{May spectral sequence!differential!crossing}

Begin with the computation $h_1 c_0 e_0 = \langle d_0, h_3, h_1^4 \rangle$
from Table \ref{tab:Massey}.
Therefore,
\[
\langle h_1 c_0 e_0, \tau, h_1^4 \rangle =
\langle \langle d_0, h_3, h_1^4 \rangle, \tau, h_1^4 \rangle,
\]
which equals 
$d_0 \langle h_3, h_1^4, \tau, h_1^4 \rangle$ by a standard formal
property of Massey products since there are no indeterminacies.

Next, compute that
$h_1^3 d_0 = \langle h_3, h_1^4, \tau, h_1^4 \rangle$ 
using May's Convergence Theorem \ref{thm:4-converge}
with the May differentials
$d_2( h_1 b_{20}) = \tau h_1^4$, 
$d_2 (h_1^2 b_{21} ) = h_1^4 h_3$, and
$d_2 ( h_1 b_{30} ) = \tau h_1^2 b_{21} + h_1 h_3 b_{20}$.
\index{Convergence Theorem!May}
Note that both subbrackets
$\langle h_1^4, \tau, h_1^4 \rangle$ and
$\langle h_3, h_1^4, \tau \rangle$ are strictly zero.

We have now shown that
$\langle h_1 c_0 e_0, \tau, h_1^4 \rangle$ equals $h_1^3 d_0^2$.
The desired formula now follows immediately.
\end{proof}

\begin{lemma}
\label{lem:_c0e0.d0}
\mbox{}
\begin{enumerate}
\item
$h_1^2 \cdot \ol{c_0 e_0} + e_0 \cdot \ol{h_1^2 c_0} = d_0^2$.
\item
$d_0 \cdot \ol{c_0 e_0} + e_0 \cdot \ol{c_0 d_0} = h_1 u$.
\end{enumerate}
\end{lemma}

\index{cofiber of tau@cofiber of $\tau$!hidden extension!compound}

\begin{proof}
For the first formula, by a matric version of Proposition \ref{prop:hidden-Massey},
we wish to compute that
\[
d_0^2 = 
\left\langle
\left[
\begin{array}{cc}
h_1^2 & e_0 
\end{array}
\right],
\left[
\begin{array}{c}
c_0 e_0 \\
h_1^2 c_0
\end{array}
\right],
\tau
\right\rangle.
\]
\index{Massey product!matric}
One might attempt to compute this with a matric version of 
May's Convergence Theorem \ref{thm:3-converge}.
However, the hypotheses of May's Convergence Theorem \ref{thm:3-converge} do not apply 
because of the presence of the later May differential 
$d_4(\Delta h_1^2) = P h_1^2 h_4$.  
\index{Convergence Theorem!May}
\index{May spectral sequence!differential!crossing}

Instead, we will show that
\[
\left\langle
\left[
\begin{array}{cc}
h_1^2 & e_0 
\end{array}
\right],
\left[
\begin{array}{c}
c_0 e_0 \\
h_1^2 c_0
\end{array}
\right],
\tau
\right\rangle
h_1^4
\]
equals $h_1^4 d_0^2$,
from which the desired bracket follows immediately.
Shuffle to obtain
\[
h_1^2 \langle c_0 e_0, \tau, h_1^4 \rangle +
e_0 \langle h_1^2 c_0, \tau, h_1^4  \rangle.
\]
By Table \ref{tab:Massey}, 
the expression equals $h_1^4 d_0^2$ as desired.
This completes the proof of the first formula.

The proof of the second formula is similar.
We wish to 
compute that
\[
h_1 u = 
\left\langle
\left[
\begin{array}{cc}
d_0 & e_0 
\end{array}
\right],
\left[
\begin{array}{c}
c_0 e_0 \\
c_0 d_0 
\end{array}
\right],
\tau
\right\rangle.
\]
Again, the hypotheses of May's Convergence Theorem \ref{thm:3-converge} do not apply.
\index{Convergence Theorem!May}
\index{May spectral sequence!differential!crossing}

Instead, we will show that
\[
\left\langle
\left[
\begin{array}{cc}
d_0 & e_0 
\end{array}
\right],
\left[
\begin{array}{c}
c_0 e_0 \\
c_0 d_0 
\end{array}
\right],
\tau
\right\rangle h_1^4
\]
equals $h_1^5 u$,
from which the desired bracket follows immediately.
Shuffle to obtain
\[
d_0 \langle c_0 e_0, \tau, h_1^4 \rangle + 
e_0 \langle c_0 d_0, \tau, h_1^4 \rangle.
\]
By Table \ref{tab:Massey},
this expression equals $h_1^2 d_0^3 + h_1^2 e_0 \cdot P e_0$.
Note that $e_0 \cdot P e_0$ equals $d_0^3 + h_1^3 u$; this is already
true in the May $E_\infty$-page.  
Therefore,
$h_1^2 d_0^3 + h_1^2 e_0 \cdot P e_0$ equals $h_1^5 u$, as desired.
\end{proof}

\begin{lemma}
\label{lem:_h1^2e0.h1^2e0^2}
$h_1^2 e_0^2 \cdot \ol{h_1^2 e_0} + d_0 e_0 g \cdot \ol{h_1^4} +
h_1^6 \cdot \ol{h_1^3 B_1} = c_0 d_0 e_0^2$.
\end{lemma}

\index{cofiber of tau@cofiber of $\tau$!hidden extension!compound}

\begin{proof}
The relation
$e_0^3 + d_0 \cdot e_0 g = h_1^5 B_1$ 
is hidden in the May spectral sequence \cite{GI14}.

By Proposition \ref{prop:hidden-Massey}, we wish to compute that
\[
c_0 d_0 e_0^2 = 
\left\langle
\left[
\begin{array}{ccc}
h_1^2 e_0^2 & 
d_0 e_0 g & 
h_1^4
\end{array}
\right],
\left[
\begin{array}{c}
h_1^2 e_0 \\
h_1^4 \\
h_1^3 B_1 
\end{array}
\right],
\tau
\right\rangle.
\]
\index{Massey product!matric}
This will follow if we can show that $h_1^4 c_0 d_0 e_0^2$ equals
\[
\left\langle
\left[
\begin{array}{ccc}
h_1^2 e_0^2 & 
d_0 e_0 g & 
h_1^4
\end{array}
\right],
\left[
\begin{array}{c}
h_1^2 e_0 \\
h_1^4 \\
h_1^3 B_1 
\end{array}
\right],
\tau
\right\rangle
h_1^4.
\]
This expression equals
\[
h_1^2 e_0^2 \langle h_1^2 e_0, \tau, h_1^4 \rangle +
d_0 e_0 g \langle h_1^4, \tau, h_1^4 \rangle +
h_1^4 \langle h_1^3 B_1, \tau, h_1^4 \rangle.
\]
The first two terms can be computed with Table \ref{tab:Massey}.
The possible non-zero values for the third bracket are multiples of $h_0$,
which means that the third term is zero in any case.

The desired formula now follows.
\end{proof}

\begin{lemma}
\label{lem:_h1d0u.h1^2}
\end{lemma}

\subsection{The Adams $E_2$-page for the cofiber of $\tau$}
\label{subsctn:E2}

Having resolved hidden extensions,
we can now
state our main theorem about $E_2(C\tau)$.

\index{cofiber of tau@cofiber of $\tau$!Adams spectral sequence!E2-page@$E_2$-page}

\begin{thm}
\label{thm:E2}
The $E_2$-page of the Adams spectral sequence for $C\tau$
is depicted in \cite{Isaksen14a}
through the 70-stem.
Table \ref{tab:Ctau-E2} lists 
the $E_2(S^{0,0})$-module generators of
$E_2(C\tau)$ through the 70-stem.
\end{thm}

For most of the generators in Table \ref{tab:Ctau-E2},
the notation $\ol{x}$ is unambiguous.
In other words, in each relevant degree, there is just a single element
$\ol{x}$ of $E_2(C\tau)$ that projects to $x$ in $E_2(S^{0,0})$.

However, there are several cases in which there is a choice of 
representative for $\ol{x}$ because of the presence of an element in the same
degree in the image of the map $E_2(S^{0,0}) \map E_2(C\tau)$.
One such example occurs in the 56-stem with $\ol{\tau h_0 g m}$.
The presence of $h_2 x'$ means that there are actually
two possible choices for $\ol{\tau h_0 g m}$.
\index{cofiber of tau@cofiber of $\tau$!ambiguous generator}

Table \ref{tab:Ctau-ambiguous}
lists all such examples of $E_2(S^{0,0})$-module generators of 
$E_2(C\tau)$ for which there is some ambiguity.
In some cases, we have given an algebraic specification of one element
of $E_2(C\tau)$ to serve as the generator.
These choices are essentially arbitrary, 
but it is important to be consistent with the notation
between different arguments.

In some cases, we have not given a definition because an algebraic description
is not readily available, and also because it does not seem to matter for
later analysis.
The reader is strongly warned to be cautious when working with these undefined
elements.  

The generator $\ol{h_1 i_1}$ deserves one additional comment.
In this case, the presence of $\tau h_1 G$ and $B_6$ means that there are
four possible choices for this generator.  
We have given two algebraic specifications for
$\ol{h_1 i_1}$, which determines a unique element from these four.
\index{i1@$i_1$}
\index{G@$G$}
\index{B6@$B_6$}


\section{Adams differentials for the cofiber of $\tau$}
\label{sctn:Adams}

We have now computed the $E_2$-page of the Adams spectral
sequence for $C\tau$. 
See \cite{Isaksen14a} for a chart of $E_2(C\tau)$ 
through the 70-stem.
\index{Adams chart!cofiber of tau@cofiber of $\tau$}

The next step is to compute the Adams differentials.
The main point is to compute the Adams $d_r$ differentials on 
the $E_r(S^{0,0})$-module generators of $E_r(C\tau)$.
Then one can compute the Adams $d_r$ differential on any element,
using the Adams $d_r$ differentials for $E_r(S^{0,0})$ given in 
Tables \ref{tab:Ext-gen}, \ref{tab:Adams-d3}, \ref{tab:Adams-d4},
and \ref{tab:Adams-d5}.
\index{cofiber of tau@cofiber of $\tau$!Adams spectral sequence!differential}

\subsection{Adams $d_2$ differentials for the cofiber of $\tau$}
\label{subsctn:Ctau-d2}

\index{cofiber of tau@cofiber of $\tau$!Adams spectral sequence!d2@$d_2$}

\begin{prop}
\label{prop:Ctau-Adams-d2}
Table \ref{tab:Ctau-E2} lists some values of the motivic Adams $d_2$ differential
for $C\tau$.  The motivic Adams $d_2$ differential is zero on all other
$E_2(S^{0,0})$-module generators of $E_2(C\tau)$, through the
70-stem, with the possible exceptions that:
\begin{enumerate}
\item
$d_2(\ol{h_1 i_1})$ might equal $h_1 h_5 c_0 d_0$.
\index{i1@$i_1$}
\index{h5c0d0@$h_5 c_0 d_0$}
\item
$d_2 (\ol{h_1 r_1})$ might equal $\tau h_1 G_0$.
\index{r1@$r_1$}
\index{G0@$G_0$}
\end{enumerate}
\end{prop}

\begin{proof}
We use 
several different approaches to establish the 
Adams $d_2$ differentials:
\begin{enumerate}
\item
From an Adams differential $d_2(x) = y$ in $E_2(S^{0,0})$,
push forward along the inclusion $S^{0,0} \map C\tau$ of the bottom cell
to obtain the same formula in $E_2(C\tau)$.
\index{cofiber of tau@cofiber of $\tau$!bottom cell}
\item
From an Adams differential $d_2(x) = y$ in $E_2(S^{0,0})$,
use the projection $C\tau \map S^{1,-1}$ 
and pull back to $d_2 (\ol{x} ) = \ol{y}$
in $E_2(C\tau)$, up to a possible error term that belongs to the image
of the inclusion $E_2(S^{0,0}) \map E_2(C\tau)$ of the bottom cell.
\item
Push forward a differential from $E_2(S^{0,0})$ as in (1),
and then use a hidden extension in $E_2(C\tau)$.
For example, $d_2 (\ol{c_0 d_0} ) = P d_0$
because
$h_0 \cdot \ol{c_0 d_0} = i$ in $E_2(C\tau)$ and
$d_2(i) = P h_0 d_0$ in $E_2(S^{0,0})$.
\item
Work $h_1$-locally.
\index{h1@$h_1$@localization}
For example, 
consider the hidden extensions
$h_1^2 \cdot \ol{c_0 e_0} + e_0 \cdot \ol{h_1^2 c_0} = d_0^2$
and 
$h_1^2 \cdot \ol{c_0 d_0} + d_0 \cdot \ol{h_1^2 c_0} = P e_0$
from Table \ref{tab:Ctau-misc-extn}.
It follows that
$d_2(\ol{c_0 e_0}) = h_1^2 \cdot \ol{c_0 d_0} + P e_0$.
\end{enumerate}

Most of the differentials are computed with straightforward applications of
these techniques.
The remaining cases are computed in the following lemmas.
\end{proof}

The chart of $E_2(C\tau)$ in \cite{Isaksen14a} indicates the 
Adams $d_2$ differentials, all of which are implied by
the calculations in Tables \ref{tab:Ext-gen} and \ref{tab:Ctau-E2}.

\begin{lemma}
\label{lem:Ctau-d2-h1^2e0g}
$d_2(\ol{h_1^2 e_0 g}) = h_1^2 e_0 \cdot \ol{h_1^2 e_0} + c_0 d_0 e_0$.
\end{lemma}

\begin{proof}
Table \ref{tab:Ext-gen} gives the 
differential $d_2(h_1^2 e_0 g) = h_1^4 e_0^2$ in $E_2(S^{0,0})$
Therefore, $d_2(\ol{h_1^2 e_0 g})$ is either
$h_1^2 e_0 \cdot \ol{h_1^2 e_0}$ or 
$h_1^2 e_0 \cdot \ol{h_1^2 e_0} + c_0 d_0 e_0$.
However,
$d_2(h_1^2 e_0 \cdot \ol{h_1^2 e_0}) = h_1^2 c_0 d_0^2$, so
$h_1^2 e_0 \cdot \ol{h_1^2 e_0}$ cannot be the target of a $d_2$
differential.
\end{proof}

\begin{lemma}
\label{lem:Ctau-d2-t^2h1g^2}
$d_2 ( \ol{\tau^2 h_1 g^2} ) = z$.
\end{lemma}

\begin{proof}
We will argue that $z$ must be zero in $E_\infty(C\tau)$.
There is only one possible differential that can kill it.

Table \ref{tab:extn-refs} gives
a classical extension
$\eta \cdot \{g^2 \} = \{ z \}$ 
in $\pi_{41}$.
This implies that there must be a hidden relation
$\tau \cdot \{ \tau^2 h_1 g^2 \} = \{ z \}$ in $\pi_{41,22}$.
In particular, $\{ z \}$ is divisible by $\tau$ in $\pi_{*,*}$.
This means that $\{ z \}$ maps to zero in 
$\pi_{41,22}(C\tau)$.
\end{proof}

\begin{lemma}
\label{lem:Ctau-d2-h1d0u}
\mbox{}
\begin{enumerate}
\item
$d_2( \ol{h_1 d_0 u} ) = P u'$.
\item
$d_2( \ol{P h_1 d_0 u} ) = P^2 u'$.
\end{enumerate}
\end{lemma}

\begin{proof}
Table \ref{tab:Ctau-E2} implies that
$d_2 (d_0 \cdot \ol{h_1 v}) = d_0 \cdot \ol{h_1^3 u}$.
By Lemma \ref{lem:_h1d0u.h1^2}, this equals
$h_1^2 \cdot \ol{h_1 d_0 u} + P v'$.

Therefore,
$h_1^2 \cdot d_2(\ol{h_1 d_0 u})$ equals $d_2(P v')$.
By Table \ref{tab:Ext-gen},
$d_2(P v')$ equals $P h_1^2 u'+ \tau h_0 d_0^4$ in $E_2(S^{0,0})$.
Therefore,
$d_2(P v')$ equals $P h_1^2 u'$ in $E_2(C\tau)$.
It follows that 
$d_2(\ol{h_1 d_0 u})$ must equal $P u'$.
This establishes the first formula.

The second formula follows by multiplying the
first formula by $P h_1$.
\end{proof}

\begin{lemma}
\label{lem:Ctau-d2-D4}
$d_2( \ol{D_4}) = h_1 \cdot \ol{B_6} + Q_2$.
\end{lemma}

\begin{proof}
Pull back the differential $d_2 (D_4) = h_1 B_6$ from
$E_2(S^{0,0})$ to conclude that  
$d_2( \ol{D_4}) = h_1 \cdot \ol{B_6}$ modulo a possible error term
that comes from pushing forward from $E_2(S^{0,0})$.
To establish the error term, use that $h_0 \cdot \ol{D_4} = D_2$
and that $d_2( D_2) = h_0 Q_2$.
\end{proof}

\begin{lemma}
\label{lem:Ctau-d2-h1c0x'}
$d_2(\ol{h_1 c_0 x'}) = P h_1 x'$.
\end{lemma}

\begin{proof}
Table \ref{tab:Ctau-E2} implies that 
$d_2(e_0 \cdot \ol{v'})$ equals
$h_1^2 e_0 \cdot \ol{u'} + h_1^2 d_0 \cdot \ol{v'} +
e_0 \cdot \ol{\tau h_0 d_0 e_0^2}$.
Recall from \cite{GI14} the relation
$e_0 u' + d_0 v' = h_1^2 c_0 x'$, which is hidden in the 
May spectral sequence.
This implies that
$d_2(e_0 \cdot \ol{v'})$ equals
$h_1^3 \cdot \ol{h_1 c_0 x'} + e_0 \cdot \ol{\tau h_0 d_0 e_0^2}$.

There is a hidden extension $h_1 e_0 \cdot \ol{\tau h_0 d_0 e_0^2} = P e_0 v$.
Therefore,
$d_2(h_1 e_0 \cdot \ol{v'})$ equals 
$h_1^4 \cdot \ol{h_1 c_0 x'} + P e_0 v$,
so 
$h_1^4 \cdot d_2(\ol{h_1 c_0 x'})$ must equal $d_2(P e_0 v)$.

By Table \ref{tab:Ext-gen},
$d_2(P e_0 v) = P h_1^2 d_0 v + P h_1^2 e_0 u$ in $E_2(S^{0,0})$.
This equals $P h_1^5 x'$ by \cite{GI14}.
It follows that $d_2(\ol{h_1 c_0 x'})$ equals
$P h_1 x'$.
\end{proof}

\begin{lemma}
\label{lem:Ctau-d2-_c0Q2}
$d_2 (\ol{c_0 Q_2} ) = 0$.
\end{lemma}

\begin{proof}
Start with the relation $h_1 \cdot \ol{c_0 Q_2} = P h_1 \cdot \ol{D_4}$,
which follows from Lemma \ref{lem:c0-i1}.
Using Lemma \ref{lem:Ctau-d2-D4},
it follows that $h_1 \cdot d_2 (\ol{c_0 Q_2} ) = P h_1^2 \cdot \ol{B_6} + P h_1 Q_2$.
We know from \cite{Bruner97} that $P h_1 Q_2 = h_1^2 q_1$,
and we know from Lemma \ref{lem:hidden-B6-Ph1} that
$P h_1^2 \cdot \ol{B_6} = h_1^2 q_1$.
\end{proof}

\begin{remark}
\label{rem:d2-h1^2e0.e0g}
We emphasize the calculation
$d_2 ( e_0 g \cdot \ol{h_1^2 e_0} ) = 
h_1^6 \cdot \ol{h_1^3 B_1} + c_0 d_0 e_0^2$,
which follows from the Leibniz rule and 
Lemma \ref{lem:_h1^2e0.h1^2e0^2}.
This implies that $h_1^6 \cdot \ol{h_1^3 B_1}$ equals $c_0 d_0 e_0^2$
in $E_3(C\tau)$.
This formula is critical for later
Adams differentials.
\end{remark}


\subsection{Adams $d_3$ differentials for the cofiber of $\tau$}
\label{subsctn:Ctau-d3}

See \cite{Isaksen14a} for a chart of $E_3(C\tau)$.
This chart is complete through
the 70-stem; however, the Adams $d_3$ differentials
are complete only through the 64-stem.

\index{cofiber of tau@cofiber of $\tau$!Adams spectral sequence!E3-page@$E_3$-page}

\begin{remark}
There are a number of classes in $E_2(S^{0,0})$ 
that do not survive to $E_3(S^{0,0})$,
but their images in $E_2(C\tau)$ do survive to
$E_3(C\tau)$.  The first few examples of this phenomenon are
$h_0 y$, $h_0 c_2$, and $h_0^4 Q'$.  These elements give rise to
$E_3(S^{0,0})$-module generators of $E_3(C\tau)$.
\index{y@$y$}
\index{c2@$c_2$}
\index{Q'@$Q'$}
\end{remark}

\begin{remark}
Note the class in the 55-stem labeled ``?".  This class is either
$\ol{h_1 i_1}$ or $\ol{h_1 i_1} + \tau h_1 G$, depending on whether
$d_2(\ol{h_1 i_1})$ is zero or non-zero.
\index{i1@$i_1$}
\index{G@$G$}
In the first case, we have that
$h_1^5 \cdot \ol{h_1 i_1} = \tau g^3$,
from which 
$d_3 (\ol{h_1 i_1})$ would equal $h_1 B_8$.
\index{g3@$g^3$}
\index{B8@$B_8$}
In the second case,
we have that 
$h_1^5 \cdot (\ol{h_1 i_1} + \tau h_1 G) = \tau g^3 + h_1^4 h_5 c_0 e_0$,
from which
$d_3 (\ol{h_1 i_1} + \tau h_1 G)$ would equal zero.
\index{h5c0e0@$h_5 c_0 e_0$}
\end{remark}

The next step is to compute Adams $d_3$ differentials
on the $E_3(S^{0,0})$-module generators of $E_3(C\tau)$.

\begin{prop}
\label{prop:Ctau-Adams-d3}
Table \ref{tab:Ctau-d3} lists some values of the motivic Adams $d_3$ differential
for $C\tau$.  The motivic Adams $d_3$ differential is zero on all other
$E_3(S^{0,0})$-module generators of $E_3(C\tau)$, through the
65-stem, with the possible exception that
$d_3( \ol{h_1 i_1} )$ equals $h_1 B_8$,
if $\ol{h_1 i_1}$ survives to $E_3(C\tau)$.
\end{prop}

\index{cofiber of tau@cofiber of $\tau$!Adams spectral sequence!d3@$d_3$}
\index{i1@$i_1$}
\index{B8@$B_8$}

\begin{proof}
The techniques for establishing these differentials 
are the same as in 
the proof of Proposition \ref{prop:Ctau-Adams-d2} for $d_2$ differentials,
except that the $h_1$-local calculations are no longer useful.
The few remaining cases are computed in the following lemmas.
\end{proof}

The chart of $E_3(C\tau)$ in \cite{Isaksen14a} indicates the 
Adams $d_3$ differentials, all of which are implied by
the calculations in Tables \ref{tab:Adams-d3} and \ref{tab:Ctau-d3}.
The differentials are complete only through the 64-stem.
Beyond the 64-stem, there are a number of unknown differentials.
\index{Adams chart!cofiber of tau@cofiber of $\tau$}

\begin{lemma}
\label{lem:Ctau-d3-_h1h3g}
\mbox{}
\begin{enumerate}
\item
$d_3 (\ol{h_1 h_3 g} ) = d_0^2$.
\item
$d_3 (\ol{h_1 h_3 g^2}) = d_0 e_0^2$.
\end{enumerate}
\end{lemma}

\begin{proof}
We showed in Lemma \ref{lem:t-h1h3g} that
both
$\{ d_0^2 \}$ and $\{d_0 e_0^2 \}$ are divisible by $\tau$ in $\pi_{*,*}$.
Therefore, the classes $d_0^2$ and $d_0 e_0^2$ of $E_\infty(S^{0,0})$
must map to zero in $E_\infty(C\tau)$.  For each element, there is
just one possible differential that can hit it.
\end{proof}

\begin{lemma}
$d_3(\ol{h_1^2 g_2}) = 0$.
\end{lemma}

\begin{proof}
The only other possibility is that $d_3(\ol{h_1^2 g_2})$ equals $N$.
We showed in Lemma \ref{lem:tau-h1^2g2}
that the elements of $\{ N \}$ are not divisible by $\tau$ in $\pi_{*,*}$.
Therefore, $\{ N \}$ maps to $\pi_{*,*}(C\tau)$ non-trivially.
The only possibility is that $N$ is non-zero in $E_\infty(C\tau)$.
\end{proof}

\begin{lemma}
\label{lem:Ctau-d3-_h1G3}
$d_3(\ol{h_1 G_3} ) = \ol{\tau h_0 e_0^3}$.
\end{lemma}

\begin{proof}
Table \ref{tab:Massey}
shows that $h_1 G_3 = \langle h_3, h_1^3, P h_1^2 h_5 \rangle$.
It follows that
$c_0 \cdot h_1 G_3 = \langle c_0, h_3, h_1^3 \rangle P h_1^2 h_5$.
Table \ref{tab:Massey} shows that 
$\langle c_0, h_3, h_1^3 \rangle = h_1^2 e_0$.

We have now shown that $c_0 \cdot h_1 G_3 = P h_1^4 h_5 e_0$.
It follows that either
$c_0 \cdot \ol{h_1 G_3} = h_1^2 \cdot \ol{P h_1^2 h_5 e_0}$ or
$c_0 \cdot \ol{h_1 G_3}= h_1^2 \cdot \ol{P h_1^2 h_5 e_0} + h_1^2 B_{21}$.
In either case,
$c_0 \cdot \ol{h_1 G_3} = h_1^2 \cdot \ol{P h_1^2 h_5 e_0}$ in $E_3(C\tau)$
since $h_1^2 B_{21}$ is hit by an Adams $d_2$ differential.

Since $d_3 ( h_1^2 \cdot \ol{P h_1^2 h_5 e_0} ) = d_0 u'$ is non-zero,
we conclude that $d_3(\ol{h_1 G_3})$ is also non-zero, and there is
just one possible non-zero value.
\end{proof}

\begin{lemma}
$d_3 (\ol{h_1 d_1 g}) = 0$.
\end{lemma}

\begin{proof}
The only other possibility is that
$d_3 (\ol{h_1 d_1 g}) = h_1^2 G_3$.
If this were the case, then
$\{ h_1^2 G_3 \}$ in $\pi_{53,30}$ would be divisible by $\tau$.
If $\{ h_1^2 G_3 \}$ were divisible by $\tau$, then
the only possibility would be that
$\tau \{ h_1 d_1 g \} = \{ h_1^2 G_3 \}$.
However, $\tau \{ h_1 d_1 g\}$ is zero 
by Lemma \ref{lem:t.eta-d1}.
\end{proof}

\begin{lemma}
\label{lem:Ctau-d3-_h1^3D4}
$d_3( \ol{h_1^3 D_4}) = h_1 B_{21}$.
\end{lemma}

\begin{proof}
Recall from Lemma \ref{lem:_D4.h1^3c0} that
$h_1^3 c_0 \cdot \ol{D_4}$ equals either 
$h_2 B_5$ or $h_2 B_5 + h_1^2 X_3$.
It follows that $c_0 \cdot \ol{h_1^3 D_4}$ equals either
$h_2 B_5$ or $h_2 B_5 + h_1^2 X_3$.
However, these two elements are equal in $E_3(C\tau)$ since
$h_1^2 X_3$ is the target of an Adams $d_2$ differential.

We know that $d_3 (h_2 B_5) = h_1 B_8 d_0$ 
by Table \ref{tab:Adams-d3}.
It follows that $d_3(\ol{h_1^3 D_4})$ is non-zero, and there
is just one possibility.
\end{proof}

\begin{lemma}
\label{lem:Ctau-d3-_Ph5c0e0}
$d_3 (\ol{P h_5 c_0 e_0}) = \ol{h_1^2 c_0 x'} + U$.
\end{lemma}

\begin{proof}
First note that either
$h_1 \cdot \ol{P h_5 c_0 e_0} = P h_1 \cdot \ol{h_5 c_0 e_0}$
or
$h_1 \cdot \ol{P h_5 c_0 e_0} = P h_1 \cdot \ol{h_5 c_0 e_0} + h_1^2 q_1$.
In either case,
$d_3(h_1 \cdot \ol{P h_5 c_0 e_0}) = P h_1 \cdot \ol{h_1^2 B_8}$ since
$d_3 (\ol{h_5 c_0 e_0}) = \ol{h_1^2 B_8}$ and
$d_3 (h_1^2 q_1) = 0$.

Finally, we must compute that
$P h_1 \cdot \ol{h_1^2 B_8} = h_1 \cdot \ol{h_1^2 c_0 x'} + h_1 U$.
Because of the relation $B_8 \cdot P h_1 = c_0 x'$,
either
$P h_1 \cdot \ol{h_1^2 B_8} = h_1^2 \cdot \ol{h_1 c_0 x'}$ or
$P h_1 \cdot \ol{h_1^2 B_8} = h_1^2 \cdot \ol{h_1 c_0 x'} + h_1 U$.
The second case must be correct because this is the element
that survives to $E_3(C\tau)$.
\end{proof}

\subsection{Adams $d_4$ differentials for the cofiber of $\tau$}
\label{subsctn:Ctau-d4}

See \cite{Isaksen14a} for a chart of $E_4(C\tau)$.
This chart is complete through
the 64-stem.  Beyond the 64-stem, 
because of unknown earlier differentials,
the actual $E_4$-page is a subquotient of what is shown in the chart.
\index{Adams chart!cofiber of tau@cofiber of $\tau$}

The next step is to compute Adams $d_4$ differentials
on the $E_4(S^{0,0})$-module generators of $E_4(C\tau)$.

\begin{prop}
\label{prop:Ctau-d4}
The motivic Adams $d_4$
differential for the cofiber of $\tau$ is zero
on all $E_4(S^{0,0})$-module generators of $E_4(C\tau)$
through the 63-stem, except that:
\begin{enumerate}
\item
$d_4( h_0^{15} h_6) = \ol{\tau P^2 h_0 d_0^2 e_0}$.
\item
$d_4( \ol{j_1} )$ might equal $B_{21}$.
\end{enumerate}
\end{prop}

\index{cofiber of tau@cofiber of $\tau$!Adams spectral sequence!d4@$d_4$}
\index{h6@$h_6$}
\index{P2d02e0@$P^2 d_0^2 e_0$}

\begin{proof}
For degree reasons, there are very few possible differentials.
The only difficult cases are addressed in Lemmas \ref{lem:d4-h0D2}
and \ref{lem:d4_h3d1g}.
\end{proof}

The chart of $E_4(C\tau)$ in \cite{Isaksen14a} indicates the 
Adams $d_4$ differentials, all of which are implied by
the calculations in Proposition \ref{prop:Ctau-d4} and 
Table \ref{tab:Adams-d4}.
The differentials are complete only through the 63-stem.
Beyond the 63-stem, there are a number of unknown differentials.
\index{Adams chart!cofiber of tau@cofiber of $\tau$}

\begin{remark}
Recall that $h_0^{15} h_6$ does not survive to $E_4(S^{0,0})$,
so this element is an $E_4(S^{0,0})$-module generator of $E_4(C\tau)$.
This is the reason that the formula for $d_4(h_0^{15} h_6)$ appears
in the statement of Proposition \ref{prop:Ctau-d4}.
\index{h6@$h_6$}
\end{remark}

\begin{remark}
\label{rem:Ctau-d4-C'}
The possible differential
$d_4(C') = h_2 B_{21}$ in $E_4(S^{0,0})$ 
mentioned in Proposition \ref{prop:Adams-d4}
occurs if and only if
$d_4 (\ol{j_1}) = B_{21}$ in $E_4(C\tau)$.
This follows immediately from the relation $h_2 \cdot \ol{j_1} = C'$.
\index{C'@$C'$}
\index{B21@$B_{21}$}
\index{j1@$j_1$}
\end{remark}

\begin{lemma}
\label{lem:d4-h0D2}
$d_4(h_0 D_2) = 0$.
\end{lemma}

\begin{proof}
We showed in Lemma \ref{lem:tau-D11}
that $h_0 h_2 h_5 i$ detects an element $\alpha$ of $\pi_{57,30}$
that is not divisible by $\tau$.
Therefore, $\alpha$ maps to a non-zero element of
$\pi_{57,30}(C\tau)$.
The only possibility is that
this element of 
$\pi_{57,30}(C\tau)$ is detected by $h_1 Q_1$.
In particular, $h_1 Q_1$ cannot equal $d_4(h_0 D_2)$.
\end{proof}

\begin{lemma}
\label{lem:d4_h3d1g}
$d_4 (\ol{h_3 d_1 g}) = 0$.
\end{lemma}

\begin{proof}
The only other possibility is that 
$d_4 (\ol{h_3 d_1 g})$ equals $P h_1^3 h_5 e_0$.
We showed in Lemma \ref{lem:t.h3d1g} that 
the element $\{ P h_1^3 h_5 e_0\}$ of $\pi_{59,33}$ is not
divisible by $\tau$.
Therefore, $P h_1^3 h_5 e_0$ is not hit by a differential
in the Adams spectral sequence for $C\tau$.
\end{proof}

\subsection{Higher Adams differentials for the cofiber of $\tau$}
\label{subsctn:higher-diff}

At this point, we are nearly done.  There is just one more
differential to compute.

\index{cofiber of tau@cofiber of $\tau$!Adams spectral sequence!d5@$d_5$}

\begin{lemma}
\label{lem:d5-h2h5}
$d_5( h_2 h_5) = 0$.
\end{lemma}

\begin{proof}
The only other possibility is that $d_5(h_2 h_5)$ equals $h_1 q$.
We showed in Lemma \ref{lem:t.eta-d1}
that the element $\{h_1 q\}$ of $\pi_{33,18}$
is not divisible by $\tau$.  Therefore, $h_1 q$ cannot be hit by a
differential in the Adams spectral sequence for the cofiber of $\tau$.
\end{proof}

The $E_4(C\tau)$ chart in \cite{Isaksen14a} indicates the
very few $d_5$ differentials along with the $d_4$ differentials.


\subsection{The Adams $E_\infty$-page for the cofiber of $\tau$}
\label{subsctn:htpy-Ctau}

Using the Adams differentials given in 
Table \ref{tab:Ctau-E2}, Table \ref{tab:Ctau-d3},
and Proposition \ref{prop:Ctau-d4}, 
as well as the Adams differentials for $S^{0,0}$ given in 
Tables \ref{tab:Ext-gen}, \ref{tab:Adams-d3}, \ref{tab:Adams-d4},
and \ref{tab:Adams-d5},
we can now directly compute the $E_\infty$-page of the
Adams spectral sequence for $C\tau$.

\index{cofiber of tau@cofiber of $\tau$!Adams spectral sequence!Einfinity-page@$E_\infty$-page}

\begin{thm}
\label{thm:Ctau-Einfty}
The $E_\infty$-page of the Adams spectral sequence for $C\tau$
is depicted in \cite{Isaksen14a}.
This chart is complete through the 63-stem.
Beyond the 63-stem,
$E_\infty(C\tau)$ is a subquotient of what is shown in the chart.
\end{thm}
\index{Adams chart!cofiber of tau@cofiber of $\tau$}

Through the 63-stem, all unknown differentials
are indicated as dashed lines.
Beyond the 63-stem, there are a number of unknown differentials.

In a range,
we now have a complete understanding of $E_\infty(C\tau)$, which is the 
associated graded object of $\pi_{*,*}(C\tau)$ with respect to the Adams filtration.
In order to better understand $\pi_{*,*}(C\tau)$ itself,
we would like to compute the maps of homotopy groups induced by
the
inclusion $j: S^{0,0} \map C\tau$ of the bottom cell and the projection
$q: C\tau \map S^{1,-1}$ to the top cell.
\index{cofiber of tau@cofiber of $\tau$!homotopy group}
\index{cofiber of tau@cofiber of $\tau$!top cell}
\index{cofiber of tau@cofiber of $\tau$!bottom cell}

\begin{prop}
\label{prop:Ctau-bottom}
\index{cofiber of tau@cofiber of $\tau$!bottom cell}
The map $j_*: \pi_{*,*} \map \pi_{*,*} C\tau$ induced by the inclusion
of the bottom cell is described as follows, through the 59-stem.
Let $\alpha$ be an element of $\pi_{*,*}$ detected by $a$ in 
$E_\infty(S^{0,0})$.  
\begin{enumerate}
\item
If $a$ does not equal $h_0 h_2 h_5 i$,
then $j_*(\alpha)$ is detected by $j_*(a)$ in $E_\infty(C\tau)$.
\index{h5i@$h_5 i$}
\item
If $a$ equals $h_0 h_2 h_5 i$,
then $j_*(\alpha)$ is detected by $h_1 Q_1$ in $E_\infty(C\tau)$.
\index{Q1@$Q_1$}
\end{enumerate}
\end{prop}

\begin{proof}
This is a straightforward calculation, using that there is an induced map
$E_\infty(S^{0,0}) \map E_\infty(C\tau)$.
\end{proof}

It is curious that the Adams filtration hides so little about the 
map $j_*$.

\begin{prop}
\label{prop:Ctau-top}
The map $q_*: \pi_{*,*} C\tau \map \pi_{*-1,*+1}$ induced by the projection
to the top cell is described as follows, through the 59-stem.
\index{cofiber of tau@cofiber of $\tau$!top cell}
\begin{enumerate}
\item
An element of $\pi_{*,*}(C\tau)$ in the image of
$j_*: \pi_{*,*} \map \pi_{*,*}(C\tau)$ 
(as described by Proposition \ref{prop:Ctau-bottom})
maps to $0$ in $\pi_{*-1,*+1}$.
\index{cofiber of tau@cofiber of $\tau$!bottom cell}
\item
An element of $\pi_{*,*}(C\tau)$ detected by $\ol{x}$ in
$E_\infty(C\tau)$ maps to an element of $\pi_{*-1,*+1}$ detected by $x$
in $E_\infty(S^{0,0})$.
\item
The remaining possibilities are described in Table \ref{tab:Ctau-top}.
\end{enumerate}
\end{prop}

\begin{proof}
The part of $q_*$ that is not hidden by the Adams filtration is described
in (1) and (2).
The part of $q_*$ that is hidden by the Adams filtration is described
in Table \ref{tab:Ctau-top}.  These are the only possible values that
are compatible with the long exact sequence
\[
\xymatrix@1{
\cdots \ar[r] & \pi_{*,*+1} \ar[r] & \pi_{*,*} \ar[r] &
\pi_{*,*}(C\tau) \ar[r] & \pi_{*-1,*+1} \ar[r] & \cdots.
}
\]
\end{proof}


\section{Hidden Adams extensions for the cofiber of $\tau$}
\label{sctn:Ctau-hidden}

Finally, we will consider hidden extensions by $2$, $\eta$, and
$\nu$ in the motivic stable homotopy groups $\pi_{*,*}(C\tau)$
of the cofiber of $\tau$.
\index{cofiber of tau@cofiber of $\tau$!hidden extension!two}
\index{cofiber of tau@cofiber of $\tau$!hidden extension!eta@$\eta$}
\index{cofiber of tau@cofiber of $\tau$!hidden extension!nu@$\nu$}
We will show in Lemma \ref{lem:ANSS-Ctau-hidden} that
there are no hidden $\tau$ extensions in $\pi_{*,*}(C\tau)$.

Recall from Proposition \ref{prop:3bracket-cofiber} that a hidden
extension by $\alpha$ in $\pi_{*,*}(C\tau)$
is the same as a Toda bracket in $\pi_{*,*}$
of the form $\langle \tau, \beta, \alpha \rangle$.
Many such Toda brackets are detected in
$\Ext$ by a corresponding Massey product of the form
$\langle \tau, b, a \rangle$.
In this circumstance, the extension by $\alpha$ is already detected
in $E_\infty(C\tau)$.
\index{Toda bracket}
\index{Massey product}

However, there are some Toda brackets of the form
$\langle \tau, \beta, \alpha \rangle$
that are not detected by Massey products in $\Ext$.
In this section, we will study such Toda brackets methodically.

\begin{prop}
\label{prop:Ctau-hidden}
Table \ref{tab:Ctau-hidden} shows some hidden extensions
by $2$, $\eta$, and $\nu$ in $\pi_{*,*}(C\tau)$.
Through the 59-stem, there are no other hidden extensions
by $2$, $\eta$, and $\nu$, except that:
\begin{enumerate}
\item
there might be a hidden $2$ extension from
$\ol{h_1 d_1 g}$ to $h_1 B_8$.
\index{d1g@$d_1 g$}
\index{B8@$B_8$}
\item
there might be a hidden $2$ extension from
$Q_2$ to $h_1 Q_1$.
\index{Q2@$Q_2$}
\index{Q1@$Q_1$}
\item
there might be a hidden $2$ extension from
$h_1^2 D_4$ to $P h_1^3 h_5 e_0$.
\index{D4@$D_4$}
\index{Ph5e0@$P h_5 e_0$}
\item
there might be a hidden $\eta$ extension from
$\ol{h_1^2 g_2}$ to $h_0 B_2$.
\index{g2@$g_2$}
\index{B2@$B_2$}
\item
there might be a hidden $\nu$ extension from
$B_6$ to $h_1 D_{11}$.
\index{B6@$B_6$}
\index{D11@$D_{11}$}
\item
there might be a hidden $\nu$ extension from
$\ol{\tau h_2 d_1 g}$ to $B_{21}$.
\index{d1g@$d_1 g$}
\index{B21@$B_{21}$}
\item
if $\ol{h_1 i_1} + \tau h_1 G$ survives to $E_\infty(C\tau)$, then
there might be a hidden $\nu$ extension from
$\ol{h_1 i_1} + \tau h_1 G$ to $h_1 D_{11}$.
\index{i1@$i_1$}
\index{G@$G$}
\index{D11@$D_{11}$}
\item
if $\ol{j_1}$ survives to $E_\infty(C\tau)$, then
there might be a hidden $2$ extension from
$\ol{j_1}$ to $h_1 \cdot \ol{h_3 G_3}$.
\index{j1@$j_1$}
\index{h3G3@$h_3 G_3$}
\end{enumerate}
\end{prop}

\begin{proof}
Some of the extensions are detected by the projection
$q: C\tau \map S^{1,-1}$ to the top cell,
and some of the extensions are detected by the inclusion
$j: S^{0,0} \map C\tau$ of the bottom cell.
\index{cofiber of tau@cofiber of $\tau$!top cell}
\index{cofiber of tau@cofiber of $\tau$!bottom cell}
The remaining cases are established in the following lemmas.
\end{proof}

\begin{remark}
\label{rem:Ctau-eta-_h1^2g2}
The possible hidden $\eta$ extension on $\ol{h_1^2 g_2}$
is connected to some of the other uncertainties in our calculations.
\index{g2@$g_2$}
Suppose that there is a hidden $\tau$ extension from $h_1 i_1$ to $h_1 B_8$
in $\pi_{*,*}$ (see Remark \ref{rem:tau-h1i1}).
\index{i1@$i_1$}
\index{B8@$B_8$}
Then 
$\nu \{C\} + \tau \{i_1\}$ is detected by $B_8$, and 
there is a hidden $\nu$ extension in $\pi_{*,*}(C\tau)$
from $C$ to $B_8$.
\index{C@$C$}
If $\{ \ol{h_1^2 g_2} \} \eta$ were zero, then we could further compute that
\[
\{B_8\} = \{ \ol{h_1^2 g_2} \} \nu^2 =
\{ \ol{h_1^2 g_2} \} \langle \eta, \nu, \eta \rangle =
\langle \{ \ol{h_1^2 g_2} \}, \eta, \nu \rangle \eta
\]
in $\pi_{*,*}(C\tau)$.
However, $\{B_8\}$ cannot be divisible by $\eta$ in $\pi_{*,*}(C\tau)$.
Therefore, $\{ \ol{h_1^2 g_2} \} \eta$ would be non-zero in $\pi_{*,*}(C\tau)$.
\end{remark}

\begin{lemma}
\label{lem:Ctau-nu-_h1^4h5}
There is no hidden $\nu$ extension on $\ol{h_1^4 h_5}$.
\end{lemma}

\begin{proof}
The only other possibility is that 
there is a hidden $\nu$ extension from $\ol{h_1^4 h_5}$ to $u$.
We will show that the Toda bracket
$\langle \tau, \eta^3 \eta_5, \nu \rangle$ does not contain $\{ u\}$.

The bracket 
contains $\langle \tau \eta^3, \eta_5, \nu \rangle$,
which equals $\langle 4 \nu, \eta_5, \nu \rangle$.
This bracket contains $4 \langle \nu, \eta_5, \nu \rangle$.
Note that $\langle \nu, \eta_5, \nu \rangle$
intersects $\{ h_1 h_3 h_5 \}$,
but $4 \langle \nu, \eta_5, \nu \rangle$ is zero.

Finally, the bracket
$\langle \tau, \eta^3 \eta_5, \nu \rangle$ has indeterminacy
generated by $\tau \{ h_3 d_1 \}$ and $\tau^2 \{ c_1 g\}$.
Therefore, $\{ u\}$ is not in the bracket.
\end{proof}

\begin{lemma}
\label{lem:Ctau-eta-h0y}
There is a hidden $\eta$ extension from $h_0 y$ to $u$.
\end{lemma}

\begin{proof}
Table \ref{tab:Ctau-top} shows that 
projection to the top cell maps 
$\{ h_0 y\}$ to $\{\tau h_2 e_0^2 \}$ in $\pi_{37,21}$.
\index{cofiber of tau@cofiber of $\tau$!top cell}
The bracket
$\langle \tau, \{ \tau h_2 e_0^2 \}, \eta \rangle$ contains
$\langle \{ \tau^2 e_0^2 \}, \nu, \eta \rangle$ which contains 
$\{u\}$ by Table \ref{tab:Toda}.

The indeterminacy of $\langle \tau, \{ \tau h_2 e_0^2 \}, \eta \rangle$
is generated by
$\tau \sigma \{d_1\}$, and $\eta \{h_0^2 h_3 h_5\}$.
Note that $\tau \sigmabar \kappabar$ is equal to $\eta \{h_0^2 h_3 h_5\}$,
as shown in Remark \ref{rem:eta-h0^2h3h5}.
Since $\{u\}$ is not in the indeterminacy,
the bracket does not contain zero.
\end{proof}

\begin{lemma}
\label{lem:Ctau-nu-h0c2}
There is no hidden $\nu$ extension on $h_0 c_2$.
\end{lemma}

\begin{proof}
According to Table \ref{tab:Ctau-top},
the projection to the top cell takes
the elements of $\pi_{41,22}(C\tau)$ detected by $h_0 c_2$
to elements of $\pi_{40,23}$ that are detected by $h_1 h_3 d_1$.
\index{cofiber of tau@cofiber of $\tau$!top cell}
These elements of $\pi_{40,23}$ must also be annihilated by $\tau$,
so they must be $\eta \sigma \{d_1\}$ and $\eta \sigma \{d_1\} + \{\tau h_0^2 g^2\}$.

It remains to compute the Toda bracket
$\langle \tau, \eta \sigma \{d_1\}, \nu \rangle$.  This bracket contains
$\langle \tau, \eta \{d_1\}, 0 \rangle$, which equals zero.
\end{proof}

\begin{lemma}
\label{lem:Ctau-2-h0c2}
There is no hidden $2$ extension on $h_0 c_2$.
\end{lemma}

\begin{proof}
We showed in Lemma \ref{lem:Ctau-nu-h0c2} that there is 
no hidden $\nu$ extension on $h_0 c_2$.
Therefore, there cannot be a hidden $2$
extension from $h_0 c_2$ to $\ol{\tau h_0^2 g^2}$.

There are no other possible hidden $2$ extensions on $h_0 c_2$.
\end{proof}

\begin{lemma}
\label{lem:Ctau-2-_h3^3g}
There is no hidden $2$ extension on $h_3 \cdot \ol{h_3^2 g}$.
\end{lemma}

\begin{proof}
The projection to the top cell detects that
$h_3 \cdot \ol{h_3^2 g}$ is the target of a hidden $\eta$ extension from $h_0 c_2$.
\index{cofiber of tau@cofiber of $\tau$!top cell}
Therefore, $h_3 \cdot \ol{h_3^2 g}$ cannot support a hidden $2$ extension.
\end{proof}

\begin{lemma}
\label{lem:Ctau-eta-_th2c1g}
There is no hidden $\eta$ extension on $\ol{\tau h_2 c_1 g}$.
\end{lemma}

\begin{proof}
The projection to the top cell takes the 
element $\{ \ol{\tau h_2 c_1 g} \}$ of $\pi_{43,23}(C\tau)$
to an element of $\pi_{42,24}$ that is detected by $\tau h_2 c_1 g$.
\index{cofiber of tau@cofiber of $\tau$!top cell}
Two elements of $\pi_{42,24}$ are detected by
$\tau h_2 c_1 g$, but only one element is killed by $\tau$.
The relation $\eta \{h_0^2 h_3 h_5\} = \tau \sigmabar \kappabar$
from Remark \ref{rem:eta-h0^2h3h5} implies that
$\nu \sigmabar \kappabar$ is the element of $\pi_{42,24}$
that is killed by $\tau$ and detected by $\tau h_2 c_1 g$.
Therefore,
the top cell detects that there is no hidden
$\eta$ extension on $\ol{\tau h_2 c_1 g}$.
\end{proof}

\begin{lemma}
\label{lem:Ctau-nu-d0r}
There is a hidden $\nu$ extension from $d_0 r$ to $h_1 u'$.
\end{lemma}

\begin{proof}
The inclusion of the bottom cell shows that
there is a hidden $2$ extension from $e_0 r$ to $h_1 u'$ 
in $\pi_{47,26}(C\tau)$. 
\index{cofiber of tau@cofiber of $\tau$!bottom cell}
The hidden $\nu$ extension on $d_0 r$
is an immediate consequence.
\end{proof}


%% file: stable-stems-ANSS.tex

\chapter{Reverse engineering the Adams-Novikov spectral sequence}
\label{ch:ANSS}

\setcounter{thm}{0}


In this chapter, we will show that 
the classical Adams-Novikov $E_2$-page is identical to the 
the motivic stable homotopy
groups $\pi_{*,*}(C\tau)$ of the cofiber of $\tau$ computed
in Chapter \ref{ch:Ctau}.
Moreover, the classical Adams-Novikov differentials and
hidden extensions can also
be deduced from prior knowledge of motivic stable homotopy groups.
We will apply this program to provide detailed computational
information about the classical Adams-Novikov spectral sequence
in previously unknown stems.
\index{Adams-Novikov spectral sequence}
\index{Adams-Novikov spectral sequence!classical}
\index{cofiber of tau@cofiber of $\tau$!homotopy group}

In fact, the classical
Adams-Novikov spectral sequence appears to be identical to the
$\tau$-Bockstein spectral sequence converging to stable motivic homotopy groups.
We have only a computational understanding of this curious phenomenon.
Our work calls for a more conceptual study of this relationship.
\index{tau-Bockstein spectral sequence@$\tau@-Bockstein spectral sequence}

The simple pattern of weights in the motivic Adams-Novikov spectral sequence 
is the key idea that allows this program to proceed.
\index{weight}
See Theorem \ref{thm:motE2} for more explanation.
For example, for simple degree reasons,
there can be no hidden $\tau$ extensions in the motivic 
Adams-Novikov spectral sequence.
Also for simple degree reasons,
there are no ``exotic" Adams-Novikov differentials;
each non-zero motivic differential corresponds to a classical non-zero analogue.

\subsection*{Outline}

Section \ref{sctn:motANSS} describes the motivic 
Adams-Novikov spectral sequence in general terms.
Section \ref{sctn:ANSS-Ctau} deals with specific properties of the
motivic Adams-Novikov spectral sequence for the cofiber of $\tau$.
The main point is that this spectral sequence collapses.
Section \ref{sctn:ANSS-calculate} carries out the translation
of information about $\pi_{*,*}(C\tau)$ into information
about the classical Adams-Novikov spectral sequence.

Chapter \ref{ch:table} contains a series of tables that summarize the essential
computational facts in a concise form.
Tables \ref{tab:ANSS-hidden-2}, \ref{tab:ANSS-hidden-eta},
and \ref{tab:ANSS-hidden-nu} list the extensions by
$2$, $\eta$, and $\nu$ that are hidden in the Adams-Novikov spectral sequence.

Table \ref{tab:Adams-ANSS} gives a correspondence between
elements of the classical Adams $E_\infty$-page and elements of the 
classical Adams-Novikov $E_\infty$-page.
\index{Adams spectral sequence!comparison to Adams-Novikov spectral sequence}
\index{Adams-Novikov spectral sequence!comparison to Adams spectral sequence}
When possible, the table also gives an element of
$\pi_*$ that is detected by these $E_\infty$ elements.

Tables \ref{tab:ANSS-boundary} and \ref{tab:ANSS-non-permanent}
list the classical Adams-Novikov elements that are boundaries
and that support differentials respectively.  The tables list
the corresponding elements of $\pi_{*,*}(C\tau)$.

\subsection*{Classical Adams-Novikov inputs}

\index{Adams-Novikov spectral sequence!classical}
The point of this chapter is to deduce
information about the Adams-Novikov spectral sequence from prior
knowledge of the motivic stable homotopy groups obtained in 
Chapters \ref{ch:Adams-diff}, \ref{ch:Adams-hidden}, and \ref{ch:Ctau}.
To avoid circularity, 
Chapters \ref{ch:Adams-diff}, \ref{ch:Adams-hidden}, and \ref{ch:Ctau}
intentionally avoid use of the 
Adams-Novikov spectral sequence whenever possible.  However, we need a few 
computational facts about the Adams-Novikov spectral sequence 
in Chapter \ref{ch:Adams-hidden}:
\begin{enumerate}
\item
Lemma \ref{lem:tau-D11} shows that a certain possible hidden
$\tau$ extension does not occur in the 57-stem.  See also 
Remark \ref{rem:tau-D11}.  For this, we use that $\beta_{12/6}$ is the only element
in the Adams-Novikov spectral sequence in the 58-stem with filtration 2
that is not divisible by $\alpha_1$
\cite{Shimomura81}.
\item
Lemma \ref{lem:2-h0h5i} establishes a hidden
$2$ extension in the 54-stem.  See also Remark \ref{rem:2-h0h5i}.
For this, we use that $\beta_{10/2}$ is the only element of the 
Adams-Novikov spectral sequence in the 54-stem with filtration $2$
that is not divisible by $\alpha_1$,
and that this element maps to $\Delta^2 h_2^2$ in the
Adams-Novikov spectral sequence for $\tmf$ \cite{Bauer08} \cite{Shimomura81}.
\index{topological modular forms}
\end{enumerate}

\subsection*{Some examples}

\begin{ex}
Consider the element $\{ h_1^2 h_3 g\}$ of $\pi_{29,18}$.
This element is killed by $\tau^2$ but not by $\tau$.
\index{h3g@$h_3 g$}

The Adams-Novikov element $\alpha_1 z_{28}$ detects
$\{ h_1^2 h_3 g \}$ (see the charts in \cite{Isaksen14e}).
Therefore, $\tau^2 \alpha_1 z_{28}$ must be hit by some 
Adams-Novikov differential.
This implies that there is a classical Adams-Novikov $d_5$ differential
from the 30-stem to the 29-stem.
This differential is well-known \cite{Ravenel86}.
\end{ex}

\begin{ex}
Consider the element $\{h_1^7 h_5 e_0 \}$ of $\pi_{55,33}$.
This element is killed by $\tau^4$ but not by $\tau^3$.
\index{h5e0@$h_5 e_0$}

The Adams-Novikov element $\alpha_1 z_{54,10}$ detects
$\{ h_1^7 h_5 e_0 \}$ (see the charts in \cite{Isaksen14e}).
Therefore,
$\tau^4 \alpha_1 z_{54,10}$ must be hit by some 
Adams-Novikov differential.
This implies that there is a classical Adams-Novikov
$d_9$ differential from the 56-stem to the 55-stem.
This differential lies far beyond previous calculations.
\end{ex}

\section{The motivic Adams-Novikov spectral sequence}
\label{sctn:motANSS}

We adopt the following notation for the classical
Adams-Novikov spectral sequence.

\index{Adams-Novikov spectral sequence}
\index{Adams-Novikov spectral sequence!classical}

\begin{defn}
\label{defn:clE}
Let $\clE{r}$ (and $\clE{\infty}$) be the pages of the classical
Adams-Novikov spectral sequence for $S^{0}$.
We write 
$\clE[s,f]{r}$ for the part of $\clE{r}$ in stem $s$ and filtration $f$.
\end{defn}

The even Adams-Novikov differentials $d_{2r}$ are all zero, so 
we will only consider $\clE{r}$ 
when $r$ is odd (or is $\infty$).

We now describe the motivic
Adams-Novikov spectral sequence.
Recall that $BPL$ is the motivic analogue of the
classical Brown-Peterson spectrum $BP$.

\begin{defn}
\label{defn:motE}
Let $\motE{r}$ (and $\motE{\infty}$) be the pages of the motivic 
Adams-Novikov spectral sequence for the motivic sphere $S^{0,0}$.
We write $\motE[s,f,w]{2}$ for the part of $\motE{2}$ in stem $s$, filtration $f$,
and weight $w$.
\end{defn}

Our goal is to describe the motivic Adams-Novikov spectral sequence
in terms of the
classical Adams-Novikov spectral sequence, as in 
\cite{HKO11}*{Theorem 8 and Section 4}.

\begin{defn}
\label{defn:motE2[0]}
Define the tri-graded object $\olmotE{2}$
such that:
\begin{enumerate}
\item
$\olmotE[s,f,\frac{s+f}{2}]{2}$ is isomorphic to $\clE[s,f]{2}$.
\item
$\olmotE[s,f,w]{2}$ is zero if $w \neq \frac{s+f}{2}$.
\end{enumerate}
\end{defn}

The following theorem completely describes the
motivic $\motE{2}$-page in terms of the classical
$\clE{2}$-page.

\begin{thm}
\label{thm:motE2}
\cite{HKO11}*{Theorem 8 and Section 4}
The $\motE{2}$-page of the motivic Adams-Novikov spectral
sequence is isomorphic to 
the tri-graded object $\olmotE{2} \otimes_{\Z_2} \Z_2[\tau]$,
where $\tau$ has degree $(0,0,-1)$.
\end{thm}

\index{Adams-Novikov spectral sequence!E2-page@$E_2$-page}

In other words, in order to produce the motivic $E_2$-page,
start with the classical $E_2$-page.  At degree $(s,f)$, replace each copy of
$\Z_2$ or $\Z/2^n$ with a copy of $\Z_2[\tau]$ or $\Z/2^n[\tau]$, where the
generator has weight $\frac{s+f}{2}$.

We will now compare the classical and motivic Adams-Novikov spectral sequences.
As we have seen in earlier chapters, $\tau$-localization
corresponds to passage from the motivic to classical situations.

\begin{thm}
\label{thm:ANSS-compare}
After inverting $\tau$,
the motivic Adams-Novikov spectral sequence is isomorphic to
the classical Adams-Novikov spectral sequence tensored over $\Z_2$
with $\Z_2[\tau^{\pm 1}]$.
\end{thm}

\index{tau@$\tau$!localization}

\begin{proof}
The proof is analogous to the corresponding result for the
motivic and classical Adams spectral sequences.  See
Proposition \ref{prop:compare} and
\cite{DI10}*{Sections 3.2 and 3.4}.
\end{proof}


\section[Motivic Adams-Novikov for the cofiber of $\tau$]
{The motivic Adams-Novikov spectral sequence for the cofiber of $\tau$}
\label{sctn:ANSS-Ctau}

We will now study the motivic Adams-Novikov spectral sequence
that computes the homotopy groups of the cofiber $C\tau$ of $\tau$.

\index{cofiber of tau@cofiber of $\tau$!Adams-Novikov spectral sequence}

\begin{defn}
\label{defn:motE-Ctau}
Let $\motECtau{r}$ (and $\motECtau{\infty}$) be the pages of the motivic 
Adams-Novikov spectral sequence for $C\tau$.
We write $\motECtau[s,f,w]{2}$ for the part of $\motECtau{2}$ in stem $s$,
filtration $f$, and weight $w$.
\end{defn}

\begin{lemma}
\label{lem:ANSS-Ctau}
$\motECtau{2}$ is isomorphic to $\olmotE{2}$.
\end{lemma}

\begin{proof}
The cofiber sequence
\[
\xymatrix{
S^{0,-1} \ar[r]^\tau & S^{0,0} \ar[r] & C\tau \ar[r] & S^{1,-1}
}
\]
induces a long exact sequence
\[
\xymatrix{
\cdots \ar[r] & \motE{2} \ar[r]^\tau & \motE{2} \ar[r] & \motECtau{2} \ar[r] &
\cdots.}
\]
Theorem \ref{thm:motE2} tells us that the map $\tau: \motE{2} \map \motE{2}$
is injective, so $\motECtau{2}$ is isomorphic to the cokernel of $\tau$.
Theorem \ref{thm:motE2} tells us that this cokernel is isomorphic
to $\olmotE{2}$.
\end{proof}

\begin{lemma}
\label{lem:ANSS-Ctau-collapse}
There are no differentials in the motivic Adams-Novikov spectral sequence
for $\tau$.
\end{lemma}

\begin{proof}
Lemma \ref{lem:ANSS-Ctau} tells us that
$\motECtau{2}$ is concentrated in tridegrees $(s,f,w)$ where $s+f-2w$ equals zero.
The Adams-Novikov $d_r$ differential increases $s+f-2w$ by $r-1$.
Therefore, all differentials are zero.
\end{proof}

\begin{lemma}
\label{lem:ANSS-Ctau-hidden}
There are no hidden $\tau$ extensions in $\motECtau{\infty}$.
\end{lemma}

\begin{proof}
Let $x$ and $y$ be two elements of $\motECtau{\infty}$
of degrees $(s,f,w)$ and $(s,f',w')$ with $f'>f$.
Then $w' > w$ since $w = \frac{s+f}{2}$ and $w' = \frac{s+f'}{2}$.
For degree reasons, it is not possible that there is a hidden
$\tau$ extension from $x$ to $y$ because $\tau$ has degree $(0,-1)$.
\end{proof}

\begin{prop}
\label{prop:ANSS-Ctau}
There is an isomorphism $\pi_{*,*} (C\tau) \map \clE{2}$
that takes the group $\pi_{s,w} (C\tau)$ into $\clE[s,2w-s]{2}$.
\end{prop}

\begin{proof}
Lemma \ref{lem:ANSS-Ctau} and Definition \ref{defn:motE2[0]} say that
$\motECtau{2}$ is isomorphic to $\clE{2}$.
Lemma \ref{lem:ANSS-Ctau-collapse} implies that
$\motECtau{\infty}$ is also isomorphic to $\clE{2}$.
As in the proof of Lemma \ref{lem:ANSS-Ctau-hidden},
for degree reasons
there cannot be hidden extensions of any kind.
Therefore,
$\pi_{*,*}(C\tau)$ is also isomorphic to $\clE{2}$.
\end{proof}

\section{Adams-Novikov calculations}
\label{sctn:ANSS-calculate}

We will now provide explicit calculations of the classical 
Adams-Novikov spectral sequence.
The charts in \cite{Isaksen14e} are an essential companion to 
this section.

\subsection{The classical Adams-Novikov $E_2$-page}

\index{notation!Adams-Novikov spectral sequence}
\index{Adams-Novikov spectral sequence!notation}
\index{Adams-Novikov spectral sequence!classical!E2-page@$E_2$-page}

We use the traditional notation for elements of the 
$\alpha$ family, as described in \cite{Ravenel86}.
We draw particular to attention to $\alpha_1$ in degree $(1,1)$
and $\alpha_{2/2}$ in degree $(3,1)$.
These elements detect $\eta$ and $\nu$ respectively.

For elements not in the $\alpha$ family,
we have labelled decomposable elements as products whenever possible.
For elements that are not known to be products, we use arbitrary symbols
of the form $z_{s,f}$ and $z'_{s,f}$
for elements in the $s$-stem with filtration $f$.
When there is no amibiguity, we simplify this to $z_s$ and $z'_s$.

Our notation is unforutnately arbitrary
and does not necessarily convey deeper structure.  However, at least
it allows us to give names to every element in the spectral sequence.
Our notation is not
compatible with the standard notation for elements
of the Adams-Novikov spectral sequence \cite{Ravenel86}.

\begin{ex}
Consider the elements in degree $(46,4)$ in 
the Adams-Novikov $E_2$ chart in \cite{Isaksen14e}.
From left to right, they are
$\alpha_1^2 z_{44,2}$,
$\alpha_1 z_{45}$,
$\alpha_1 z'_{45}$,
and $\alpha_{2/2} z_{43,3}$.
\end{ex}

\begin{thm}
\label{thm:ANSS-E2}
The $E_2$-page of the classical Adams-Novikov spectral sequence
is depicted through the 59-stem in the chart in \cite{Isaksen14e}.
The chart is complete 
except for the
uncertainties described in Propositions \ref{prop:uncertain-54}
and \ref{prop:uncertain-59}, and the following:
\begin{enumerate}
\item
$\alpha_1 z_{47,3}$ might equal $2 \alpha_{2/2} z'_{45}$.
\item
If $\alpha_1 z_8 z'_{45}$ is non-zero,
then $2z_{54,6}$ might equal $\alpha_1 z_8 z'_{45}$.
\item
$\alpha_{2/2} z_{53}$ might equal $\alpha_1^2 z_{54,6}$.
\item
$2z_{57}$ might equal $\alpha_1 z_{56,2}$.
\item
$\alpha_{2/2} z_{55}$ or $\alpha_{2/2} z'_{55}$ might
equal $\alpha_1^2 z_{56,4}$.
\item
$\alpha_{2/2} z_{56,4}$ might equal $z_{59,5}$.
\item
If $z_{60,4}$ is non-zero, then $2z_{60,4}$ might equal $\alpha_1^2 z_{58,2}$.
\end{enumerate}
\end{thm}

\index{Adams-Novikov chart}

\begin{proof}
This follows immediately from Proposition \ref{prop:ANSS-Ctau}
and the calculation of 
$\pi_{*,*}(C\tau)$ given in Chapter \ref{ch:Ctau}.
The uncertainties are consequences of uncertainties in the structure
of $\pi_{*,*}(C\tau)$.
\end{proof}

\index{cofiber of tau@cofiber of $\tau$!homotopy group}

\begin{prop}
\label{prop:uncertain-54}
Modulo elements of the form $\alpha_1^n \alpha_{k/b}$,
in the 53-stem, 54-stem, and 55-stem, 
either case (1) or case (2) occurs.
\begin{enumerate}
\item
$\clE[54,6]{2}$ has order four, containing two distinct non-zero elements
$z_{54,6}$ and $\alpha_1 z_8 z'_{45} = \alpha_{2/2}^3 z'_{45}$; \\
$\clE[55,5]{2}$ is isomorphic to $\Z/2 \oplus \Z/2$ with generators
$z_{55}$ and $z'_{55}$; \\
$\alpha_1 z'_{55} = z_{56,6}$; \\
and $\alpha_{2/2}^2 z_{47,3}$ is zero in $\clE[53,5]{2}$.
\item
$\clE[54,6]{2}$ has order two; \\
$\clE[55,5]{2}$ has order two; \\
and $\alpha_{2/2}^2 z_{47,3} = z_8 z'_{45}$ in $\clE[53,5]{2}$.
\end{enumerate}
\end{prop}

\begin{proof}
In the motivic Adams spectral sequence for $C\tau$, 
there is a possible $d_3$ differential hitting $h_1 B_8$ 
discussed in Proposition \ref{prop:Ctau-Adams-d3}.
Case (1) of the proposition corresponds to the possibility that 
this differential does not occur.
Case (2) of the proposition corresponds to the possibility that this
differential does occur.
\end{proof}

\begin{prop}
\label{prop:uncertain-59}
Modulo elements of the form $\alpha_1^n \alpha_{k/b}$,
in the 59-stem and 60-stem, 
either case (1) or case (2) occurs.
\begin{enumerate}
\item
$\clE[59,5]{2}$ is isomorphic to $\Z/2 \oplus \Z/2$, with generators
$\alpha_1^2 z_{57}$ and $z_{59,5}$;  \\
and $\clE[60,4]{2}$ has order four, containing two distinct non-zero elements
$z_{60,4}$ and $\alpha_{2/2} z_{57}$.
\item
the only non-zero element of $\clE[59,5]{2}$ is
$\alpha_1^2 z_{57}$;  \\
and $\clE[60,4]{2}$ has order two.
\end{enumerate}
\end{prop}

\begin{proof}
In the motivic Adams spectral sequence for $C\tau$, there is a 
possible $d_4$ differential hitting $B_{21}$ 
discussed in Proposition \ref{prop:Ctau-d4}.
Case (1) of the proposition corresponds to the possibility that
this differential does not occur.
Case (2) of the proposition corresponds to the possibility that
this differential does occur.
\end{proof}

\begin{lemma}
Assume that case (1) of Proposition \ref{prop:uncertain-54} occurs.
Then $\alpha_1 z_{47,3}$ equals $2 \alpha_{2/2} z'_{45}$.
\end{lemma}

\begin{proof}
Case (1) of Proposition \ref{prop:uncertain-54} says that
$\alpha_{2/2}^2 z_{47,3}$ is not divisible by $\alpha_1$.
If $\alpha_1 z_{47,3}$ were zero, then we could shuffle Massey products to obtain
\[
\alpha_{2/2}^2 z_{47,3} = \langle \alpha_1, \alpha_{2/2}, \alpha_1 \rangle z_{47,3} =
\alpha_1 \langle \alpha_{2/2}, \alpha_1, z_{47,3} \rangle.
\]
Therefore, $\alpha_1 z_{47,3}$ must be non-zero.

Under the isomorphism of Proposition \ref{prop:ANSS-Ctau},
the element $z_{47,3}$ of $\clE[47,3]{2}$ corresponds to the
element $\ol{h_1^2 g_2}$ in $\pi_{47,25}(C\tau)$.
We showed in Proposition \ref{prop:Ctau-hidden} that
the only possible hidden $\eta$ extension on $\ol{h_1^2 g_2}$
takes the value $h_0 B_2$, which corresponds
in $\clE[48,26]{2}$ to $2 \alpha_{2/2} z'_{45}$.
\end{proof}

\subsection{Adams-Novikov differentials}

Having obtained the Adams-Novikov $E_2$-page, we next compute 
differentials.

\index{Adams-Novikov spectral sequence!differential}
\index{Adams-Novikov chart}

\begin{thm}
\label{thm:ANSS-diff}
The differentials in the classical Adams-Novikov spectral sequence
are depicted through the 59-stem in the chart in \cite{Isaksen14e}.
The chart is complete 
except for the following:
\begin{enumerate}
\item
if $z'_{55}$ exists in $\clE[55,5]{2}$,
then $d_3(z'_{55}) = \alpha_1 z_{53}$.
\item
if $z_{60,4}$ exists in $\clE[60,4]{2}$,
then $d_3(z_{60,4}) = z'_{59,7}$.
\end{enumerate}
\end{thm}

\begin{proof}
There is only one pattern of differentials
in the motivic Adams-Novikov spectral sequence that will give
the same answer for $\pi_{*,*}$ that was already obtained in
Chapters \ref{ch:Adams-diff} and \ref{ch:Adams-hidden}.
For example, $\{ h_1^2 h_4 c_0 \}$ in $\pi_{25,15}$
is annihilated by $\tau$.
Therefore, there must be a motivic Adams-Novikov differential
hitting $\tau \alpha_1^2 \alpha_{4/4} z_{16}$.
The only possibility is that $d_3(z_{26}) = \tau \alpha_1^2 \alpha_{4/4} z_{16}$.

Having obtained the motivic Adams-Novikov differentials in this way,
the classical Adams-Novikov differentials follow immediately.

The uncertainties in the statement of the theorem are associated with
the uncertainties in Propositions \ref{prop:uncertain-54} and \ref{prop:uncertain-59}.
\end{proof}

\subsection{The classical Adams-Novikov $E_\infty$-page}

\index{Adams-Novikov spectral sequence!Einfinity-page@$E_\infty$-page}
\index{Adams-Novikov chart}

\begin{thm}
\label{thm:ANSS-Einfty}
The $E_\infty$-page of the classical Adams-Novikov is depicted in the
chart in \cite{Isaksen14e} through the 59-stem.
The chart includes all hidden extensions by $2$, $\eta$, and $\nu$.
The chart is complete except for the uncertainties described in 
Propositions \ref{prop:Einfty-uncertain-54} and \ref{prop:Einfty-uncertain-59}, 
and the following:
\begin{enumerate}
\item
There might be a hidden $\nu$ extension from $\alpha_{2/2} z_{45}$
to $z_{51}$.
\item
There might be a hidden $2$ extension from $2\alpha_{4/4} z_{44}$
to $z_{51}$.
\end{enumerate}
\end{thm}

\begin{proof}
The $E_\infty$-page can be computed directly from
Theorems \ref{thm:ANSS-E2} and \ref{thm:ANSS-diff} because we know the
$E_2$-page and all differentials up to some specified
uncertainties.

The hidden extensions by $2$, $\eta$, and $\nu$ all follow from
extensions in $\pi_{*,*}(C\tau)$, as computed in Chapter \ref{ch:Ctau}.
\end{proof}

Tables \ref{tab:ANSS-hidden-2}, \ref{tab:ANSS-hidden-eta}, and \ref{tab:ANSS-hidden-nu}
list all of the hidden extensions by $2$, $\eta$, and $\nu$ in the
motivic Adams-Novikov spectral sequence.
\index{Adams-Novikov spectral sequence!hidden extension}

\begin{remark}
From Lemma \ref{lem:nu-h2h5d0},
the possible extension (1) 
in Theorem \ref{thm:ANSS-Einfty}
occurs if and only if the possible extension (2) occurs.
\end{remark}

\begin{prop}
\label{prop:Einfty-uncertain-54}
In the 53-stem, 54-stem, and 55-stem, 
either case (1) or case (2) occurs.
\begin{enumerate}
\item
$\alpha_1 z_8 z''_{45} = \alpha_{2/2}^3 z''_{45}$
is a non-zero element of $\clE[54,6]{\infty}$; \\
$\clE[54,8]{\infty}$ is zero; \\
$\alpha_{2/2} z_{50}$ is zero in $\clE[53,5]{\infty}$; \\
and there is a hidden
$\nu$ extension from $z_{50}$ to $z_{53}$.
\item
$\clE[54,6]{\infty}$ is zero; \\
$\alpha_1 z_{53}$ is a non-zero element of $\clE[54,8]{\infty}$; \\
$\alpha_{2/2} z_{50} = z_8 z''_{45}$
in $\clE[53,5]{\infty}$; \\
and there is a hidden $\nu$ extension
from $\alpha_{2/2}^2 z'_{45}$ to $\alpha_1 z_{53}$.
\end{enumerate}
\end{prop}

\begin{proof}
The two cases are associated with the two cases of
Proposition \ref{prop:uncertain-54}.  See also
the first uncertainty in Theorem \ref{thm:ANSS-diff}.
\end{proof}

\begin{prop}
\label{prop:Einfty-uncertain-59}
In the 59-stem, either case (1) or case (2) occurs.
\begin{enumerate}
\item
$z_{59,5}$ is the only non-zero element of $\clE[59,5]{\infty}$; \\
and 
$z_{59,7}$ is the only non-zero element of $\clE[59,7]{\infty}$.
\item
$\clE[59,5]{\infty}$ is zero; \\
and $\clE[59,7]{\infty}$ has two generators $z_{59,7}$ and $z'_{59,7}$.
\end{enumerate}
\end{prop}

\begin{proof}
The two cases are associated with the two cases of
Proposition \ref{prop:uncertain-59}.  See also
the second uncertainty in Theorem \ref{thm:ANSS-diff}.
\end{proof}


%% file: stable-stems-tables.tex

\chapter{Tables}
\label{ch:table}

\index{notation|textbf}
\begin{longtable}{llll}
\caption{Notation for $\pi_{*,*}$} \\
\toprule
element & $(s,w)$ & $\Ext$ & definition \\
\midrule \endfirsthead
\caption[]{Notation for $\pi_{*,*}$} \\
\toprule
element & $(s,w)$ & $\Ext$ & definition \\
\midrule \endhead
\bottomrule \endfoot
\label{tab:notation}
$\tau$ & $(0,-1)$ & $\tau$ \\
$2$ & $(0,0)$ & $h_0$ \\
$\eta$ & $(1,1)$ & $h_1$ \\
$\nu$ & $(3,2)$ & $h_2$ \\
$\sigma$ & $(7,4)$ & $h_3$ \\
$\epsilon$ & $(8,5)$ & $c_0$ \\
$\mu_{8k+1}$ & $(1,1) + k(8,4)$ & $P^k h_1$ \\
$\zeta_{8k+3}$ & $(3,2) + k(8,4)$ & $P^k h_2$ \\
$\kappa$ & $(14,8)$ & $d_0$ \\
$\rho_{15}$ & $(15,8)$ & $h_0^3 h_4$ \\
$\eta_4$ & $(16,9)$ & $h_1 h_4$ & $\eta^3 \cdot \eta_4 = 0$ \\
$\nu_4$ & $(18,10)$ & $h_2 h_4$ & $\nu_4 = \langle 2\sigma, \sigma, \nu \rangle$ \\
$\sigmabar$ & $(19,11)$ & $c_1$ \\
$\kappabar$ & $(20,11)$ & $\tau g$ \\
$\rho_{23}$ & $(23,12)$ & $h_0^2 i$ \\
$\theta_4$ & $(30,16)$ & $h_4^2$ \\
$\rho_{31}$ & $(31,16)$ & $h_0^{10} h_5$ \\
$\eta_5$ & $(32,17)$ & $h_1 h_5$ & $\eta_5 \in \langle \eta, 2, \theta_4 \rangle$,
    $\eta^7 \cdot \eta_5 = 0$ \\
$\theta_{4.5}$ & $(45,24)$ & $h_4^3$ & 
  $4 \theta_{4.5} \in \{ h_0 h_5 d_0 \}$, 
  $\eta \theta_{4.5} \in \{ B_1 \}$ \\
  & & & $\sigma \theta_{4.5} \notin \{ \tau h_1 h_3 g_2 \}$ \\
\end{longtable}


\index{May spectral sequence!E2-page@$E_2$-page|textbf}
\begin{longtable}{llll}
\caption{May $E_2$-page generators} \\
\toprule 
 & $(m,s,f,w)$ & $d_2$ & description \\
\midrule \endfirsthead
\caption[]{May $E_2$-page generators} \\
\toprule 
 & $(m,s,f,w)$ & $d_2$ & description \\
\midrule \endhead
\bottomrule \endfoot
\label{tab:May-E2-gen}
$h_0$ & $(1,0,1,0)$ & & $h_{10}$\\
$h_1$ & $(1,1,1,1)$ & & $h_{11}$ \\
$h_2$ & $(1,3,1,2)$ & & $h_{12}$ \\
$b_{20}$ & $(4,4,2,2)$ & $\tau h_1^3 + h_0^2 h_2$ & $h_{20}^2$ \\
$h_3$ & $(1,7,1,4)$ & & $h_{13}$ \\
$h_0(1)$ & $(4,7,2,4)$ & $h_0 h_2^2$ & $h_{20} h_{21} + h_{11} h_{30}$ \\
$b_{21}$ & $(4,10,2,6)$ & $h_2^3 + h_1^2 h_3$ & $h_{21}^2$ \\
$b_{30}$ & $(6,12,2,6)$ & $\tau h_1 b_{21} + h_3 b_{20}$ & $h_{30}^2$ \\
$h_4$ & $(1,15,1,8)$ & & $h_{14}$ \\
$h_1(1)$ & $(4,16,2,9)$ & $h_1 h_3^2$ & $h_{21} h_{22} + h_{12} h_{31}$ \\
$b_{22}$ & $(4,22,2,12)$ & $h_3^3 + h_2^2 h_4$ & $h_{22}^2$ \\
$b_{31}$ & $(6,26,2,14)$ & $h_4 b_{21} + h_2 b_{22}$ & $h_{31}^2$ \\
$b_{40}$ & $(8,28,2,14)$ & $h_4 b_{30} + \tau h_1 b_{31}$ & $h_{40}^2$ \\
$h_5$ & $(1,31,1,16)$ & & $h_{15}$ \\
$h_2(1)$ & $(4,34,2,18)$ & $h_2 h_4^2$ & $h_{22} h_{23} + h_{13} h_{32}$ \\
$h_0(1,3)$ & $(7,38,3,20)$ & $h_4^2 h_0(1) + h_0 h_2 h_2(1)$ &
  $h_{50} h_{11} h_{13} + h_{40} h_{11} h_{23} +$ \\
  & & & $+ h_{20} h_{41} h_{13} + h_{20} h_{31} h_{23}$ \\
$b_{23}$ & $(4,46,2,24)$ & $h_4^3 + h_3^2 h_5$ & $h_{23}^2$ \\
$h_0(1,2)$ & $(9,46,3,24)$ & $h_3 h_0(1,3)$ & 
   $h_{30} h_{31} h_{32} + h_{30} h_{41} h_{22} +$ \\
   & & & $+ h_{40} h_{21} h_{32} + h_{40} h_{41} h_{12} +$ \\
   & & & $+h_{50} h_{21} h_{22} + h_{50} h_{31} h_{12}$ \\
$b_{32}$ & $(6,54,2,28)$ & $h_5 b_{22} + h_3 b_{23}$ & $h_{32}^2$ \\
$b_{41}$ & $(8,58,2,30)$ & $h_5 b_{31} + h_2 b_{32}$ & $h_{41}^2$ \\
$b_{50}$ & $(10,60,2,30)$ & $h_5 b_{40} + \tau h_1 b_{41}$ & $h_{50}^2$ \\
$h_6$ & $(1,63,1,32)$ & & $h_{16}$ \\
$h_3(1)$ & $(4,70,2,36)$ & $h_3 h_5^2$ & $h_{23} h_{24} + h_{14} h_{33}$ \\   
\end{longtable}


\index{May spectral sequence!E2-page@$E_2$-page|textbf}
\index{May spectral sequence!differential!d2@$d_2$|textbf}
\begin{longtable}{ll}
\caption{May $E_2$-page relations} \\
\toprule
relation & $(m,s,f,w)$ \\
\midrule \endfirsthead
\caption[]{May $E_2$-page relations} \\
\toprule
relation & $(m,s,f,w)$ \\
\midrule \endhead
\bottomrule \endfoot
\label{tab:May-E2-reln}
$h_0 h_1$ & $(2,1,2,1)$ \\
$h_1 h_2$ & $(2,4,2,3)$ \\
$h_2 b_{20} = h_0 h_0(1)$ & $(5,7,3,4)$ \\
$h_2 h_3$ & $(2,10,2,6)$ \\
$h_2 h_0(1) = h_0 b_{21}$ & $(5,10,3,6)$ \\
$h_3 h_0(1)$ & $(5,14,3,8)$ \\
$h_0(1)^2 = b_{20} b_{21} + h_1^2 b_{30}$ & $(8,14,4,8)$ \\
$h_0 h_1(1)$ & $(5,16,3,9)$ \\
\end{longtable}


\index{May spectral sequence!differential!d4@$d_4$|textbf}
\begin{longtable}{llll}
\caption{The May $d_4$ differential} \\
\toprule
& $(m,s,f,w)$ & description & $d_4$ \\
\midrule \endfirsthead
\caption[]{The May $d_4$ differential} \\
\toprule
& $(m,s,f,w)$ & description & $d_4$ \\
\midrule \endhead
\bottomrule \endfoot
\label{tab:May-E4}
$P$ & $(8, 8, 4, 4)$ & $b_{20}^2$ & $h_0^4 h_3$ \\
$\nu$ & $(7, 15, 3, 8)$ & $h_2 b_{30}$ & $h_0^2 h_3^2$ \\
$g$ & $(8, 20, 4, 12)$ & $b_{21}^2$ & $h_1^4 h_4$ \\
$\Delta$ & $(12, 24, 4, 12)$ & $b_{30}^2$ & $\tau^2 h_2 g + P h_4$ \\
$\nu_1$ & $(7, 33, 3, 18)$ & $h_3 b_{31}$ & $h_1^2 h_4^2$ \\
$x_{34}$ & $(7, 34, 5, 18)$ & $h_0^3 h_2(1) + h_0 h_4^2 b_{20}$ & \\
$x_{35}$ & $(10, 35, 4, 18)$ & $h_0 h_3 b_{40}$ & $x_{34}$ \\
$x_{47}$ & $(13, 47, 5, 25)$ & $h_2 b_{40} h_1(1)$ & $\tau h_1^2 g_2$ \\
$x_{49}$ & $(10, 49, 4, 26)$ & $h_2 h_0(1,2)$ & $h_0 h_3 c_2$ \\
$\Delta_1$ & $(12, 52, 4, 28)$ & $b_{31}^2$ & $h_5 g + h_3 g_2$ \\
$\Gamma$ & $(16,56,4,28)$ & $b_{40}^2$ & 
	$\Delta h_5 + \tau^2 \Delta_1 h_2$ \\
$x_{59}$ & $(19,59,7,31)$ & $h_2 b_{30} b_{40} h_1(1)$ & $\tau^2 e_1 g$ \\
$x_{63}$ & $(20,63,8,33)$ & $h_1 b_{20} b_{30} h_0(1,2) +$
	& $\tau h_5 d_0 e_0$ \\
	& & $+ \tau h_1^2 b_{31} b_{40} h_0(1) +$ & \\
	& & $+ \tau b_{20} b_{31}^2 h_0(1)$ & \\
$x_{65}$ & $(20,65,8,34)$ & $h_0 h_3 b_{20} b_{31} b_{40}$ & $h_0^3 A''$ \\
$x_{68}$ & $(16,68,6,36)$ & $h_3^2 b_{31} b_{40} + \tau b_{31}^2 h_1(1)$ &
	 $\tau s_1$ \\
$\nu_2$ & $(7,69,3,36)$ & $h_4 b_{32}$ & $h_2^2 h_5^2$ \\
$x_{69}$ & $(13,69,5,36)$ & $h_3 b_{40} h_2(1)$ & $h_0^2 d_2$ \\
\end{longtable}


\index{May spectral sequence!differential!d6@$d_6$|textbf}
\begin{longtable}{llll}
\caption{The May $d_6$ differential} \\
\toprule
 & $(m,s,f,w)$ & description & $d_6$ \\
\midrule \endfirsthead
\caption[]{The May $d_6$ differential} \\
\toprule
 & $(m,s,f,w)$ & description & $d_6$ \\
\midrule \endhead
\bottomrule \endfoot
\label{tab:May-E6}
$P i$ & $(23,31,11,16)$ & $P^2 \nu$ & $h_0^5 s$ \\
$P r$ & $(22,38,10,20)$ & $P \Delta h_2^2$ & $h_0^6 x$ \\
$Y$ & $(16,45,6,24)$ & $B h_0(1)$ & $h_0^3 g_2$ \\
$\phi$ & $(17,49,7,27)$ & $B h_1 b_{21}$ & $h_1^3 h_5 d_0$ \\
$X$ & $(23,54,9,28)$ & $B h_0 b_{20} b_{30}$ & $P h_0 h_5 d_0$ \\
$P Q'$ & $(37,55,17,28)$ & $P^2 \Delta h_0^2 \nu$ & $h_0^9 X$ \\
$x_{56}$ & $(21,56,9,31)$ & $B h_1 b_{21} h_0(1)$ & $h_1^2 h_5 c_0 d_0$ \\
$x_{56}'$ & $(24,56,10,30)$ & $P B b_{21}$ & $P h_1^2 h_5 d_0$ \\
$\phi'$ & $(23,57,9,30)$ & $B h_0 b_{30} h_0(1)$ & $P h_0 h_5 e_0$ \\
$P x_{56}$ & $(29,64,13,35)$ & $$ & $P h_1^2 h_5 c_0 d_0$ \\
$P x_{56}'$ & $(32,64,14,34)$ & $$ & $P^2 h_1^2 h_5 d_0$ \\
$B_{23}$ & $(24,65,10,36)$ & $Y g$ & $h_1^2 h_5 d_0 e_0$ \\
$P \phi'$ & $(31,65,13,34)$ & $$ & $P^2 h_0 h_5 e_0$ \\
$c_0 g^3$ & $(29,68,15,41)$ & $$ & $h_1^{10} D_4$ \\
$\Delta h_0^2 Y$ & $(30,69,12,36)$ & $$ & 
	$\Delta h_0^5 g_2 + h_0 h_5 d_0 i$ \\
\end{longtable}


\index{May spectral sequence!differential!d8@$d_8$|textbf}
\begin{longtable}{llll}
\caption{The May $d_8$ differential} \\
\toprule
element & $(m,s,f,w)$ & description & $d_8$ \\
\midrule \endfirsthead
\caption[]{The May $d_8$ differential} \\
\toprule
element & $(m,s,f,w)$ & description & $d_8$ \\
\midrule \endhead
\bottomrule \endfoot
\label{tab:May-E8}
$P^2$ & $(16,16,8,8)$ & & $h_0^8 h_4$ \\
$\Delta h_3$ & $(13,31,5,16)$ & & $h_0^4 h_4^2$ \\
$g^2$ & $(16,40,8,24)$ & & $h_1^8 h_5$ \\
$w$ & $(21,45,9,25)$ & $\Delta h_1 g$ & $P h_1^5 h_5$ \\
$\Delta^2$ & $(24,48,8,24)$ & & $P^2 h_5$ \\
$\Delta c_0 g$ & $(25,52,11,29)$ & & $P h_1^4 h_5 c_0$ \\
$Q_3$ & $(13,67,5,36)$ & $\Delta_1 h_4$ & $h_1^4 h_5^2$ \\
$\Gamma h_0 h_3^2$ & $(19,70,7,36)$ & & $h_0^4 p'$ \\
\end{longtable}


\index{May spectral sequence!differential!higher|textbf}
\begin{longtable}{llll}
\caption{Higher May differentials} \\
\toprule 
element & $(m,s,f,w)$ & $d_{r}$ & value \\
\midrule \endfirsthead
\caption[]{Higher May differentials} \\
\toprule 
element & $(m,s,f,w)$ & $d_{r}$ & value \\
\midrule \endhead
\bottomrule \endfoot
\label{tab:May-higher}
$P^2 Q'$ & $(45, 63, 21, 32)$ & $d_{12}$ & $P h_0^{10} h_5 i$ \\
$P^4$ & $(32, 32, 16, 16)$ & $d_{16}$ & $h_0^{16} h_5$ \\
$\Delta^2 h_4$ & $(25, 63, 9, 32)$ & $d_{16}$ & $h_0^8 h_5^2$ \\
$P^8$ & $(64, 64, 32, 32)$ & $d_{32}$ & $h_0^{32} h_6$ \\
\end{longtable}


\newpage

\index{Adams spectral sequence!E2-page@$E_2$-page|textbf}
\index{Adams spectral sequence!differential!d2@$d_2$|textbf}
\begin{longtable}{lllll}
\caption{Adams $E_2$ generators} \\
\toprule
element & $(m,s,f,w)$ & May description & $d_2$ & reference \\
\midrule \endfirsthead
\caption[]{Adams $E_2$ generators} \\
\toprule
element & $(m,s,f,w)$ & May description & $d_2$ & reference \\
\midrule \endhead
\bottomrule \endfoot
\label{tab:Ext-gen}
$h_0$ & $(1, 0, 1, 0)$ & \\
$h_1$ & $(1, 1, 1, 1)$ & \\
$h_2$ & $(1, 3, 1, 2)$ &  \\
$h_3$ & $(1, 7, 1, 4)$ & \\
$c_0$ & $(5, 8, 3, 5)$ & $h_1 h_0(1)$ \\
$P h_1$ & $(9,9,5,5)$ & $$ \\
$P h_2$ & $(9,11,5,6)$ & $$ \\
$d_0$ & $(8, 14, 4, 8)$ & $h_0(1)^2$ \\
\end{longtable}


\index{May spectral sequence!Einfinity-page@$E_\infty$-page|textbf}
\begin{longtable}{lll}
\caption{Temporary May $E_\infty$ generators} \\
\toprule
element & $(m,s,f,w)$ & description \\
\midrule \endfirsthead
\caption[]{Temporary May $E_\infty$ generators} \\
\toprule
element & $(m,s,f,w)$ & description \\
\midrule \endhead
\bottomrule \endfoot
\label{tab:May-Einfty-temp}
$s$ & $(13, 30, 7, 16)$ & $P h_4 h_0(1) + h_0^3 b_{20} b_{31}$ \\
$P^2 s$ & $(29, 46, 15, 24)$ & \\
$S_1$ & $(25,54,11,28)$ & $h_0^2 X$ \\
$g_2'$ & $(23,60,11,32)$ & $P \Delta_1 h_0^3$ \\
$\tau P G$ & $(22,62,10,33)$ & $\tau P \Delta_1 h_1^2$ \\
$P h_5 i$ & $(24,62,12,32)$ & $$ \\
$P^4 s$ & $(45,62,23,32)$ & $$  \\
$P D_4$ & $(22,65,10,35)$ & $$ \\
$s_1$ & $(13,67,7,37)$ & $h_5 b_{21}^2 h_1(1) + h_1^3 b_{21} b_{32}$ \\
$P h_5^2$ & $(10,70,6,36)$ & \\
$\tau P^2 G$ & $(30,70,14,37)$ & $$\\
$P^2 S_1$ & $(41,70,19,36)$ & $$ \\
\end{longtable}


\index{Ext@$\Ext$!ambiguous generator|textbf}
\begin{longtable}{llll}
\caption{Ambiguous $\Ext$ generators} \\
\toprule
element & $(m,s,f,w)$ & ambiguity & definition \\
\midrule \endfirsthead
\caption[]{Ambiguous $\Ext$ generators} \\
\toprule
element & $(m,s,f,w)$ & ambiguity & definition \\
\midrule \endhead
\bottomrule \endfoot
\label{tab:Ext-ambiguous}
$f_0$ & $(8,18,4,10)$ & $\tau h_1^3 h_4$ & \\
$y$ & $(14,38,6,20)$ & $\tau^2 h_2^2 d_1$ \\
$f_1$ & $(8,40,4,22)$ & $h_1^2 h_3 h_5$ \\
$u'$ & $(25,46,11,25)$ & $\tau d_0 l$ & $\tau \cdot u' = 0$ \\
$B_2$ & $(17,48,7,26)$ & $h_0^2 h_5 e_0$ & \\
$v'$ & $(25,49,11,27)$ & $\tau e_0 l$ & $\tau \cdot v' = 0$ \\
$G_3$ & $(21,51,9,28)$ & $\tau g n$ & $h_2 \cdot G_3 = 0$ \\
$R_1$ & $(26,54,10,28)$ & $\tau P h_1 h_5 d_0$ & \\
\end{longtable}


\index{May spectral sequence!hidden extension!tau@$\tau$|textbf}
\begin{longtable}{llll}
\caption{Hidden May $\tau$ extensions} \\
\toprule
$(s,f,w)$ & $x$ & $\tau \cdot x$ & reference \\
\midrule \endfirsthead
\caption[]{Hidden May $\tau$ extensions} \\
\toprule
$(s,f,w)$ & $x$ & $\tau \cdot x$ & reference \\
\midrule \endhead
\bottomrule \endfoot
\label{tab:May-tau}
$(30,11,16)$ & $P c_0 d_0$ & $h_0^4 s$ & classical \\
$(37,9,20)$ & $\tau h_0 e_0 g$ & $h_0^4 x$ & classical \\
$(37,10,20)$ & $\tau h_0^2 e_0 g$ & $h_0^5 x$ & classical \\
$(41,5,22)$ & $h_1 f_1$ & $h_0^2 c_2$ & classical \\
$(43,9,24)$ & $\tau h_2 g^2$ & $P h_1^4 h_5$ & Lemma \ref{lem:t^2-h2g^2} \\
$(46,19,24)$ & $P^3 c_0 d_0$ & $P^2 h_0^4 s$ & classical \\
$(53,9,28)$ & $B_8$ & $P h_5 d_0$ & Lemma \ref{lem:t-B8} \\
$(53,17,28)$ & $\tau P h_0 d_0^2 e_0$ & $h_0^7 x'$ & classical \\
\end{longtable}


\index{May spectral sequence!hidden extension!h0@$h_0$|textbf}
\begin{longtable}{llll}
\caption{Hidden May $h_0$ extensions} \\
\toprule
$(s,f,w)$ & $x$ & $h_0 \cdot x$ & reference \\
\midrule \endfirsthead
\caption[]{Hidden May $h_0$ extensions} \\
\toprule
$(s,f,w)$ & $x$ & $h_0 \cdot x$ & reference \\
\midrule \endhead
\bottomrule \endfoot
\label{tab:May-h0}
$(26,7,16)$ & $h_2^2 g$ & $h_1^3 h_4 c_0$ & Lemma \ref{lem:hidden-h0h2^2g} \\
$(30,7,16)$ & $r$ & $s$ & classical \\
$(46,11,28)$ & $h_2^2 g^2$ & $h_1^7 h_5 c_0$ & Lemma \ref{lem:hidden-h0h2^2g} \\
$(46,12,25)$ & $u'$ & $\tau h_0 d_0 l$ & Lemma \ref{lem:h0-u'} \\
$(46,15,24)$ & $i^2$ & $P^2 s$ & classical \\
$(49,12,27)$ & $v'$ & $\tau h_0 e_0 l$ & Lemma \ref{lem:h0-u'} \\
$(50,11,28)$ & $g r$ & $P h_1^3 h_5 c_0$ & Lemma \ref{lem:h0-gr} \\
$(54,16,29)$ & $P u'$ & $\tau h_0 d_0^2 j$ & $\tau \cdot P u'$ \\
\end{longtable}


\index{May spectral sequence!hidden extension!h1@$h_1$|textbf}
\begin{longtable}{llll}
\caption{Hidden May $h_1$ extensions} \\
\toprule
$(s,f,w)$ & $x$ & $h_1 \cdot x$ & reference \\
\midrule \endfirsthead
\caption[]{Hidden May $h_1$ extensions} \\
\toprule
$(s,f,w)$ & $x$ & $h_1 \cdot x$ & reference \\
\midrule \endhead
\bottomrule \endfoot
\label{tab:May-h1}
$(38,6,21)$ & $x$ & $\tau h_2^2 d_1$ & classical \\
$(39,7,21)$ & $y$ & $\tau^2 c_1 g$ & classical \\
$(56,8,31)$ & $\tau h_1 G$ & $h_5 c_0 e_0$ & Lemma \ref{lem:h1-th1G} \\
$(58,10,33)$ & $h_1^2 B_6$ & $\tau h_2^2 d_1 g$ & Lemma \ref{lem:h1-h1^2B6} \\
$(59,11,33)$ & $h_1 D_{11}$ & $\tau^2 c_1 g^2$ & Lemma \ref{lem:h1-h1D11} \\
$(62,9,34)$ & $h_1 B_3$ & $h_5 d_0 e_0$ & Lemma \ref{lem:h1-th1G} \\
$(62,10,33)$ & $X_1$ & $\tau P G$ & classical \\
$(64,8,35)$ & $C'$ & $\tau d_1^2$ & classical \\
$(64,12,35)$ & $\tau P h_1 G$ & $P h_5 c_0 e_0$ & Lemma \ref{lem:h1-th1G} \\
$(67,7,37)$ & $r_1$ & $s_1$ & Lemma \ref{lem:h1-r1} \\
$(67,13,36)$ & $h_1^2 q_1$ & $h_0^4 X_3$ & Lemma \ref{lem:h1-h1^2q1} \\
$(68,10,38)$ & $C''$ & $d_1 t$ & classical \\
$(70,5,37)$ & $p'$ & $h_5^2 c_0$ & classical \\
$(60,12,39)$ & $h_1^2 X_3$ & $h_5 c_0 d_0 e_0$ & Lemma \ref{lem:h1-th1G} \\
$(70,14,37)$ & $\tau W_1$ & $\tau P^2 G$ & classical \\
\end{longtable}


\index{May spectral sequence!hidden extension!h2@$h_2$|textbf}
\begin{longtable}{llll}
\caption{Hidden May $h_2$ extensions} \\
\toprule
$(s,f,w)$ & $x$ & $h_2 \cdot x$ & reference \\
\midrule \endfirsthead
\caption[]{Hidden May $h_2$ extensions} \\
\toprule
$(s,f,w)$ & $x$ & $h_2 \cdot x$ & reference \\
\midrule \endhead
\bottomrule \endfoot
\label{tab:May-h2}
$(26,7,16)$ & $h_0 h_2 g$ & $h_1^3 h_4 c_0$ & $h_0 \cdot h_2^2 g$ \\
$(46,11,28)$ & $h_0 h_2 g^2$ & $h_1^7 h_5 c_0$ & $h_0 \cdot h_2^2 g^2$ \\
$(49,12,27)$ & $u'$ & $\tau h_0 e_0 l$ & $h_0 \cdot u'$ \\
$(50,11,28)$ & $e_0 r$ & $P h_1^3 h_5 c_0$ & Remark \ref{rem:h2-e0r} \\
$(52,12,29)$ & $v'$ & $\tau h_0 e_0 m$ & $h_0 \cdot v'$ \\
$(54,9,30)$ & $h_2 B_2$ & $h_1 h_5 c_0 d_0$ & Lemma \ref{lem:h2-h2B2} \\
$(57,16,31)$ & $P u'$ & $\tau h_0 d_0 e_0 j$ & $h_0 \cdot P u'$ \\
$(58,8,32)$ & $B_6$ & $\tau e_1 g$ & Lemma \ref{lem:h2-B6} \\
\end{longtable}


\index{May spectral sequence!hidden extension|textbf}
\begin{longtable}{lll}
\caption{Some miscellaneous hidden May extensions} \\
\toprule
$(s,f,w)$ & relation & reference \\
\midrule \endfirsthead
\caption[]{Some miscellaneous hidden May extensions} \\
\toprule
$(s,f,w)$ & relation & reference \\
\midrule \endhead
\bottomrule \endfoot
\label{tab:May-misc}
$(59,12,33)$ & $c_0 \cdot G_3 = P h_1^3 h_5 e_0$ & Lemma \ref{lem:c0-G3} \\
$(60,11,32)$ & $h_0^2 B_4 + \tau h_1 B_{21} = g'_2$ & Lemma \ref{lem:g2'} \\
$(61,10,35)$ & $c_0 \cdot i_1 = h_1^4 D_4$ & Lemma \ref{lem:c0-i1} \\
$(62,12,35)$ & $P h_1 \cdot i_1 = h_1^5 Q_2$ & Lemma \ref{lem:c0-i1} \\
$(63,10,35)$ & $c_0 \cdot B_6 = h_1^3 B_3$ & Lemma \ref{lem:c0-B6} \\
$(65,10,35)$ & $c_0 \cdot Q_2 = P D_4$ & Lemma \ref{lem:c0-i1} \\
\end{longtable}


\begin{landscape}

\index{Massey product|textbf}
\begin{longtable}{llllll}
\caption{Some Massey products in $\Ext$} \\
\toprule
$(s,f,w)$ & bracket & contains & indeterminacy & proof & used for \\
\midrule \endfirsthead
\caption[]{Some Massey products in $\Ext$} \\
\toprule
$(s,f,w)$ & bracket & contains & indeterminacy & proof & used for \\
\midrule \endhead
\bottomrule \endfoot
\label{tab:Massey}
$(2,2,1)$ & $\langle h_0, h_1, h_0 \rangle$ & $\tau h_1^2$ & &
    Proposition \ref{prop:Massey-hn} & $\langle 2, \eta, 2 \rangle$ \\
$(3,2,2)$ & $\langle h_1, h_0, h_1 \rangle$ & $h_0 h_2$ & & 
    Proposition \ref{prop:Massey-hn} & $\langle \eta, 2, \eta \rangle$ \\
$(6,2,4)$ & $\langle h_1, h_2, h_1 \rangle$ & $h_2^2$ & & 
    Proposition \ref{prop:Massey-hn} & $\langle \eta, \nu, \eta \rangle$ \\
$(8,3,5)$ & $\langle h_1, h_2, h_0 h_2 \rangle$ &
    $c_0$ & & $d_2(h_0(1)) = h_0 h_2^2$ & 
    Lemmas \ref{lem:hidden-h0h2^2g}, \ref{lem:h0-gr} \\
    & $= \langle h_1, h_0, h_2^2 \rangle$ \\
$(8,5,4)$ & $\langle h_0, h_0^3 h_3, h_0 \rangle$ & $0$ & & 
    & $\langle 2, 8\sigma, 2 \rangle$ \\
$(9,4,5)$ & $\langle h_0, c_0, h_0 \rangle$ & $\tau h_1 c_0$ & &
    Proposition \ref{prop:Massey-hn} & $\langle 2, \epsilon, 2 \rangle$ \\
$(9,7,7)$ & $\langle h_1^4, \tau, h_1^4 \rangle$ & 
  $0$ & & 
  & Lemma \ref{lem:h1^4-t-brackets} \\
\end{longtable}


\index{Massey product!matric|textbf}
\begin{longtable}{lllll}
\caption{Some matric Massey products in $\Ext$} \\
\toprule
$(s,f,w)$ & bracket & equals & proof & used in \\
\midrule \endfirsthead
\caption[]{Some matric Massey products} \\
\toprule
$(s,f,w)$ & bracket & equals & proof & used in \\
\midrule \endhead
\bottomrule \endfoot
\label{tab:Ctau-matric}
$(25,8,14)$ & 
$\left\langle
\left[
\begin{array}{cc}
h_1^2 & d_0  
\end{array}
\right], 
\left[
\begin{array}{c}
c_0 d_0 \\
h_1^2 c_0 
\end{array}
\right],
\tau
\right\rangle$ & 
  $P e_0$ & 
  \begin{minipage}{1.5in}$d_2 ( b_{20} h_0(1) ) = \tau h_1^2 c_0$ \\
  $d_2 ( b_{20} b_{21} b_{30} ) = \tau c_0 d_0$ \end{minipage} & 
  Lemma \ref{lem:_h1^2e0.c0} \\
\addlinespace
$(28,8,16)$ & 
$\left\langle
\left[
\begin{array}{cc}
c_0 & e_0
\end{array}
\right],
\left[
\begin{array}{c}
h_1^2 e_0 \\
h_1^2 c_0
\end{array}
\right],
\tau
\right\rangle$ & 
$d_0^2$ & 
\begin{minipage}{1.5in}
$d_2 ( h_1 b_{30} h_0(1) ) = \tau h_1^2 e_0$ \\
$d_2 ( b_{20} h_0(1) ) = \tau h_1^2 c_0$
\end{minipage} &
Lemma \ref{lem:_h1^2e0.c0} \\
\addlinespace
$(28,8,16)$ & 
$\left\langle
\left[
\begin{array}{cc}
h_1^2 & e_0 
\end{array}
\right],
\left[
\begin{array}{c}
c_0 e_0 \\
h_1^2 c_0
\end{array}
\right],
\tau
\right\rangle$ &
  $d_0^2$ & Lemma \ref{lem:_c0e0.d0} & Lemma \ref{lem:_c0e0.d0} \\
\addlinespace
$(40,10,22)$ & 
$\left\langle
\left[
\begin{array}{cc}
d_0 & e_0 
\end{array}
\right],
\left[
\begin{array}{c}
c_0 e_0 \\
c_0 d_0 
\end{array}
\right],
\tau
\right\rangle$ & 
  $h_1 u$ & Lemma \ref{lem:_c0e0.d0} & Lemma \ref{lem:_c0e0.d0} \\
\addlinespace
$(56,15,33)$ & 
$\left\langle
\left[
\begin{array}{ccc}
h_1^2 e_0^2 & 
d_0 e_0 g & 
h_1^4
\end{array}
\right],
\left[
\begin{array}{c}
h_1^2 e_0 \\
h_1^4 \\
h_1^3 B_1 
\end{array}
\right],
\tau
\right\rangle$ &  
  $c_0 d_0 e_0^2$ & Lemma \ref{lem:_h1^2e0.h1^2e0^2} & 
  Lemma \ref{lem:_h1^2e0.h1^2e0^2} \\
\addlinespace
$(57,15,31)$ & 
$\left\langle
\left[
\begin{array}{cc}
h_1^2 & d_0 
\end{array}
\right],
\left[
\begin{array}{c}
h_1 d_0 u \\
h_1^3 u
\end{array}
\right],
\tau
\right\rangle$ & 
  $P v'$ & 
  \begin{minipage}{1.8in}
  $d_2(h_1 b_{20} b_{30}^3 h_0(1)^2) = \tau h_1 d_0 u$ \\
  $d_2(h_1 b_{20} b_{30}^2 h_0(1)^2 ) = \tau h_1^3 u$
  \end{minipage} & 
  Lemma \ref{lem:_h1^2e0.c0} \\
\end{longtable}

\end{landscape}


\index{Adams spectral sequence!differential!classical|textbf}
\begin{longtable}{lllll}
\caption{Classical Adams differentials} \\
\toprule 
$(s,f)$ & $d_r$ & $x$ & $d_r(x)$ & reference \\
\midrule \endfirsthead
\caption[]{Classical Adams differentials} \\
\toprule 
$(s,f)$ & $d_r$ & $x$ & $d_r(x)$ & reference \\
\midrule \endhead
\bottomrule \endfoot
\label{tab:diff-refs}
$(30,6)$ & $d_3$ & $r$ & $h_1 d_0^2$ & \cite{MT67}*{Theorem 2.2.2} \\
$(31,8)$ & $d_3$ & $d_0 e_0$ & $P c_0 d_0$ & \cite{MT67}*{Proposition 4.3.1} \\
$(31,8)$ & $d_4$ & $d_0 e_0 + h_0^7 h_5$ & $P^2 d_0$ 
  & \cite{MT67}*{Corollary 4.3.2} \\
$(34,2)$ & $d_3$ & $h_2 h_5$ & $h_0 p$ & \cite{MT67}*{Proposition 3.3.7} \\
$(37,8)$ & $d_4$ & $e_0 g$ & $P d_0^2$ & \cite{MT67}*{Theorem 4.2.1} \\
$(38,2)$ & $d_4$ & $h_3 h_5$ & $h_0 x$ & \cite{MT67}*{Theorem 7.3.7} \\
$(38,4)$ & $d_3$ & $e_1$ & $h_1 t$ & \cite{Bruner84}*{Theorem 4.1} \\
$(38,6)$ & $d_2$ & $y$ & $h_0^3 x$ & \cite{MT67}*{Theorem 5.1.4} \\
\end{longtable}


\begin{landscape}

\index{Toda bracket!classical|textbf}
\begin{longtable}{lllll}
\caption{Some classical Toda brackets} \\
\toprule
$s$ & bracket & contains & indeterminacy & reference \\
\midrule \endfirsthead
\caption[]{Some classical Toda brackets} \\
\toprule
$s$ & bracket & contains & indeterminacy & reference \\
\midrule \endhead
\bottomrule \endfoot
\label{tab:bracket-refs}
$8$ & $\langle 2, \eta, \nu, \eta^2 \rangle$ & $\epsilon = \{ c_0 \}$ & $$ & 
   \cite{Tangora70b}*{Lemma 1.5} \\
$8$ & $\langle \nu, \eta, \nu \rangle$ & $\eta \sigma + \epsilon \in \{ h_1 h_3 \}$ &
   $$ & \cite{Toda62}*{p. 189} \\
$8$ & $\langle \nu^2, 2, \eta \rangle$ & $\epsilon = \{ c_0 \}$ & 
   $\eta \sigma \in \{ h_1 h_3 \}$ & \cite{Toda62}*{p. 189} \\
$9$ & $\langle 8\sigma, 2, \eta \rangle$ & $\mu_9 = \{ P h_1 \}$ & 
   $\eta^2 \sigma \in \{ h_1^2 h_3 \}$ & \cite{Toda62}*{p. 189} \\
   &&& $\eta \epsilon \in \{ h_1 c_0 \}$ &\\
$11$ & $\langle 2\sigma, 8, \nu \rangle$ & $\zeta_{11} \in \{ P h_2 \}$ & $$ & 
   \cite{Toda62}*{p. 189} \\
$15$ & $\langle \epsilon, 2, \nu^2 \rangle$ & $\eta \kappa \in \{ h_1 d_0 \}$ & $$ & 
   \cite{Toda62}*{p. 189} \\
$16$ & $\langle \eta, 2, \kappa \rangle$ & $0$ & $\eta \rho_{15} = \{ P c_0 \}$ & 
   \cite{BJM84}*{Lemma 2.4} \\
\end{longtable}

\end{landscape}


\index{Adams spectral sequence!differential!d3@$d_3$|textbf}
\begin{longtable}{llll}
\caption{Adams $d_3$ differentials} \\
\toprule
element & $(s,f,w)$ & $d_3$ & reference \\
\midrule \endfirsthead
\caption[]{Adams $d_3$ differentials} \\
\toprule
element & $(s,f,w)$ & $d_3$ & reference \\
\midrule \endhead
\bottomrule \endfoot
\label{tab:Adams-d3}
$h_0 h_4$ & $(15,2,8)$ & $h_0 d_0$ & image of $J$ \\
$r$ & $(30,6,16)$ & $\tau h_1 d_0^2$ & Table \ref{tab:diff-refs} \\
$h_0^3 h_5$ & $(31,4,16)$ & $h_0 r$ & image of $J$ \\
$\tau d_0 e_0$ & $(31,8,17)$ & $P c_0 d_0$ & Table \ref{tab:diff-refs} \\
$h_2 h_5$ & $(34,2,18)$ & $\tau h_1 d_1$ & Table \ref{tab:diff-refs} \\
$\tau e_0 g$ & $(37,8,21)$ & $c_0 d_0^2$ & Lemma \ref{lem:d3-te0g} \\
$e_1$ & $(38,4,21)$ & $h_1 t$ & Table \ref{tab:diff-refs} \\
$\tau P d_0 e_0$ & $(39,12,21)$ & $P^2 c_0 d_0$ & Lemma \ref{lem:d3-tPd0e0} \\
\end{longtable}


\index{Adams spectral sequence!differential!d4@$d_4$|textbf}
\begin{longtable}{llll}
\caption{Adams $d_4$ differentials} \\
\toprule
element & $(s,f,w)$ & $d_4$ & reference \\
\midrule \endfirsthead
\caption[]{Adams $d_4$ differentials} \\
\toprule
element & $(s,f,w)$ & $d_4$ & reference \\
\midrule \endhead
\bottomrule \endfoot
\label{tab:Adams-d4}
$\tau^2 d_0 e_0 + h_0^7 h_5$ & $(31,8,16)$ & $P^2 d_0$ & Table \ref{tab:diff-refs} \\
$\tau^2 e_0 g$ & $(37,8,20)$ & $P d_0^2$ & Table \ref{tab:diff-refs} \\
$h_3 h_5$ & $(38,2,20)$ & $h_0 x$ & Table \ref{tab:diff-refs} \\
$\tau^2 P d_0 e_0$ & $(39,12,20)$ & $P^3 d_0$ & Table \ref{tab:diff-refs} \\
$\tau^2 P^2 d_0 e_0$ & $(47,16,24)$ & $P^4 d_0$ & Table \ref{tab:diff-refs} \\
$\tau^2 g r$ & $(50,10,26)$ & $i j$ & $\tmf$ \\
$\tau^3 g m$ & $(55,11,29)$ & $P u' + \tau d_0^2 j$ & $\tmf$ \\
$\tau^2 P^3 d_0 e_0$ & $(55,20,28)$ & $P^5 d_0$ & Table \ref{tab:diff-refs} \\
$\tau^2 d_0 v$ & $(56,13,29)$ & $P^2 u$ & Table \ref{tab:diff-refs} \\
$\tau^2 e_0 g^2$ & $(57,12,32)$ & $d_0^4$ & Lemma \ref{lem:d4-t^2e0g^2} \\
$\tau^2 d_0^2 r$ & $(58,14,30)$ & $P i j$ & $\tmf$ \\
$\tau h_1 X_1$ & $(62,10,32)$ & ? \\
$R$ & $(62,10,32)$ & ? \\
$\tau^2 g v$ & $(62,13,33)$ & $P d_0 u$ & $\tmf$ \\
$C'$ & $(63,7,34)$ & ? \\
$\tau X_2$ & $(63,7,33)$ & ? \\
$\tau^2 h_1 B_{22}$ & $(63,11,33)$ & $P h_1 x'$ & Lemma \ref{lem:d4-h2B5} \\
$\tau^3 d_0^2 m$ & $(63,15,33)$ & $P^2 u' + \tau P d_0^2 j$ & $\tmf$ \\
$h_0^{18} h_6$ & $(63,19,32)$ & $P^2 h_0 i^2$ & image of $J$ \\
$\tau^2 P d_0 v$ & $(64,17,33)$ & $P^3 u$ & $\tmf$ \\
$\tau^2 P d_0^2 r$ & $(66,18,34)$ & $P^2 i j$ & $\tmf$ \\
$\tau h_2 B_5$ & $(69,11,37)$ & $h_1 d_0 x'$ & Lemma \ref{lem:d4-h2B5} \\
$\tau^2 m^2$ & $(70,14,38)$ & $d_0^2 z$ & Lemma \ref{lem:d4-t^2m^2} \\
$\tau^2 e_0 x'$ & $(70,14,36)$ & $P^2 x'$ & Lemma \ref{lem:d4-h2B5} \\
\end{longtable}


\index{Adams spectral sequence!differential!d5@$d_5$|textbf}
\begin{longtable}{llll}
\caption{Adams $d_5$ differentials} \\
\toprule
element & $(s,f,w)$ & $d_5$ & reference \\
\midrule \endfirsthead
\caption[]{Adams $d_5$ differentials} \\
\toprule
element & $(s,f,w)$ & $d_5$ & reference \\
\midrule \endhead
\bottomrule \endfoot
\label{tab:Adams-d5}
$\tau P h_5 e_0$ & $(56,9,29)$ & $\tau d_0 z$ & Lemma \ref{lem:d5-tPh5e0} \\
$A'$ & $(61,6,32)$ & ? \\
$\tau h_1 H_1$ & $(63,6,33)$ & ? \\
$\tau h_1^2 X_1$ & $(63,11,33)$ & ? \\
$h_0^{22} h_6$ & $(63,23,32)$ & $P^6 d_0$ & image of $J$ \\
\end{longtable}


\begin{landscape}

\index{Toda bracket|textbf}
\begin{longtable}{llllll}
\caption{Some Toda brackets} \\
\toprule
$(s,w)$ & bracket & contains & indeterminacy & proof & used in \\
\midrule \endfirsthead
\caption[]{Some Toda brackets} \\
\toprule
$(s,w)$ & bracket & contains & indeterminacy & proof & used in \\
\midrule \endhead
\bottomrule \endfoot
\label{tab:Toda}
$(2,1)$ & $\langle 2, \eta, 2 \rangle$ & $\tau \eta^2 = \{ \tau h_1^2 \}$ & 
     & $\tau h_1^2 = \langle h_0, h_1, h_0 \rangle$ &
     Lemmas \ref{lem:2-h1h5}, \ref{lem:2-e0r},  \\
     & & & & & \ref{lem:eta-ti1} \\
$(3,2)$ & $\langle \eta, 2, \eta \rangle$ & 
     $\{ 2 \nu, 6 \nu \} = \{ h_0 h_2 \}$ & $4 \nu = \{ \tau h_1^3 \}$ &
     $h_0 h_2 = \langle h_1, h_0, h_1 \rangle$ &
     Lemma \ref{lem:sigma-h1h3h5} \\
$(6,4)$ & $\langle \eta, \nu, \eta \rangle$ & $\nu^2 = \{ h_2^2 \}$ & &
     $h_2^2 = \langle h_1, h_2, h_1 \rangle$ &
     Lemmas \ref{lem:eta-h1f1}, \ref{lem:eta-C} \\
$(8,4)$ & $\langle 2, 8 \sigma, 2 \rangle$ & $0$ & &
    $0 = \langle h_0, h_0^3 h_3, h_0 \rangle$ & 
    Lemma \ref{lem:bracket-tau-nu-kappabar} \\
$(8,5)$ & $\langle 2 \nu, \nu, \eta \rangle$ &
   $\epsilon = \{ c_0 \}$ &
   $\eta \sigma \in \{ h_1 h_3 \}$ &
   $c_0 = \langle h_0 h_2, h_2, h_1 \rangle$ &
   Lemmas \ref{lem:t-h1h3g}, \ref{lem:epsilon-kappabar} \\
$(8,5)$ & $\langle \nu, \eta, \nu \rangle$ & 
   $\epsilon + \eta \sigma \in \{ h_1 h_3 \}$ & &
   Table \ref{tab:bracket-refs} & 
   Lemmas \ref{lem:epsilon-q}, \ref{lem:Ctau-eta-h0y} \\
$(9,5)$ & $\langle \eta, 2, 8 \sigma \rangle$ & $\mu_9 = \{ P h_1 \}$ & 
    $\tau \eta^2 \sigma \in \{ \tau h_1^2 h_3\}$ &
    $P h_1 = \langle h_1, h_0, h_0^3 h_3 \rangle$ &
    Lemma \ref{lem:compound-h1^2h3h5} \\
\end{longtable}

\end{landscape}


\index{Adams spectral sequence!hidden extension!classical|textbf}
\begin{longtable}{lllll}
\caption{Classical Adams hidden extensions} \\
\toprule
$(s,f)$ & type & from & to & reference \\
\midrule \endfirsthead
\caption[]{Classical Adams hidden extensions} \\
\toprule
$(s,f)$ & type & from & to & reference \\
\midrule \endhead
\bottomrule \endfoot
\label{tab:extn-refs}
$(20,6)$ & $\nu$ & $h_0^2 g$ & $P h_1 d_0$ & \cite{MT67}*{Theorem 2.1.1} \\
$(21,5)$ & $\eta$ & $h_1 g$ & $P d_0$ & \cite{MT67}*{Theorem 2.1.1} \\
$(21,6)$ & $2$ & $h_0 h_2 g$ & $P h_1 d_0$ & \cite{MT67}*{Theorem 2.1.1} \\
$(30,2)$ & $\nu$ & $h_4^2$ & $p$ & \cite{BMT70}*{Proposition 3.3.5} \\
$(30,2)$ & $\sigma$ & $h_4^2$ & $x$ & \cite{BMT70}*{Proposition 3.5.1} \\
$(32,6)$ & $\nu$ & $q$ & $h_1 e_0^2$ & \cite{BMT70}*{Proposition 3.3.1} \\
$(40,8)$ & $2$ & $g^2$ & $h_1 u$ & \cite{BMT70}*{Proposition 3.4.3} \\
$(40,8)$ & $\eta$ & $g^2$ & $z$ & \cite{BMT70}*{Corollary 3.4.2} \\
$(41,10)$ & $\eta$ & $z$ & $d_0^3$ & \cite{BMT70}*{Proposition 3.4.1} \\
$(45,3)$ & $4$ & $h_4^3$ & $h_0 h_5 d_0$ & \cite{BJM84}*{Theorem 3.3(i)} \\
$(45,3)$ & $\eta$ & $h_4^3$ & $B_1$ & \cite{BJM84}*{Theorem 3.1(i)} \\
$(45,9)$ & $\eta$ & $w$ & $d_0 l$ & \cite{BJM84}*{Theorem 3.1(iv)} \\
$(46,11)$ & $\eta$ & $d_0 l$ & $P u$ & \cite{BJM84}*{Theorem 3.1(ii)} \\
$(47,10)$ & $\eta$ & $e_0 r$ & $d_0 e_0^2$ & \cite{BJM84}*{Theorem 3.1(vi)} \\
\end{longtable}


\index{Adams spectral sequence!hidden extension!tau@$\tau$|textbf}
\begin{longtable}{llll}
\caption{Hidden Adams $\tau$ extensions} \\
\toprule
$(s,f,w)$ & from & to & reference \\
\midrule \endfirsthead
\caption[]{Hidden Adams $\tau$ extensions} \\
\toprule
$(s,f,w)$ & from & to & reference \\
\midrule \endhead
\bottomrule \endfoot
\label{tab:Adams-tau}
$(22,7,13)$ & $c_0 d_0$ & $P d_0$ & cofiber of $\tau$ \\
$(23,8,14)$ & $h_1 c_0 d_0$ & $P h_1 d_0$ & cofiber of $\tau$ \\
$(28,6,17)$ & $h_1 h_3 g$ & $d_0^2$ & Lemma \ref{lem:t-h1h3g} \\
$(29,7,18)$ & $h_1^2 h_3 g$ & $h_1 d_0^2$ & Lemma \ref{lem:t-h1h3g} \\
$(40,9,23)$ & $\tau h_0 g^2$ & $h_1 u$ & Lemma \ref{lem:t-th0g^2} \\
$(41,9,23)$ & $\tau^2 h_1 g^2$ & $z$ & Lemma \ref{lem:t-th1g^2} \\
$(42,11,25)$ & $c_0 e_0^2$ & $d_0^3$ & cofiber of $\tau$ \\
$(43,12,26)$ & $h_1 c_0 e_0^2$ & $h_1 d_0^3$ & cofiber of $\tau$ \\
$(47,12,26)$ & $h_1 u'$ & $P u$ & cofiber of $\tau$ \\
$(48,10,29)$ & $h_1 h_3 g^2$ & $d_0 e_0^2$ & Lemma \ref{lem:t-h1h3g} \\
$(49,11,30)$ & $h_1^2 h_3 g^2$ & $h_1 d_0 e_0^2$ & Lemma \ref{lem:t-h1h3g} \\
$(52,10,29)$ & $h_1 G_3$ & $\tau^2 e_0 m$ & cofiber of $\tau$ \\
$(53,9,29)$ & $B_8$ & $x'$ & cofiber of $\tau$ \\
$(53,11,30)$ & $h_1^2 G_3$ & $d_0 u$ & cofiber of $\tau$ \\
$(54,8,31)$ & $h_1 i_1$ & $?$ \\
$(54,10,30)$ & $h_1 B_8$ & $h_1 x'$ & cofiber of $\tau$ \\
$(54,11,32)$ & $h_1^6 h_5 e_0$ & $\tau e_0^2 g$ & cofiber of $\tau$ \\
$(55,12,33)$ & $h_1^7 h_5 e_0$ & $\tau h_1 e_0^2 g$ & cofiber of $\tau$ \\
$(55,13,31)$ & $\tau^2 h_1 e_0^2 g$ & $d_0 z$ & Lemma \ref{lem:t-th1g^2} \\
$(59,7,33)$ & $j_1$ & $?$ \\
$(59,12,33)$ & $P h_1^3 h_5 e_0$ & $\tau d_0 w$ & Lemma \ref{lem:t-Ph1^3h5e0} \\
\end{longtable}

\index{Adams spectral sequence!hidden extension!tau@$\tau$|textbf}
\index{Adams spectral sequence!hidden extension!tentative|textbf}
\begin{longtable}{llll}
\caption{Tentative hidden Adams $\tau$ extensions} \\
\toprule
$(s,f,w)$ & from & to & reference \\
\midrule \endfirsthead
\caption[]{Tentative hidden Adams $\tau$ extensions} \\
\toprule
$(s,f,w)$ & from & to & reference \\
\midrule \endhead
\bottomrule \endfoot
\label{tab:Adams-tau-tentative}
$(60,13,34)$ & $\tau^2 h_0 g^3$ & $d_0 u' + \tau d_0^2 l$ & cofiber of $\tau$ \\
$(61,13,35)$ & $\tau^2 h_1 g^3$ & $d_0 e_0 r$ & cofiber of $\tau$ \\
$(62,14,37)$ & $h_1^6 h_5 c_0 e_0$ & $d_0^2 e_0^2$ & cofiber of $\tau$ \\
$(63,15,38)$ & $h_0^2 h_2 g^3$ & $h_1 d_0^2 e_0^2$ & cofiber of $\tau$ \\
$(65,9,36)$ & $h_1^2 X_2$ & $\tau B_{23}$ & Lemma \ref{lem:tau-h1^2X2} \\
$(66,10,37)$ & $h_1^3 X_2$ & $\tau h_1 B_{23}$ & Lemma \ref{lem:tau-h1^2X2} \\
$(66,14,37)$ & $h_1^5 X_1$ & $\tau^2 d_0 e_0 m$ & cofiber of $\tau$ \\
$(67,11,38)$ & $h_1^4 X_2$ & $B_8 d_0$ & Lemma \ref{lem:tau-h1^4X2} \\
$(67,13,37)$ & $B_8 d_0$ & $d_0 x'$ & Lemma \ref{lem:tau-B8d0} \\
$(67,15,38)$ & $h_0 e_0 g r$ & $d_0^2 u$ & cofiber of $\tau$ \\
$(68,14,41)$ & $h_1 h_3 g^3$ & $e_0^4$ & cofiber of $\tau$ \\
$(69,15,42)$ & $h_1^2 h_3 g^3$ & $h_1 e_0^4$ & cofiber of $\tau$ \\
\end{longtable}


\index{Adams spectral sequence!hidden extension!two|textbf}

\begin{longtable}{llll}
\caption{Hidden Adams $2$ extensions} \\
\toprule
$(s,f,w)$ & from & to & reference \\
\midrule \endfirsthead
\caption[]{Hidden Adams $2$ extensions} \\
\toprule
$(s,f,w)$ & from & to & reference \\
\midrule \endhead
\bottomrule \endfoot
\label{tab:Adams-2}
$(23,6,14)$ & $h_0 h_2 g$ & $ h_1 c_0 d_0$ & Lemma \ref{lem:2-h0h2g} \\
$(23,6,13)$ & $\tau h_0 h_2 g$ & $ P h_1 d_0$ & Lemma \ref{lem:2-h0h2g} \\
$(40,8,22)$ & $\tau^2 g^2$ & $ h_1 u$ & Table \ref{tab:extn-refs} \\
$(43,10,26)$ & $h_0 h_2 g^2$ & $ h_1 c_0 e_0^2$ & Lemma \ref{lem:2-h0h2g} \\
$(43,10,25)$ & $\tau h_0 h_2 g^2$ & $ h_1 d_0^3$ & Lemma \ref{lem:2-h0h2g} \\
$(47,10,26)$ & $e_0 r$ & $ h_1 u'$ & Lemma \ref{lem:2-e0r} \\
$(47,10,25)$ & $\tau e_0 r$ & $ P u$ & Lemma \ref{lem:2-e0r} \\
$(51,6,28)$ & $h_0 h_3 g_2$ & $?$ \\
$(54,9,28)$ & $h_0 h_5 i$ & $\tau^4 e_0^2 g$ &
   Lemma \ref{lem:2-h0h5i} \\
$(59,7,33)$ & $j_1$ & $?$ \\
\end{longtable}

\index{Adams spectral sequence!hidden extension!two|textbf}
\index{Adams spectral sequence!hidden extension!tentative|textbf}
\begin{longtable}{llll}
\caption{Tentative hidden Adams $2$ extensions} \\
\toprule
$(s,f,w)$ & from & to & reference \\
\midrule \endfirsthead
\caption[]{Tentative hidden Adams $2$ extensions} \\
\toprule
$(s,f,w)$ & from & to & reference \\
\midrule \endhead
\bottomrule \endfoot
\label{tab:Adams-2-tentative}
$(60,12,33)$ & $\tau^3 g^3$ & $ d_0 u' + \tau d_0^2 l$ & Lemma \ref{lem:2-t^3g^3} \\
$(63,14,37)$ & $\tau h_0 h_2 g^3$ & $ h_1 d_0^2 e_0^2$ & Lemma \ref{lem:2-t^3g^3} \\
$(67,14,37)$ & $\tau e_0 g r$ & $ d_0^2 u$ & Lemma \ref{lem:2-t^3g^3} \\
\end{longtable}


\index{Adams spectral sequence!hidden extension!eta@$\eta$|textbf}
\begin{longtable}{llll}
\caption{Hidden Adams $\eta$ extensions} \\
\toprule
$(s,f,w)$ & from & to & reference \\
\midrule \endfirsthead
\caption[]{Hidden Adams $\eta$ extensions} \\
\toprule
$(s,f,w)$ & from & to & reference \\
\midrule \endhead
\bottomrule \endfoot
\label{tab:Adams-eta}
$(15,4,8)$ & $h_0^3 h_4$ & $P c_0$ & image of $J$ \\
$(21,5,12)$ & $\tau h_1 g$ & $c_0 d_0$ & Lemma \ref{lem:eta-th1g} \\
$(21,5,11)$ & $\tau^2 h_1 g$ & $P d_0$ & Lemma \ref{lem:eta-th1g} \\
$(23,9,12)$ & $h_0^2 i$ & $P^2 c_0$ & image of $J$ \\
$(31,11,16)$ & $h_0^{10} h_5$ & $P^3 c_0$ & image of $J$ \\
$(38,4,20)$ & $h_0^2 h_3 h_5$ & $\tau^2 c_1 g$ & Lemma \ref{lem:eta-h0^2h3h5} \\
$(39,17,20)$ & $P^2 h_0^2 i$ & $P^4 c_0$ & image of $J$ \\
$(40,8,21)$ & $\tau^3 g^2$ & $z$ & Table \ref{tab:extn-refs} \\
$(41,5,23)$ & $h_1 f_1$ & $\tau h_2 c_1 g$ & Lemma \ref{lem:eta-h1f1} \\
$(41,9,24)$ & $\tau h_1 g^2$ & $c_0 e_0^2$ & Lemma \ref{lem:eta-th1g} \\
$(41,10,22)$ & $z$ & $\tau d_0^3$ & Lemma \ref{lem:eta-th1g} \\
$(45,3,24)$ & $h_3^2 h_5$ & $B_1$ & Lemma \ref{lem:eta-h3^2h5} \\
$(45,9,24)$ & $\tau w$ & $\tau d_0 l + u'$ & Table \ref{tab:extn-refs} \\
$(46,11,24)$ & $\tau^2 d_0 l$ & $P u$ & Table \ref{tab:extn-refs} \\
$(47,10,26)$ & $e_0 r$ & $\tau d_0 e_0^2$ & Table \ref{tab:extn-refs} \\
$(47,20,24)$ & $h_0^7 Q'$ & $P^5 c_0$ & image of $J$ \\
$(52,11,28)$ & $\tau^2 e_0 m$ & $d_0 u$ & Lemma \ref{lem:eta-t^2e0m} \\
$(54,12,29)$ & $\tau^3 e_0^2 g$ & $d_0 z$ & Lemma \ref{lem:eta-t^3e0^2g} \\
$(55,25,28)$ & $P^4 h_0^2 i$ & $P^6 c_0$ & image of $J$ \\
$(58,8,30)$ & $\tau h_1 Q_2$ & $?$ \\
\end{longtable}

\index{Adams spectral sequence!hidden extension!eta@$\eta$|textbf}
\index{Adams spectral sequence!hidden extension!tentative|textbf}
\begin{longtable}{llll}
\caption{Tentative hidden Adams $\eta$ extensions} \\
\toprule
$(s,f,w)$ & from & to & reference \\
\midrule \endfirsthead
\caption[]{Tentative hidden Adams $\eta$ extensions} \\
\toprule
$(s,f,w)$ & from & to & reference \\
\midrule \endhead
\bottomrule \endfoot
\label{tab:Adams-eta-tentative}
$(59,13,32)$ & $\tau d_0 w$ & $d_0 u' + \tau d_0^2 l$ & Lemma \ref{lem:eta-td0w} \\
$(60,12,33)$ & $\tau^3 g^3$ & $d_0 e_0 r$ & Lemma \ref{lem:eta-td0w} \\
$(61,13,35)$ & $\tau^2 h_1 g^3$ & $d_0^2 e_0^2$ & Lemma \ref{lem:eta-td0w} \\
$(61,14,34)$ & $d_0 e_0 r$ & $\tau^2 d_0^2 e_0^2$ & Lemma \ref{lem:eta-td0w} \\
$(63,26,32)$ & $h_0^{25} h_6$ & $P^7 c_0$ & image of $J$ \\
$(65,13,35)$ & $\tau^2 g w$ & $\tau^2 d_0 e_0 m$ & Lemma \ref{lem:eta-td0w} \\
$(66,11,36)$ & $\tau h_1 B_{23}$ & $B_8 d_0$ & Lemma \ref{lem:eta-th1B23} \\
$(66,11,35)$ & $\tau^2 h_1 B_{23}$ & $d_0 x'$ & Lemma \ref{lem:eta-th1B23} \\
$(66,15,36)$ & $\tau^2 d_0 e_0 m$ & $d_0^2 u$ & Lemma \ref{lem:eta-td0w} \\
\end{longtable}


\index{Adams spectral sequence!hidden extension!nu@$\nu$|textbf}
\begin{longtable}{llll}
\caption{Hidden Adams $\nu$ extensions} \\
\toprule
$(s,f,w)$ & from & to & reference \\
\midrule \endfirsthead
\caption[]{Hidden Adams $\nu$ extensions} \\
\toprule
$(s,f,w)$ & from & to & reference \\
\midrule \endhead
\bottomrule \endfoot
\label{tab:Adams-nu}
$(20,6,12)$ & $h_0^2 g$ & $h_1 c_0 d_0$ & Lemma \ref{lem:nu-h0^2g} \\
$(20,6,11)$ & $\tau h_0^2 g$ & $P h_1 d_0$ & Lemma \ref{lem:nu-h0^2g} \\
$(22,4,13)$ & $h_2 c_1$ & $h_1^2 h_4 c_0$ & Lemma \ref{lem:nu-h2c1} \\
$(26,6,15)$ & $\tau h_2^2 g$ & $h_1 d_0^2$ & Lemma \ref{lem:nu-th2^2g} \\
$(30,2,16)$ & $h_4^2$ & $p$ & Table \ref{tab:extn-refs} \\
$(32,6,17)$ & $q$ & $\tau^2 h_1 e_0^2$ & Table \ref{tab:extn-refs} \\
$(39,9,21)$ & $u$ & $\tau d_0^3$ & Lemma \ref{lem:nu-u} \\
$(40,10,24)$ & $h_0^2 g^2$ & $h_1 c_0 e_0^2$ & Lemma \ref{lem:nu-h0^2g} \\
$(40,10,23)$ & $\tau h_0^2 g^2$ & $h_1 d_0^3$ & Lemma \ref{lem:nu-h0^2g} \\
$(42,8,25)$ & $h_2 c_1 g$ & $h_1^6 h_5 c_0$ & Lemma \ref{lem:nu-h2c1} \\
$(45,3,24)$ & $h_3^2 h_5$ & $B_2$ & Lemma \ref{lem:nu-h3^2h5} \\
$(45,4,24)$ & $h_0 h_3^2 h_5$ & $h_0 B_2$ & Lemma \ref{lem:nu-h3^2h5} \\
$(45,9,24)$ & $\tau w$ & $\tau^2 d_0 e_0^2$ & Lemma \ref{lem:nu-u} \\
$(46,7,25)$ & $N$ & $P h_1^2 h_5 c_0$ & Lemma \ref{lem:nu-h2c1} \\
$(46,10,27)$ & $\tau h_2^2 g^2$ & $h_1 d_0 e_0^2$ & Lemma \ref{lem:nu-th2^2g} \\
$(48,6,26)$ & $h_2 h_5 d_0$ & $?$ \\
$(51,8,28)$ & $h_2 B_2$ & $h_1 B_8$ & Lemma \ref{lem:nu-h2B2} \\
$(51,8,27)$ & $\tau h_2 B_2$ & $h_1 x'$ & Lemma \ref{lem:nu-h2B2} \\
$(52,10,29)$ & $h_1 G_3$ & $\tau^2 h_1 e_0^2 g$ & Lemma \ref{lem:nu-h1G3} \\
$(52,11,28)$ & $\tau^2 e_0 m$ & $d_0 z$ & Lemma \ref{lem:nu-h1G3} \\
$(53,7,30)$ & $i_1$ & $?$ \\
$(54,11,32)$ & $h_1^6 h_5 e_0$ & $h_2 e_0^2 g$ & Lemma \ref{lem:nu-h1^6h5e0} \\
\end{longtable}

\index{Adams spectral sequence!hidden extension!nu@$\nu$|textbf}
\index{Adams spectral sequence!hidden extension!tentative|textbf}
\begin{longtable}{llll}
\caption{Tentative hidden Adams $\nu$ extensions} \\
\toprule
$(s,f,w)$ & from & to & reference \\
\midrule \endfirsthead
\caption[]{Tentative hidden Adams $\nu$ extensions} \\
\toprule
$(s,f,w)$ & from & to & reference \\
\midrule \endhead
\bottomrule \endfoot
\label{tab:Adams-nu-tentative}
$(57,10,30)$ & $h_0 h_2 h_5 i$ & $\tau^2 d_0^2 l$ & Lemma \ref{lem:nu-h0h2h5i} \\
$(59,13,32)$ & $\tau d_0 w$ & $\tau^2 d_0^2 e_0^2$ & Lemma \ref{lem:nu-td0w} \\
$(59,12,33)$ & $P h_1^3 h_5 e_0$ & $\tau d_0^2 e_0^2$ & Lemma \ref{lem:nu-td0w} \\
$(60,14,35)$ & $\tau h_0^2 g^3$ & $h_1 d_0^2 e_0^2$ & Lemma \ref{lem:nu-th0^2g^3} \\
$(62,12,37)$ & $h_2 c_1 g^2$ & $h_1^8 D_4$ & Lemma \ref{lem:nu-h2c1g^2} \\
$(65,13,36)$ & $\tau g w + h_1^4 X_1$ & $\tau^2 e_0^4$ & Lemma \ref{lem:nu-td0w} \\
\end{longtable}


\index{Adams spectral sequence!hidden extension|textbf}
\begin{longtable}{lllll}
\caption{Some miscellaneous hidden Adams extensions} \\
\toprule
$(s,f,w)$ & type & from & to & reference \\
\midrule \endfirsthead
\caption[]{Some miscellaneous hidden Adams extensions} \\
\toprule
$(s,f,w)$ & type & from & to & reference \\
\midrule \endhead
\bottomrule \endfoot
\label{tab:Adams-misc-extn}
$(16,2,9)$ & $\sigma$ & $h_1 h_4$ & $h_4 c_0$ & Lemma \ref{lem:sigma-h1h4} \\
$(20,4,11)$ & $\epsilon$ & $\tau g$ & $d_0^2$ & Lemma \ref{lem:epsilon-kappabar} \\
$(40,8,23)$ & $\epsilon$ & $\tau g^2$ & $d_0 e_0^2$ & 
   Lemma \ref{lem:epsilon-kappabar} \\
$(32,6,17)$ & $\epsilon$ & $q$ & $h_1 u$ & Lemma \ref{lem:epsilon-q} \\
$(45,3,24)$ & $\epsilon$ & $h_3^2 h_5$ & $B_8$ & Lemma \ref{lem:epsilon-h3^2h5} \\
$(30,2,16)$ & $\nu_4$ & $h_4^2$ & $h_2 h_5 d_0$ & Lemma \ref{lem:nu4-h4^2} \\
$(30,2,16)$ & $\eta_4$ & $h_4^2$ & $h_1 h_5 d_0$ & Lemma \ref{lem:eta4-h4^2} \\
$(45,?,24)$ & $\kappa$ & $h_3^2 h_5$ or $h_5 d_0$ & $B_{21}$ &
   Lemma \ref{lem:kappa-h3^2h5} \\
$(45,?,24)$ & $\kappabar$ & $h_3^2 h_5$ or $h_5 d_0$ & $\tau B_{23}$ &
   Lemma \ref{lem:kappabar-h3^2h5} \\
   & & & & (tentative) \\
\end{longtable}


\newpage

\index{Adams spectral sequence!hidden extension!compound|textbf}
\begin{longtable}{lll}
\caption{Some compound hidden Adams extensions} \\
\toprule
$(s,w)$ & relation & reference\\
\midrule \endfirsthead
\caption[]{Some compound hidden Adams extensions} \\
\toprule
$(s,w)$ & relation & reference\\
\midrule \endhead
\bottomrule \endfoot
\label{tab:Adams-compound-extn}
$(9,6)$ & $\nu^3 + \eta^2 \sigma = \eta \epsilon$ & \cite{Toda62} \\
$(40,22)$ & $\nu \{ h_2^2 h_5\} + \eta \sigma \eta_5 = \epsilon \eta_5$ &
   Lemma \ref{lem:compound-h1^2h3h5} \\
\end{longtable}


\index{cofiber of tau@cofiber of $\tau$!hidden extension!h0@$h_0$|textbf}
\begin{longtable}{llll}
\caption{Hidden $h_0$ extensions in $E_2(C\tau)$} \\
\toprule
$(s,f,w)$ & $\ol{x}$ & $h_0 \cdot \ol{x}$ & reference \\
\midrule \endfirsthead
\caption[]{Hidden $h_0$ extensions in $E_2(C\tau)$} \\
\toprule
$(s,f,w)$ & $\ol{x}$ & $h_0 \cdot \ol{x}$ & reference \\
\midrule \endhead
\bottomrule \endfoot
\label{tab:Ctau-h0}
$(11,4,6)$ & $\ol{h_1^2 c_0}$ & $P h_2$ & 
   $d_2 ( b_{20} h_0(1) ) = \tau h_1^2 c_0$ \\ 
$(11,4,6) $ & $\ol{P^k h_1^2 c_0}$ & $P^{k+1} h_2$ &  
   $d_2 ( b_{20}^{2k+1} h_0(1) ) = \tau P^k h_1^2 c_0$ \\
   $+k(8,4,4)$ \\
$(23,6,12)$ & $\ol{c_0 d_0}$ & $i$ & 
   $d_2 ( b_{20} b_{30} h_0(1) ) = \tau c_0 d_0$ \\
$(26,6,14)$ & $\ol{c_0 e_0}$ & $j$ & 
   Lemma \ref{lem:_c0e0.h0} \\
$(26,6,14)$ & $\ol{P^k c_0 e_0}$ & $P^k j$ & 
   Lemma \ref{lem:_c0e0.h0} \\
   $+k(8,4,4)$ \\
$(39,14,20)$ & $\ol{P^2 c_0 d_0}$ & $P^2 i$ & 
   $d_2 ( b_{20}^5 b_{30} h_0(1) ) = \tau P^2 c_0 d_0$ \\
\end{longtable}


\index{cofiber of tau@cofiber of $\tau$!hidden extension!h1@$h_1$|textbf}
\begin{longtable}{llll}
\caption{Hidden $h_1$ extensions in $E_2(C\tau)$} \\
\toprule
$(s,f,w)$ & $\ol{x}$ & $h_1 \cdot \ol{x}$ & reference \\
\midrule \endfirsthead
\caption[]{Hidden $h_1$ extensions in $E_2(C\tau)$} \\
\toprule
$(s,f,w)$ & $\ol{x}$ & $h_1 \cdot \ol{x}$ & reference \\
\midrule \endhead
\bottomrule \endfoot
\label{tab:Ctau-h1}
$(35,5,19)$ & $\ol{h_1^2 d_1}$ & $t$ & 
  $d_2 ( h_1 b_{30} b_{22} ) = \tau h_1^2 d_1$ \\
  & & & $d_2 ( b_{21}^2 h_1(1) ) = h_1^3 d_1$ \\
$(39,7,22)$ & $h_1^3 h_5 \cdot \ol{h_1^4}$ & $\tau g^2$ & 
  $d_8 ( g^2 ) = h_1^8 h_5$ \\
$(41,8,22)$ & $\ol{\tau h_0 g^2}$ & $v$ & 
  $d_4 ( \Delta e_0 ) = \tau^2 h_0 g^2$ \\
$(44,8,23)$ & $\ol{\tau^2 h_2 g^2}$ & $\tau w$ & 
  $d_8 ( w ) = P h_1^5 h_5$ \\
$(46,12,24)$ & $\ol{\tau h_0 d_0^2 e_0}$ & $P u$ & 
  $d_4 ( P \Delta d_0 ) = \tau^2 h_0 d_0^2 e_0$ \\
$(49,12,26)$ & $\ol{\tau h_0 d_0 e_0^2}$ & $P v$ & 
  $d_4 ( \Delta P e_0 ) = \tau^2 h_0 d_0 e_0^2$ \\
$(52,12,28)$ & $\ol{\tau h_0 e_0^3}$ & $d_0 u$ & Lemma \ref{lem:h1-th0e0^3} \\
\end{longtable}


\index{cofiber of tau@cofiber of $\tau$!hidden extension!h2@$h_2$|textbf}
\begin{longtable}{llll}
\caption{Hidden $h_2$ extensions in $E_2(C\tau)$} \\
\toprule
$(s,f,w)$ & $\ol{x}$ & $h_2 \cdot \ol{x}$ & reference \\ 
\midrule \endfirsthead
\caption[]{Hidden $h_2$ extensions in $E_2(C\tau)$} \\
\toprule
$(s,f,w)$ & $\ol{x}$ & $h_2 \cdot \ol{x}$ & reference \\ 
\midrule \endhead
\bottomrule \endfoot
\label{tab:Ctau-h2}
$(28,4,15)$ & $\ol{h_3 g}$ & $n$ & $d_2 ( b_{30} h_1(1)) = \tau h_3 g$ \\
$(42,8,22)$ & $\ol{\tau^2 h_1 g^2}$ & ? & \\
$(43,7,23)$ & $\ol{\tau h_2 c_1 g}$ & $N$ & $d_4 ( \Delta c_1 ) = \tau^2 h_2 c_1 g$ \\
$(44,9,24)$ & $h_2 \cdot \ol{\tau h_0 g^2}$ & $e_0 r$ & 
  $d_4 ( \Delta h_0 g ) = \tau^2 h_0 h_2 g^2$ \\
$(47,5,25)$ & $\ol{h_1^2 g_2}$ & $C$ & 
  $d_4 ( h_2 b_{40} h_1(1) ) = \tau h_1^2 g_2$ \\
$(48,8,27)$ & $\ol{h_3 g^2}$ & $g n$ & 
  $d_2 ( b_{21}^2 b_{30} h_1(1) ) = \tau h_3 g^2$ \\
$(54,8,30)$ & $\ol{h_1 d_1 g}$ & $h_1^2 B_6$ & 
  $d_2 ( b_{21}^2 b_{30} b_{22} + h_2^2 b_{21}^2 b_{40} ) = \tau h_1 d_1 g$ \\
$(58,5,30)$ & $\ol{D_4}$ & $A$ & $d_2 ( b_{30} h_0(1,2) ) = \tau D_4$ \\
\end{longtable}


\index{cofiber of tau@cofiber of $\tau$!hidden extension|textbf}
\begin{longtable}{lll}
\caption{Some miscellaneous hidden extensions in $E_2(C\tau)$} \\
\toprule
$(s,f,w)$ & relation & reference \\ 
\midrule \endfirsthead
\caption[]{Some miscellaneous hidden extensions in $E_2(C\tau)$} \\
\toprule
$(s,f,w)$ & relation & reference \\ 
\midrule \endhead
\bottomrule \endfoot
\label{tab:Ctau-misc-extn}
$(25,8,14)$ &
  $h_1^2 \cdot \ol{c_0 d_0} + d_0 \cdot \ol{h_1^2 c_0} = P e_0$ &
  Lemma \ref{lem:_h1^2e0.c0} \\
$(28,8,16)$ & 
  $c_0 \cdot \ol{h_1^2 e_0} + e_0 \cdot \ol{h_1^2 c_0} = d_0^2$ & 
  Lemma \ref{lem:_h1^2e0.c0} \\
$(28,8,16)$ & 
  $h_1^2 \cdot \ol{c_0 e_0} + e_0 \cdot \ol{h_1^2 c_0} = d_0^2$ &
  Lemma \ref{lem:_c0e0.d0} \\
$(40,10,22)$ & 
  $d_0 \cdot \ol{c_0 e_0} + e_0 \cdot \ol{c_0 d_0} = h_1 u$ &
  Lemma \ref{lem:_c0e0.d0} \\
$(56,15,33)$ & 
  $h_1^2 e_0^2 \cdot \ol{h_1^2 e_0} + d_0 e_0 g \cdot \ol{h_1^4} +
   h_1^6 \cdot \ol{h_1^3 B_1} = c_0 d_0 e_0^2$ & 
  Lemma \ref{lem:_h1^2e0.h1^2e0^2} \\
$(57,15,31)$ & 
  $h_1^2 \cdot \ol{h_1 d_0 u} + d_0 \cdot \ol{h_1^3 u} = P v'$ &
  Lemma \ref{lem:_h1^2e0.c0} \\
$(59,9,32)$ & 
  $h_1^3 \cdot \ol{B_6} + h_2 \cdot \ol{\tau h_2 d_1 g} = h_1^2 Q_2$ &
  Lemma \ref{lem:B6.h1^3} \\
$(65,11,34)$ &
  $P h_1 \cdot \ol{B_6} = h_1 q_1$ & Lemma \ref{lem:hidden-B6-Ph1} \\
\end{longtable}


\index{cofiber of tau@cofiber of $\tau$!Adams spectral sequence!E2-page@$E_2$-page|textbf}
\index{cofiber of tau@cofiber of $\tau$!generator|textbf}
\begin{longtable}{llll}
\caption{$E_2(C\tau)$ generators} \\
\toprule
$(s,f,w)$ & $x$ & $d_2(x)$ & reference \\
\midrule \endfirsthead
\caption[]{$E_2(C\tau)$ generators} \\
\toprule
$(s,f,w)$ & $x$ & $d_2(x)$ & reference \\
\midrule \endhead
\bottomrule \endfoot
\label{tab:Ctau-E2}
$(5,3,3)$ & $\ol{h_1^4}$ \\
$(11,4,6)$ & $\ol{h_1^2 c_0}$ \\
$(5,3,3)$ & $\ol{P^k h_1^4}$ \\
$+k(8,4,4)$ \\
$(11,4,6)$ & $\ol{P^k h_1^2 c_0}$ \\
$+k(8,4,4)$ \\
$(20,5,11)$ & $\ol{h_1^2 e_0}$ & $d_0 \cdot \ol{h_1^4}$ & top cell \\
$(23,6,12)$ & $\ol{c_0 d_0}$ & $P d_0$ &
  $h_0 \cdot \ol{c_0 d_0} = i$ \\
\end{longtable}


\newpage

\index{cofiber of tau@cofiber of $\tau$!ambiguous generator|textbf}
\begin{longtable}{llll}
\caption{Ambiguous $E_2(C\tau)$ generators} \\
\toprule
$(s,f,w)$ & $x$ & ambiguity & definition \\
\midrule \endfirsthead
\caption[]{Ambiguous $E_2(C\tau)$ generators} \\
\toprule
$(s,f,w)$ & $x$ & ambiguity & definition \\
\midrule \endhead
\bottomrule \endfoot
\label{tab:Ctau-ambiguous}
$(55,7,30)$ & $\ol{h_1 i_1}$ & $\tau h_1 G$ & 
  $h_1^5 \cdot \ol{h_1 i_1} = \tau g^3$ \\
 & & $B_6$ & $c_0 \cdot \ol{h_1 i_1} = h_1^5 \cdot \ol{D_4}$ \\
$(56,11,30)$ & $\ol{\tau h_0 g m}$ & $h_2 x'$ & \\
$(59,11,32)$ & $\ol{\tau h_2 g m}$ & $h_0 B_{21}$ & \\
$(62,11,32)$ & $\ol{P h_5 c_0 d_0}$ & $h_0 R$ & $h_0 \cdot \ol{P h_5 c_0 d_0} = 0$ \\
$(68,6,36)$ & $\ol{h_1 r_1}$ & $\tau h_1 Q_3$ & \\
\end{longtable}


\index{cofiber of tau@cofiber of $\tau$!Adams spectral sequence!d3@$d_3$|textbf}
\begin{longtable}{llll}
\caption{Adams $d_3$ differentials for $E_3(C\tau)$} \\
\toprule
$(s,f,w)$ & $x$ & $d_3(x)$ & reference \\
\midrule \endfirsthead
\caption[]{Adams $d_3$ differentials for $E_3(C\tau)$} \\
\toprule
$(s,f,w)$ & $x$ & $d_3(x)$ & reference \\
\midrule \endhead
\bottomrule \endfoot
\label{tab:Ctau-d3}
$(29,5,16)$ & $\ol{h_1 h_3 g}$ & $d_0^2$ &
   Lemma \ref{lem:Ctau-d3-_h1h3g} \\
$(47,17,24)$ & $h_0^4 Q'$ & $P^4 d_0$ & 
   $d_3 (h_0^5 Q')$ in bottom cell \\
$(49,9,28)$ & $\ol{h_1 h_3 g^2}$ & $d_0 e_0^2$ &
   Lemma \ref{lem:Ctau-d3-_h1h3g} \\
$(51,6,27)$ & $\ol{h_1^2 h_5 e_0}$ & $\ol{h_1^3 B_1}$ &
   $h_1^6 h_5 \cdot \ol{h_1^2 e_0} = \tau e_0 g^2$ \\
   & & & $h_1^6 \cdot \ol{h_1^3 B_1} = c_0 d_0 e_0^2$ in $E_3(C\tau)$ \\
   & & & $d_3(\tau e_0 g^2)$ in bottom cell \\
$(53,9,28)$ & $\ol{h_1 G_3}$ & $\ol{\tau h_0 e_0^3}$ &
   Lemma \ref{lem:Ctau-d3-_h1G3}\\
$(54,7,28)$ & $\ol{h_5 c_0 d_0}$ & $x'$ &
    $h_0 \cdot \ol{h_5 c_0 d_0} = h_5 i$ \\
    & & & $d_3(h_5 i )$ in bottom cell \\
$(55,7,30)$ & $\ol{h_1 i_1}$ or & $?$ \\
& $\ol{h_1 i_1} + \tau h_1 G$ & \\
$(56,6,29)$ & $\ol{B_6}$ & $\ol{\tau h_2 g n}$ & top cell \\
$(57,7,30)$ & $\ol{h_5 c_0 e_0}$ & $\ol{h_1^2 B_8}$ & top cell \\
$(59,10,31)$ & $\ol{P h_1^2 h_5 e_0}$ & $x' \cdot \ol{h_1^4}$ & top cell \\
$(61,8,33)$ & $\ol{h_1^3 D_4}$ & $h_1 B_{21}$ &
   Lemma \ref{lem:Ctau-d3-_h1^3D4} \\
$(65,11,34)$ & $\ol{P h_5 c_0 e_0}$ & $h_1 \cdot \ol{ h_1 c_0 x'} + U$ &
  Lemma \ref{lem:Ctau-d3-_Ph5c0e0} \\
$(68,10,37)$ & $\ol{h_1^2 D_4 c_0}$ & $B_8 d_0$ &
   $h_1 \cdot \ol{h_1^2 D_4 c_0} = h_2 B_5$ \\
   & & & $d_3(h_2 B_5)$ in bottom cell \\
\end{longtable}


\index{cofiber of tau@cofiber of $\tau$!top cell|textbf}
\begin{longtable}{lll}
\caption{Projection to the top cell of $C\tau$} \\
\toprule
& element of & element of \\
$(s,f,w)$ & $E_\infty(C\tau)$ & $E_\infty(S^{0,0})$ \\
\midrule \endfirsthead
\caption[]{Projection to the top cell of $C\tau$} \\
\toprule
& element of & element of \\
$(s,f,w)$ & $E_\infty(C\tau)$ & $E_\infty(S^{0,0})$ \\
\midrule \endhead
\bottomrule \endfoot
\label{tab:Ctau-top}
$(30,6,16)$ & $r$ & $h_1 d_0^2$ \\
$(34,2,18)$ & $h_2 h_5$ & $h_1 d_1$ \\
$(38,7,20)$ & $h_0 y$ & $\tau h_2 e_0^2$ \\
$(41,4,22)$ & $h_0 c_2$ & $h_1 h_3 d_1$ \\
$(44,10,24)$ & $d_0 r$ & $h_1 d_0^3$ \\
$(50,10,28)$ & $g r$ & $h_1 d_0 e_0^2$ \\
$(55,7,30)$ & $B_6$ & $h_2 g n$ \\
$(56,10,29)$ & $Q_1$ & $d_0 z$ \\
$(57,7,30)$ & $Q_2$ & $\tau g t$ \\
$(58,7,30)$ & $h_0 D_2$ & $D_{11}$ \\
$(58,11,32)$ & $P h_1^2 h_5 e_0$ & $\tau h_2 e_0^2 g$ \\
$(59,8,33)$ & $h_1^2 D_4$ & $h_2^2 d_1 g$ \\
\end{longtable}


\index{cofiber of tau@cofiber of $\tau$!hidden extension|textbf}
\begin{longtable}{lllll}
\caption{Hidden Adams extensions in $E_\infty(C\tau)$} \\
\toprule
$(s,f,w)$ & type & from & to & reference \\
\midrule \endfirsthead
\caption[]{Hidden Adams extensions in $E_\infty(C\tau)$} \\
\toprule
$(s,f,w)$ & type & from & to & reference \\
\midrule \endhead
\bottomrule \endfoot
\label{tab:Ctau-hidden}
$(35,5,19)$ & $\eta$ & $h_2 h_5$ & $\ol{h_3^2 g}$ & top cell \\
$(39,9,21)$ & $\eta$ & $h_0 y$ & $u$ & Lemma \ref{lem:Ctau-eta-h0y} \\
$(41,9,22)$ & $\nu$ & $h_0 y$ & $\ol{\tau h_0^2 g^2}$ & top cell \\
$(42,6,23)$ & $\eta$ & $h_0 c_2$ & $h_3 \cdot \ol{h_3^2 g}$ & top cell \\
$(47,12,26)$ & $\nu$ & $d_0 r$ & $h_1 u'$ & Lemma \ref{lem:Ctau-nu-d0r} \\
$(48,8,26)$ & $\eta$ & $\ol{h_1^2 g_2}$ & $?$ & \\
$(54,10,30)$ & $2$ & $\ol{h_1 d_1 g}$ & $?$ & \\
$(57,11,30)$ & $2$ & $Q_2$ & $?$ & \\
$(57,11,30)$ & $\nu$ & $h_0 h_5 i$ & $h_1 Q_1$ & bottom cell \\
$(58,10,32)$ & $\nu$ & $\ol{h_1 i_1} + \tau h_1 G$ & $?$ & \\
$(58,10,32)$ & $\nu$ & $B_6$ & $?$ & \\
$(59,10,32)$ & $\nu$ & $\ol{\tau h_2 d_1 g}$ & $?$ & \\
$(59,9,31)$ & $\eta$ & $h_0 D_2$ & $\ol{h_3 G_3}$ & top cell \\
$(59,12,33)$ & $2$ & $h_1^2 D_4$ & $?$ & \\
$(60,10,32)$ & $\nu$ & $Q_2$ & $h_1 \cdot \ol{h_3 G_3}$ & top cell \\
$(60,10,32)$ & $2$ & $\ol{j_1}$ & $?$ & \\
\end{longtable}


\index{Adams-Novikov spectral sequence!hidden extension!two|textbf}
%

\index{Adams-Novikov spectral sequence!hidden extension!eta@$\eta$|textbf}
\begin{longtable}{lll}
\caption{Hidden Adams-Novikov $\eta$ extensions} \\
\toprule
$(s,f,w)$ & from & to \\
\midrule \endfirsthead
\caption[]{Hidden Adams-Novikov $\eta$ extensions} \\
\toprule
$(s,f,w)$ & from & to \\
\midrule \endhead
\bottomrule \endfoot
\label{tab:ANSS-hidden-eta}
$(37,3,20)$  & $\alpha_{4/4} z_{30}$ & $\tau \alpha_{2/2}^2 z_{32,4}$ \\
$(38,2,20)$  & $z_{38}$ & $\tau^2 z_{39,7}$ \\
$(39,3,21)$  & $\alpha_1 z'_{38}$ & $\tau^2 z_{40,8}$ \\
$(41,5,23)$  & $\alpha_1 z_{40,4}$ & $\tau \alpha_{2/2} z_{39,7}$ \\
$(47,5,26)$  & $2 z_{47,5}$ & $\tau^2 z_{48}$ \\
$(58,6,32)$ & $\alpha_1^2 z_{56,4}$ & $\tau^2 z_{59,11}$ \\
\end{longtable}

\index{Adams-Novikov spectral sequence!hidden extension!nu@$\nu$|textbf}
%


\index{Adams spectral sequence!comparison to Adams-Novikov spectral sequence|textbf}
\index{Adams-Novikov spectral sequence!comparison to Adams spectral sequence|textbf}
\index{comparison between Adams and Adams-Novikov spectral sequence|textbf}
\begin{longtable}{llll}
\caption{Correspondence between classical Adams and Adams-Novikov $E_\infty$} \\
\toprule
$s$ & Adams & Adams-Novikov & detects \\
\midrule \endfirsthead
\caption[]{Correspondence between classical Adams and Adams-Novikov $E_\infty$} \\
\toprule
$s$ & Adams & Adams-Novikov & detects \\
\midrule \endhead
\bottomrule \endfoot
\label{tab:Adams-ANSS}
$0$ & $h_0^k$ & $2^k$ & $2^k$ \\
$1$ & $h_1$ & $\alpha_1$ & $\eta$ \\
$2$ & $h_1^2$ & $\alpha_1^2$ & $\eta^2$ \\
$3$ & $h_2$ & $\alpha_{2/2}$ & $\nu$ \\
$3$ & $h_0 h_2$ & $2\alpha_{2/2}$ & $2\nu$ \\
$3$ & $h_0^2 h_2$ & $\alpha_1^3$ & $4\nu$ \\
$6$ & $h_2^2$ & $\alpha_{2/2}^2$ & $\nu^2$ \\
$7$ & $h_0^k h_3$ & $2^k \alpha_{4/4}$ & $\sigma$ \\
$8$ & $h_1 h_3$ & $\alpha_1 \alpha_{4/4}$ & $\eta \sigma$ \\
$8$ & $c_0$ & $z_8 + \alpha_1 \alpha_{4/4}$ & $\epsilon$ \\
$9$ & $h_1 c_0$ & $\alpha_1 z_8 + \alpha_1^2 \alpha_{4/4}$ & $\eta \epsilon$ \\
$9$ & $h_1^2 h_3$ & $\alpha_1^2 \alpha_{4/4}$ & $\eta^2 \sigma$ \\
$8k+1$ & $P^k h_1$ & $\alpha_{4k+1}$ & $\mu_{8k+1}$ \\
$8k+2$ & $P^k h_1^2$ & $\alpha_1 \alpha_{4k+1}$ & $\eta \mu_{8k+1}$ \\
$8k+3$ & $P^k h_2$ & $2\alpha_{4k+2/3}$ & $\zeta_{8k+3}$ \\
$8k+3$ & $P^k h_0 h_2$ & $4 \alpha_{4k+2/3}$ & $2\zeta_{8k+3}$ \\
$8k+3$ & $P^k h_0^2 h_2$ & $\alpha_1^2 \alpha_{4k+1}$ & $4\zeta_{8k+3}$ \\
$14$ & $h_3^2$ & $\alpha_{4/4}^2$ & $\sigma^2$ \\
$14$ & $d_0$ & $z_{14}$ & $\kappa$ \\
$15$ & $h_0^{k+3} h_4$ & $2^k \alpha_{8/5}$ & $2^k \rho_{15}$ \\
$15$ & $h_1 d_0$ & $\alpha_1 z_{14}$ & $\eta \kappa$ \\
$16$ & $h_1 h_4$ & $z_{16}$ & $\eta_4$ \\
$8k+8$ & $P^k c_0$ & $\alpha_1 \alpha_{4k+4/b}$ & $\eta \rho_{8k+7}$ \\
$17$ & $h_1^2 h_4$ & $\alpha_1 z_{16}$ & $\eta \eta_4$ \\
$17$ & $h_2 d_0$ & $\alpha_{2/2} z_{14}$ & $\nu \kappa$ \\
$8k+9$ & $P^k h_1 c_0$ & $\alpha_1^2 \alpha_{4k+4/b}$ & $\eta^2 \rho_{8k+7}$ \\
$18$ & $h_2 h_4$ & $z_{18}$ & $\nu_4$ \\
$18$ & $h_0 h_2 h_4$ & $2z_{18}$ & $2 \nu_4$ \\
$18$ & $h_1^3 h_4$ & $\alpha_1^2 z_{16}$ & $4 \nu_4$ \\
$19$ & $c_1$ & $z_{19}$ & $\sigmabar$ \\
$20$ & $g$ & $z_{20,2}$ & $\kappabar$ \\
$20$ & $h_0 g$ & $z_{20,4}$ & $2\kappabar$ \\
$20$ & $h_0^2 g$ & $2z_{20,4}$ & $4\kappabar$ \\
$21$ & $h_2^2 h_4$ & $\alpha_{2/2} z_{18}$ & $\nu \nu_4$ \\
$21$ & $h_1 g$ & $\alpha_1 z_{20,2}$ & $\eta \kappabar$ \\
$22$ & $h_2 c_1$ & $\alpha_{2/2} z_{19}$ & $\nu \sigmabar$ \\
$22$ & $P d_0$ & $\alpha_1^2 z_{20,2}$ & $\eta^2 \kappabar$ \\
$23$ & $h_4 c_0$ & $\alpha_{4/4} z_{16}$ & $\sigma \eta_4$ \\
$23$ & $h_2 g$ & $z_{23}$ & $\nu \kappabar$ \\
$23$ & $h_0 h_2 g$ & $2 z_{23}$ & $2 \nu \kappabar$ \\
$23$ & $P h_1 d_0$ & $4 z_{23}$ & $4 \nu \kappabar$ \\
$23$ & $h_0^{k+2} i$ & $2^k \alpha_{12/4}$ & $2^k \rho_{23}$ \\
$24$ & $h_1 h_4 c_0$ & $\alpha_1 \alpha_{4/4} z_{16}$ & $\eta \sigma \eta_4$ \\
$26$ & $h_2^2 g$ & $\alpha_{2/2} z_{23}$ & $\nu^2 \kappabar$ \\
$28$ & $d_0^2$ & $z_{28}$ & $\kappa^2$ \\
$30$ & $h_4^2$ & $z_{30}$ & $\theta_4$ \\
$31$ & $h_1 h_4^2$ & $\alpha_1 z_{30}$ & $\eta \theta_4$ \\
$31$ & $n$ & $z_{31}$ & $$ \\
$31$ & $h_0^{k+10} h_5$ & $2^k \alpha_{16/6}$ & $2^k \rho_{31}$ \\
$32$ & $h_1 h_5$ & $z'_{32,2}$ & $\eta_5$ \\
$32$ & $d_1$ & $z_{32,4}$ & $$ \\
$32$ & $q$ & $z_{32,2}$ & $$ \\
$33$ & $h_1^2 h_5$ & $\alpha_1 z'_{32,2}$ & $\eta \eta_5$ \\
$33$ & $p$ & $\alpha_2 z_{30}$ & $\nu \theta_4$ \\
$33$ & $h_1 q$ & $\alpha_1 z_{32,2}$ & $$ \\
$34$ & $h_0 h_2 h_5$ & $2 z_{34,2}$ & $$ \\
$34$ & $h_1^3 h_5$ & $\alpha_1^2 z'_{32,2}$ & $\eta^2 \eta_5$ \\
$34$ & $h_2 n$ & $\alpha_{2/2} z_{31}$ & $$ \\
$34$ & $e_0^2$ & $z_{34,6}$ & $\kappa \kappabar$ \\
$35$ & $h_2 d_1$ & $\alpha_{2/2} z_{32,4}$ & $$ \\
$35$ & $h_1 e_0^2$ & $\alpha_1 z_{34,6}$ & $\eta \kappa \kappabar$ \\
$36$ & $t$ & $\alpha_1^2 z_{34,2}$ & $$ \\
$37$ & $h_2^2 h_5$ & $\alpha_{2/2} z_{34,2}$ & $$ \\
$37$ & $x$ & $\alpha_{4/4} z_{30}$ & $\sigma \theta_4$ \\
$38$ & $h_0^2 h_3 h_5$ & $z_{38}$ & $$ \\
$38$ & $h_0^3 h_3 h_5$ & $2z_{38}$ & $$ \\
$38$ & $h_2^2 d_1$ & $\alpha_{2/2}^2 z_{32,4}$ & $$ \\
$39$ & $h_1 h_3 h_5$ & $z_{39,3}$ & $$ \\
$39$ & $h_5 c_0$ & $z'_{39,3}$ & $$ \\
$39$ & $h_3 d_1$ & $\alpha_{4/4} z_{32,4}$ & $$ \\
$39$ & $h_2 t$ & $z_{39,7}$ & $$ \\
$39$ & $u$ & $\alpha_1 z'_{38}$ & $$ \\
$39$ & $P^2 h_0^{k+2} i$ & $2^k \alpha_{20/4}$ & $2^k \rho_{39}$ \\
$40$ & $h_1^2 h_3 h_5$ & $\alpha_1 z_{39,3}$ & $$ \\
$40$ & $f_1$ & $z_{40,4}$ & $$ \\
$40$ & $h_1 h_5 c_0$ & $\alpha_1 z'_{39,3}$ & $$ \\
$40$ & $P h_1 h_5$ & $z_{40,2}$ & $$ \\
$40$ & $g^2$ & $\alpha_1^4 z_{36}$ & $\kappabar^2$ \\
$40$ & $h_1 u$ & $z_{40,8}$ & $2 \kappabar^2$ \\
$41$ & $h_1 f_1$ & $\alpha_1 z_{40,4}$ & $$ \\
$41$ & $P h_1^2 h_5$ & $\alpha_1 z_{40,2}$ & $$ \\
$41$ & $z$ & $\alpha_1^5 z_{36}$ & $\eta \kappabar^2$ \\
$42$ & $P h_2 h_5$ & $2 z_{42}$ & $$ \\
$42$ & $P h_0 h_2 h_5$ & $4 z_{42}$ & $$ \\
$42$ & $P h_1^3 h_5$ & $\alpha_1^2 z_{40,2}$ & $$ \\
$42$ & $d_0^3$ & $\alpha_1^6 z_{36}$ & $\eta^2 \kappabar^2$ \\
$44$ & $g_2$ & $z_{44,4}$ & $$ \\
$44$ & $h_0 g_2$ & $2z_{44,4}$ & $$ \\
$44$ & $h_0^2 g_2$ & $4z_{44,4}$ & $$ \\
$45$ & $h_3^2 h_5$ & $z'_{45}$ & $\theta_{4.5}$ \\
$45$ & $h_0 h_3^2 h_5$ & $2z'_{45}$ & $2 \theta_{4.5}$ \\
$45$ & $h_5 d_0$ & $z_{45} + 2z'_{45}$ & $$ \\
$45$ & $h_1 g_2$ & $\alpha_1 z_{44,4}$ & $$ \\
$45$ & $h_0 h_5 d_0$ & $4z'_{45}$ & $4 \theta_{4.5}$ \\
$45$ & $h_0^2 h_5 d_0$ & $8z'_{45}$ & $8 \theta_{4.5}$ \\
$45$ & $w$ & $\alpha_1 z_{44,2}$ & $$ \\
$46$ & $h_1 h_5 d_0$ & $\alpha_1 z_{45}$ & $$ \\
$46$ & $B_1$ & $\alpha_1 z'_{45}$ & $\eta \theta_{4.5}$ \\
$46$ & $N$ & $\alpha_{2/2} z_{43,3}$ & $$ \\
$46$ & $d_0 l$ & $\alpha_1^2 z_{44,2}$ & $$ \\
$47$ & $h_2 g_2$ & $\alpha_{2/2} z_{44,4}$ & $$ \\
$47$ & $P h_5 c_0$ & $\alpha_{8/5} z'_{32,2}$ & $$ \\
$47$ & $h_1 B_1$ & $\alpha_1^2 z'_{45}$ & $\eta^2 \theta_{4.5}$ \\
$47$ & $e_0 r$ & $2 z_{47,5}$ & $$ \\
$47$ & $P u$ & $4z_{47,5}$ & $$ \\
$47$ & $h_0^{k+7} Q'$ & $2^k \alpha_{24/5}$ & $2^k \rho_{47}$ \\
$48$ & $h_2 h_5 d_0$ & $\alpha_{2/2} z_{45}$ & $$ \\
$48$ & $B_2$ & $\alpha_{2/2} z'_{45}$ & $\nu \theta_{4.5}$ \\
$48$ & $h_0 B_2$ & $2\alpha_{2/2} z'_{45}$ & $2\nu \theta_{4.5}$ \\
$48$ & $P h_1 h_5 c_0$ & $\alpha_1 \alpha_{8/5} z'_{32,2}$ & $$ \\
$48$ & $d_0 e_0^2$ & $z_{48}$ & $\kappa^2 \kappabar$ \\
$50$ & $h_5 c_1$ & $z'_{50}$ & $$ \\
$50$ & $C$ & $z_{50}$ & $$ \\
$51$ & $h_3 g_2$ & $\alpha_{4/4} z_{44,4}$ & $$ \\
$51$ & $h_0 h_3 g_2$ & $2 \alpha_{4/4} z_{44,4}$ & $$ \\
$51$ & $h_2 B_2$ & $\alpha_{2/2}^2 z'_{45}$ & $\nu^2 \theta_{4.5}$ \\
$51$ & $g n$ & $z_{51}$ & $$ \\
$52$ & $h_1 h_3 g_2$ & $\alpha_1 \alpha_{4/4} z_{44,4}$ & $$ \\
$52$ & $d_1 g$ & $z_{52,8}$ & $$ \\
$52$ & $e_0 m$ & $z_{52,6}$ & $$ \\
$53$ & $h_2 h_5 c_1$ & $\alpha_{2/2} z'_{50}$ & $$ \\
$53$ & $h_2 C$ & $z_8 z'_{45}$ or $z_{53}$ & $$ \\
$53$ & $x'$ & $z_8 z'_{45}$ or $z_{53}$ & $\epsilon \theta_{4.5}$ \\
$53$ & $d_0 u$ & $\alpha_1 z_{52,6}$ & $$ \\
$54$ & $h_0 h_5 i$ & $z_{54,2}$ & $$ \\
$54$ & $h_1 x'$ & $\alpha_1 z_8 z'_{45}$ or $\alpha_1 z_{53,5} $ & $$ \\
$54$ & $e_0^2 g$ & $z_{54,10}$ & $\kappa \kappabar^2$ \\
$55$ & $P^4 h_0^{k+2} i$ & $2^k \alpha_{28/4}$ & $2^k \rho_{55}$ \\
$57$ & $h_0 h_2 h_5 i$ & $\alpha_{2/2} z_{54,2}$ & $$ \\
$58$ & $h_1 Q_2$ & $\alpha_1 z_{57}$ & $$ \\
$59$ & $B_{21}$ & $z_{59,5}$ or $z'_{59,7}$ & $\kappa \theta_{4.5}$ \\
$59$ & $d_0 w$ & $z_{59,7}$ & $$ \\
\end{longtable}


\index{Adams-Novikov spectral sequence!boundary}
\begin{longtable}{lll}
\caption{Classical Adams-Novikov boundaries} \\
\toprule
$(s,f)$ & boundary & $\pi_{*,*}(C\tau)$ \\
\midrule \endfirsthead
\caption[]{Classical Adams-Novikov boundaries} \\
\toprule
$(s,f)$ & boundary & $\pi_{*,*}(C\tau)$ \\
\midrule \endhead
\bottomrule \endfoot
\label{tab:ANSS-boundary}
$(4,4) + k(1,1)$ & $\alpha_1^{k+4}$ & $h_1^{k+4}$ \\
$(10,4) + k(1,1)$ & $\alpha_1^{k+3} \alpha_{4/4}$ & $h_1^{k+2} c_0$ \\
$(12,4) + k(1,1)$ & $\alpha_1^{k+3} \alpha_5$ & $P h_1^{k+4}$ \\
$(18,4) + k(1,1)$ & $\alpha_1^{k+3} \alpha_{8/5}$ & $h_1^{k+2} P c_0$ \\
$(20,4) + k(1,1)$ & $\alpha_1^{k+3} \alpha_9$ & $P^2 h_1^{k+4}$ \\
$(26,4) + k(1,1)$ & $\alpha_1^{k+3} \alpha_{12/4}$ & $h_1^{k+2} P^2 c_0$ \\
$(28,4) + k(1,1)$ & $\alpha_1^{k+3} \alpha_{13}$ & $P^3 h_1^{k+4}$ \\
$(34,4) + k(1,1)$ & $\alpha_1^{k+3} \alpha_{16/6}$ & $h_1^{k+2} P^3 c_0$ \\
$(36,4) + k(1,1)$ & $\alpha_1^{k+3} \alpha_{17}$ & $P^4 h_1^{k+4}$ \\
$(42,4) + k(1,1)$ & $\alpha_1^{k+3} \alpha_{20/4}$ & $h_1^{k+2} P^4 c_0$ \\
$(44,4) + k(1,1)$ & $\alpha_1^{k+3} \alpha_{21}$ & $P^5 h_1^{k+4}$ \\
$(50,4) + k(1,1)$ & $\alpha_1^{k+3} \alpha_{24/5}$ & $h_1^{k+2} P^5 c_0$ \\
$(54,4) + k(1,1)$ & $\alpha_1^{k+3} \alpha_{25}$ & $P^6 h_1^{k+4}$ \\
$(58,4) + k(1,1)$ & $\alpha_1^{k+3} \alpha_{28/4}$ & $h_1^{k+2} P^6 c_0$ \\
\midrule
$(25, 5)$ & $\alpha_1^2 \alpha_{4/4} z_{16}$ & $h_1^2 h_4 c_0$ \\
$(29, 7)$ & $\alpha_1 z_{28}$ & $h_1^2 h_3 g$ \\
$(33, 5)$ & $\alpha_1 z_{32,4}$ & $h_1 d_1$ \\
$(34, 6)$ & $\alpha_1^2 z_{32,4}$ & $h_1^2 d_1$ \\
$(35, 5)$ & $\alpha_1^3 z'_{32,2}$ & $h_1^4 h_5$ \\
$(36, 6)$ & $\alpha_1^4 z'_{32,2}$ & $h_1^5 h_5$ \\
$(37, 7)$ & $\alpha_1^5 z'_{32,2}$ & $h_1^6 h_5$ \\
$(37, 7)$ & $\alpha_{2/2} z_{34,6}$ & $h_2 e_0^2$ \\
$(38, 8)$ & $\alpha_1^6 z'_{32,2}$ & $h_1^7 h_5$ \\
$(40, 6)$ & $\alpha_1 \alpha_{4/4} z_{32,4}$ & $h_1 h_3 d_1$ \\
$(40, 8)$ & $2 z_{40,8}$ & $h_0^2 g^2$ \\
$(41, 5)$ & $\alpha_1^2 z'_{39,3}$ & $h_1^2 h_5 c_0$ \\
$(41, 7)$ & $\alpha_1^2 \alpha_{4/4} z_{32,4}$ & $h_1^2 h_3 d_1$ \\
$(42, 6)$ & $\alpha_1^3 z'_{39,3}$ & $h_1^3 h_5 c_0$ \\
$(42, 8)$ & $\alpha_{2/2} z_{39,7}$ & $h_2 c_1 g$ \\
$(43, 7)$ & $\alpha_1^4 z'_{39,3}$ & $h_1^4 h_5 c_0$ \\
$(43, 9)$ & $z_{43,9}$ & $h_2 g^2$ \\
$(43, 9)$ & $2 z_{43,9}$ & $h_0 h_2 g^2$ \\
$(43, 9)$ & $4 z_{43,9}$ & $h_1 c_0 e_0^2$ \\
$(44, 8)$ & $\alpha_1^5 z'_{39,3}$ & $h_1^5 h_5 c_0$ \\
$(45, 9)$ & $\alpha_1^6 z'_{39,3}$ & $h_1^6 h_5 c_0$ \\
$(46, 6)$ & $\alpha_1^2 z_{44,4}$ & $h_1^2 g_2$ \\
$(46, 10)$ & $\alpha_{2/2} z_{43,9}$ & $h_2^2 g^2$ \\
$(49, 5)$ & $\alpha_1^2 \alpha_{8/5} z'_{32,2}$ & $P h_1^2 h_5 c_0$ \\
$(49, 11)$ & $\alpha_1 z_{48}$ & $h_1^2 h_3 g^2$ \\
$(53, 9)$ & $\alpha_1 z_{52,8}$ & $h_1 d_1 g$ \\
$(54, 8)$ & ? $\alpha_1 z_{53}$ & $h_1 i_1$ \\
$(54, 8)$ & $\alpha_{2/2} z_{51}$ & $h_2 g n$ \\
$(54, 10)$ & $\alpha_1^2 z_{52,8}$ & $h_1^2 d_1 g$ \\
$(55, 9)$ & $\alpha_{2/2} z_{52,8}$ & $h_2 d_1 g$ \\
$(55, 9)$ & $\alpha_1^2 z_{53}$ & $h_1^2 i_1$ \\
$(55, 11)$ & $\alpha_1 z_{54,10}$ & $h_1^7 h_5 e_0$ \\
$(56, 8)$ & $\alpha_1^2 z_{54,6}$ & $g t$ \\
$(56, 10)$ & $\alpha_1^3 z_{53}$ & $h_1^3 i_1$ \\
$(57, 5)$ & $\alpha_1 z_{56,4}$ & $D_{11}$ \\
$(57, 11)$ & $\alpha_1^4 z_{53}$ & $h_1^4 i_1$ \\
$(57, 11)$ & $\alpha_{2/2} z_{54,10}$ & $h_2 d_0 g^2$ \\
$(58, 6)$ & $\alpha_{4/4}^2 z_{44,4}$ & $h_3^2 g_2$ \\
$(58, 6)$ & $\alpha_1^2 z_{56,4}$ & $h_1 D_{11}$ \\
$(58, 10)$ & $\alpha_{2/2}^2 z_{52,8}$ & $h_2^2 d_1 g$ \\
$(58, 12)$ & $\alpha_1^5 z_{53}$ & $h_1^5 i_1$ \\
$(59, 5)$ & $\alpha_1^2 z_{57}$ & $h_1^2 Q_2$ \\
$(59, 7)$ & ? $z'_{59,7}$ & $j_1$ \\
$(59, 9)$ & $\alpha_{4/4} z_{52,8}$ & $h_3 d_1 g$ \\
$(59, 11)$ & $z_{59,11}$ & $c_1 g^2$ \\
\end{longtable}


\index{Adams-Novikov spectral sequence!non-permanent class}
\begin{longtable}{lll}
\caption{Classical Adams-Novikov non-permanent classes} \\
\toprule
$(s,f)$ & class & $\pi_{*,*}(C\tau)$ \\
\midrule \endfirsthead
\caption[]{Classical Adams-Novikov non-permanent classes} \\
\toprule
$(s,f)$ & class & $\pi_{*,*}(C\tau)$ \\
\midrule \endhead
\bottomrule \endfoot
\label{tab:ANSS-non-permanent}
$(5,1) + k(1,1)$ & $\alpha_1^k \alpha_3$ & $h_1^k \cdot \ol{h_1^4}$  \\
$(11,1) + k(1,1)$ & $\alpha_1^k \alpha_{6/3}$ & $h_1^k \cdot \ol{h_1^2 c_0}$  \\
$(13,1) + k(1,1)$ & $\alpha_1^k \alpha_7$ & $h_1^k \cdot \ol{P h_1^4}k$  \\
$(19,1) + k(1,1)$ & $\alpha_1^k \alpha_{10/3}$ & $h_1^k \cdot \ol{P h_1^2 c_0}$  \\
$(21,1) + k(1,1)$ & $\alpha_1^k \alpha_{11}$ & $h_1^k \cdot \ol{P^2 h_1^4}$ \\
$(27,1) + k(1,1)$ & $\alpha_1^k \alpha_{14/3}$ & $h_1^k \cdot \ol{P^2 h_1^2 c_0}$ \\
$(29,1) + k(1,1)$ & $\alpha_1^k \alpha_{15}$ & $h_1^k \cdot \ol{P^3 h_1^4}$ \\
$(35,1) + k(1,1)$ & $\alpha_1^k \alpha_{18/3}$ & $h_1^k \cdot \ol{P^3 h_1^2 c_0}$ \\
$(37,1) + k(1,1)$ & $\alpha_1^k \alpha_{19}$ & $h_1^k \cdot \ol{P^4 h_1^4}$ \\
$(43,1) + k(1,1)$ & $\alpha_1^k \alpha_{22/3}$ & $h_1^k \cdot \ol{P^4 h_1^2 c_0}$ \\
$(45,1) + k(1,1)$ & $\alpha_1^k \alpha_{23}$ & $h_1^k \cdot \ol{P^5 h_1^4}$ \\
$(51,1) + k(1,1)$ & $\alpha_1^k \alpha_{26/3}$ & $h_1^k \cdot \ol{P^5 h_1^2 c_0}$ \\
$(53,1) + k(1,1)$ & $\alpha_1^k \alpha_{27}$ & $h_1^k \cdot \ol{P^6 h_1^4}$ \\
$(59,1) + k(1,1)$ & $\alpha_1^k \alpha_{30/3}$ & $h_1^k \cdot \ol{P^6 h_1^2 c_0}$ \\
\midrule
$(26,2)$ & $z_{26}$ & $\ol{h_1^2 h_4 c_0}$  \\
$(30,2)$ & $z'_{30}$ & $r$  \\
$(34,2)$ & $z_{34,2}$ & $h_2 h_5$ \\
$(35,3)$ & $\alpha_1 z_{34,2}$ & $\ol{h_3^2 g}$ \\
$(36,2)$ & $z_{36}$ & $\ol{h_1^4 h_5}$ \\
$(37,3)$ & $\alpha_1 z_{36}$ & $h_1 \cdot \ol{h_1^4 h_5}$ \\
$(38,2)$ & $z_{38}'$ & $h_0 y$ \\
$(38,4)$ & $\alpha_1^2 z_{36}$ & $h_1^2 \cdot \ol{h_1^4 h_5}$ \\
$(39,5)$ & $\alpha_1^3 z_{36}$ & $h_1^3 \cdot \ol{h_1^4 h_5}$ \\
$(41,3)$ & $z_{41}$ & $h_0 c_2$ \\
$(41,3)$ & $\alpha_{2/2} z'_{38}$ & $\ol{\tau h_0^2 g^2}$ \\
$(42,2)$ & $z_{42}$ & $\ol{h_1^2 h_5 c_0}$ \\
$(42,4)$ & $\alpha_1 z_{41}$ & $h_3 \cdot \ol{h_3^2 g}$ \\
$(43,3)$ & $\alpha_1 z_{42}$ & $h_1 \cdot \ol{h_1^2 h_5 c_0}$ \\
$(43,3)$ & $z_{43,3}$ & $\ol{\tau h_2 c_1 g}$ \\
$(44,2)$ & $z_{44,2}$ & $\ol{\tau^2 h_2 g^2}$ \\
$(44,4)$ & $\alpha_1^2 z_{42}$ & $h_1^2 \cdot \ol{h_1^2 h_5 c_0}$ \\
$(44,4)$ & $z'_{44,4}$ & $\ol{\tau h_0 h_2 g^2}$ \\
$(44,4)$ & $2 z'_{44,4}$ & $d_0 r$ \\
$(45,5)$ & $\alpha_1^3 z_{42}$ & $h_1^3 \cdot \ol{h_1^2 h_5 c_0}$ \\
$(46,6)$ & $\alpha_1^4 z_{42}$ & $h_1^4 \cdot \ol{h_1^2 h_5 c_0}$ \\
$(47,3)$ & $z_{47,3}$ & $\ol{h_1^2 g_2}$ \\
$(47,5)$ & $z_{47,5}$ & $\ol{\tau h_2^2 g^2}$ \\
$(50,2)$ & $z_{50,2}$ & $\ol{P h_1^2 h_5 c_0}$ \\
$(50,6)$ & $\alpha_{2/2} z_{47,5}$ & $g r$ \\
$(54,6)$ & $z_{54,6}$ & $\ol{h_1 d_1 g}$ \\
$(55,5)$ & $z_{55}$ & $B_6$ \\
$(55,7)$ & $\alpha_1 z_{54,6}$ & $\ol{h_1^2 d_1 g}$ \\
$(55,5)$ & ? $z'_{55}$ & $\ol{h_1 i_1} + \tau h_1 G$ \\
$(56,2)$ & $z_{56,2}$ & $Q_1$ \\
$(56,4)$ & $z_{56,4}$ & $\ol{\tau h_2 d_1 g}$ \\
$(56,6)$ & $z_{56,6}$ & $\ol{h_1^2 i_1} + h_5 c_0 e_0$ \\$(57,3)$ & $z_{57}$ & $Q_2$ \\
$(57,7)$ & $\alpha_1 z_{56,6}$ & $h_1 \cdot \ol{h_1^2 i_1} + h_1 h_5 c_0 e_0$ \\
$(58,2)$ & $z_{58,2}$ & $h_0 D_2$ \\
$(58,6)$ & $z'_{58,6}$ & $P h_1^2 h_5 e_0$ \\
$(58,8)$ & $\alpha_1^2 z_{56,6}$ & $h_1^2 \cdot \ol{h_1^2 i_1}+ h_1^2 h_5 c_0 e_0$ \\
$(59,3)$ & $z_{59,3}$ & $\ol{h_3^2 g_2}$ \\
$(59,3)$ & $\alpha_1 z_{58,2}$ & $\ol{h_3 G_3}$ \\
$(59,7)$ & $z_{59,7}$ & $h_1^2 D_4$ \\
$(59,9)$ & $\alpha_1^3 z_{56,6}$ & $h_1^3 \cdot \ol{h_1^2 i_1} + h_1^3 h_5 c_0 e_0$ \\
$(60,2)$ & $z_{60,2}$ & $\ol{h_1^2 Q_2}$ \\
$(60,4)$ & ? $z_{60,4}$ & $\ol{j_1}$ \\
$(60,4)$ & $\alpha_{2/2} z_{57}$ & $h_1 \cdot \ol{h_3 G_3}$ \\
$(60,6)$ & $z_{60,6}$ & $\ol{h_3 d_1 g}$ \\
\end{longtable}